\documentclass[11pt,a4paper]{amsart}
\pdfoutput=1
\usepackage{amssymb}
\usepackage{amsmath}
\usepackage{mathtools}
\usepackage{amsthm}
\usepackage{amsfonts}
\usepackage{bbm} 
\usepackage{dutchcal} 
\usepackage{tikz-cd} 

\usepackage{todonotes} 
\usepackage{graphicx} 
\usepackage[toc,page]{appendix} 
\usepackage{kpfonts} 
\usepackage[mathscr]{euscript} 
\usepackage[margin=2.5cm]{geometry} 
\usepackage{enumerate} 
\usepackage{comment} 
\usepackage{subfiles} 

\usepackage{url}
\usepackage[style = alphabetic,
            backref = true,
            backrefstyle = two]{biblatex}
\renewcommand\mathfrak[1]{\mbox{\usefont{U}{euf}{m}{n}#1}}

\definecolor{winered}{rgb}{0.5,0,0}
\usepackage[citecolor = winered, 
            urlcolor = winered, 
            linkcolor = winered,
            menucolor = winered, 
            colorlinks = true]{hyperref}

\DefineBibliographyStrings{english}{%
  backrefpage = {p.}, 
  backrefpages = {pp.}, 
}

\usepackage[foot]{amsaddr}

\newtheoremstyle{theoremdd}
{\topsep}{\topsep}{\upshape}{0pt}{\bfseries}{.}{ }{\thmname{#1}\thmnumber{ #2}\thmnote{ (#3)}}

\theoremstyle{definition}
\newtheorem{Th}{Theorem}[section]
\newtheorem{Lemma}[Th]{Lemma}
\newtheorem{Cor}[Th]{Corollary}
\newtheorem{Prop}[Th]{Proposition}
\newtheorem{Def}[Th]{Definition}

\newtheorem{Rem}[Th]{Remark}
\newtheorem{Ex}[Th]{Example}

\newtheorem{Not}[Th]{Notation}

\newcommand{\C}{\mathbb{C}}
\newcommand{\R}{\mathbb{R}}
\newcommand{\Z}{\mathbb{Z}}

\newcommand{\N}{\mathbb{N}}
\newcommand{\ncat}{\mathbf} 
\newcommand{\cat}{\mathcal} 
\newcommand{\cons}{\text} 
\renewcommand{\u}{\underline}
\newcommand{\colim}{\text{colim}}
\newcommand{\ncolim}[1]{\underset{#1}{\colim} \,}
\newcommand{\coeq}{\text{coeq}}
\newcommand{\Pre}{\ncat{Pre}}
\newcommand{\Set}{\ncat{Set}}
\newcommand{\Sh}{\ncat{Sh}}
\newcommand{\Obj}{\text{Obj} \, }
\newcommand{\Mor}{\text{Mor} \, }
\newcommand{\op}{\text{op}} 
\newcommand{\Match}{\text{Match}}
\newcommand{\Fam}{\text{Fam}}
\newcommand{\im}{\text{im}}
\newcommand{\coim}{\text{coim}}
\newcommand{\reg}{\text{reg}}
\newcommand{\Sub}{\cons{Sub}}

\newcommand{\sat}[1]{\text{sat}(#1)}

\newcommand{\Gro}[1]{\text{Gro}(#1)}

\newcommand{\Cone}{\text{Cone}}
\newcommand{\fin}{\text{fin}}
\newcommand{\Fun}{\cons{Fun}}
\newcommand{\fp}{\text{fp}}
\newcommand{\Gir}{\text{Gir}}
\newcommand{\open}{\text{open}}
\newcommand{\Man}{\ncat{Man}}

\makeatletter
\newtheorem*{rep@theorem}{\rep@title}
\newcommand{\newreptheorem}[2]{%
\newenvironment{rep#1}[1]{%
 \def\rep@title{#2 \ref{##1}}%
 \begin{rep@theorem}}%
 {\end{rep@theorem}}}
\makeatother

\newreptheorem{theorem}{Theorem}

\newreptheorem{lemma}{Lemma}

\newreptheorem{cor}{Corollary}
\newreptheorem{prop}{Proposition}

\usetikzlibrary{calc}
\usetikzlibrary{decorations.pathmorphing}

\tikzset{curve/.style={settings={#1},to path={(\tikztostart)
    .. controls ($(\tikztostart)!\pv{pos}!(\tikztotarget)!\pv{height}!270:(\tikztotarget)$)
    and ($(\tikztostart)!1-\pv{pos}!(\tikztotarget)!\pv{height}!270:(\tikztotarget)$)
    .. (\tikztotarget)\tikztonodes}},
    settings/.code={\tikzset{quiver/.cd,#1}
        \def\pv##1{\pgfkeysvalueof{/tikz/quiver/##1}}},
    quiver/.cd,pos/.initial=0.35,height/.initial=0}

\tikzset{tail reversed/.code={\pgfsetarrowsstart{tikzcd to}}}
\tikzset{2tail/.code={\pgfsetarrowsstart{Implies[reversed]}}}
\tikzset{2tail reversed/.code={\pgfsetarrowsstart{Implies}}}

\definecolor{emilioeditcolor}{rgb}{0.94, 0.97, 1.0}

\title{Coverages and Grothendieck Toposes}
\author{Emilio Minichiello*}
\address{*CUNY CityTech}
\email{eminichiello67@gmail.com}

\begin{document}

\begin{abstract}
These notes detail the basics of the theory of Grothendieck toposes from the viewpoint of coverages. Typically one defines a site as a (small) category equipped with a Grothendieck topology. However, it is often desirable to generate a Grothendieck topology from a smaller structure, such as a Grothendieck pretopology, but these require some pullbacks to exist in your underlying category. There is an even more light-weight structure one can generate a Grothendieck topology from called a coverage. Coverages don't require any limits or colimits to exist in the underlying category.

We prove in detail several results about coverages, such as closing coverages under refinement and composition, to be what we call a saturated coverage, which doesn't change its category of sheaves. We show that Grothendieck topologies are in bijection with saturated coverages. We give an explicit description of the saturated coverage and the Grothendieck topology generated from a coverage. 

We furthermore give a readable account of some of the most important parts of Grothendieck topos theory, with an emphasis placed on coverages. These include constructing sheafification using the plus construction and also in ``one go,'' the equivalence between left exact localizations of presheaf toposes and saturated coverages, morphisms of sites using the fully general notion of covering flatness, points of a Grothendieck topos and Giraud's theorem. We show that Giraud's theorem is equivalent to Rezk's notion of weak descent. Also included is a section devoted to many examples of sites and Grothendieck toposes appearing in the literature, and appendices covering set theory and category theory background, localization and locally presentable categories.
\end{abstract}
\maketitle

\setcounter{tocdepth}{1} 
\tableofcontents

\newpage
    
\section{Introduction}

The theory of sheaves has become a widely-used and powerful tool in modern mathematics. Sheaves were first developed by Jean Leray in a prison camp during World War II, see the wonderful article \cite{miller2000leray} for more history of the origins of sheaf theory. Their original area of application was algebraic topology, but before long the influence of sheaf theory spread to other areas, especially algebraic geometry. Grothendieck took the theory of sheaves to new heights, inventing the concept of what is now known as a Grothendieck topology. This notion was then central to the definition of \'{e}tale cohomology in algebraic geometry, which was itself key to proving the Weil conjectures.

Today, sheaf theory has expanded horizontally, in the sense that it is being applied to ever more areas of mathematics and science \cite{rosiak2022sheaf, kearney2020sheafnetwork, hansen2020opiniondynamicsdiscoursesheaves, curry2014sheaves}, and vertically, in the sense of new developments in the theory of sheaves of spaces, i.e. higher topos theory \cite{lurie2009higher}.

There are already the wonderful books \cite{johnstone2002sketches, maclane2012sheaves, goldblatt2014topoi} amongst many other resources for the subject of topos theory. So I must justify why I wrote these notes. As a PhD student, I studied diffeological spaces, which are a certain generalization of smooth manifolds. It turns out that the category of diffeological spaces is equivalent to the category of concrete sheaves on the site of open subsets of cartesian spaces (see Example \ref{ex diffeological spaces}). 

I wanted to study these diffeological spaces using the modern machinery of higher topos theory as in \cite{schreiber2013dcct}. However, there was a certain technical difficulty that presented itself. There are many different sites that one can use that give equivalent categories of diffeological spaces (see Example \ref{ex diffeological spaces}), but only one site (Example \ref{ex j good coverage}) had a particular property\footnote{Briefly, many more simplicial presheaves tend to be $\infty$-stacks on $(\ncat{Cart}, j_{\text{good}})$ than on $(\ncat{Man}, j_{\text{emb}})$, see \cite[Remark 3.2.2]{fiorenza2011cechcocyclesdifferentialcharacteristic}, and also the former site has projective cofibrant \v{C}ech nerves for its covering families. This is a very important technical convenience for working with the projective model structure on simplicial presheaves. Furthermore the remarkable results \cite[Lemma 3.3.29, Proposition 3.3.30]{sati2022equivariant}, \cite[Proposition 4.13]{pavlov2022numerable} hold for the former site.} that was necessary for the homotopical theory of higher sheaves. Most of the time in differential geometry, one studies a Grothendieck pretopology (Definition \ref{def grothendieck pretopology}), from which it is easy to generate the resulting Grothendieck topology. But the site I mentioned above, that I needed for homotopical reasons, was not a Grothendieck pretopology. Instead it was merely a coverage. This made me realize that I really needed to better understand the theory of coverages. However, nearly all of the theory of coverages I could find in the literature was covered in three references \cite{johnstone2002sketches, shulman2012exact, low2016categories}.

In order to help myself understand this theory enough to write my PhD thesis, (which I was able to achieve, it consisted of the two papers \cite{Minichiello2024bundles, minichiello2024derham}) I decided to write these notes. It helped me immensely, and I hope it can help you as well. My justification for writing these notes consists of three points: 
\begin{enumerate}
    \item most of the references on topos theory are very long, and as a graduate student I found it quite difficult to navigate these massive texts to find the most important points. These notes are meant to be short, and to be readable by motivated graduate students. Hence I have included as many proofs as I could without extending the text massively, and included as many details in the proofs as I could within reason.
    \item most of the references on topos theory focus extensively on Grothendieck topologies. This is reasonable, as any coverage can be completed to a unique Grothendieck topology. However, in the context of differential geometry, I really needed to be able to manipulate the covering families directly. Doing this with Grothendieck topologies can be a bit cumbersome. I felt that a detailed discussion with careful proofs of closing coverages under desirable properties was missing from the literature.
    \item many of the most impressive applications of topos theory are to logic and algebraic geometry. I felt it was time to make a text that avoided these well-known examples and focus more on lesser-known areas of application, such as in differential geometry. I included an extensive section of examples (Section \ref{section some sites}) that I hope will be useful and interesting.
\end{enumerate}

For these notes I decided to focus particularly on material and subjects that I had the hardest time pinning down. You will also see this reflected in the appendix material, which consists of three subjects that I found vitally important but also somewhat more difficult to get a handle on. My hope is that these notes will serve as a useful base of material for a young researcher or graduate student to refer back to, and pave the way for learning about homotopical sheaf theory.

Now let us briefly detail the different sections of these notes. In Section \ref{section sites and sheaves} we introduce the main objects of study: coverages, presheaves and sheaves. In Section \ref{section coverage closures} we study how to close arbitrary coverages under two closure operations: refinement closure and composition closure. We prove that closing a coverage under these operations does not change its sheaves. In Section \ref{section saturated and grothendieck coverages} we discuss saturated coverages, which are precisely those coverages that are both refinement closed and composition closed. We show that saturated coverages are in bijection with Grothendieck topologies. In Section \ref{section sheafification and lex localizations} we prove the Little Giraud Theorem (Theorem \ref{th lex localizations <-> grothendieck toposes}), that left exact localizations of presheaf toposes are equivalent to Grothendieck toposes. We also give a detailed construction of sheafification using the plus construction applied twice, and then also explain another construction which ``sheafifies in one go.'' In Section \ref{section some sites} we provide a long list of examples of sites, Grothendieck toposes and categories of concrete sheaves that show up in the literature. We also discuss several ways to deal with the underlying site being a large category. In Section \ref{section morphisms of sites}, we discuss morphisms between sites, using the most general possible notion of covering flatness and prove the Comparison Lemma (Theorem \ref{th comparison lemma}). In Section \ref{section points of a site} we apply the theory of morphisms of sites to discuss points of sites and corresponding points of Grothendieck toposes. In Section \ref{section girauds theorem} we prove Giraud's Theorem (Theorem \ref{th Giraud's Theorem}), which characterizes the conditions under which a category is a Grothendieck topos, and we show that this is equivalent to the notion of weak descent. We give a little taste of how this generalizes to the case of $\infty$-toposes. Finally we include three appendices: Appendix \ref{section background} covers some background on set theory, category theory and presheaf toposes, Appendix \ref{section localizations} covers the theory of localizations of categories, and Appendix \ref{section locally presentable categories} covers the theory of locally presentable categories.

\subsection{What's been left out}

The following topics have been left out of these notes. 
\begin{enumerate}
\item There is little to no mention of algebraic geometry in these notes. The reasoning for this is due to my own ignorance on the subject, and because the wonderful resource the Stacks Project \cite{stacksproject} already exists. We highly recommend the reader go there to learn more. 
\item We do not cover \v{C}ech cohomology or sheaf cohomology of sites, you can find that in \cite[\href{https://stacks.math.columbia.edu/tag/01FQ}{Tag 01FQ}]{stacksproject}.
\item We do not cover categorical logic or the use of Lawvere-Tierney topologies in these notes. Both of these are important parts of topos theory, and are covered very well in the following references \cite{maclane2012sheaves, goldblatt2014topoi, johnstone2002sketches, lurie2018categoricallogic, awodey2009introduction}.
\end{enumerate}

\subsection{What the reader should know}

While we collect some helpful prerequisite information in the appendices, in order to keep these notes to a reasonable size, we had to omit certain topics from the appendices and make them assumed knowledge of the reader. Here we list those topics and references for them.
\begin{itemize}
    \item Basic category theory: (co)limits, adjunctions, natural transformations, etc. \cite{riehl2017category},
    \item Coend Calculus \cite{loregian2021co},
    \item Tensoring and Powering \cite[Section 3.7]{riehl2014categorical},
    \item the pasting lemma for pullbacks and pushouts \cite{bauer2012pastinglaw},
    \item Cartesian Closure and Local Cartesian Closure \cite{huang2022lcc}.
\end{itemize}

\subsection{Conventions and Notation}
\begin{itemize}
    \item We use a bold font $\ncat{C}$ to denote named categories like $\ncat{Set}$. We use the font $\cat{C}$ for unnamed categories. 
    \item We take as our set-theoretic foundation a pair $\mathbb{U} \in \mathbb{V}$ of Grothendieck universes, see Section \ref{section set theory}. We call sets that are elements of $\mathbb{U}$ small and elements of $\mathbb{V}$ large.
    \item All categories $\cat{C}$ are understood to be locally small (have small hom-sets) and all sets are understood to be small sets.
    \item We assume the law of excluded middle and the axiom of choice hold.
    \item The reader should be warned that we will not always cite the first reference where a result or definition was given. Rather we will sometimes refer to what we consider to be a readable reference which may then itself give the correct attribution of a result or definition. 
\end{itemize}

\subsection{Corrections}
This is a living document that will be updated periodically. As with any large non-peer reviewed document, it will certainly contain mistakes and typos. Please reach out by email if you find any such mistakes. I will make sure to mention you here.

For help finding errors, thanks to Quique Ruiz and Cheyne Glass.

\subsection*{Acknowledgements}
I want to thank my advisor Mahmoud Zeinalian, for introducing me to the wonderful world of sheaves. I also want to thank Matt Cushman, Cheyne Glass, Severin Bunk and Dmitri Pavlov for many helpful discussions about sheaves.

I also wish to acknowledge that these notes build heavily on the work of Peter Johnstone \cite{johnstone2002sketches}, Zhen Lin Low \cite{low2016categories} and Michael Shulman \cite{shulman2012exact}. Without these three references, these notes could not exist.

\section{Sites and Sheaves} \label{section sites and sheaves}

In this section, we introduce the main objects of study of these notes, sites and sheaves.

\subsection{Coverages}

The theory of sites is usually presented using Grothendieck topologies or pretopologies \cite{maclane2012sheaves}. These are particular kinds of collections of covering morphisms. Here we follow \cite{johnstone2002sketches} in using coverages. Coverages are more general collections of covering morphisms than Grothendieck (pre)topologies. Their definition is simpler, and apply to more examples, but proving theorems using them is typically more difficult.

\begin{Def} \label{def family of morphisms}
Let $\cat{C}$ be a category, and $U \in \cat{C}$. A \textbf{family of morphisms} (or \textbf{family} for short) over $U$ is a set of morphisms $r = \{r_i : U_i \to U \}_{i \in I}$ in $\cat{C}$ with codomain $U$. A \textbf{refinement} of a family of morphisms $t = \{t_j : V_j \to U \}_{j \in J}$ over $U$ consists of a family of morphisms $r = \{r_i : U_i \to U \}_{i \in I}$, a function $\alpha: I \to J$, which we call the \textbf{index map} of the refinement, and for each $i \in I$ a map $f_i: U_i \to V_{\alpha(i)}$, which we call the $i$th \textbf{component} of $f$, making the following diagram commute:
\begin{equation*}
    \begin{tikzcd}
	{U_i} && {V_{\alpha(i)}} \\
	& U
	\arrow["{f_i}", from=1-1, to=1-3]
	\arrow["{r_i}"', from=1-1, to=2-2]
	\arrow["{t_{\alpha(i)}}", from=1-3, to=2-2]
\end{tikzcd}
\end{equation*}
If $r$ is a refinement of $t$, with maps $f_i: U_i \to V_{\alpha(i)}$, then we write $f : r \to t$. Let $\Fam(U)$ denote the category whose objects are families over $U$ and whose morphisms are refinements. We say that $r$ refines $t$, and write $r \leq t$ if there exists a refinement $f : r \to t$.
\end{Def}

\begin{Not} \label{not canonical families}
Given any object $U \in \cat{C}$, there always exist several families over $U$. Let $\varnothing_U$ denote the \textbf{empty family} over $U$, and let $y_U$ denote the \textbf{maximal family} over $U$, which consists of all those morphisms with codomain $U$. If $f: V \to U$ is a morphism, then let $(f)$ denote the corresponding singleton family $(f) = \{ f : V \to U \}$.
\end{Not}

\begin{Def} \label{def pushforward of a family}
If $g : V \to U$ is map in $\cat{C}$ and $t = \{t_j :  V_j \to V \}_{j \in J}$ is a family on $V$, then let 
$$g_*(t) = \{ V_j \xrightarrow{t_j} V \xrightarrow{g} U \}_{j \in J}$$
denote the \textbf{pushforward} of $t$ by $g$. This makes $\Fam$ into a (strict) functor $\Fam_*: \cat{C} \to \ncat{Cat}$ by setting $\Fam_*(g) = g_*$ and that sends a refinement $f : t \to s$ to the refinement $g_*(f) : g_*(t) \to g_*(s)$ whose components agree with $f$.
We can also pullback families in the following way. Given $g$ as above and $r = \{ r_i : U_i \to U \}$ a family over $U$, let
\begin{equation*}
    g^*(r) = \{ f: W \to V \, \mid \, gf \text{ factors through a map $r_i : U_i \to U$ } \}.
\end{equation*}
This similarly makes $\Fam$ into a (strict) functor $\Fam^* : \cat{C}^{\text{op}} \to \ncat{Cat}$ by setting $\Fam^*(g) = g^*$\footnote{In fact, if $r$ is a family over $U$, then $g^*(r)$ is a sieve (Definition \ref{def sieve}) over $V$. If $h : r \to t$ is a refinement, then $g^*(r) \subseteq g^*(t)$.}.
\end{Def}

\begin{Rem}
In what follows we will often drop the indexing set from a family $r = \{r_i : U_i \to U \}_{i \in I}$, referencing it by $r$ or $\{r_i : U_i \to U \}$ when we wish to make the morphism names and domains known.
\end{Rem}

\begin{Def} \label{def collection of families}
A \textbf{collection of families} $j$ on a small category $\cat{C}$ consists of a set $j(U)$ for each $U \in \cat{C}$, whose elements $r \in j(U)$ are families of morphisms over $U$. We write $r \in j$ to mean that there exists some $U \in \cat{C}$ such that $r \in j(U)$.
\end{Def}

\begin{Def} \label{def coverage}
We say that a collection of families $j$ on a small category $\cat{C}$ is a \textbf{coverage} if
\begin{itemize}
    \item for every $U \in \cat{C}$, $(1_U) \in j(U)$, and\footnote{There are some versions of the definition of coverage in the literature that do not require this axiom, such as \cite[Definition A.2.1.9]{johnstone2002sketches}, and there are some that do, such as \cite[Definition A.2.8]{low2016categories}. We include it for several technical reasons that will become more apparent as we go along.}
    \item for every $U \in \cat{C}$, $r \in j(U)$ and map $g : V \to U$ in $\cat{C}$, there exists a family $t \in j(V)$ such that $g_*(t) \leq r$.
\end{itemize} 
If $j$ is a coverage on $\cat{C}$, then we call families $r \in j(U)$ \textbf{covering families} over $U$. If a map $r_i: U_i \to U$ belongs to a covering family $r \in j(U)$, then we say that $r_i$ is a \textbf{covering map}. If $\cat{C}$ is a small category, and $j$ is a coverage on $\cat{C}$, then we call the pair $(\cat{C}, j)$ a \textbf{site}. We also think of $j(U)$ as a poset, where $r \leq r'$ in $j(U)$ if both $r$ and $r'$ are covering families over $U$ and there is a refinement $f : r \to r'$.
\end{Def}

\begin{Rem}
Unravelling Definition \ref{def coverage} means that for every $U \in \cat{C}$, $r \in j(U)$, and $g : V \to U$, there exists a $t \in j(V)$ such that for every $t_j : V_j \to V$ in $t$ there exists a map $r_i : U_i \to U$ in $r$ and a map $s_j: V_j \to U_i$ in $\cat{C}$ making the following diagram commute:
\begin{equation*} \label{eqn coverage def}
    \begin{tikzcd} 
	{V_j} & {U_i} \\
	V & U
	\arrow["{t_j}"', from=1-1, to=2-1]
	\arrow["{s_j}", from=1-1, to=1-2]
	\arrow["{r_i}", from=1-2, to=2-2]
	\arrow["g"', from=2-1, to=2-2]
\end{tikzcd}
\end{equation*}
\end{Rem}

\begin{Ex} \label{ex set joint epi coverage}
Let $j_{\text{epi}}$ denote the collection of families on $\ncat{FinSet}$\footnote{here we let $\ncat{FinSet}$ denote a small skeleton of the category of finite sets. Explicitly, let $\ncat{FinSet}$ denote the category whose objects are the sets $[n] = \{1, \dots, n \}$ with $[0] =\varnothing$ and whose morphisms are functions.} where $r = \{ r_i : S_i \to S \}_{i \in I}$ belongs to $j_{\text{epi}}(S)$ if and only if $r$ is \textbf{jointly epimorphic}, i.e. that $\sum_{i \in I} r_i : \sum_{i \in I} S_i \to S$ is an epimorphism. Given a map $g : T \to S$ of sets, pulling back each $r_i$ gives a function
\begin{equation*}
    \begin{tikzcd}
	{g^*(S_i)} & {S_i} \\
	T & S
	\arrow[from=1-1, to=1-2]
	\arrow["{g^*(r_i)}"', from=1-1, to=2-1]
	\arrow["\lrcorner"{anchor=center, pos=0.125}, draw=none, from=1-1, to=2-2]
	\arrow["{r_i}", from=1-2, to=2-2]
	\arrow["g"', from=2-1, to=2-2]
\end{tikzcd}
\end{equation*}
and $\sum_{i \in I} g^*(r_i)$ is an epimorphism. Indeed, if $t \in T$, then there exists an $i \in I$ and an $s_i \in S_i$ such that $r_i(s_i) = g(t)$. Then $\sum_{i \in I} g^*(r_i)(t, s_i) = t$. Thus $j_{\text{epi}}$ forms a coverage on $\ncat{FinSet}$.
\end{Ex}

\begin{Ex} \label{ex open cover coverage}
Let $X$ be a topological space and let $\mathcal{O}(X)$ denote the poset of open subsets of $X$, where namely $U \leq V$ if and only if $U$ and $V$ are open subsets of $X$ and $U \subseteq V$. Let $j_X$ denote the collection of families on $\mathcal{O}(X)$ such that $j_X(U)$ is the set of all open covers of an open subset $U$, namely $\{ U_i \subseteq U \}_{i \in I} \in j_X(U)$ if $\bigcup_i U_i = U$. This collection of families is a coverage. Indeed, suppose that $\{U_i \subseteq U \}$ is an open cover of $U$ and $V \leq U$ is an open subset of $U$. Then $\{ V \cap U_i \subseteq V \}$ is an open cover of $V$, and $V \cap U_i \leq U_i$ for every $i \in I$. We call $j_X$ the \textbf{open cover coverage} of $X$, and call $(\mathcal{O}(X), j_X)$ the \textbf{site associated with $X$}.
\end{Ex}

\begin{Ex} \label{ex basis for a topology as a coverage}
Now suppose that $X$ is a topological space with topology $\tau$, and let $\mathscr{B}$ be a basis for $\tau$ in the classical sense, namely $\mathscr{B} \subseteq \tau$ is a subcollection of open sets such that the elements of $\mathscr{B}$ cover $X$ and for every open set $U$ and elements $x \in U$, there exists a basis element $B_x \in \mathscr{B}$ such that $x \in B_x$ and $B_x \subseteq U$. Let us show that $\mathscr{B}$ defines a coverage $j_{\mathscr{B}}$ on $\mathcal{O}(X)$. If $U \in \mathcal{O}(X)$, then let $j_{\mathscr{B}}(U)$ denote those open covers of $U$ consisting only of elements in $\mathscr{B}$ (and also containing the identity family $(1_U)$). So if $V \subseteq U$ is an open subset, and $\{ B_i \subseteq U \} \in j_{\mathscr{B}}(U)$, then every point $y \in V$ belongs to some $B_i$. So for every $y \in V$ there exists a basis element $B'_y \in \mathscr{B}$ such that $y \in B'_y$ and $B'_y \subseteq V \cap B_i$. Then $\{ B'_y \}_{y \in V}$ forms an open cover of $V$. Thus $j_{\mathscr{B}}$ is a coverage on $\mathcal{O}(X)$. In fact, if we let $\mathcal{O}(X, \mathscr{B})$ denote the full subcategory of $\mathcal{O}(X)$ only consisting of basis elements, then $j_{\mathscr{B}}$ also defines a coverage on this category by the same argument.
\end{Ex}

\begin{Ex} \label{ex canonical coverages}
Given a small category $\cat{C}$, there are several canonical coverages one can consider, recall Notation \ref{not canonical families}:
\begin{enumerate}
    \item let $j_{\text{triv}}$ denote the \textbf{trivial coverage}, where for every $U \in \cat{C}$, $j_{\text{triv}}(U) = \{ (1_U) \}$,
    \item let $j_{\text{iso}}$ denote the \textbf{isomorphism coverage}, where for every $U \in \cat{C}$, $j_{\text{iso}}(U)$ consists of those singleton families $(f)$ over $U$ where $f$ is an isomorphism,
    \item let $j_{\text{max}}$ denote the \textbf{maximal coverage}, where for every $U \in \cat{C}$, $j_{\text{max}}(U) = \{(1_U), y_U \}$.
\end{enumerate}
\end{Ex}

\begin{Ex} \label{ex slice site}
Suppose that $(\cat{C}, j)$ is a site with $U \in \cat{C}$. Let $\cat{C}_{/U}$ denote the slice category, and $\pi_{/U} : \cat{C}_{/U} \to \cat{C}$ denote the projection functor. Let $j_{/U}$ denote the collection of families on $\cat{C}_{/U}$ where if $g : V \to U$ is a morphism in $\cat{C}$, then $\{ f_i : r_i \to g \} \in j_{/U}(g)$ if and only if $\{\pi_{/U}(f_i)\} \in j(V)$. It is easy to see that $j_{/U}$ forms a coverage on $\cat{C}_{/U}$. We call $(\cat{C}_{/U}, j_{/U})$ the \textbf{slice site} at $U$.
\end{Ex}

\begin{Ex} \label{ex induced site}
Given a site $(\cat{C}, j)$ and a  subcategory $\cat{D} \hookrightarrow \cat{C}$, we say that a coverage $j'$ on $\cat{D}$ is \textbf{induced} if every covering family $r'$ in $j'$ is of the form $r' = r \cap \Mor(\cat{D})$. If $j'$ is induced then we write $j' = j|_{\cat{D}}$, and call it the induced coverage.
\end{Ex}


\subsection{(Pre)Sheaves}

\begin{Def} \label{def presheaf, section, matching family, amalgamation}
A \textbf{presheaf} on a category $\cat{C}$ is a functor $X: \cat{C}^{\op} \to \ncat{Set}$. An element $x \in X(U)$ for an object $U \in \cat{C}$ is called a \textbf{section} over $U$. If $f: U \to V$ is a map in $\cat{C}$, and $x \in X(V)$, then we sometimes denote $X(f)(x)$ by $x|_U$. Given a family $r = \{r_i : U_i \to U \}_{i \in I}$ over $U$, an \textbf{intersection square} on $r$ is a commutative diagram of the form
\begin{equation} \label{eq intersection square}
\begin{tikzcd}
	U_{ij} & {U_j} \\
	{U_i} & U
	\arrow["{r_i}"', from=2-1, to=2-2]
	\arrow["{r_j}", from=1-2, to=2-2]
	\arrow["u_i"', from=1-1, to=2-1]
	\arrow["u_j", from=1-1, to=1-2]
\end{tikzcd}
\end{equation}
where $r_i, r_j \in r$ and $u_i$, $u_j$, and $U_{ij}$ are arbitrary. An \textbf{$X$-matching family} over $r$ is a collection of sections $\{x_i \in X(U_i) \}_{i \in I}$ such that given any intersection square (\ref{eq intersection square}), $X(u_i)(x_i) = X(u_j)(x_j)$ for all $i,j \in I$. If the presheaf is clear from context we may say that $\{ x_i \}$ is a matching family for $r$.

If $X$ is a presheaf on $\cat{C}$, and $r$ is a family of morphisms on $U$, then let $\text{Match}(r,X)$ denote the set of $X$-matching families over $r$. Given a $X$-matching family $\{ x_i \}$ over $r$, an \textbf{amalgamation} for $\{ x_i \}$ is a section $x \in X(U)$ such that $X(r_i)(x) = x_i$ for all $i$.

If $U \in \cat{C}$, then let $y(U)$ denote the presheaf on $\cat{C}$ where for $V \in \cat{C}$, $y(U)(V) = \cat{C}(V,U)$, and if $f : V \to V'$ is a morphism in $\cat{C}$, then $y(U)(f) : y(U)(V') \to y(U)(V)$ is the precomposition map. We call $y(U)$ a \textbf{representable presheaf}, and more specifically the representable on $U$.
\end{Def}

\begin{Lemma}\label{lem pullback of matching family by refinement is a matching family}
Given a presheaf $X$ on a category $\cat{C}$, suppose that $r = \{r_i: U_i \to U \}_{i \in I}$ and $t = \{t_j: V_j \to U \}_{j \in J}$ are families over $U$ and $f : r \to t$ is a refinement, with index map $\alpha: I \to J$. If $\{ x_j \}_{j \in J}$ is an $X$-matching family over $t$, then $\{ X(f_i)(x_{\alpha(i)}) \}_{i \in I}$ is an $X$-matching family over $r$.
\end{Lemma}

\begin{proof}
Suppose we have morphisms $u_i: U_{ij} \to U_i$ and $u_j: U_{ij} \to U_j$ such that $r_i u_i = r_j u_j$. Then we have a commuting diagram:
\[\begin{tikzcd}
	{U_i} & {U_{ij}} & {U_j} \\
	{V_{\alpha(i)}} & U & {V_{\alpha(j)}}
	\arrow["{u_i}"', from=1-2, to=1-1]
	\arrow["{u_j}", from=1-2, to=1-3]
	\arrow["{f_i}"', from=1-1, to=2-1]
	\arrow["{f_j}", from=1-3, to=2-3]
	\arrow["{r_i}"{description}, from=1-1, to=2-2]
	\arrow["{r_j}"{description}, from=1-3, to=2-2]
	\arrow["{t_{\alpha(i)}}"', from=2-1, to=2-2]
	\arrow["{t_{\alpha(j)}}", from=2-3, to=2-2]
\end{tikzcd}\]
and therefore
$$X(u_i)X(f_i)(x_{\alpha(i)}) = X(f_i u_i)(x_{\alpha(i)}) = X(f_j u_j)(x_{\alpha(j)}) = X(u_j)X(f_j)(x_{\alpha(j)}),$$
since $\{x_j \}$ is a matching family for $t$. Thus $\{ X(f_i)(x_{\alpha(i)}) \}$ is a matching family for $r$.
\end{proof}

\begin{Rem}
If $f: r \to t$ is a refinement of families of morphisms and $\{ x_j \}$ is a matching family for a presheaf $X$ over $t$, let $f^* \{ x_j \}$ denote the corresponding matching family $\{X(f_i)(x_{\alpha(i)})\}$. Given $U \in \cat{C}$ and a presheaf $X$ on $\cat{C}$, this makes $\Match(-,X)$ into a functor $\Match(-,X) : \Fam(U)^\op \to \ncat{Set}$ by setting $\Match(f)(\{ x_j \}) = f^*\{x_j\}$. We call $f^* \{ x_j \}$ the \textbf{pullback matching family} of $\{ x_j \}$ by $f$.
\end{Rem}

\begin{Ex}
Let us consider the example of the site $(\mathcal{O}(\R^1), j_{\R^1})$ of the real line. The family $\mathcal{U}_1 = \{ (n, n+2) \subseteq \R^1 \}_{n \in \Z}$ forms an open cover of $\R^1$. The family $\mathcal{U}_2 = \{ (m, m + 3) \subseteq \R^1 \}_{m \in \Z}$ forms another open cover of $\R^1$, and there is a refinement $f: \mathcal{U}_1 \to \mathcal{U}_2$ with index map $\alpha : \mathbb{N} \to \mathbb{N}$ given by the identity map. If $X = y(\R)$ is the representable presheaf on $\R$, and $\{x_m \}$ is an $X$-matching family of continuous functions $x_m: (m, n+3) \to \R$ for $\mathcal{U}_2$, then we get the pullback matching family $f^*\{x_m \} = \{X(f_n)(x_{\alpha(n)})\}_{n \in \Z} = \{ x_{n}|_{(n,n+2)} \}_{n \in \Z}$ on $\mathcal{U}_1$. 
\end{Ex}

\begin{Def} \label{def separated and sheaf}
Given a presheaf $X$ on a category $\cat{C}$ and a family of morphisms $r$ on an object $U$ in $\cat{C}$, there is a canonical map
\begin{equation} \label{eq restriction map of presheaf to matching families}
    \text{res}_{r,X}: X(U) \to \text{Match}(r,X)
\end{equation}
which is defined for an element $x \in X(U)$ to be the matching family $\text{res}_{r,X}(x) = \{ X(r_i)(x) \}$ of $X$ over $r$. We say that $X$ is \textbf{separated} on $r$ if $\text{res}_{r,X}$ is injective. We say that $X$ is a \textbf{sheaf} on $r$ if $\text{res}_{r,X}$ is bijective. 

Given a collection of families $j$ on $\cat{C}$, we say that $X$ is (separated) a $j$-sheaf on an object $U \in \cat{C}$ if $X$ is (separated) a sheaf on $r$ for every $r \in j(U)$.

Given a site $(\cat{C}, j)$, we call $X$ a $j$-\textbf{sheaf} if it is a $j$-sheaf on every object $U$ of $\cat{C}$. Let $\ncat{Sh}(\cat{C}, j)$ denote the full subcategory of $\ncat{Pre}(\cat{C})$ whose objects are $j$-sheaves.
\end{Def}

\begin{Def} \label{def grothendieck topos}
We say that a category $\cat{D}$ is a \textbf{Grothendieck topos} or just \textbf{topos} for short, if there exists a site $(\cat{C}, j)$ and an equivalence $\cat{D} \simeq \Sh(\cat{C}, j)$.
\end{Def}

\begin{Lemma}
Given a category $\cat{C}$, a presheaf $X$ on $\cat{C}$, a refinement $f: r \to t$ of families over an object $U \in \cat{C}$, and an $X$-matching family $\{ x_j \}$ over $t$, if $x$ is an amalgamation for $\{x_j \}$, then $x$ is an amalgamation for $f^* \{x_j\}$.
\end{Lemma}

\begin{proof}
Since $f : r \to t$ is a refinement, $X(r_i)(x) = X(t_{\alpha(i)} f_i)(x) = X(f_i)X(t_\alpha(i))(x)$. But $\{x_j \}$ is an $X$-matching family over $t$, so $X(t_{\alpha(i)})(x) = x_{\alpha(i)}$ for all $i$. Thus $X(r_i)(x) = X(f_i)(x_{\alpha(i)})$, so $x$ is an amalgamation for $f^*\{x_j \}$.
\end{proof}

\begin{Lemma} \label{lem sheaves on empty covers}
Given a category $\cat{C}$ and a presheaf $X$ on $\cat{C}$, then $X$ is a sheaf on the empty family $\varnothing$ over an object $U \in \cat{C}$ if and only if $X(U) \cong *$.
\end{Lemma}

\begin{proof}
Given a family $r$ of morphisms over an object $U \in \cat{C}$, $X$ is a sheaf on $r$ if $X(U) \to \Match(r, X)$ is an isomorphism. When $r = \varnothing$, there exists only a single matching family, namely the empty one.
\end{proof}

\begin{Def} \label{def singular objects in a site}
We say that an object $U$ in a site $(\cat{C}, j)$ is \textbf{singular} if the empty family 
 of morphisms $\varnothing$ is a $j$-covering family on $U$. We say that a site $(\cat{C}, j)$ is \textbf{nonsingular} if it has no singular objects, and \textbf{singular} otherwise.
\end{Def}

\begin{Ex}
Given a topological space $X$, the site $(\mathcal{O}(X), j_X)$ from Example \ref{ex open cover coverage} is singular, because the initial object $\varnothing$ in $\mathcal{O}(X)$ has exactly two covering families, the set $\{ 1_\varnothing : \varnothing \to \varnothing \}$ and the empty family $\varnothing$. Thus in order for a presheaf $A$ on $X$ to be a sheaf, $A(\varnothing) \cong *$.
\end{Ex}

\begin{Lemma} \label{lem always a sheaf on an isomorphism}
Given an object $U$ in a category $\cat{C}$ and a nonempty family of morphisms $r$ over $U$ consisting of isomorphisms, every presheaf $X$ on $\cat{C}$ is a sheaf on $r$.
\end{Lemma}

\begin{proof}
Let $r = \{r_i : U_i \to U \}_{i \in I}$ be a family of isomorphisms over $U$, $X$ a presheaf on $\cat{C}$ and suppose that $\{ x_i \}$ is an $X$-matching family over $r$. We want to show that $\{ x_i \}$ has a unique amalgamation.

First we note that $X(r_i^{-1})(x_i) = X(r_j^{-1})(x_j)$ for all $i,j \in I$, because $\{x_i \}$ is an $X$-matching family over $r$ and the following commutative diagram is an intersection square
\begin{equation*}
\begin{tikzcd}
	U & {U_j} \\
	{U_i} & U
	\arrow["{r_j^{-1}}", from=1-1, to=1-2]
	\arrow["{r_i^{-1}}"', from=1-1, to=2-1]
	\arrow["{r_j}", from=1-2, to=2-2]
	\arrow["{r_i}"', from=2-1, to=2-2]
\end{tikzcd}    
\end{equation*}
Letting $x = X(r_i^{-1})(x_i)$, we see that $x$ is an amalgamation of $\{x_i \}$.

Now suppose that $y \in X(U)$ is also an amalgamation of $\{x_i \}$. Then $X(r_i)(y) = x_i$ for all $i \in I$, so $X(r_i^{-1})X(r_i)(y) = X(r_i^{-1})(x_i) = x$. Therefore $y = x$. Thus $x$ is a unique amalgamation, and so $X$ is a sheaf on $r$.
\end{proof}

\begin{Ex}
 If $\cat{C}$ is a category, then we can consider the empty, identity and isomorphism coverages of Example \ref{ex canonical coverages} on $\cat{C}$. Lemma \ref{lem always a sheaf on an isomorphism} proves that every presheaf on $(\cat{C}, j_\text{iso})$ and $(\cat{C}, j_{\text{id}})$ is a sheaf. It is similarly easy to see that every presheaf on $(\cat{C}, j_{\text{max}})$ is a sheaf. Therefore we have
 \begin{equation}
     \Pre(\cat{C}) = \Sh(\cat{C}, j_{\text{iso}}) = \Sh(\cat{C}, j_{\text{id}}) = \Sh(\cat{C}, j_{\text{max}}).
\end{equation}
\end{Ex}

\begin{Not}
Since every category of presheaves $\ncat{Pre}(\cat{C})$ is in particular a category of sheaves $\ncat{Sh}(\cat{C}, j_\text{id})$, every presheaf category is a Grothendieck topos. We will refer to these kinds of toposes as \textbf{presheaf toposes}, and if $(\cat{C}, j)$ is a site, then we will sometimes call $\ncat{Sh}(\cat{C}, j)$ its \textbf{sheaf topos}.   
\end{Not}

\begin{Lemma} \label{lem coverage if pullbacks exist}
Suppose that $j$ is a coverage on a category $\cat{C}$ such that pullbacks along covering maps exist. If $r = \{ r_i: U_i \to U \}_{i \in I}$ is a covering family of $U$, then a collection $\{s_i \in X(U_i) \}$ of sections of a presheaf $X$ on $\cat{C}$ is a matching family for $r$ if and only if for every pullback square of the form
\begin{equation*}
\begin{tikzcd}
	{U_i \times_U U_j} & {U_j} \\
	{U_i} & U
	\arrow["{r_i}"', from=2-1, to=2-2]
	\arrow["{r_j}", from=1-2, to=2-2]
	\arrow["{\pi_i}"', from=1-1, to=2-1]
	\arrow["{\pi_j}", from=1-1, to=1-2]
\end{tikzcd}
\end{equation*}
it follows that $X(\pi_i)(s_i) = X(\pi_j)(s_j)$.
\end{Lemma}

\begin{proof}
$(\Rightarrow)$ This is clear.
 
$(\Leftarrow)$ Suppose we have maps $g: V \to U_i$ and $h : V \to U_j$, then we have the following commutative diagram:
\begin{equation*}
    \begin{tikzcd}
	V \\
	& {U_i \times_U U_j} & {U_j} \\
	& {U_i} & U
	\arrow["{\pi_i}"', from=2-2, to=3-2]
	\arrow["{\pi_j}", from=2-2, to=2-3]
	\arrow["{r_i}"', from=3-2, to=3-3]
	\arrow["{r_j}", from=2-3, to=3-3]
	\arrow["g"', curve={height=12pt}, from=1-1, to=3-2]
	\arrow["h", curve={height=-12pt}, from=1-1, to=2-3]
	\arrow["k"{description}, from=1-1, to=2-2]
\end{tikzcd}
\end{equation*}
where $k$ is the unique map to the pullback. The following holds:
$$ X(g)(s_i) = X(\pi_i k)(s_i) = X(k) X(\pi_i)(s_i) = X(k) X(\pi_j)(s_j) = X(\pi_j k)(s_j) = X(h)(s_j).$$
\end{proof}

\begin{Ex} \label{ex pullbacks in topological sites}
If $X$ is a topological space, then $\mathcal{O}(X)$ has pullbacks where if $U_i \subseteq U$ and $U_j \subseteq U$ are morphisms, then their pullback is given by their intersection $U_i \cap U_j$. Hence by Lemma \ref{lem coverage if pullbacks exist}, to check that a family of sections is a matching family on $\mathcal{O}(X)$ it is enough to check that they match on double intersections.
\end{Ex}

It is useful to compare different coverages on the same underlying category. 

\begin{Def} \label{def coverage comparisons}
Given a small category $\cat{C}$ and coverages $j, j'$ on $\cat{C}$, we write
\begin{itemize}
    \item $j \subseteq j'$, if for every $U \in \cat{C}$ it follows that $j(U) \subseteq j'(U)$,
    \item $j \leq j'$, if for every $r' \in j'$ there exists a $r \in j$ and a refinement $r \leq r'$,
\end{itemize}
We say that $j'$ \textbf{contains} $j$ if $j \subseteq j'$, and that $j$ \textbf{refines} $j'$ if $j \leq j'$. We say two coverages $j, j'$ on the same underlying small category $\cat{C}$ are \textbf{equivalent} if $j \leq j'$ and $j' \leq j$.
\end{Def}

\begin{Lemma} \label{lem coverage comparison results}
Given coverages $j, j'$ on a small category $\cat{C}$, 
\begin{enumerate}
    \item if $j \leq j'$, then $\Sh(\cat{C}, j) \subseteq \Sh(\cat{C}, j')$.
    \item if $j \subseteq j'$, then $\Sh(\cat{C}, j') \subseteq \Sh(\cat{C}, j)$,
\end{enumerate}
\end{Lemma}

\begin{proof}
(1) follows from Lemma \ref{lem sheaf on refinement}, and (2) follows because if $j \subseteq j'$, then $j' \leq j$ by just taking the identity refinement for every covering family.
\end{proof}

\begin{Ex} \label{ex bases give same sheaves as topology}
Suppose that $(X, \tau)$ is a topological space with topology $\tau$ and $\mathscr{B}$ is a basis for $\tau$. We can consider the two coverages $j_X$ (Example \ref{ex open cover coverage}) and $j_\mathscr{B}$ (Example \ref{ex basis for a topology as a coverage}) on $\mathcal{O}(X)$. Clearly $j_{\mathscr{B}} \subseteq j_X$, and thus $\Sh(\mathcal{O}(X), j_X) \subseteq \Sh(\mathcal{O}(X), j_\mathscr{B})$. Now let us show that $j_{\mathscr{B}} \leq j_X$. Suppose that $U$ is an open subset of $X$ and $\mathcal{U} = \{U_i \subseteq U \}$ is an open cover. Then every point $x \in U$ belongs to some $U_i$, and therefore there exists a basis element $B_x \in \mathscr{B}$ such that $x \in B_x$ and $B_x \subseteq U_i$. Thus $B = \{ B_x \}_{x \in U}$ forms an open cover of $U$ and $B \leq \mathcal{U}$. Thus $j_X$ and $j_{\mathscr{B}}$ are equivalent, and hence $\Sh(\mathcal{O}(X), j_X) = \Sh(\mathcal{O}(X), j_{\mathscr{B}})$.
\end{Ex}

\section{Sieves}

\begin{Def} \label{def sieve}
Given a category $\cat{C}$ with $U \in \cat{C}$, a \textbf{sieve} $R$ over $U$ is a family of morphisms over $U$ that is closed under precomposition, namely if $r_i : U_i \to U \in R$, then for any morphism $g: V \to U_i$, the composite $r_i g : V \to U$ is also in $R$. Let $\cons{Sieve}(U)$ denote the full subcategory of $\cons{Fam}(U)$ on the sieves over $U$. Note that if $R$ and $T$ are sieves, then $R \leq T$ if and only if $R \subseteq T$. Let $\cons{sieve}(U)$ denote the poset truncation of $\cons{Sieve}(U)$. 
\end{Def}

\begin{Lemma} \label{lem sieves and subfunctors iso}
Given a category $\cat{C}$ with $U \in \cat{C}$, there is an isomorphism of posets
\begin{equation}
    \cons{sieve}(U) \cong \cons{Sub}(y(U)),
\end{equation}
where $\cons{Sub}(y(U))$ denotes the poset of subobjects $R \hookrightarrow y(U)$ in $\Pre(\cat{C})$.
\end{Lemma}

\begin{proof}
If $R \hookrightarrow y(U)$ is a subobject, then consider $\phi(R) = \bigcup_{V \in \mathscr{C}} R(V)$. We want to show that this is a sieve on $U$, namely that it is closed under precomposition. So suppose that $f: V \to U \in R(V)$ and $g: W \to V$ is a map in $\cat{C}$. Since $R$ is a subfunctor, that means that the following diagram commutes:
\begin{equation*}
    \begin{tikzcd}
	{R(V)} & {R(W)} \\
	{\cat{C}(V,U)} & {\cat{C}(W,U)}
	\arrow[hook, from=1-1, to=2-1]
	\arrow[hook, from=1-2, to=2-2]
	\arrow["{g^*}"', from=2-1, to=2-2]
	\arrow["{R(g)}", from=1-1, to=1-2]
\end{tikzcd}
\end{equation*}
Thus $\phi(R)$ is a closed under precomposition. Furthermore if $R \hookrightarrow T \hookrightarrow y(U)$, then $\phi(R) \subseteq \phi(T)$. So $\phi$ defines a map $\phi: \cons{Sub}(y(U)) \to \cons{sieve}(U)$ of posets.

Conversely, if $R = \{f: V \to U \}$ is a sieve on $U$, then let $\psi(R)$ be the presheaf defined by setting $\psi(R)(V)$ to be the set of maps in $R$ with domain $V$. It is easy to check that this is a subfunctor $\psi(R) \hookrightarrow y(U)$ and that if $R \subseteq T$ where $T$ is a sieve, then $\psi(R) \hookrightarrow \psi(T) \hookrightarrow y(U)$. 

It is not hard to show that these constructions are inverse to each other, and thus define an isomorphism of posets.
\end{proof}

\begin{Not} \label{not sieves} 
Thanks to the above result, when we refer to a sieve, we leave ambiguous whether we are thinking of it as a family of morphisms closed under precomposition or as a subfunctor. 
\end{Not}

\begin{Def} \label{def sieve generated by a family of morphisms}
Given a category $\cat{C}$, with $U \in \cat{C}$, and a family of morphisms $r = \{ r_i : U_i \to U \}$ over $U$, let $\overline{r}$ denote the set of morphisms $f : V \to U$ such that $f$ factors through some $r_i \in r$:
\begin{equation*}
    \begin{tikzcd}
	V && U \\
	& {U_i}
	\arrow["f", from=1-1, to=1-3]
	\arrow["{s_f}"', from=1-1, to=2-2]
	\arrow["{r_i}"', from=2-2, to=1-3]
\end{tikzcd}
\end{equation*}
We say that $R$ is the sieve \textbf{generated} by $r$ if $R = \overline{r}$, and that $r$ is a \textbf{generating family} for $R$. It is easy to see that $R$ is then the smallest sieve such that $r \subseteq R$. 
\end{Def}

\begin{Lemma} \label{lem refinement under sifted closure}
Given two families of morphisms $r,r'$ on an object $U$ in a small category $\cat{C}$, then $r \leq r'$ if and only if $\overline{r} \subseteq \overline{r'}$.
\end{Lemma}

\begin{proof}
Let $r = \{r_i : U_i \to U \}_{i \in I}$ and $r' = \{r'_j : U'_j \to U \}_{j \in J}$. 
$(\Rightarrow)$ If $r \leq r'$, then for every $i \in I$, there is a commutative diagram
\begin{equation*}
    \begin{tikzcd}
	{U_i} && {U_j'} \\
	& U
	\arrow[from=1-1, to=1-3]
	\arrow["{r_i}"', from=1-1, to=2-2]
	\arrow["{r'_j}", from=1-3, to=2-2]
\end{tikzcd}
\end{equation*}
for some $j \in J$. Hence $r \subseteq \overline{r}'$. But this implies that $\overline{r} \subseteq \overline{r}'$.

$(\Leftarrow)$ Suppose that $\overline{r} \subseteq \overline{r'}$. Then $r \subseteq \overline{r'}$, which is basically the definition of $r \leq r'$.
\end{proof}

\begin{Prop}[{\cite[Lemma 3.5]{nlab:sieve}}] \label{prop sieves are coequalizers of generating family}
Given a category $\cat{C}$ with $X \in \cat{C}$ and a sieve $R$ over $X$, where $r = \{r_i : U_i \to X\}$ is a generating family for $R$, let $U = \sum_{i \in I} y(U_i)$. Then there is an isomorphism of presheaves:
\begin{equation} \label{eq sieves as colimits of generators}
    R \cong \coeq \left( U \times_{y(X)} U \rightrightarrows U \right)
\end{equation}
where the right hand side denotes the coequalizer of the two projection maps from the pullback $U \times_{y(X)} U$.
\end{Prop}

\begin{proof}
Given $V \in \cat{C}$, let $C \coloneqq \coeq \left( U \times_{y(X)} U \rightrightarrows U \right)(V)$. Then we have
\begin{equation*}
\begin{aligned}
 	C & \cong \coeq \left( \Pre(\cat{C})(y(V),U \times_{y(X)} U) \rightrightarrows \Pre(\cat{C})(y(V),U) \right) \\
 	& \cong \coeq \left( \Pre(\cat{C})(y(V),U) \times_{\Pre(\cat{C})(y(V),y(X))} \Pre(\cat{C})(y(V),U) \rightrightarrows \Pre(\cat{C})(y(V),U) \right) \\
 	& \cong \coeq \left( \Pre(\cat{C})(y(V),U) \times_{\cat{C}(V,X)}\Pre(\cat{C})(y(V),U) \rightrightarrows \Pre(\cat{C})(y(V),U) \right) \\
 	& \cong \coeq \left( \left( \sum_i \cat{C}(V,U_i) \right) \times_{\cat{C}(V, X)} \left( \sum_j \cat{C}(V,U_j) \right) \rightrightarrows \sum_i \cat{C}(V,U_i) \right),
 	\end{aligned}
 \end{equation*}
where the third isomorphism holds by the Yoneda lemma and that fact that colimits are computed objectwise. 

Explicitly, $C$ is the set of morphisms $V \to U_i$ for some $i \in I$, modulo the following equivalence relation: for morphisms $f: V \to U_i$, $f': V \to U_j$, then $f \sim f'$ if and only if the following diagram commutes:
\begin{equation*}
\begin{tikzcd}
	V & {U_j} \\
	{U_i} & X
	\arrow["f"', from=1-1, to=2-1]
	\arrow["{r_i}"', from=2-1, to=2-2]
	\arrow["{r_j}", from=1-2, to=2-2]
	\arrow["{{f'}}", from=1-1, to=1-2]
\end{tikzcd}
\end{equation*}
Define the map $\phi: C \to R(V)$ as follows: given $f: V \to U_i$, post-compose with $r_i: U_i \to X$. This is clearly well defined and injective by definition. If $f' \in R(V)$, then $f'$ can be written as a composite $f' = r_i f''$ where $f'' : V \to U_i$ is some morphism in $\cat{C}$. Thus $\phi(f'') = f'$, so $\phi$ is surjective, and therefore a bijection. Naturality is easy to show.
\end{proof}

\begin{Rem}
Since any sieve is a generating family for itself, this means that we can always write any sieve $R \hookrightarrow y(U)$ as 
\begin{equation} \label{eqn sieves as colimits}
R \cong \coeq \left( \sum_{g: W \to U, f: V \to U  \in R} y(W) \times_{y(U)} y(V) \rightrightarrows \sum_{f \in R} y(V) \right).
\end{equation}
\end{Rem}

Another convenient characterization of a sieve as a colimit is given by the following result.

\begin{Lemma} \label{lem sieve as coequalizer}
Given a small category $\cat{C}$ and a sieve $R \hookrightarrow y(U)$ we have
\begin{equation*}
    R \cong \text{coeq} \left( \sum_{W \xrightarrow{g} V \xrightarrow{f} U \in R } y(W) \rightrightarrows \sum_{V \xrightarrow{f} U \in R} y(V) \right)
\end{equation*}
where the left hand sum is indexed by all pairs of maps $(g,f)$ such that $fg \in R$, and the two maps are given by projection $(g,f) \mapsto g$ and composition $(g,f) \mapsto fg$.
\end{Lemma}

\begin{proof}
This follows from the coYoneda Lemma \ref{lem coyoneda lemma}, and the conventional way of writing any colimit as a coequalizer of coproducts.
\end{proof}

\begin{Lemma}[{\cite[C2.1 Lemma 2.1.3]{johnstone2002sketches}}] \label{lem sheaf on covering family iff on sieve it generates}
Given a category $\cat{C}$, a presheaf $X$ on $\cat{C}$ is a sheaf on a family of morphisms $r = \{ U_i \to U \}$ if and only if it is a sheaf on the sieve $R = \overline{r}$ it generates.
\end{Lemma}

\begin{proof}
$(\Rightarrow)$ First we note that there is a refinement $i: r \to \overline{r} = R$ given componentwise by the identity map. Now if $\{ x_f \}_{f \in R}$ is a matching family for $R$, then we can restrict this family $\{ x_f \}|_r \coloneqq i^*\{ x_f \}$ to a matching family on $r$ by Lemma \ref{lem pullback of matching family by refinement is a matching family}. Since $X$ is a sheaf on $r$, there is an amalgamation $x \in X(U)$. We want to show that this is an amalgamation for $\{x_f \}$. Note that since $x$ is an amalgamation for $\{x_f \}|_r$, we know that $X(r_i)(x) = x_{r_i}$, so the above is equivalent to asking that $X(s_f)(x_{r_i}) = x_f$ whenever $f$ factors as $r_i s_f$ for some $s_f : V \to U_i$. Consider the commutative diagram:
\begin{equation*}
    \begin{tikzcd}
	V & {U_i} \\
	V & U
	\arrow["{r_i}", from=1-2, to=2-2]
	\arrow[Rightarrow, no head, from=1-1, to=2-1]
	\arrow["{s_f}", from=1-1, to=1-2]
	\arrow["f"', from=2-1, to=2-2]
\end{tikzcd}
\end{equation*}
since both $f$ and $r_i$ belong to $R$ and $\{x_f \}$ is a matching family for $R$, this implies that $X(1_V)(x_f) = x_f = X(s_f)(x_{r_i})$, which proves that $x$ is an amalgamation for $\{x_f\}$. 

$(\Leftarrow)$ Suppose $X$ is a sheaf on $R$, and $\{ x_i \}$ is a matching family for $r$. By definition every morphism in $R$ factors through some $r_i \in r$. Using the axiom of choice, we can construct a refinement $\pi: R \to r$, where we choose $\pi(r_i) = r_i$ for all $r_i \in r$. Then $\pi^* \{ x_i \}$ is a matching family for $R$ such that $i^* \pi^* \{ x_i \} = \{ x_i \}$. Since $X$ is a sheaf on $R$, $\{y_f \}$ amalgamates uniquely to some $y$ on $X(U)$. Thus $y$ is also an amalgamation of $\{x_i \}$. Let us show that $y$ is a unique amalgamation of $\{x_i \}$. If $z$ were another amalgamation of $\{x_i \}$, then $X(r_i)(z) = x_i$, so for every morphism $f \in R$, which can be written as a composite $f = r_i s_f$, then
$$X(r_i s_f)(z) = X(s_f) X(r_i)(z) = X(s_f)(x_i) = y_f.$$
Thus $z$ is also an amalgamation of $\{y_f \}$, so $y = z$. Thus $X$ is a sheaf on $r$.
\end{proof}

\begin{Lemma}
Given a morphism $g: V \to U$, recall the definition of pullback of families by $g$ (Definition \ref{def pushforward of a family}). If $R$ is a sieve over $U$, then as a presheaf $g^*(R)$ is a pullback
\begin{equation*}
    \begin{tikzcd}
	{g^*(R)} & R \\
	{y(V)} & {y(U)}
	\arrow["g"', from=2-1, to=2-2]
	\arrow[hook, from=1-2, to=2-2]
	\arrow[hook, from=1-1, to=2-1]
	\arrow[from=1-1, to=1-2]
	\arrow["\lrcorner"{anchor=center, pos=0.125}, draw=none, from=1-1, to=2-2]
\end{tikzcd}
\end{equation*}
Similarly if $S$ is a sieve over $V$, then as a presheaf $g_*(S)$ is the image (Definition \ref{def image}) of the composite map
\begin{equation*}
   S \hookrightarrow y(V) \xrightarrow{g} y(U). 
\end{equation*}
\end{Lemma}

\begin{Lemma} \label{lem adjunction on sieve posets}
Let $\cat{C}$ be a category and $g : V \to U$ be a morphism. Then we have an adjunction
\begin{equation}
    g_* : \cons{sieve}(V) \rightleftarrows \cons{sieve}(U) : g^*.
\end{equation}
where $g_*$ is the pushforward map on families (Definition \ref{def pushforward of a family}).
\end{Lemma}

\begin{Lemma} \label{lem pullback of sieve by map in sieve is maximal sieve}
Let $\cat{C}$ be a category and $R \hookrightarrow y(U)$ a sieve. A map $g: V \to U$ in $\cat{C}$ belongs to $R$ if and only if $g^*(R) = y(V)$.
\end{Lemma}

\begin{proof}
$(\Rightarrow)$ If $R$ is a sieve on $U$, then $g^*(R) = \{ f: W \to V \, : \, gf \in R \}$. But $R$ is a sieve and $g \in R$, so every map $f: W \to V$ has this property, since $R$ is closed under precomposition.
$(\Leftarrow)$ Suppose that $g^*(R) = y(V)$. Then $g_*(1_V) \in R$, thus $g \in R$.
\end{proof}

\begin{Def}
We say that a coverage $j$ is \textbf{sifted} if every $r \in j(U)$ is a sieve\footnote{Strictly speaking, there do not exist any sifted coverages, because Definition \ref{def coverage} requires that $(1_U)$ be a covering family. As soon as a sieve $R$ contains the identity map, then $R = y(U)$. Thus for $j$ to be a sifted coverage, we instead require that $y(U)$ be a covering sieve for every object $U$.}. We call covering families of sifted coverages \textbf{covering sieves}.
\end{Def}

\begin{Rem}
We will see later that being sifted is a very useful property for a coverage to have when one wishes to prove certain theorems. One immediate convenience is the following result.
\end{Rem}

\begin{Lemma} \label{lem convenience of using sieves for matching families}
Let $R$ be a sieve on an object $U$ in a category $\cat{C}$ and $X$ a presheaf on $\cat{C}$. A collection $\{ s_f \in X(V) \}_{f \in R}$ of sections for every $f: V \to U$ in $R$ is a matching family if and only if 
$$X(g)(s_f) = s_{fg}$$
for every morphism $g: W \to V$ in $\cat{C}$.
\end{Lemma}

\begin{proof}
$(\Rightarrow)$ Suppose $\{s_f \}$ is a matching family, then consider the commutative diagram:
\begin{equation*}
    \begin{tikzcd}
	W & W \\
	V & U
	\arrow["g"', from=1-1, to=2-1]
	\arrow["f"', from=2-1, to=2-2]
	\arrow["gf", from=1-2, to=2-2]
	\arrow[Rightarrow, no head, from=1-1, to=1-2]
\end{tikzcd}
\end{equation*}
this implies that $X(g)(s_f) = s_{fg}$.

$(\Leftarrow)$ Suppose we have a commutative diagram:
\begin{equation*}
    \begin{tikzcd}
	A & {V'} \\
	V & U
	\arrow["h", from=1-1, to=1-2]
	\arrow["g"', from=1-1, to=2-1]
	\arrow["f"', from=2-1, to=2-2]
	\arrow["{f'}", from=1-2, to=2-2]
\end{tikzcd}
\end{equation*}
where $f,f' \in R$. Then $X(g)(s_f) = s_{fg} = s_{f'h} = X(h)(s_{f'})$, thus $\{ s_f \}$ is a matching family.
\end{proof}

\begin{Lemma} \label{lem matching family in bijection with presheaf maps for sieves}
Given a category $\cat{C}$, a presheaf $X$ on $\cat{C}$, and a sieve $R \hookrightarrow y(U)$ there is a canonical map
\begin{equation*}
    \Pre(\cat{C})(R,X) \to \text{Match}(R, X)
\end{equation*}
and furthermore this map is a bijection.
\end{Lemma}

\begin{proof}
The canonical map $\Pre(\cat{C})(R,X) \to \text{Match}(R,X)$ can be described as follows. Given a map of presheaves $h: R \to X$ and a morphism $f: V \to U$ in $R$ we obtain a commutative diagram
\begin{equation*}
    \begin{tikzcd}
	{R(U)} & {R(V)} \\
	{X(U)} & {X(V)}
	\arrow["{R(f)}", from=1-1, to=1-2]
	\arrow["{h_U}"', from=1-1, to=2-1]
	\arrow["{h_V}", from=1-2, to=2-2]
	\arrow["{X(f)}"', from=2-1, to=2-2]
\end{tikzcd}
\end{equation*}
and $f \in R(V)$ so we obtain an element $h_V(f) \in X(V)$. We claim that the set of sections $\{h_V(f) \}_{f \in R}$ is a matching family of $X$ over $R$. By Lemma \ref{lem convenience of using sieves for matching families}, it is enough to show that if $g: W \to V$ is an arbitrary morphism in $\cat{C}$, then $X(g)(h_V(f)) = X(h_V(gf))$. However this follows by naturality of $h$.

Now to see that the map $\Pre(\cat{C})(R,X) \to \text{Match}(R,X)$ is a bijection, note that any matching family $\{s_f \}$ for $X$ over $R$ gives rise to a map of presheaves $h: R \to X$ defined componentwise by $h_V(f) = s_f$. Being a matching family guarantees that this is a natural transformation. It is easy to see this defines an inverse to the canonical map and thus establishes the bijection.
\end{proof}

\begin{Cor} \label{cor sheaf condition on a sieve}
Given a presheaf $X$ on a category $\cat{C}$, and a sieve $R$ over an object $U \in \cat{C}$, $X$ is a sheaf on $R$ if and only if the canonical map
\begin{equation}
    X(U) \cong \Pre(\cat{C})(y(U),X) \to \Pre(\cat{C})(R,X)
\end{equation}
induced by precomposing with the inclusion $R \hookrightarrow y(U)$ is a bijection.
\end{Cor}

\begin{proof}
This follows from Lemma \ref{lem convenience of using sieves for matching families}.
\end{proof}

From the above results, we obtain a convenient description of the sheaf condition for sieves. If $X$ is a presheaf and $R$ a sieve on $U$, then an $X$-matching family on $R$ is a map $m : R \to X$, and an amalgamation for $m$ is a map $x : y(U) \to X$ making the following diagram commute
\begin{equation*}
    \begin{tikzcd}
	R & X \\
	{y(U)}
	\arrow["m", from=1-1, to=1-2]
	\arrow[hook, from=1-1, to=2-1]
	\arrow["x"', dashed, from=2-1, to=1-2]
\end{tikzcd}
\end{equation*}
So $X$ is a sheaf on $R$ if such amalgamations always exist and are unique.

\begin{Rem}
The following description of the sheaf condition is what is typically written in textbooks, which only consider the sifted case. We see it as a special case of our formalism here.
\end{Rem}

\begin{Cor} \label{cor classical sheaf condition}
Given a small category $\cat{C}$ and $R \hookrightarrow y(U)$ a sieve on $\cat{C}$ generated by a family of morphisms $r = \{r_i: U_i \to U \}$, then a presheaf $X$ is a sheaf on $R$ if and only if the following diagram is an equalizer:
$$X(U) \to \prod_{i \in I} X(U_i) \rightrightarrows \prod_{i,j \in I} X(U_i) \times_{X(U)} X(U_j).$$
\end{Cor}

\begin{proof}
This follows from combining Corollary \ref{cor sheaf condition on a sieve} and Proposition \ref{prop sieves are coequalizers of generating family}.
\end{proof}

Another convenient form for sheaves on sieves is given by the following result.

\begin{Lemma} \label{lem sheaves on sieves condition}
Given a small category $\cat{C}$ and a sieve $R \hookrightarrow y(U)$, a presheaf $X$ is a sheaf on $R$ if and only if
\begin{equation*}
    X(U) \cong \lim_{V \xrightarrow{f} U \in R} X(V).
\end{equation*}
\end{Lemma}

\begin{proof}
By Corollary \ref{cor sheaf condition on a sieve}, $X$ is a sheaf on $R$ if and only if
\begin{equation*}
    X(U) \to \Pre(\cat{C})(R, X)
\end{equation*}
is a bijection. By the coYoneda Lemma \ref{lem coyoneda lemma}, we have
\begin{equation*}
    \begin{aligned}
       \Pre(\cat{C})(R, X) &\cong \Pre(\cat{C})\left(\ncolim{y(V) \to R} y(U), X\right) \\
       & \cong \lim_{y(V) \to R} \Pre(\cat{C})(y(V), X) \\
       & \cong \lim_{f : V \to U \in R} X(V).
    \end{aligned}
\end{equation*}
\end{proof}

\begin{Def} \label{def sifted closure}
If $j$ is a coverage, then let $\overline{j}$ denote the collection of families where $R \in \overline{j}(U)$ if $R = \overline{r}$ for some $r \in j(U)$. We call $\overline{j}$ the \textbf{sifted closure} of $j$.   
\end{Def}

\begin{Lemma}
Given a coverage $j$ on a category $\cat{C}$, then $\overline{j}$ is a sifted coverage of $\cat{C}$.
\end{Lemma}

\begin{proof}
Clearly $\overline{j}$ is sifted. We wish to show it is a coverage. Suppose we have a covering family $R \in \overline{j}(U)$, and a map $g: V \to U$. We wish to show that there is a covering family $R' \in \overline{j}(V)$ such that for every map $k \in R'$, $gk$ factors through some $l \in R$. Since $R = \overline{r}$, we know that since $j$ is a coverage, there exists some covering family $t \in j(V)$ with the corresponding property. Let $R' = \overline{t}$. Then for every $k \in R'$, there exists a $k_j$ such that $k = t_j k_j$, and a map $s_j$ making the following diagram commute:
\begin{equation*}
    \begin{tikzcd}
	W \\
	{V_j} & {U_i} \\
	V & U
	\arrow["{k_j}", from=1-1, to=2-1]
	\arrow["{t_j}", from=2-1, to=3-1]
	\arrow["g"', from=3-1, to=3-2]
	\arrow["{r_i}", from=2-2, to=3-2]
	\arrow["{s_j}", from=2-1, to=2-2]
	\arrow["k"', curve={height=18pt}, from=1-1, to=3-1]
\end{tikzcd}
\end{equation*}
but then $gk$ factors through some $r_i$, so $gk \in R$. Thus $g_*(R') \subseteq R$, and $\overline{j}$ is a coverage.
\end{proof}

\begin{Cor} \label{cor sheaves on sifted closure}
A presheaf $X$ on a category $\cat{C}$ is a sheaf for a coverage $j$ if and only if it is a sheaf for $\overline{j}$. In other words, $\ncat{Sh}(\cat{C}, j) = \ncat{Sh}(\cat{C}, \overline{j})$. 
\end{Cor}

\begin{proof}
This follows from Lemma \ref{lem sheaf on covering family iff on sieve it generates}. 
\end{proof}

\section{Coverage Closures} \label{section coverage closures}
In this section, we detail two closure operations corresponding to different properties that are desirable for a coverage to have.

\subsection{Composition Closure}
In this section we describe composition closure. The usual reference for this material would probably be \cite[Section C.2]{johnstone2002sketches}, but we note that the statement of \cite[Lemma C.2.1.7]{johnstone2002sketches} as written is incorrect, see Zhen Lin Low's MathOverflow answer \cite{lowMOanswerLemma2.1.7}. The proof of Lemma \ref{lem sheaf on composite family} is inspired by the proof given in the MathOverflow answer above.

\begin{Def} \label{def composition closed coverage}
If $(\cat{C}, j)$ is a site, then we say that $j$ is \textbf{composition closed} if the following condition holds: if $r = \{r_i: U_i \to U \} \in j(U)$, and for each $i$ there is a $t^i = \{t^i_j: V_{ij} \to U_i \} \in j(U_i)$, then $(r \circ t) \coloneqq \bigcup_i (r_i)_*(t^i) \in j(U)$.
\end{Def}

\begin{Rem}
It will be very useful in what follows to visualize covering families and their composites as infinitely wide, but finitely high trees. For example a covering family $r = \{ r_i: U_i \to U \}$ over $U$ can be viewed as an infinitely wide tree of height $1$:
\begin{equation*}
\begin{tikzcd}
	{U_i} & \dots & {U_{i'}} \\
	& U
	\arrow["{r_i}"', from=1-1, to=2-2]
	\arrow["{r_{i'}}", from=1-3, to=2-2]
\end{tikzcd}
\end{equation*}
Now suppose that $t^i = \{t^i_j : V^i_j \to U_i \}$ is a covering family on $U_i$. Then since $(1_{U_j})$ is a covering family for every $U_j$, if $(\cat{C}, j)$ is composition closed, the composite family 
\begin{equation*}
    (r_i)_*(t^i) \cup \left( \bigcup_{j \in (I \setminus i)} (r_j)_*(1_{U_j}) \right) = (r_i)_*(t^i) \cup (r \setminus r_i)
\end{equation*}
is a covering family over $U$, and can be visualized as follows:
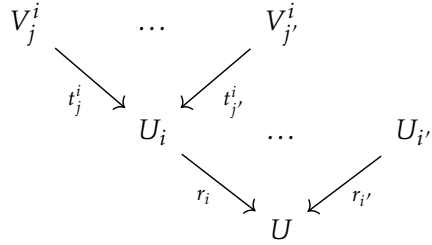
\begin{equation*}
    \begin{tikzcd}
	{V^i_j} & \dots & {V^i_{j'}} \\
	& {U_i} & \dots & {U_{i'}} \\
	&& U
	\arrow["{r_i}"', from=2-2, to=3-3]
	\arrow["{r_{i'}}", from=2-4, to=3-3]
	\arrow["{t^i_j}"', from=1-1, to=2-2]
	\arrow["{t^i_{j'}}", from=1-3, to=2-2]
\end{tikzcd}
\end{equation*}
We can thus compose covering families in an operadic fashion. We will formalize this intuition now.
\end{Rem}

\begin{Def}[{\cite[Appendix A.2]{low2016categories}}] \label{def tree}
Let $\cat{C}$ be a category with $U \in \cat{C}$. A \textbf{path} on $U$ is a (possibly empty) finite ordered tuple $(f_m, \dots, f_1)$ of morphisms such that $\text{dom}(f_i) = \text{cod}(f_{i+1})$ for $0 < i < m$ and $\text{cod}(f_1) = U$, and we say that $m$ is the \textbf{length} of the path. Given a path $(f_m, \dots, f_1)$, we call $(f_n, \dots, f_1)$ a \textbf{prefix} if $n \leq m$. The \textbf{leaf node} of a path $(f_m, \dots, f_1)$ is the domain of $f_m$.

Given a set $T$ of paths on $U$, we say that a path $(f_m, \dots, f_1) \in T$ is \textbf{maximal} if for every path $(g_n, \dots, g_1) \in T$ with $m \leq n$ and $(f_m, \dots, f_1) = (g_m, \dots, g_1)$, then $m = n$.  A \textbf{tree} $T$ on $U$ is a small set of paths on $U$ satisfying the following conditions:
\begin{enumerate}
    \item The empty path is an element of $T$,
    \item if $(f_{m+1}, f_m, \dots, f_1) \in T$ then $(f_m, \dots, f_1) \in T$,
    \item every path is a prefix of a maximal path,
    \item there exists a finite maximal path length\footnote{We have added this condition from Low's original definition to ease our induction proofs.}, which we call the \textbf{height} of the tree $T$.
\end{enumerate}
Given a tree $T$ on an object $U$, let $T^{\circ}$ denote the family of morphisms $$T^{\circ} = \{U_m \xrightarrow{f_m} U_{m-1} \to \dots \to U_1 \xrightarrow{f_1} U \, : \, (f_1, \dots, f_m) \text{ is a maximal path of } T \},$$
we call this the \textbf{composite} of the tree $T$. If $T$ contains only the empty branch, then let $T^\circ = \{ 1_U \}$.

We say a subset $T' \subseteq T$ is a \textbf{sub-tree} if it is also a tree on $U$. Given a tree $T$ and $n \geq 0$, let $T_{\leq n}$ denote the sub-tree of $T$ of paths of length less than or equal to $n$.
\end{Def}

\begin{Def}[{\cite[Appendix A.2]{low2016categories}}] \label{def j-tree}
Let $(\cat{C}, j)$ be a site, with $U \in \cat{C}$. We say that a tree $T$ on $U$ is a \textbf{$j$-tree} if:
\begin{enumerate}
    \item The set of morphisms $f: V \to U$ such that $(f) \in T$, is either empty or a covering family in $j(U)$, and
    \item for every path $f = (f_m, \dots, f_1) \in T$, the set of morphisms $g: V \to U_m$ such that $(g, f_m, \dots, f_1) \in T$, is either empty or a covering family of $U_m = \text{dom}(f_m)$.
\end{enumerate}
\end{Def}

\begin{Rem}
A $j$-tree can be more easily described as the result of an iterative procedure as follows. Start with a covering family $r$ on $U$. Then suppose that for some subset of domains $U_i$ of the covering maps there are covering families $t^i$. Attach those families, extending the tree. Now keep doing this at every stage, attaching new covering famliies and extending the tree, but eventually stop at some finite stage $n$. This gives a $j$-tree, and conversely every $j$-tree can be obtained in this way. This immediately gives the following result, which is the key feature of $j$-trees.
\end{Rem}

\begin{Lemma} \label{lem j-trees closed under composition}
If $T$ is a $j$-tree and $\{U_i \}_{i \in I}$ is its set of leaf nodes, and for every $i \in I$, $T_i$ is a $j$-tree on $U_i$, then the composite tree is a $j$-tree.
\end{Lemma}

\begin{Lemma} \label{lem sheaf on composite family}
Let $(\cat{C}, j)$ be a site and $X$ a $j$-sheaf. If $X$ is a sheaf on a family of morphisms $r = \{r_i: U_i \to U \}_{i \in I}$, and for every $i \in I$, there is a covering family $t^i = \{t^i_j: V^i_j \to U_i \}_{j \in J_i}$, then $X$ is a sheaf on $(r \circ t) = \bigcup_i (r_i)_*(t^i)$, where $(r_i)_*(t^i)$ is the pushforward family from Definition \ref{def pushforward of a family}.
\end{Lemma}

\begin{proof}
Let $\{x^i_j \in X(V^i_j) \}_{i \in I, j \in J_i}$ be a matching family of $X$ over $(r \circ t)$. We wish to show that there is a unique amalgamation.

Now for each $i \in I$, consider the set $\{x^i_j \in X(V^i_j) \}_{j \in J_i}$. This is a matching family for $X$ over $t^i$. Since $X$ is a sheaf and $t^i$ is a covering family, there exists a unique amalgamation $x_i \in X(U_i)$ such that $X(t^i_j)(x_i) = x^i_j$. We wish to prove that $\{x_i \in X(U_i) \}_{i \in I}$ is a matching family for $r$. Doing this will take some work. So suppose that we have a commutative diagram of the form
\begin{equation*}
    \begin{tikzcd}
	B & {U_{i'}} \\
	{U_i} & U
	\arrow["{r_i}"', from=2-1, to=2-2]
	\arrow["{r_{i'}}", from=1-2, to=2-2]
	\arrow["g"', from=1-1, to=2-1]
	\arrow["h", from=1-1, to=1-2]
\end{tikzcd}
\end{equation*}
we wish to show that $X(g)(x_i) = X(h)(x_{i'})$. Now since $t^i$ is a covering family, and $g: B \to U_i$ is an arbitrary map, there exists a covering family $s = \{s_k : W_k \to B \}_{k \in K}$ over $B$ and a refinement $f : g_*(s) \to t^i$, with index map $\alpha : K \to J_i$. Similarly since $t^{i'}$ is a covering family, there exists a covering family $s' = \{s'_{k'} : W'_{k'} \to B \}_{k' \in K'}$ over $B$ and a refinement $f' : h_*(s') \to t^{i'}$, with index map $\beta: K' \to J_{i'}$. Thus for every $k \in K$ and $k' \in K'$, we obtain the following commutative diagram
\begin{equation*}
  \begin{tikzcd}
	& {W'_{k'}} & {V^{i'}_{\beta(k')}} \\
	{W_k} & B & {U_{i'}} \\
	{V^i_{\alpha(k)}} & {U_i} & U
	\arrow["{r_i}"', from=3-2, to=3-3]
	\arrow["{r_{i'}}", from=2-3, to=3-3]
	\arrow["g"', from=2-2, to=3-2]
	\arrow["h", from=2-2, to=2-3]
	\arrow["{t^{i'}_{\beta(k')}}", from=1-3, to=2-3]
	\arrow["{s'_{k'}}", from=1-2, to=2-2]
	\arrow["{f'_{k'}}", from=1-2, to=1-3]
	\arrow["{t^i_{\alpha(k)}}"', from=3-1, to=3-2]
	\arrow["{f_k}"', from=2-1, to=3-1]
	\arrow["{s_k}"', from=2-1, to=2-2]
\end{tikzcd}  
\end{equation*}
Now since $s'$ is a covering family of $W'_{k'}$, and $s_k : W_k \to B$ is an arbitrary map, there exists a covering family $q = \{q_\ell: Q_\ell \to W_k \}_{\ell \in L}$, and a refinement $\sigma : (b_k)_*(q) \to s'$, with index map $\gamma : L \to K'$. Thus for every $\ell$, we obtain the following commutative diagram
\begin{equation*}
    \begin{tikzcd}
	{Q_\ell} & {W'_{\gamma(\ell)}} & {V^{i'}_{\beta \gamma(\ell)}} \\
	{W_k} & B & {U_{i'}} \\
	{V^i_{\alpha(k)}} & {U_i} & U
	\arrow["{r_i}"', from=3-2, to=3-3]
	\arrow["{r_{i'}}", from=2-3, to=3-3]
	\arrow["g"', from=2-2, to=3-2]
	\arrow["h", from=2-2, to=2-3]
	\arrow["{t^{i'}_{\beta \gamma(\ell)}}", from=1-3, to=2-3]
	\arrow["{s'_{\gamma(\ell)}}", from=1-2, to=2-2]
	\arrow["{f'_{\gamma(\ell)}}", from=1-2, to=1-3]
	\arrow["{t^i_{\alpha(k)}}"', from=3-1, to=3-2]
	\arrow["{f_k}"', from=2-1, to=3-1]
	\arrow["{s_k}"', from=2-1, to=2-2]
	\arrow["{q_\ell}"', from=1-1, to=2-1]
	\arrow["{\sigma_\ell}", from=1-1, to=1-2]
\end{tikzcd}
\end{equation*}
Now notice that $X(f_k q_\ell)(x^i_{\alpha(k)}) = X(f'_{\gamma(\ell)} \sigma_\ell)(x^{i'}_{\gamma(\ell)})$, because $\{x^i_j \}$ is a matching family for $(r \circ t)$. Now 
$$X(f'_{\gamma(\ell)} \sigma_\ell)(x^{i'}_{\gamma(\ell)}) = X(t^{i'}_{\beta \gamma(\ell)} f'_{\gamma(\ell)} \sigma_\ell)(x_{i'}) = X(h s_k q_\ell)(x_{i'}),$$
and similarly
$$X(f_k q_\ell)(x^i_\alpha(k)) = X(t^i_\alpha(k) f_k q_\ell)(x_i) = X(g s_k q_\ell)(x_i).$$
Thus for every $\ell$ we have 
\begin{equation} \label{eq composite closure proof matching family}
    X(g s_k q_\ell)(x_i) = X(h s_k q_\ell)(x_{i'}).
\end{equation}
Let $z = X(g s_k)(x_i)$ and $z' = X(h s_k)(x_{i'})$. Then consider the sets $\{X(q_\ell)(z) \}_{\ell \in L}$ and $\{X(q_\ell)(z') \}_{\ell \in L}$. These are both matching families for $X$ over $q$, which is a covering family. Since $X$ is a sheaf, $z$ is a unique amalgamation of $\{X(q_\ell)(z) \}$ and $z'$ is a unique amalgamation of $\{X(q_\ell)(z')\}$. But by (\ref{eq composite closure proof matching family}), $z$ is also an amalgamation for $\{X(q_\ell)(z') \}$. Thus $z = z'$. Therefore \begin{equation} \label{eq composite closure proof matching family 2}
  X(s_k)X(g)(x_i) = X(s_k)X(h)(x_{i'}).
\end{equation}

Now let us repeat this trick. Namely let $y = X(g)(x_i)$ and $y' = X(h)(x_{i'})$, and consider the sets $\{X(s_k)(y) \}$ and $\{X(s_k)(y') \}$. These are a matching family for $X$ over $s$. Since $X$ is a sheaf, $y$ is a unique amalgamation of $\{X(s_k)(y) \}$ and $y'$ is a unique amalgamation of $\{X(s_k)(y') \}$. But by (\ref{eq composite closure proof matching family 2}), $y$ is also an amalgamation for $\{X(s_k)(y') \}$. Thus $y = X(g)(x_i) = X(h)(x_{i'}) = y'$, as we wished to show. Thus $\{ x_i \}$ is a matching family for $r$.

Now since $X$ is a sheaf on $r$, there exists a unique amalgamation $x \in X(U)$ such that $X(r_i)(x) = x_i$ for all $i \in I$. This is clearly an amalgamation for $\{x^i_j \}$ as $X(t^i_j r_i)(x) = X(t^i_j)(x_i) = x^i_j$ for every $i \in I, j \in J_i$. Thus $X$ is a sheaf on $(r \circ t)$.
\end{proof}

\begin{Lemma} \label{lem sheaf on j-trees}
Suppose that $(\cat{C}, j)$ is a site, $U \in \cat{C}$, and $X$ is a $j$-sheaf. If $T$ is a $j$-tree, then $X$ is a sheaf on $T^\circ$.
\end{Lemma}

\begin{proof}
We will prove this by induction on the height of $T$. For the base case, if the height of $T$ is $0$, then by Lemma \ref{lem always a sheaf on an isomorphism}, $X$ is a sheaf on $T^\circ = (1_U)$. For the inductive step, suppose that $X$ is a sheaf on the composite of every $j$-tree of height $n$, and suppose that $T$ is a $j$-tree of height $n+1$. Let us show that $X$ is a sheaf on $T^\circ$. By assumption $X$ is a sheaf on $T_{\leq n}^\circ$. But by Lemma \ref{lem sheaf on composite family}, this implies that $X$ is a sheaf on $T^\circ$, as $T$ can be constructed from $T_{\leq n}$ by adding on covering families to its leaf nodes.
\end{proof}

\begin{Def} \label{def j-tree closed}
Given a site $(\cat{C}, j)$, we say that it is \textbf{$j$-tree closed} if for every object $U$ and every $j$-tree $T$ on $U$, its composite $T^\circ$ is a covering family on $U$.
\end{Def}

\begin{Lemma}[{\cite[Lemma A.2.10]{low2016categories}}] \label{lem composition closed iff j-tree closed}
Let $(\cat{C}, j)$ be a site. Then $j$ is a composition closed coverage if and only if it is $j$-tree closed.
\end{Lemma}

\begin{proof}
$(\Leftarrow)$ If $j$ is closed under $j$-trees, then clearly it is composition closed. $(\Rightarrow)$ Conversely, we can build any height $n$ $j$-tree by taking $n$-many composites. Thus $j$ is closed under $j$-trees.
\end{proof}

\begin{Def} \label{def composition closure}
 Given a site $(\cat{C}, j)$, let $\text{comp}(j)$ denote the collection of families on $\cat{C}$ where $\text{comp}(j)(U)$ consists of those families of morphisms of the form $T^\circ$ for some $j$-tree $T$ on $U$.   
\end{Def}

\begin{Lemma} \label{lem can pullback j-trees}
Given a site $(\cat{C}, j)$, a $j$-tree $T$ on an object $U \in \cat{C}$ and a morphism $f : V \to U$, there exists a $j$-tree $S$ on $V$ such that $f_*(S^\circ) \leq T^\circ$.
\end{Lemma}

\begin{proof}
Let us prove this by induction on the height of $T$. When the height of $T$ is $1$, then $T^\circ$ is a covering family, and then the result holds by the fact that $j$ is a coverage.

Now suppose that the hypothesis holds for every height $n$ $j$-tree on $U$, and let $T$ be a height $n+1$ $j$-tree on $U$. Then $T_{\leq n}$ is a height $n$ $j$-tree on $U$, and so by the hypothesis there exists a $j$-tree $S$ on $V$ such that $f_*(S) \leq T_{\leq n}^\circ$.

So for each $s_a : V_a \to V$ in $S^\circ$, there exists a map $s^i_a: V_a \to U_i$ such that the following diagram commutes
\begin{equation*}
\begin{tikzcd}[ampersand replacement=\&]
	{V_j} \& {U_i} \\
	V \& U
	\arrow["{s^i_j}", from=1-1, to=1-2]
	\arrow["{s_j}"', from=1-1, to=2-1]
	\arrow["{r_i}", from=1-2, to=2-2]
	\arrow["f"', from=2-1, to=2-2]
\end{tikzcd}    
\end{equation*}
where $r_i \in T^\circ_{\leq n}$. Now if $r_i \in T^\circ$, then we are done. If not, then there must be a path $(f_n, \dots, f_1)$ that is maximal in $T_{\leq n}$ but is not maximal in $T$ such that $f_n \circ \dots \circ f_1 = r_i$. Since it is not maximal, that means there is a covering family $t_i = \{t^i_k : U^i_k \to U_i \}$ such that each $(t^i_k, f_n, \dots, f_1)$ is a maximal path in $T$. Since $j$ is a coverage, there exists a covering family $q_a = \{q^a_\ell : V^a_\ell \to V_j \}$ such that $(s^i_a)_*(q_a) \leq t_i$. By composing the $j$-tree $S$ with the covering family $q_a$ for each $s_j \in S^\circ$, we obtain a new $j$-tree $Q$ such that $f_*(Q^\circ) \leq T^\circ$.
\end{proof}

\begin{Lemma} \label{lem comp(j) smallest composition closed coverage}
If $(\cat{C}, j)$ is a site, then $\text{comp}(j)$ is a coverage on $\cat{C}$, and it is the smallest composition closed coverage on $\cat{C}$ containing $j$, in the sense that $j \subseteq \text{comp}(j)$ (Definition \ref{def coverage comparisons}) and if $j'$ is a composition closed coverage such that $j \subseteq j'$, then $\text{comp}(j) \subseteq j'$.
\end{Lemma}

\begin{proof}
Let us show that $\text{comp}(j)$ is a coverage. Suppose that $r \in \text{comp}(j)(U)$ and $f: V \to U$ is a morphism. Since $r$ is a covering family, there exists a $j$-tree $T$ such that $T^\circ = r$. Now by Lemma \ref{lem can pullback j-trees}, there exists a $j$-tree $S$ on $V$ such that $f_*(S^\circ) \leq T^\circ$. Thus $\text{comp}(j)$ is a coverage, and by Lemma \ref{lem j-trees closed under composition}, it is composition closed. The fact that it is the smallest composition closed coverage containing $j$ is not hard to prove using Lemma \ref{lem composition closed iff j-tree closed}.
\end{proof}

\begin{Lemma} \label{lem j-trees have common refinement}
Given a site $(\cat{C}, j)$ and an object $U \in \cat{C}$, if $T$ and $S$ are $j$-trees on $U$, then there exists a $j$-tree $R$ such that $R^\circ \leq T^\circ$ and $R^\circ \leq S^\circ$.
\end{Lemma}

\begin{proof}
Suppose that we fix a map $s_i : V_i \to U$ in $S^\circ$, then by Lemma \ref{lem can pullback j-trees}, there is a $j$-tree $Q^i$ on $V_i$ such that $(s_i)_*((Q^i)^\circ) \leq T^\circ$. Thus we can compose to obtain a new $j$-tree $R = \cup_i (s_i \circ Q^i)$ on $U$. By its construction we have $R^\circ \leq T^\circ$. But we also have $R^\circ \leq S^\circ$, by taking any map $r \in R^\circ$ and factoring it as $r = s_i q^i$, with $s_i \in S^\circ$ and $q^i \in (Q^i)^\circ$. Thus $R$ is a $j$-tree that refines both $T$ and $S$.
\end{proof}

\begin{Prop} \label{prop sheaf iff sheaf on composition closure}
Given a site $(\cat{C}, j)$, a presheaf $X$ on $\cat{C}$ is a $j$-sheaf if and only if it is a $\text{comp}(j)$-sheaf. In other words 
$$\Sh(\cat{C}, j) = \Sh(\cat{C}, \text{comp}(j)).$$
\end{Prop}

\begin{proof}
This follows from Lemma \ref{lem sheaf on composite family}.
\end{proof}

\subsection{Refinement Closure}
In this section we describe the second closure condition on coverages, refinement closure.

\begin{Def} \label{def refinement closed coverage}
Given a site $(\cat{C}, j)$ we say that the coverage $j$ is \textbf{refinement closed} if for every $U \in \cat{C}$, $j(U)$ satisfies the following condition: if $r \in j(U)$ and $f: r \to t$ is a refinement, then $t \in j(U)$.
\end{Def}

\begin{Lemma}[{\cite[Lemma 2.1.6]{johnstone2002sketches}}] \label{lem sheaf on refinement}
Suppose $j$ is a coverage on a category $\cat{C}$. If $X$ is a sheaf on $\cat{C}$, $r \in j(U)$ is a covering family, $t = \{ t_j: V_j \to U \}$ is a family of morphisms and $f: r \to t$ is a refinement, then $X$ is a sheaf on $t$. 
\end{Lemma}

\begin{proof}
Suppose that $X$ is a sheaf, and $r = \{r_i: U_i \to U \}$ is a covering family. Suppose that $f: r \to t$ is a refinement, so that there exists a function $\alpha: I \to J$ such that for every $i$, there is a map $f_i: U_i \to V_{\alpha(i)}$ such that $r_i = t_{\alpha(i)} f_i $.
 
 Let $\{ a_j \}$ be a matching family for $t$, and let $\{ b_i \} = f^* \{a_j \}$ be the pullback matching family on $r$. Since $X$ is a sheaf, and $r$ is a covering family, there exists a unique amalgamation $b \in X(U)$ such that $X(r_i)(b) = b_i$.
 
Now since $r$ is a covering family, for every $t_j: V_j \to U$, there exists a covering family $p^j = \{ p^j_k: W^j_k \to V_j \}_{k \in K}$ and maps $l_k: W_k^j \to U_i$ for some $i$ making the following diagram commute:
\[\begin{tikzcd}
	{W^j_k} & {U_i} & {V_{\alpha(i)}} \\
	{V_j} & U & U
	\arrow[Rightarrow, no head, from=2-2, to=2-3]
	\arrow["{p^j_k}"', from=1-1, to=2-1]
	\arrow["{t_j}"', from=2-1, to=2-2]
	\arrow["{r_i}", from=1-2, to=2-2]
	\arrow["{f_i}", from=1-2, to=1-3]
	\arrow["{t_{\alpha(i)}}", from=1-3, to=2-3]
	\arrow["{l_k}", from=1-1, to=1-2]
\end{tikzcd}\]
 Now we wish to show that $b$ is an amalgamation for $\{ a_j \}$, so that $X(t_j)(b) = a_j$ for every $j$. We will do this by exploiting the fact that $p^j$ is a covering family.
 
First note that since $\{ a_j \}$ is a matching family, we have
$$X(p^j_k)(a_j) = X(f_i l_k)(a_{\alpha(i)}).$$
This implies that
$$X(p^j_k)(a_j) = X(l_k)X(f_i)(a_{\alpha(i)}) = X(l_k)(b_i) = X(l_k) X(r_i)(b) = X(p^j_k) X(t_j)(b).$$
But it is clear that $\{ X(p^j_k)(a_j)\}_{k \in K}$ is an $X$-matching family on $p^j$ by functoriality. Since $X$ is a sheaf, and $p^j$ is a covering family, there exists a unique amalgamation on $V_j$. But both $a_j$ and $X(t_j)(b)$ are such amalgamations. Thus $X(t_j)(b) = a_j$. Thus $b$ is an amalgamation for $\{ a_j \}$.
\end{proof}

\begin{Def} \label{def refinement closure}
Let $\text{ref}(j)$ denote the collection of families where $\text{ref}(j)(U)$ consists of those families $t$ over $U$ such that there exists a covering family $r \in j(U)$ and a refinement $f: r \to t$.
\end{Def}

\begin{Lemma} \label{lem refinement closure smallest refinement closed coverage}
If $(\cat{C}, j)$ is a site, then $\text{ref}(j)$ is a coverage on $\cat{C}$, and it is the smallest refinement closed coverage on $\cat{C}$ containing $j$.
\end{Lemma}

\begin{proof}
Let us show that $\text{ref}(j)$ is a coverage. Suppose that $f: W \to V$ is a morphism in $\cat{C}$ and $t = \{ t_j: V_j \to V \}$ is a covering family of $\text{ref}(j)(V)$. Then there exists a refinement $g: r \to t$ where $r = \{ r_i: U_i \to U \} \in j(V)$. Thus there exists a family $s = \{ s_j: W_j \to W \} \in j(W)$ that factors through some $r_i: U_i \to U$. Therefore it factors through $t_{\alpha(i)}: V_{\alpha(i)} \to V$ where $\alpha$ is determined by the refinement $g$. Since $s \in j(W)$, and $j \subseteq \text{ref}(j)$, we have $s \in \text{ref}(j)(W)$. Thus $\text{ref}(j)$ is a coverage. Now if $j'$ is a refinement closed coverage such that $j \subseteq j'$, then for every $r \in j$ and every refinement $r \leq t$, it follows that $t \in j'$. Thus $\text{ref}(j) \subseteq j'$.
\end{proof}

\begin{Prop} \label{prop sheaf iff sheaf on refinement closure}
If $(\cat{C}, j)$ is a site, then a presheaf $X$ on $\cat{C}$ is a $j$-sheaf if and only if it is a $\text{ref}(j)$-sheaf. In other words
$$\Sh(\cat{C}, j) = \Sh(\cat{C}, \text{ref}(j)).$$
\end{Prop}

\begin{proof}
This follows from Lemma \ref{lem sheaf on refinement}.
\end{proof}

Now in the case that $j$ is a sifted coverage, we notice that there exists a refinement $R \leq T$ of sieves if and only if $R \subseteq T$. This observation and Lemma \ref{lem sheaf on refinement} proves the following result.

\begin{Cor} \label{cor sheaves closed under containing covering sieves}
If $j$ is a sifted coverage on $\cat{C}$, $X$ is a sheaf on $\cat{C}$, $R \in j(U)$ is a covering sieve, and $T$ is a sieve on $U$, then $X$ is a sheaf on $T$ if $R \subseteq T$.
\end{Cor}

\begin{Lemma} \label{lem ref closure preserves composition closure}
If $j$ is a composition closed coverage on a category $\cat{C}$, then so is $\text{ref}(j)$.
\end{Lemma}

\begin{proof}
Suppose that $r = \{r_i : U_i \to U \}_{i \in I}$ is a $\text{ref}(j)$-covering family on $U$, so that there exists a $j$-covering family $r'$ and a refinement $r' \leq r$. Furthermore suppose that for every $i$, there is a $\text{ref}(j)$-covering family $t^i = \{t^i_j : V^i_j \to U_i \}$ on $U_i$. Then there exists a $j$-covering family $(t')^i$ and a refinement $(t')^i \leq t^i$. We wish to show that $(r \circ t)$ is a $\text{ref}(j)$ covering family. For each $i \in I$, we have the following commutative diagram
\begin{equation*}
\begin{tikzcd}
	{W^{i'}_k} & {(V')^i_{j'}} & {V^i_j} \\
	& {U'_{i'}} & {U_i} \\
	&& U
	\arrow["{r_i}", from=2-3, to=3-3]
	\arrow["{t^i_j}", from=1-3, to=2-3]
	\arrow[from=2-2, to=2-3]
	\arrow["{r'_{i'}}"', from=2-2, to=3-3]
	\arrow[from=1-2, to=1-3]
	\arrow["{(t')^i_{j'}}"{description}, from=1-2, to=2-3]
	\arrow[from=1-1, to=1-2]
	\arrow["{s^{i'}_k}"', from=1-1, to=2-2]
\end{tikzcd}
\end{equation*}
where $s^{i'} = \{s^{i'}_k : W_k \to U'_{i'} \}$ is a $j$-covering family on $U'_{i'}$ which exists since there is a map $U'_{i'} \to U_i$, $(t')^i$ is a covering family and $j$ is a coverage. We can thus take the composite $(r' \circ s)$, since $j$ is composition closed, and obtain a $j$-covering family on $U$ that refines $(r \circ t)$. Thus $(r \circ t)$ is a $\text{ref}(j)$ covering family, and thus $\text{ref}(j)$ is composition closed.
\end{proof}

\section{Local Epimorphisms}

\begin{Def} \label{def local epimorphism of presheaves}
Given a site $(\cat{C},j)$, a morphism $f: X \to Y$ of presheaves over $\cat{C}$ is a \textbf{$j$-local epimorphism}, if for every section $s: y(U) \to Y$ there exists a $j$-tree $T$ on $U$ such that for every $r_i : U_i \to U$ in $T^\circ$, there exists a section $s_i : y(U_i) \to X$ making the following diagram commute
\begin{equation*}
\begin{tikzcd}
	{y(U_i)} & X \\
	{y(U)} & Y
	\arrow["{s_i}", from=1-1, to=1-2]
	\arrow["{r_i}"', from=1-1, to=2-1]
	\arrow["f", from=1-2, to=2-2]
	\arrow["s"', from=2-1, to=2-2]
\end{tikzcd} 
\end{equation*}
In other words, $s_*(T^\circ) \leq f$.
\end{Def}

\begin{Lemma} \label{lem local epi on comp closed coverage}
If $(\cat{C}, j)$ is a composition closed coverage, then a map $f: X \to Y$ of presheaves is a $j$-local epimorphism if and only if for every section $s: y(U) \to Y$ there exists a covering family $r = \{r_i: U_i \to U \}$ and local sections $s_i: y(U_i) \to X$ making the following diagram commute for every $i$
\begin{equation*}
    \begin{tikzcd}
	{y(U_i)} & X \\
	{y(U)} & Y
	\arrow["f", from=1-2, to=2-2]
	\arrow["s"', from=2-1, to=2-2]
	\arrow["{r_i}"', from=1-1, to=2-1]
	\arrow["{s_i}", from=1-1, to=1-2]
\end{tikzcd}
\end{equation*}
i.e. $s_*(r) \leq f$.
\end{Lemma}

\begin{Rem} \label{rem epis are local epis}
Note that any epimorphism of presheaves is a local epimorphism. Let us call a morphism $f: X \to Y$ that satisfies the condition of Lemma \ref{lem local epi on comp closed coverage} a \textbf{strong $j$-local epimorphism}. Clearly strong $j$-local epimorphisms are $j$-local epimorphisms, but the converse is not true, see Remark \ref{rem counterexample of strong local epis}.
\end{Rem}

\begin{Lemma} \label{lem local epimorphisms closed under composition}
Given a site $(\cat{C}, j)$, if $f: X \to Y$ and $g: Y \to Z$ are $j$-local epimorphisms of presheaves on $\cat{C}$, then $gf : X \to Z$ is a $j$-local epimorphism.
\end{Lemma}

\begin{proof}
Given a section $s : y(U) \to Z$, since $g$ is a $j$-local epimorphism, there exists a $j$-tree $T$ on $U$, such that every $t_i \in T^\circ$ factors through $g$, providing the following commutative diagram
\begin{equation*}
   \begin{tikzcd}
	& X \\
	{y(U_i)} & Y \\
	{y(U)} & Z
	\arrow["s"', from=3-1, to=3-2]
	\arrow["f", from=1-2, to=2-2]
	\arrow["g", from=2-2, to=3-2]
	\arrow["{t_i}"', from=2-1, to=3-1]
	\arrow["{s_i}"', from=2-1, to=2-2]
\end{tikzcd} 
\end{equation*}
But then $s_i$ provides a section of $Y$, and since $f$ is a $j$-local epimorphism, there exists a $j$-tree $T_i$ on $U_i$ that factors through $f$. By Lemma \ref{lem j-trees closed under composition}, the composition of $T$ and $T_i$ for every $i$ is itself a $j$-tree, and thus $gf$ is a $j$-local epimorphism.
\end{proof}

\begin{Lemma} \label{lem composition is local epi implies local epi}
Let $f: X \to Y$ and $g : Y \to Z$ be maps of presheaves. If $gf$ is a $j$-local epimorphism, then so is $g$.
\end{Lemma}

\begin{proof}
Given a section $s: y(U) \to Z$, we can find a $j$-tree $T$ on $U$ such that $s_*(T^\circ) \leq gf$. But this means that $s_*(T^\circ) \leq g$, so $g$ is a $j$-local epimorphism.
\end{proof}

\begin{Lemma} \label{lem local epi iff pullback by section local epi}
Given a site $(\cat{C}, j)$ a map $f : X \to Y$ of presheaves on $\cat{C}$ is a $j$-local epimorphism if and only if for every section $s: y(U) \to Y$, the pullback $s^*(f) : y(U) \times_Y X \to y(U)$ is a $j$-local epimorphism.
\end{Lemma}

\begin{proof}
$(\Rightarrow)$ Suppose that $f$ is a $j$-local epimorphism, $s: y(U) \to Y$ is a section and $t : y(V) \to y(U)$ is a section. Then there exists a $j$-tree $T$ on $V$ such that $(st)_*(T^\circ) \leq f$. But this implies that $t_*(T^\circ) \leq s^*(f)$.

$(\Leftarrow)$ Given a section $s: y(U) \to Y$, if the pullback $s^*(f)$ is a $j$-local epimorphism, then taking $1_{y(U)} : y(U) \to y(U)$, there exists a $j$-tree $T$ on $U$ such that $T^\circ \leq s^*(f)$, which is equivalent to $s_*(T^\circ) \leq f$. Thus $f$ is a $j$-local epimorphism.
\end{proof}

\begin{Lemma} \label{lem local epi iff pullback local epi}
Given a site $(\cat{C}, j)$, a map $f: X \to Y$ is a $j$-local epimorphism if and only if for every map $g: Z \to Y$ the pullback $g^*(f) : Z \times_Y X \to Z$ is a $j$-local epimorphism.
\end{Lemma}

\begin{proof}
$(\Rightarrow)$ By Lemma \ref{lem local epi iff pullback by section local epi} the map $g^*(f)$ is a $j$-local epimorphism if and only if $s^*(g^*(f))$ is a $j$-local epimorphism for every section $s : y(U) \to Z$. But in the following commutative diagram, the outer rectangle is a pullback
\begin{equation*}
    \begin{tikzcd}
	{y(U) \times_Y X} & {Z \times_Y X} & X \\
	{y(U)} & Z & Y
	\arrow[from=1-1, to=1-2]
	\arrow["{s^*(g^*(f))}"', from=1-1, to=2-1]
	\arrow["\lrcorner"{anchor=center, pos=0.125}, draw=none, from=1-1, to=2-2]
	\arrow[from=1-2, to=1-3]
	\arrow["{g^*(f)}", from=1-2, to=2-2]
	\arrow["\lrcorner"{anchor=center, pos=0.125}, draw=none, from=1-2, to=2-3]
	\arrow["f", from=1-3, to=2-3]
	\arrow["s"', from=2-1, to=2-2]
	\arrow["g"', from=2-2, to=2-3]
\end{tikzcd}
\end{equation*}
and $gs : y(U) \to Y$ is a section. Thus again by Lemma \ref{lem local epi iff pullback by section local epi}, this implies that $f$ is a $j$-local epimorphism.
$(\Leftarrow)$ Taking $g = 1_Y$ shows that $f$ is a $j$-local epimorphism.
\end{proof}

\begin{Not}
Given a map $f: X \to Y$ in a category $\cat{C}$ with pullbacks, let $\Delta_f: X \to X \times_Y X$ denote the \textbf{slice diagonal map}, which is the unique map obtained by the universal property of pullbacks
\begin{equation*}
\begin{tikzcd}
	X \\
	& {X \times_Y X} & X \\
	& X & Y
	\arrow["{\pi_0}"', from=2-2, to=3-2]
	\arrow["f"', from=3-2, to=3-3]
	\arrow["{\pi_1}", from=2-2, to=2-3]
	\arrow["f", from=2-3, to=3-3]
	\arrow["{1_X}"', curve={height=12pt}, from=1-1, to=3-2]
	\arrow["{1_X}", curve={height=-12pt}, from=1-1, to=2-3]
	\arrow["{\Delta_f}"{description}, dashed, from=1-1, to=2-2]
\end{tikzcd}
\end{equation*}
Note that this map is a monomorphism.
\end{Not}

\begin{Lemma} \label{lem map is mono iff slice diagonal iso}
A map $f: X \to Y$ in a category $\cat{C}$ with pullbacks is a monomorphism if and only if the slice diagonal $\Delta_f: X \to X \times_Y X$ is an isomorphism.
\end{Lemma}

\begin{proof}
First let us show that $f$ is a monomorphism if and only if every map $h: Z \to X \times_Y X$ factors through the slice diagonal map $\Delta_f$. Suppose that $f$ is a mono. Let $h_0 = \pi_0 h$ and $h_1 = \pi_1 h$. Then we have $f h_0 = f h_1$, and since $f$ is a mono, this implies that $h_0 = h_1$. Now $h$ factors through $\Delta_f$ because $\Delta_f h_0 = \Delta_f pi_0 h = h$. Conversely, if every map to $X\times_Y X$ factors through $\Delta_f$, then any pair of maps $h_0, h_1 : Z \to X$ such that $f h_0 = f h_1$ determines a map $h : Z \to X \times_Y X$ that factors through $\Delta_f$ which implies that $h_0 = h_1$, so $f$ is a mono.

Now let us show that every map $h: Z \to X \times_Y X$ factors through the diagonal $\Delta_f$ if and only if $\Delta_f$ is an isomorphism. Suppose every map to $X \times_Y X$ factors through $\Delta_f$. Then $\cat{C}(Z, \Delta_f) : \cat{C}(Z, X) \to \cat{C}(Z, X \times_Y X)$ is injective because $\Delta_f$ is a mono, and it is surjective by assumption. Thus by the Yoneda Lemma $\Delta_f$ is an isomorphism. Conversely, if $\Delta_f$ is an isomorphism, then obviously every map to $X \times_Y X$ factors through $\Delta_f$ by its inverse.
\end{proof}

\begin{Def}
Given a site $(\cat{C}, j)$, we say a map $f: X \to Y$ of presheaves on $\cat{C}$ is a \textbf{$j$-local monomorphism} if the slice diagonal map $\Delta_f: X \to X \times_Y X$ is a $j$-local epimorphism. We say that $f$ is a $j$-\textbf{local isomorphism} if it is both a $j$-local monomorphism and a $j$-local epimorphism.
\end{Def}

\begin{Rem}
Similarly to Remark \ref{rem epis are local epis} all isomorphisms and monomorphisms of presheaves are local isomorphisms and local monomorphisms, respectively.
\end{Rem}

\begin{Lemma} \label{lem characterization of local mono}
Given a site $(\cat{C}, j)$, a map $f: X \to Y$ of presheaves on $\cat{C}$ is a $j$-local monomorphism if and only if for every pair of sections $s, s' : y(U) \to X$ such that $f(s) = f(s')$, there exists a $j$-tree $T$ on $U$ such that for every $r_i : U_i \to U$ in $T^\circ$, the restrictions of $s, s'$ are equal, i.e. $X(r_i)(s) = X(r_i)(s')$.
\end{Lemma}

\begin{Lemma} \label{lem local monos closed under composition}
Given a site $(\cat{C}, j)$, if $f: X \to Y$ and $g: Y \to Z$ are maps of presheaves that are additionally $j$-local monomorphisms, then $gf$ is a $j$-local monomorphism.
\end{Lemma}

\begin{proof}
Suppose that $s, s' : y(U) \to X$ are sections such that $gf(s) = gf(s')$. Then since $g$ is a $j$-local monomorphism and $f(s), f(s') : y(U) \to Y$ are sections such that $g(f(s)) = g(f(s'))$, there exists a $j$-tree $T$ on $U$ such that for every $r_i : U_i \to U$ in $T^\circ$, $X(r_i)(f(s)) = X(r_i)(f(s'))$. Thus $f(X(r_i)(s)) = f(X(r_i)(s'))$. Since $f$ is a $j$-local monomorphism, this implies that there exists a $j$-tree $T^i$ such that for every $t^i_j : V^i_j \to U_i$, $X(t^i_j)X(r_i)(s) = X(t^i_j)X(r_i)(s')$. Composing $T$ with all of the $T^i$, we obtain a $j$-tree $T'$ such that for every $r_i t^i_j \in (T')^\circ$, $X(r_i t^i_j)(s) = X(r_i t^i_j)(s')$. Thus $gf$ is a $j$-local monomorphism.
\end{proof}

\begin{Cor} \label{cor local isos closd under composition}
Given a site $(\cat{C}, j)$, the class of $j$-local isomorphisms is closed under composition.
\end{Cor}
 
\begin{Lemma}[{\cite[Lemma 16.2.4]{kashiwara2006}}] \label{lem properties of local monos and local isos}
Given a site $(\cat{C}, j)$ and maps $f: X \to Y$, $g :Y \to Z$ of presheaves on $\cat{C}$:
\begin{enumerate}
    \item if $f$ is a $j$-local monomorphism, then for every morphism $h : W \to Y$, the pullback $h^*(f) : W \times_Y X \to W$ is a $j$-local monomorphism,
    \item if $f$ is a $j$-local isomorphism, then for every morphism $h : W \to Y$, the pullback $h^*(f) : W \times_Y X \to W$ is a $j$-local isomorphism,
    \item if $gf$ is a $j$-local epimorphism, and $g$ is a $j$-local monomorphism, then $f$ is a $j$-local epimorphism,
    \item if $gf$ is a $j$-local monomorphism, and $f$ is a $j$-local epimorphism, then $g$ is a $j$-local monomorphism,
    \item if $gf$ is a $j$-local monomorphism, then $f$ is a $j$-local monomorphism,
    \item if $f: X \to Y$ and $g: Y \to Z$ are maps of presheaves on $\cat{C}$, then if any two morphisms in $\{f, g, gf \}$ are $j$-local isomorphisms, then all of them are $j$-local isomorphisms.
\end{enumerate}
\end{Lemma}

\begin{proof}
(1) We have the following pasting diagram of pullback squares
\begin{equation*}
\begin{tikzcd}
	{W \times_Y X} & X \\
	{W \times_Y (X \times_Y X) \cong (W \times_Y X) \times_W (W \times_Y X)} & {X \times_Y X} \\
	W & Y
	\arrow[from=1-1, to=1-2]
	\arrow[from=1-1, to=2-1]
	\arrow["{\Delta_f}"', from=1-2, to=2-2]
	\arrow["f", curve={height=-24pt}, from=1-2, to=3-2]
	\arrow[from=2-1, to=2-2]
	\arrow[from=2-1, to=3-1]
	\arrow[from=2-2, to=3-2]
	\arrow[""{name=0, anchor=center, inner sep=0}, "h"', from=3-1, to=3-2]
	\arrow["\lrcorner"{anchor=center, pos=0.125}, draw=none, from=1-1, to=0]
	\arrow["\lrcorner"{anchor=center, pos=0.125}, draw=none, from=2-1, to=0]
\end{tikzcd}
\end{equation*}
and since the upper left map is $\Delta_{h^*(f)}$, and $\Delta_f$ is a $j$-local epi by assumption, this implies $\Delta_{h^*(f)}$ is a $j$-local epi by Lemma \ref{lem local epi iff pullback local epi}. Thus $h^*(f)$ is a $j$-local monomorphism.

(2) This follows from (1) and Lemma \ref{lem local epi iff pullback local epi}.

(3) We have the following pasting diagram of pullback squares
\begin{equation*}
    \begin{tikzcd}
	X & Y \\
	{X \times_Z Y} & {Y \times_Z Y} & Y \\
	X & Y & Z
	\arrow["f", from=1-1, to=1-2]
	\arrow["q"', from=1-1, to=2-1]
	\arrow["\lrcorner"{anchor=center, pos=0.125}, draw=none, from=1-1, to=2-2]
	\arrow["{\Delta_g}", from=1-2, to=2-2]
	\arrow[Rightarrow, no head, from=1-2, to=2-3]
	\arrow["k"', from=2-1, to=2-2]
	\arrow[from=2-1, to=3-1]
	\arrow["\lrcorner"{anchor=center, pos=0.125}, draw=none, from=2-1, to=3-2]
	\arrow["\ell"', from=2-2, to=2-3]
	\arrow[from=2-2, to=3-2]
	\arrow["\lrcorner"{anchor=center, pos=0.125}, draw=none, from=2-2, to=3-3]
	\arrow["g", from=2-3, to=3-3]
	\arrow["f"', from=3-1, to=3-2]
	\arrow["g"', from=3-2, to=3-3]
\end{tikzcd}
\end{equation*}
Now $\ell k$ is a $j$-local epi since $gf$ is, and similarly $q$ is a $j$-local epi since $\Delta_g$ is. Thus $f$ is a $j$-local epi.

(4) We have the following pasting diagram of pullback squares
\begin{equation*}
    \begin{tikzcd}
	{X \times_Z X} & {Y \times_ZX} & X \\
	{X \times_Z Y} & {Y \times_Z Y} & Y \\
	X & Y & Z
	\arrow["{\ell'}", from=1-1, to=1-2]
	\arrow["{k'}"', from=1-1, to=2-1]
	\arrow["\lrcorner"{anchor=center, pos=0.125}, draw=none, from=1-1, to=2-2]
	\arrow[from=1-2, to=1-3]
	\arrow["k"', from=1-2, to=2-2]
	\arrow["\lrcorner"{anchor=center, pos=0.125}, draw=none, from=1-2, to=2-3]
	\arrow["f", from=1-3, to=2-3]
	\arrow["\ell", from=2-1, to=2-2]
	\arrow[from=2-1, to=3-1]
	\arrow["\lrcorner"{anchor=center, pos=0.125}, draw=none, from=2-1, to=3-2]
	\arrow[from=2-2, to=2-3]
	\arrow[from=2-2, to=3-2]
	\arrow["\lrcorner"{anchor=center, pos=0.125}, draw=none, from=2-2, to=3-3]
	\arrow["g", from=2-3, to=3-3]
	\arrow["f"', from=3-1, to=3-2]
	\arrow["g"', from=3-2, to=3-3]
\end{tikzcd}
\end{equation*}
and $\ell, \ell', k, k'$ are all $j$-local epimorphisms since $f$ is. Thus $q = \ell k' = \ell' k$ is a $j$-local epi. Note that the following square commutes
\begin{equation*}
    \begin{tikzcd}
	X & {X \times_Z X} \\
	Y & {Y \times_Z Y}
	\arrow["{\Delta_{gf}}", from=1-1, to=1-2]
	\arrow["f"', from=1-1, to=2-1]
	\arrow["q", from=1-2, to=2-2]
	\arrow["{\Delta_g}"', from=2-1, to=2-2]
\end{tikzcd}
\end{equation*}
But $q$ and $\Delta_{gf}$ are $j$-local epimorphisms, therefore by Lemma \ref{lem composition is local epi implies local epi}, this implies that $\Delta_g$ is a $j$-local epi, so $g$ is a $j$-local mono.

(5) Note that the following commutative square is a pullback
\begin{equation*}
\begin{tikzcd}
	X & X \\
	{X \times_Y X} & {X \times_Z X}
	\arrow[Rightarrow, no head, from=1-1, to=1-2]
	\arrow["{\Delta_f}"', from=1-1, to=2-1]
	\arrow["\lrcorner"{anchor=center, pos=0.125}, draw=none, from=1-1, to=2-2]
	\arrow["{\Delta_{gf}}", from=1-2, to=2-2]
	\arrow["k", from=2-1, to=2-2]
\end{tikzcd}
\end{equation*}
where $k$ is the canonical map induced by the universal property of the pullback. Thus $\Delta_f$ is a $j$-local epi since $\Delta_{gf}$ is. 

(6) By Corollary \ref{cor local isos closd under composition}, if $f$ and $g$ are $j$-local isos, then so is $gf$.

Now suppose that $gf$ and $g$ are $j$-local isos. Then by (3), $f$ is a $j$-local epi, and by (4), $f$ is a $j$-local mono and therefore a $j$-local iso.

Finally suppose that $gf$ and $f$ are $j$-local isos. By Lemma \ref{lem composition is local epi implies local epi}, this implies that $g$ is a $j$-local epi, and by (5), this implies that $g$ is a $j$-local mono.
\end{proof}

\section{Saturated and Grothendieck Coverages} \label{section saturated and grothendieck coverages}

\subsection{Saturated Coverages}

In Section \ref{section coverage closures}, we considered refinement and composition closed coverages and their corresponding closure operations. In this section we will consider coverages closed under both conditions and compare such coverages with Grothendieck coverages/topologies.

\begin{Def}[{\cite[Definition A.2.8(d)]{low2016categories}}]
We say that a coverage $j$ on a category $\cat{C}$ is \textbf{saturated} if it is refinement and composition closed.
\end{Def}

\begin{Rem}
Our saturated coverages are precisely Shulman's weakly $\kappa$-ary topologies for $\kappa$ the cardinality of the Grothendieck universe $\mathbb{U}$ (see Definition \ref{def grothendieck universe}), though we require the underlying categories to be small, see \cite[Definition 3.1]{shulman2012exact}.
\end{Rem}

\begin{Ex} \label{ex open site is saturated}
For any topological space $X$, the open cover coverage $j_X$ on $\mathcal{O}(X)$ is saturated. This is because if $\mathcal{U} = \{ U_i \subseteq U\}$ is an open cover of an open subset $U \subseteq X$, and $\mathcal{V} = \{V^i_j \subseteq U_i \}$ is an open cover of $U_i$, then $(\mathcal{U} \setminus \{U_i \}) \cup \mathcal{V}$ is an open cover of $U$, so $j_X$ is composition closed. Now if $\mathcal{U}$ and $\mathcal{V}$ are collections of open subsets of $U \subseteq X$ such that $\mathcal{V}$ is an open cover and there is a refinement $\mathcal{U} \leq \mathcal{V}$, then $\mathcal{U}$ must be an open cover of $U$ as well. Thus $j_X$ is refinement closed, and therefore it is saturated.
\end{Ex}

\begin{Ex} \label{ex jointly epi coverage is saturated}
The jointly epimorphic coverage  $(\ncat{FinSet}, j_{\text{epi}})$ from Example \ref{ex set joint epi coverage} is saturated.
\end{Ex}

\begin{Def}[{\cite[Definition A.2.15(d)]{low2016categories}}] \label{def saturating family}
Let $(\cat{C}, j)$ be a site. Say that a family of morphisms $r = \{r_i: U_i \to U \}$ on $U$ is $j$-\textbf{saturating} if the map
\begin{equation}
    \overline{r} \hookrightarrow y(U)
\end{equation}
given by the inclusion of the sieve generated by $r$ to the maximal sieve on $U$ is a $j$-local epimorphism (Definition \ref{def local epimorphism of presheaves}).
\end{Def}

\begin{Lemma} \label{lem family saturating iff j-tree coverage cond holds}
Let $(\cat{C}, j)$ be a site. A family of morphisms $r = \{r_i : U_i \to U \}$ is saturating if and only if for any map $g: V \to U$, there exists a $j$-tree $T$ on $V$ and a refinement $g_*(T^\circ) \leq r$.
\end{Lemma}

\begin{proof}
Follows immediately from the definitions.
\end{proof}

\begin{Lemma} \label{lem family saturating iff refined by j-tree}
Let $(\cat{C}, j)$ be a site. A family of morphisms $r = \{r_i : U_i \to U \}$ is saturating if and only if there exists a $j$-tree $T$ on $U$ and a refinement $T^\circ \leq r$.
\end{Lemma}

\begin{proof}
$(\Rightarrow)$ Suppose that $r$ is saturating. Then if we let $g = 1_U$, then Lemma \ref{lem family saturating iff j-tree coverage cond holds} implies that there exists a $j$-tree $T$ on $U$ and a refinement $f : T^\circ \to r$.

$(\Leftarrow)$ Suppose that $r$ is refined by a $j$-tree $T$, and suppose that $g: V \to U$ is a morphism in $\cat{C}$, then by Lemma \ref{lem can pullback j-trees}, there exists a $j$-tree $S$ on $V$ and a refinement $g_*(S^\circ) \leq T^\circ$, which implies that $g_*(S^\circ) \leq r$. Thus by Lemma \ref{lem family saturating iff j-tree coverage cond holds}, $r$ is saturating.
\end{proof}

\begin{Cor}
If $(\cat{C}, j)$ is composition closed, then a family of morphisms $t$ on $U$ is saturating if and only if there exists a covering family $r$ on $U$ and a refinement $r \leq t$.
\end{Cor}

\begin{Def}\label{def saturation closure}
Given a site $(\cat{C}, j)$, let $\sat{j}$ denote the collection of families on $\cat{C}$ where $\sat{j}(U)$ is the set of saturating families of morphisms on $U$. We call this the \textbf{saturation closure} of $j$.
\end{Def}

\begin{Lemma} \label{lem sat closure is comp then ref closure}
Given a site $(\cat{C}, j)$ the saturation closure $\sat{j}$ is precisely the refinement closure of the composition closure of $j$:
\begin{equation*}
    \sat{j} = \text{ref}(\text{comp}(j)).
\end{equation*}
\end{Lemma}

\begin{proof}
This follows from Lemma \ref{lem family saturating iff refined by j-tree}.
\end{proof}

\begin{Prop}[{\cite[Proposition A.2.15]{low2016categories}}] \label{prop sat closure smallest saturated coverage}
Given a site $(\cat{C}, j)$, the collection of families $\sat{j}$ is the smallest saturated coverage on $\cat{C}$ containing $j$. 
\end{Prop}

\begin{proof}
By Lemma \ref{lem sat closure is comp then ref closure}, $\sat{j}$ is clearly refinement closed, and by Lemma \ref{lem ref closure preserves composition closure}, $\sat{j}$ is also composition closed. Thus $\sat{j}$ is saturated. If $j'$ is a saturated coverage containing $j$, then by Lemma \ref{lem comp(j) smallest composition closed coverage}, $\text{comp}(j) \subseteq j'$. But by Lemma \ref{lem refinement closure smallest refinement closed coverage}, it follows that $\sat{j} \subseteq j'$.
\end{proof}

\begin{Cor} \label{cor saturating iff covering}
A coverage $j$ on a category $\cat{C}$ is saturated if and only if $\sat{j} = j$. Thus if $j$ is a saturated coverage, then a family $r$ on $U$ is covering if and only if it is saturating.
\end{Cor}

\begin{Prop} \label{prop sheaf iff sheaf on saturation closure}
Given a site $(\cat{C}, j)$, a presheaf $X$ is a $j$-sheaf if and only if it is a $\sat{j}$-sheaf. In other words
$$\Sh(\cat{C}, j) = \Sh(\cat{C}, \sat{j}).$$
\end{Prop}

\begin{proof}
This follows from Proposition \ref{prop sheaf iff sheaf on composition closure}, Proposition \ref{prop sheaf iff sheaf on refinement closure} and Lemma \ref{lem sat closure is comp then ref closure}.
\end{proof}

\begin{Lemma} \label{lem local epi iff local epi on saturation closure}
Let $j$ be a coverage on $\cat{C}$. Then a morphism $f : X \to Y$ of presheaves on $\cat{C}$ is a $j$-local epimorphism (mono/iso) if and only if it is a $\sat{j}$-local epimorphism (mono/iso).
\end{Lemma}

\begin{proof}
$(\Rightarrow)$ Suppose that $f: X \to Y$ is a $j$-local epimorphism. Then for every $s : y(U) \to Y$, there exists a $j$-tree $T$ on $U$ such that $s_*(T^\circ) \leq f$. But then $T^\circ$ is a covering family in $\sat{j}$. Thus by Lemma \ref{lem local epi on comp closed coverage}, $f$ is a $\sat{j}$-local epimorphism.

$(\Leftarrow)$ If $f: X \to Y$ is a $\sat{j}$-local epimorphism, then for every $s: y(U) \to Y$, there exists a covering family $r \in \sat{j}(U)$ such that $s_*(r) \leq \{f\}$. But by Lemma \ref{lem family saturating iff refined by j-tree}, $r$ is a $\sat{j}$-covering family if and only if it is $j$-saturating. This implies that there exists a $j$-tree $T$ on $U$ such that $T^\circ \leq r$. Thus $s_*(T^\circ) \leq \{f\}$, so $f$ is a $j$-local epimorphism.
\end{proof}

\begin{Lemma} \label{lem saturating iff sifted closure is saturating}
Let $(\cat{C}, j)$ be a site. If $r$ is a family on $U$, then $r$ is saturating if and only if $\overline{r}$ is saturating. Furthermore every $j$-covering family is saturating. Therefore $j \subseteq \sat{j}$ and $\overline{j} \subseteq \sat{j}$.
\end{Lemma}

\begin{proof}
$(\Rightarrow)$ This follows just from the definition of coverage (Definition \ref{def coverage}) and $j$-local epimorphism (Definition \ref{def local epimorphism of presheaves}). 

$(\Leftarrow)$ If $\overline{r}$ is saturating, then since $\overline{\overline{r}} = \overline{r}$, this implies that $\overline{r} \hookrightarrow y(U)$ is a $j$-local epimorphism, and thus $r$ is saturating. Now if $r \in j$ is a covering family, then it is saturating by taking $r$ to refine itself by the identiy in Lemma \ref{lem family saturating iff refined by j-tree}.
\end{proof}

\begin{Cor} \label{cor saturated coverages contain their sifted closure}
If $j$ is a saturated coverage, then $\overline{j} \subseteq j$.
\end{Cor}

\begin{Lemma} \label{lem covering families closed under intersection in saturated coverages}
If $j$ is a saturated coverage on a category $\cat{C}$, $U \in \cat{C}$, and $r, r' \in j(U)$, then there exists a covering family $r'' \in j(U)$ and refinements $r'' \leq r$ and $r'' \leq r'$. In other words, $j(U)$ is a finitely cofiltered (Definition \ref{def filtered category}) poset.
\end{Lemma}

\begin{proof}
Since covering families are $j$-trees, by Corollary \ref{lem j-trees have common refinement}, there exists a $j$-tree $R$ on $U$ and refinements $R^\circ \leq r$ and $R^\circ \leq r'$. But $j$ is saturated, so by Lemma \ref{lem composition closed iff j-tree closed}, $r'' = R^\circ$ is a covering family that refines both $r$ and $r'$.
\end{proof}

\begin{Lemma} \label{lem family is saturating iff sum of maps is local epi}
Given a family of morphisms $r = \{ r_i : U_i \to U \}_{i \in I}$ in a site $(\cat{C}, j)$, the corresponding map $\sum_i r_i : \sum_{i \in I} y(U_i) \to y(U)$ is a $j$-local epimorphism if and only if the corresponding family $r$ is $j$-saturating.
\end{Lemma}

\begin{proof}
By Proposition \ref{prop sieves are coequalizers of generating family} we have the following commutative diagram
\begin{equation*}
    \begin{tikzcd}
	{\sum_{i,j} y(U_i) \times_{y(U)} y(U_j)} & {\sum_i y(U_i)} & {\overline{r}} \\
	&& {y(U)}
	\arrow["\pi", from=1-2, to=1-3]
	\arrow[shift right=2, from=1-1, to=1-2]
	\arrow[shift left=2, from=1-1, to=1-2]
	\arrow["{\sum_i r_i}"', from=1-2, to=2-3]
	\arrow[hook, from=1-3, to=2-3]
\end{tikzcd}
\end{equation*}
where $\pi$ is a coequalizer, and hence an epimorphism. Thus if $r$ is saturating, then $\overline{r} \hookrightarrow y(U)$ is a $j$-local epi, which by Lemma \ref{lem local epimorphisms closed under composition} implies that $\sum_i r_i$ is a $j$-local epi. Conversely, if $\sum_i r_i$ is a $j$-local epi, then by Lemma \ref{lem composition is local epi implies local epi}, $\overline{r} \hookrightarrow y(U)$ is a $j$-local epi, and thus $r$ is saturating.
\end{proof}

\begin{Def} \label{def sifted saturated coverage}
Let $J$ be a sifted coverage on a category $\cat{C}$. We say that $J$ is a \textbf{sifted-saturated coverage} if it is composition closed and it satisfies the following weaker version of refinement closure which we call \textbf{sifted refinement closure}: If $R$ is a covering sieve on $U$, and $T$ is a sieve on $U$ such that $R \subseteq T$, then $T$ is a covering sieve on $U$.
\end{Def}

\begin{Lemma} \label{lem sifted closure of saturated is sifted saturated}
Let $(\cat{C}, j)$ be a site. If $j$ is a saturated coverage, then $\overline{j}$ is a sifted-saturated coverage.
\end{Lemma}

\begin{proof}
Suppose that $R \in \overline{j}(U)$ is a covering sieve, $T$ is a sieve on $U$ and $R \subseteq T$. By definition of $\overline{j}(U)$, there exists a covering family $r = \{r_i : U_i \to U \} \in j(U)$ such that $R = \overline{r}$. Thus there exists a refinement $r \leq R$. Since $j$ is refinement closed, this implies that $R \in j(U)$, and therefore $T \in j(U)$. Since $T = \overline{T}$, this implies that $T \in \overline{j}(U)$. Thus $\overline{j}$ is refinement closed. Now suppose that for every $f : V \to U \in R$, there exists a $\overline{j}$-covering sieve $T_f \in J(V)$ with $j$-covering family $t^f = \{t^f_j : V_j \to V\}$ such that $T_f = \overline{t^f}$. We want to show that $(R \circ T) = \bigcup_{f \in R} f_*(T_f)$ is a $\overline{j}$-covering sieve on $U$. Now it is easy to see that $R = \bigcup_i (r_i)_*(y(U_i))$ and $T_f = \bigcup_j (t^f_j)_*(y(V_j))$. Since $j$ is saturated, and $(1_W) \in j(W)$ for every $W \in \cat{C}$, then by Lemma \ref{lem saturating iff sifted closure is saturating}, this implies that $y(W) \in j(W)$. Thus $(R \circ T) = (r \circ y_r \circ t \circ y_t)$, where we are letting $y_r$ and $y_t$ denote the collection of maximal sieves $y(W)$ for every $W$ in the domains of $r$ and $t$ respectively. But $j$ is composition closed so $(r \circ y_r \circ t \circ y_t) = (R \circ T)$ is a covering sieve. Thus $\overline{j}$ is sifted-saturated.
\end{proof}

\subsection{Grothendieck Coverages}

We will now consider the more traditional notion of a Grothendieck coverage. The work of the previous section will be used to show how a coverage generates a Grothendieck coverage, and this is useful in applications. However the main usage of Grothendieck coverages is in proving theorems about Grothendieck toposes.

\begin{Def} \label{def Grothendieck coverage}
A \textbf{Grothendieck coverage}\footnote{More commonly known as a \textbf{Grothendieck topology}} on an (essentially small) category $\cat{C}$ is a sifted collection of families $J$ on $\cat{C}$ such that
\begin{itemize}
	\item[(G1)] $y(U) \in J(U)$ for every $U \in \cat{C}$,
	\item[(G2)] for any sieve $R \in J(U)$ and any morphism $g: V \to U$, $g^*(R) \in J(V)$, and
	\item[(G3)] if $R \in J(U)$, $R'$ is a sieve on $U$, and $g^*(R') \in J(V)$ for every $g \in R(V)$, then $R' \in J(U)$.
\end{itemize}
We refer to a category $\cat{C}$ equipped with a Grothendieck coverage a \textbf{Grothendieck site}. Note that (G1) and (G2) imply that a Grothendieck coverage is a sifted coverage.
\end{Def}

\begin{Lemma}[{\cite[Lemma 3.12]{jardine2015local}}] \label{lem covering sieve props}
Let $(\cat{C}, J)$ be a Grothendieck site. Then the following hold:
\begin{enumerate}
	\item if $R,R'$ are sieves on $U$, $R \subseteq R'$ and $R$ is a covering sieve, then $R'$ is a covering sieve,
	\item if $R,R'$ are covering sieves, then $R \cap R'$ is a covering sieve.
\end{enumerate}
\end{Lemma}

\begin{proof}
(1) Let $g : V \to U \in R$, then by Lemma \ref{lem pullback of sieve by map in sieve is maximal sieve}, we know that $g^*R = g^*R' = y(V)$, which is a covering sieve of $V$ by (G1). Since this is true for all $g \in R$, then $R'$ is a covering sieve by (G3).

(2) Let $g \in R$, then $g^* (R \cap R') = g^* R \cap g^* R' = y(V) \cap g^* R' = g^* R'$ by Lemma \ref{lem pullback of sieve by map in sieve is maximal sieve}. By $(G2)$, $g^* R' \in J(V)$, thus by $(G3)$, $R \cap R'$ is a covering sieve.
\end{proof}

\begin{Lemma} \label{lem intersections of grothendieck coverages}
Given a small category $\cat{C}$ and a set $\{J_i \}_{i \in I}$ of Grothendieck coverages on $\cat{C}$, let $J = \bigcap_{i \in I} J_i$ denote the sifted collection of families where $R \in J(U)$ if and only if $R \in J_i(U)$ for all $i \in I$. Then $J$ is a Grothendieck coverage.
\end{Lemma}

\begin{proof}
Since $y(U) \in J_i(U)$ for all $i \in I$, $J$ satisfies (G1). If $R \in J(U)$, then $R \in J_i(U)$ for all $i$, and since each $J_i$ is a Grothendieck coverage $g^*(R) \in J_i(V)$ for every map $g : V\to U$, so $g^*(R) \in J(V)$, so $J$ satisfies (G2). Similar reasoning shows that $J$ satisfies (G3).
\end{proof}

\begin{Rem} \label{rem grothendieck coverage generated by collection of families}
Thanks to Lemma \ref{lem intersections of grothendieck coverages}, we can generate a Grothendieck coverage given any collection of families. Indeed, suppose that $\cat{C}$ is a small category and $j$ is a collection of families on $\cat{C}$. Then let $\Gro{j}$ denote the intersection of all the Grothendieck coverages $J$ such that $j \subseteq J$.
\end{Rem}

\begin{Lemma}[{\cite[Lemma 3.12.3]{jardine2015local}}] \label{lem grothendieck coverages are composition closed}
If $(\cat{C}, J)$ is a Grothendieck site, and $R \in J(U)$ is a covering sieve, such that for every $f: V \to U$ with $f \in R$, there is a covering sieve $T_f \in J(V)$. Then $(R \circ T) = \bigcup_{f \in R} f_*(T_f)$ is a covering sieve. In other words, $J$ is composition closed.
\end{Lemma}

\begin{proof}
First consider $f^* (R \circ T) = \bigcup_{g \in R} f^* g_*(T_g)$. Now if $g = f$, then $T_f \subseteq f^* f_*(T_f)$ Thus $T_f \subseteq f^* (R \circ T)$, and since $T_f$ is a covering sieve, by Lemma \ref{lem covering sieve props}, $f^*(R \circ T)$ is a covering sieve. Since $f^*(R \circ T)$ is a covering sieve for every $f \in R$, by $(G3)$, therefore $(R \circ T)$ is a covering sieve.
\end{proof}

\begin{Rem} \label{rem Grothendieck coverages are not refinement closed}
Note that while Grothendieck coverages are composition closed, they are \textbf{not} refinement closed. Indeed, Grothendieck coverages must be sifted, but if $R$ is a covering sieve of a Grothendieck coverage $J$ over $U$, then there exists a refinement $R \leq (1_U)$, but $(1_U)$ is not a sieve, so it cannot belong to $J(U)$. However, they are sifted refinement closed, which is easy to see. This means that Grothendieck coverages and saturated coverages are two different kinds of mathematical objects, but we will show they are equivalent in a certain sense.
\end{Rem}

\begin{Prop} \label{prop grothendieck coverages are precisely sifted saturated coverages}
Given a category $\cat{C}$, the Grothendieck coverages on $\cat{C}$ are precisely the sifted-saturated coverages on $\cat{C}$.
\end{Prop}

\begin{proof}
Suppose that $J$ is a Grothendieck coverage, then by Lemma \ref{lem covering sieve props} and Lemma \ref{lem grothendieck coverages are composition closed}, $J$ is a sifted-saturated coverage. Conversely suppose that $J$ is a sifted-saturated coverage. We wish to show that $J$ is a Grothendieck coverage. Suppose that $R \in J(U)$ and $g: V \to U$ is a morphism. Since $J$ is a coverage, there exists a covering sieve $T \in J(V)$ such that $g_*(T) \subseteq R$. But by Lemma \ref{lem adjunction on sieve posets}, this is equivalent to $T \subseteq g^*(R)$. Since $J$ is sifted-refinement closed, this implies that $g^*(R) \in J(V)$. Now let us show (G3). Suppose that $R \in J(U)$, $R' \subseteq y(U)$ and for every $g \in R$, $g^*(R') \in J(V)$. Then $(R \circ R^*(R')) = \bigcup_{g \in R} g_* g^*(R')$ is a covering sieve since $J$ is composition closed, and for every $g \in R$, it follows that $g_* g^*(R') \subseteq R'$, thus $(R \circ R^*(R')) \subseteq R'$. But $J$ is sifted-refinement closed, so $R' \in J(V)$.
\end{proof}

\begin{Lemma} \label{lem covering sieve iff local iso iff local epi}
Suppose that $(\cat{C}, J)$ is a Grothendieck site. Then a sieve $R$ on an object $U$ is a covering sieve for $J$ if and only if the inclusion $R \hookrightarrow y(U)$ is a $J$-local isomorphism if and only if it is a $J$-local epimorphism.
\end{Lemma}

\begin{proof}
$(\Rightarrow)$ Suppose that $R$ is a $J$-covering sieve on $U$. Let us show that $R \hookrightarrow y(U)$ is a $J$-local epimorphism. If $x : y(V) \to y(U)$ is a section, then since $R \in J(U)$ is a covering family and $J$ is a coverage, then there exists a covering family $S \in J(V)$ such that for every $s_i : V_i \to V$ in $S$, it factors through some map in $R$. Thus $R \hookrightarrow y(U)$ is a $J$-local epimorphism. Since it is a monomorphism, this implies that it is also a $J$-local isomorphism.

$(\Leftarrow)$ Suppose that $R \hookrightarrow y(U)$ is a $J$-local epimorphism. Then it is a $J$-local isomorphism. We want to show that $R$ is a $J$-covering sieve on $U$. Given any map $f : V \to U$, since $R \hookrightarrow y(U)$ is a $J$-local epimorphism, there is a $J$-covering sieve $S_f$ on $V$ such that $S_f \subseteq f^*(R)$. Thus $f^*(R)$ is a $J$-covering sieve for every $f$ in the maximal sieve $y(U)$. Thus by (G3), $R \in J(U)$.
\end{proof}

Recall that if $f : X \to Y$ is a map of presheaves, then $f$ factors through its image presheaf $\im(f)$ (Definition \ref{def image presheaf}).

\begin{Lemma} \label{lem local epi on groth cvg iff covering sieve}
Given a Grothendieck site $(\cat{C}, J)$, a map $f: X \to y(U)$ of presheaves is a $J$-local epimorphism if and only if $\im(f) \hookrightarrow y(U)$ is a $J$-covering sieve.
\end{Lemma}

\begin{proof}
 $(\Rightarrow)$ Suppose that $f : X \to y(U)$ is a $J$-local epimorphism, and we factor it as
 \begin{equation*}
 X \xrightarrow{e_f} \im(f) \xhookrightarrow{i_f} y(U).
 \end{equation*}
 Then by Lemma \ref{lem composition is local epi implies local epi}, $i_f$ is a $J$-local epimorphism. Thus by Lemma \ref{lem covering sieve iff local iso iff local epi}, $\im(f)$ is a $J$-covering sieve. 
 
 $(\Leftarrow)$ Conversely if $\im(f) \hookrightarrow y(U)$ is a $J$-covering sieve, then by Lemma \ref{lem covering sieve iff local iso iff local epi}, $i_f$ is a $J$-local epimorphism, and $e_f$ is an epimorphism, hence a $J$-local epimorphism, and thus by Lemma \ref{lem local epimorphisms closed under composition}, $f$ is a $J$-local epimorphism.
\end{proof}

\begin{Lemma} \label{lem local epi on groth cvg iff pullback is local epi}
Given a Grothendieck site $(\cat{C}, J)$, a map $f: X \to Y$ of presheaves is a $J$-local epimorphism if and only if for every section $s: y(U) \to Y$, the pullback $s^*(f) : y(U) \times_Y X \to y(U)$ is a $J$-local epimorphism.
\end{Lemma}

\begin{proof}
This is a special case of Lemma \ref{lem local epi iff pullback by section local epi}.
\end{proof}

\subsection{Comparison of Saturated and Grothendieck Coverages}
In this section we compare saturated coverages and Grothendieck coverages, showing that for any category $\cat{C}$, there is an isomorphism between the corresponding posets of saturated and Grothendieck coverages. This leads to Corollary \ref{cor grothendieck closure preserves sheaves} showing that to every coverage $j$ there exists a unique smallest Grothendieck coverage containing $\overline{j}$ and with the same sheaves as $j$.

\begin{Lemma} \label{lem sifted closure of saturated coverage is Grothendieck}
Let $j$ be a saturated coverage on a category $\cat{C}$. Then $\overline{j}$ is a Grothendieck coverage.
\end{Lemma}

\begin{proof}
This follows from Lemma \ref{lem sifted closure of saturated is sifted saturated}.
\end{proof}

\begin{Def} \label{def interior coverage}
Suppose that $J$ is a sifted coverage on $\cat{C}$, let $J^\circ$ denote the collection of families on $\cat{C}$ where a family $r$ on $U$ belongs to $J^\circ(U)$ if $\overline{r} \in J(U)$. We call this the \textbf{interior coverage} of $J$.
\end{Def}

\begin{Lemma} \label{lem interior of Grothendieck coverage is saturated}
If $J$ is a Grothendieck coverage on a category $\cat{C}$, then $J^\circ$ is a saturated coverage.
\end{Lemma}

\begin{proof}
Suppose that $r \in J^\circ(U)$ and $r \leq t$. Then $\overline{r} \in J(U)$ and $\overline{r} \subseteq \overline{t}$. Thus $\overline{t} \in J(U)$, so $t \in J^\circ(U)$. Thus $J^\circ$ is refinement closed. Now given a $J^\circ$-covering family $r = \{ r_i : U_i \to U \}_{i \in I}$ on $U$ and for each $i$ a $J^\circ$-covering family $t^i$ on $U_i$, then we want to show that $\overline{(r \circ t)} \in J(U)$. Now note that $\{t^i \}_{i \in I}$ is an $I$-indexed family of $J^\circ$-covering families. Let $\overline{I}$ denote the indexing set for $\overline{r}$. We wish to extend $\{ \overline{t^i} \}$ to a new set $\{S^j \}$ of $J^\circ$-covering sieves indexed by $j \in \overline{I}$. For $i \in I \subseteq \overline{I}$ let $S^i = \overline{t^i}$. Now if $g \in \overline{r}$, then $g$ factors as $g = r_k g_j$ for some morphism $g_j$ and some covering map $r_k$. Then set $S^j = g_j^*(\overline{t^k})$. This is a covering sieve since $J$ is a Grothendieck coverage. Thus $(\overline{r} \circ S)$ is a covering sieve on $U$ and furthermore $(\overline{r} \circ S) \subseteq \overline{(r \circ t)}$. Thus by Lemma \ref{lem covering sieve props}, $\overline{(r \circ t)}$ is a covering sieve. Thus $J^\circ$ is a saturated coverage.
\end{proof}

\begin{Def}
Given a category $\cat{C}$, let $\ncat{SatCvg}(\cat{C})$ denote the (large) poset of saturated coverages equipped with the $\subseteq$ relation. Similarly let $\ncat{GroCvg}(\cat{C})$ denote the (large) poset of Grothendieck coverages. The constructions $\overline{(-)}$ and $(-)^\circ$ defined above can easily be seen to define maps of posets.
\end{Def}

\begin{Prop}
The maps of (large) posets
\begin{equation}
\begin{tikzcd}
	{\ncat{SatCvg}} && {\ncat{GroCvg}}
	\arrow[""{name=0, anchor=center, inner sep=0}, "{{\overline{(-)}}}", curve={height=-18pt}, from=1-1, to=1-3]
	\arrow[""{name=1, anchor=center, inner sep=0}, "{{(-)^\circ}}", curve={height=-18pt}, from=1-3, to=1-1]
	\arrow["\dashv"{anchor=center, rotate=-90}, draw=none, from=0, to=1]
\end{tikzcd}
\end{equation}
form an isomorphism of posets.
\end{Prop}

\begin{proof}
Let us show that $\overline{J^\circ} = J$ and $(\overline{j})^\circ = j$. Clearly $j \subseteq (\overline{j})^\circ$, so let us prove the converse. If $r \in (\overline{j})^\circ$, then $\overline{r} \in \overline{j}$, so by Collorary \ref{cor saturated coverages contain their sifted closure}, $\overline{r} \in j$. But $j$ is saturated, so by Lemma \ref{lem saturating iff sifted closure is saturating}, this implies that $r \in j$. Now for $J$ a Grothendieck coverage, it is clear that $\overline{J^\circ} \subseteq J$, so let us prove the converse. Suppose that $R \in J$ and $r$ is a family such that $\overline{r} = R$. Then $r \in J^\circ$, so $\overline{r} = R \in \overline{J^\circ}$. Thus $J \subseteq \overline{J^\circ}$.
\end{proof}

\begin{Def} \label{def Grothendieck closure of a coverage}
Given a coverage $j$ on a category $\cat{C}$, let 
$$\Gro{j} \coloneqq \overline{\sat{j}}.$$
We call this the \textbf{Grothendieck closure} of $j$. 
\end{Def}

\begin{Lemma} \label{lem sifted closure has same Grothendieck closure}
Given a coverage $j$ on a category $\cat{C}$, we have
\begin{equation}
    \Gro{j} = \Gro{\overline{j}}.
\end{equation}
Furthermore, $\Gro{j}$ is the smallest Grothendieck coverage containing $\overline{j}$.
\end{Lemma}

\begin{proof}
By Lemma \ref{lem saturating iff sifted closure is saturating}, we have that $\overline{j} \subseteq \sat{j}$, and therefore $\sat{\overline{j}} \subseteq \sat{j}$. Thus $\Gro{\overline{j}} = \overline{\sat{\overline{j}}} \subseteq \overline{\sat{j}} = \Gro{j}$. Conversely, it is easy to see that $j \subseteq \left(\overline{\sat{\overline{j}}}\right)^\circ = \sat{\overline{j}}$, since if $r \in j$, then $\overline{r} \in \overline{j}$ and therefore $r \in \sat{\overline{j}}$. Thus $\sat{j} \subseteq \sat{\overline{j}}$ and thus $\Gro{j} \subseteq \Gro{\overline{j}}$.

Now suppose that $J$ is a Grothendieck coverage such that $\overline{j} \subseteq J$. Then $j \subseteq (\overline{j})^\circ \subseteq J^\circ$, but $J^\circ$ is saturated by Lemma \ref{lem interior of Grothendieck coverage is saturated}, so $\sat{j} \subseteq J^\circ$. Thus $\Gro{j} \subseteq J$.
\end{proof}

\begin{Cor} \label{cor grothendieck closure preserves sheaves}
Given a coverage $j$ on a category $\cat{C}$, there exists a Grothendieck coverage $\Gro{j}$ containing $\overline{j}$ and such that $\Sh(\cat{C},j) = \Sh(\cat{C}, \Gro{j})$.
\end{Cor}

\begin{Lemma} \label{lem local epi iff local epi on groth closure}
Given a site $(\cat{C}, j)$ a map $f : X \to Y$ of presheaves on $\cat{C}$ is a $j$-local epimorphism (mono/iso) if and only if it is a $\Gro{j}$-local epimorphism (mono/iso).
\end{Lemma}

\begin{proof}
By Lemma \ref{lem local epi iff local epi on saturation closure}, we can assume that $j$ is a saturated coverage. Thus it is sufficient to show that $f$ is a $j$-local epimorphism (mono/iso) if and only if it is a $\overline{j}$-local epimorphism (mono/iso).

$(\Rightarrow)$ If $f$ is a $j$-local epimorphism, then by Corollary \ref{cor saturated coverages contain their sifted closure}, $\overline{j} \subseteq j$, so $f$ is a $\overline{j}$-local epimorphism. This similarly implies that if $f$ is a $j$-local monomorphism or isomorphism, then it is a $\overline{j}$-local monomorphism or isomorphism.

$(\Leftarrow)$ Since every $j$-covering family is contained within a $\overline{j}$-covering family, if $f$ lifts against all the $\overline{j}$-covering families, then it lifts against all the $j$-covering families. Thus $f$ is a $j$-local epimorphism (mono/iso).
\end{proof}

\begin{Lemma} \label{lem meets in saturated coverages}
Let $j$ be a saturated coverage on a small category $\cat{C}$, let $U \in \cat{C}$ and let $r, r' \in j(U)$. If we consider $j(U)$ as a preorder, ordered by refinement, then the meet of $r$ and $r'$ exists and is given by
\begin{equation*}
    r \wedge r' = \overline{r} \cap \overline{r}'.
\end{equation*}
In particular, this implies that $\sat{j}(U)$ is a finitely cofiltered category (Definition \ref{def filtered category}).
\end{Lemma}

\begin{proof}
First we note that $\overline{r} \cap \overline{r}'$ is $j$-saturating because the following diagram
\begin{equation*}
    \begin{tikzcd}
	{\overline{r} \cap \overline{r}'} & {\overline{r}} \\
	{\overline{r}'} & {y(U)}
	\arrow[hook, from=1-1, to=1-2]
	\arrow[hook, from=1-1, to=2-1]
	\arrow[hook, from=1-2, to=2-2]
	\arrow[hook, from=2-1, to=2-2]
\end{tikzcd}
\end{equation*}
is a pullback, and both $\overline{r} \hookrightarrow y(U)$ and $\overline{r}' \hookrightarrow y(U)$ are $j$-local epimorphisms since $r, r' \in j(U)$ and thus are $j$-saturating by Corollary \ref{cor saturating iff covering}. Hence $\overline{r} \cap \overline{r}' \hookrightarrow \overline{r}'$ is a $j$-local epimorphism, and thus the composite $\overline{r} \cap \overline{r}' \hookrightarrow y(U)$ is a $j$-local epimorphism by Lemma \ref{lem j-local epis form a system of local epis}. 

Now $\overline{\overline{r} \cap \overline{r}'} = \overline{r} \cap \overline{r}' \subseteq \overline{r}$. Thus by Lemma \ref{lem refinement under sifted closure} we have $\overline{r} \cap \overline{r}' \leq r$, and similarly for $r'$. Now if $t \leq r$ and $t \leq r'$, then $\overline{t} \subseteq \overline{r} \cap \overline{r}' = \overline{\overline{r} \cap \overline{r}'}$. Thus again by Lemma \ref{lem refinement under sifted closure}, we have $t \leq \overline{r} \cap \overline{r}'$.
\end{proof}

\begin{Rem}
Given a saturated coverage $j$, for every $U \in \cat{C}$,  $j(U)$ is a preorder under refinement, and we can consider $\overline{j}(U)$ as a poset under inclusion. If we let $j(U)_\sim$ denote the quotient of $j(U)$ by $r \sim r'$ if $r \leq r'$ and $r' \leq r$, then we get an isomorphism of posets
\begin{equation*}
    \begin{tikzcd}
	{j(U)_{\sim}} && {\overline{j}(U)}
	\arrow["{\overline{(-)}}", curve={height=-12pt}, from=1-1, to=1-3]
	\arrow["{\iota}", curve={height=-12pt}, from=1-3, to=1-1]
\end{tikzcd}
\end{equation*}
where $\iota$ is just inclusion. Basically this is just the fact that $r \leq \overline{r}$ and $\overline{r} \leq r$.
\end{Rem}

The following result is very important for Section \ref{section sheafification and lex localizations}. It says in the language of Appendix \ref{section localizations}, that $j$-sheaves are $W$-local, where $W$ is the class of $j$-local isomorphisms.

\begin{Prop} \label{prop sheaves are local iso local}
Let $(\cat{C}, J)$ be a Grothendieck coverage, $Z$ a $J$-sheaf, and let $f: X \to Y$ be a $J$-local isomorphism. If $g: X \to Z$ is a map of presheaves, then there exists a unique map $h : Y \to Z$ such that $hf = g$.
\end{Prop}

\begin{proof}
Suppose that $s : y(U) \to Y$ is a section. Then since $J$ is composition closed, there exists a $J$-covering sieve $R \hookrightarrow y(U)$ such that for every $r_i : U_i \to U$ in $R$ there is a map $s_i : y(U_i) \to X$ making the following diagram commute
\begin{equation*}
\begin{tikzcd}[ampersand replacement=\&]
	{y(U_i)} \& X \& Z \\
	{y(U)} \& Y
	\arrow["{s_i}", from=1-1, to=1-2]
	\arrow["{r_i}"', from=1-1, to=2-1]
	\arrow["g", from=1-2, to=1-3]
	\arrow["f", from=1-2, to=2-2]
	\arrow["s"', from=2-1, to=2-2]
\end{tikzcd}
\end{equation*}
Now let us show that $\{ g(s_i) \in Z(U_i) \}$ is a $Z$-matching family. First note that if we have a commutative diagram
\begin{equation*}
    \begin{tikzcd}
	U_{ij} & {U_j} \\
	{U_i} & U
	\arrow["{\pi_j}", from=1-1, to=1-2]
	\arrow["{\pi_i}"', from=1-1, to=2-1]
	\arrow["{r_j}", from=1-2, to=2-2]
	\arrow["{r_i}"', from=2-1, to=2-2]
\end{tikzcd}
\end{equation*}
then we obtain sections $X(\pi_i)(s_i), X(\pi_j)(s_j) \in X(U_{ij})$ such that 
$$f(X(\pi_i)(s_i)) = Y(\pi_i)(f(s_i)) = Y(\pi_i)Y(r_i)(s) = Y(\pi_j)Y(r_j)(s) = Y(\pi_j)(f(s_j)) = f(X(\pi_j)(s_j)).$$ 
Now since $f$ is a local monomorphism, there exists a $J$-covering sieve $T^{ij}$ on $U_{ij}$ such that if $t_k : V_k \to U_{ij} \in T^{ij}$, then $X(t_k)X(\pi_i)(s_i) = X(t_k)X(\pi_j)(s_j)$. Thus $\{ g(X(t_k)X(\pi_i)(s_i)) = g(X(t_k)X(\pi_j)(s_j) \}_k$ is a $Z$-matching family on $T^{ij}$. But both $g(X(\pi_i)(s_i)) = Z(\pi_i)(g(s_i))$ and $g(X(\pi_j)(s_j)) = Z(\pi_j)(g(s_j))$ are amalgamations on $Z(U_{ij})$. But $Z$ is a $J$-sheaf, and hence has unique amalgamations, so $Z(\pi_i)(g(s_i)) = Z(\pi_j)(g(s_j))$. Thus $\{g(s_i) \}$ is a $Z$-matching family for $R$, so there exists a unique amalgamation $h(s) \in Z(U)$ such that $Z(r_i)(h(s)) = g(s_i)$. Let $h : Y(U) \to Z(U)$ be defined objectwise by $s \mapsto h(s)$. 

Now let us show that the above construction didn't depend on the choice of covering families. Suppose that there is another covering sieve $R' = \{r'_{i'} : U_{i'} \to U \}_{i' \in I'}$ on $U$ with corresponding sections $s'_{i'} \in X(U_{i'})$. Then as above, we obtain two $Z$-matching families $\{g(s_i) \}_{i \in I}$ and $\{g(s'_{i'}) \}_{i' \in I'}$ on $R$ and $R'$ respectively. Now since $J$ is Grothendieck coverage, by Lemma \ref{lem covering sieve props}, the sieve $R \cap R'$ is a $J$-covering sieve. 

Thus we have two $Z$-matching families $\{g(s_\ell) \}_{\ell \in I \cap I'}$, $\{ g(s'_{\ell}) \}_{\ell \in I \cap I'}$ on $R \cap R'$. Furthermore $f(s_\ell) = Y(r_\ell)(s) = f(s'_\ell)$ for every $r_\ell \in R \cap R'$. Thus since $f$ is a local mono, there exists a $J$-covering sieve $T^\ell$ on $U_\ell$ such that for every $t^\ell_k : V^\ell_k \to U_\ell \in T^\ell$, we have
$X(t^\ell_k)(s_\ell) = X(t^\ell_k)(s'_\ell).$ Thus $Z(t^\ell_k)(g(s_\ell)) = Z(t^\ell_k)(g(s'_\ell))$. Therefore $\{ Z(t^\ell_k)(g(s_\ell)) = Z(t^\ell_k)(g(s'_\ell)) \}_k$ is a $Z$-matching family for $T^\ell$, with amalgamations $g(s_\ell)$ and $g(s'_\ell)$. Since $Z$ is a $J$-sheaf, amalgamations on $T^\ell$ are unique, so this implies that $g(s_\ell) = g(s'_\ell)$ for every $\ell \in I \cap I'$.

Now if $h(s)$ is an amalgamation for $\{g(s_i) \}_{i \in I}$ and $h'(s)$ is an amalgamation for $\{g(s'_{i'}) \}_{i' \in I'}$, then both $h(s)$ and $h'(s)$ are amalgamations for $\{g(s_\ell) = g(s'_\ell) \}_{\ell \in I \cap I'}$. But $Z$ is a $J$-sheaf, and therefore amalgamations for $R \cap R'$ are unique, therefore $h(s) = h'(s)$. Thus the assignment $s \mapsto h(s)$ doesn't depend on the chosen covering families.

Let us show that the assignment $h_U : Y(U)\to Z(U)$ extends to a natural transformation $h : Y \to Z$. If $a : V \to U$ is a morphism in $\cat{C}$, then since $J$ is a coverage, there exists a covering sieve $T = \{t_j : V_j \to V \}_{j \in J}$ an index map $\alpha : J \to I$ and maps $\ell_j : V_j \to U_\alpha(j)$. Taking $\{Z(\ell_j)(g(s_{\alpha(j)}))\} = \{g(s_{\alpha(j)} \ell_j) \}$, we get a $Z$-matching family on $T$, and hence get the unique amalgamation $h(sa)$. But $Z(t_j)Z(a)(h(s)) = Z(a t_j)(h(s)) = Z(\ell_j)Z(r_{\alpha(j)})(h(s)) = Z(\ell_j)g(s_{\alpha(j)})$. Thus $Z(a)(h(s))$ is also an amalgamation, and hence $h(Y(a)(s)) = h(sa) = Z(a)(h(s))$.

Now let us show that $hf = g$. If $s \in Y(U)$ is a section such that $s = f(x)$ for some $x \in X(U)$, then $\{ X(a)(x) \}_{a \in y_U}$ is an $X$-matching family for the maximal covering sieve $y_U = \{a : V \to U\}$, such that $f(X(a)(x)) = Y(a)(f(x)) = Y(a)(s)$. Thus $\{ g(X(a)(x)) = Z(a)(g(x)) \}$ is a $Z$-matching family for $y_U$, with unique amalgamation $g(x)$ and thus $h(s) = h(f(x)) = g(x)$.

Uniqueness of $h$ follows from uniqueness of amalgamations.
\end{proof}

\begin{Cor} \label{cor sheaves are local iso local even without saturation}
The hypothesis of Proposition \ref{prop sheaves are local iso local} holds for an arbitrary site $(\cat{C}, j)$, not just Grothendieck sites.
\end{Cor}

\begin{proof}
This follows from Lemma \ref{lem local epi iff local epi on groth closure}.
\end{proof}

\section{Sheafification and Lex Localizations} \label{section sheafification and lex localizations}

In this section, we use the theory of localizations as in Appendix \ref{section localizations} to study Grothendieck toposes. The main goal for this section is Theorem \ref{th lex localizations <-> grothendieck toposes}, also called the ``Little Giraud Theorem'', which characterizes Grothendieck toposes by lex localizations of presheaf topoi. This is well-trodden ground, see \cite[Corollary 2.1.11]{johnstone2002sketches}. However, we diverge from the usual way of proving this using Lawvere-Tierney topologies and instead use systems of local epimorphisms and isomorphisms.

\subsection{Systems of Local Epimorphisms}

\begin{Def}[{\cite[Section 16.1]{kashiwara2006}}] \label{def system of local epis}
Given a presheaf topos $\Pre(\cat{C})$ a \textbf{system of local epimorphisms} consists of a class $E = \{f : X \to Y \}$ of morphisms of presheaves such that:
\begin{enumerate}
    \item if $f: X \to Y$ is an epi, then $f \in E$,
    \item if $f : X \to Y$ is in $E$ and $g: Y \to Z$ is in $E$, then $gf \in E$,
    \item if $gf$ as above is in $E$, then $g \in E$,
    \item an arbitrary morphism $f: X \to Y$ is in $E$ if and only if for every section $s: y(U) \to Y$ the pullback $s^*(f) : y(U)\times_Y X \to X$ is in $E$.
\end{enumerate}
If $E$ is a system of local epimorphisms on $\Pre(\cat{C})$ and $f \in E$ we call $f$ an $E$-local epimorphism or $E$-local epi for short. We will use the statements $f \in E$ and $f$ is an $E$-local epi interchangeably.
\end{Def}

\begin{Lemma} \label{lem alternate axiom for system of local epi}
The axiom (4) of Definition \ref{def system of local epis} is equivalent to the following (4'): a map $f: X \to Y$ of presheaves is an $E$-local epi if and only if for every arbitrary map $g: Z \to Y$, the pullback map $g^*(f) : Z \times_Y X \to Z$ is an $E$-local epi. 
\end{Lemma}

\begin{proof}
$(4' \Rightarrow 4)$ If $f : X \to Y$ is an $E$-local epi, then for every section $s: y(U) \to Y$, the pullback morphism $s^*(f)$ is an $E$-local epi. Now conversely suppose that every $s^*(f)$ is in $E$. We want to show that $f$ is in $E$. Let $g: Z \to Y$ be a map of presheaves, and suppose that $s : y(U) \to Z$ is a section. Consider the following diagram of pullback squares
\begin{equation*}
    \begin{tikzcd}
	{y(U)\times_Z (Z \times_Y X)} & {Z \times_Y X} & X \\
	{y(U)} & Z & Y
	\arrow[from=1-1, to=1-2]
	\arrow["{s^*(g^*(f))}"', from=1-1, to=2-1]
	\arrow["\lrcorner"{anchor=center, pos=0.125}, draw=none, from=1-1, to=2-2]
	\arrow[from=1-2, to=1-3]
	\arrow["{g^*(f)}"', from=1-2, to=2-2]
	\arrow["\lrcorner"{anchor=center, pos=0.125}, draw=none, from=1-2, to=2-3]
	\arrow["f", from=1-3, to=2-3]
	\arrow["s"', from=2-1, to=2-2]
	\arrow["g"', from=2-2, to=2-3]
\end{tikzcd}
\end{equation*}
since both squares are pullbacks, the outer square is a pullback. Thus by assumption $s^*(g^*(f))$ is an $E$-local epi. But since $s$ was arbitrary, this proves that $g^*(f)$ is in $E$. But since $g$ was arbitrary, this proves that $f$ is in $E$.
$(4 \Rightarrow 4')$ If $f$ is in $E$, then by a similar argument as above, for every map $g: Z \to Y$, $g^*(f)$ is in $E$. Conversely if every $g^*(f)$ is in $E$, then taking $g$ to be the identity shows that $f$ is in $E$.
\end{proof}

\begin{Lemma} \label{lem j-local epis form a system of local epis}
Given a coverage $j$ on $\cat{C}$, the class of $j$-local epimorphims forms a system of local epimorphisms.
\end{Lemma}

\begin{proof}
(1) is clear, (2) follows from Lemma \ref{lem local epimorphisms closed under composition}, and (3) follows from Lemma \ref{lem composition is local epi implies local epi}. Now suppose that $f : X \to Y$ is a map of presheaves and for all sections $s: y(U) \to Y$ the pullback $s^*(f) : y(U) \times_Y X \to y(U)$ is a $j$-local epimorphism. We want to show that $f$ is a $j$-local epimorphism. So suppose that $s : y(U) \to Y$ is a section. Then $s^*(f)$ is a $j$-local epimorphism, so if we consider the identity map $1_{y(U)} : y(U) \to y(U)$, then there exists a $j$-tree $T$ on $U$ such $T^\circ \leq s^*(f)$, which implies that $s_*(T^\circ) \leq f$. Since this holds for any section $s$, this shows that $f$ is a $j$-local epimorphism. Thus (4).$(\Leftarrow)$ holds.

Conversely, suppose that $f$ is a local epi. We want to show that for every $s: y(U) \to Y$, the map $s^*(f)$ is a local epi. Suppose that $s' : y(V) \to y(U)$ is a section, then since $f$ is a local epi, there exists a $j$-tree $T$ on $V$ and a refinement $g : (ss')_*(T^\circ) \to f$ as in the following commutative diagram
\begin{equation}
    \begin{tikzcd}
	{y(V_i)} & {y(U)\times_YX} & X \\
	{y(V)} & {y(U)} & Y
	\arrow["f", from=1-3, to=2-3]
	\arrow["s"', from=2-2, to=2-3]
	\arrow[from=1-2, to=1-3]
	\arrow["{s^*(f)}", from=1-2, to=2-2]
	\arrow["{s'}"', from=2-1, to=2-2]
	\arrow["{t_i}"', from=1-1, to=2-1]
	\arrow["{g_i}", curve={height=-18pt}, from=1-1, to=1-3]
	\arrow[dashed, from=1-1, to=1-2]
\end{tikzcd}
\end{equation}
By the universal property of the pullback, there is a unique dotted map as above making the diagram commute, and therefore defining a refinement $(s')_*(T^\circ) \leq s^*(f)$ as we wanted to show, thus proving (4).$(\Rightarrow)$. 
\end{proof}

\begin{Def}
Given a system $E$ of local epimorphisms on $\Pre(\cat{C})$, let $j(E)$ denote the collection of families on $\cat{C}$ where for every $U \in \cat{C}$, a family $r \in j(E)(U)$ if and only if $\overline{r} \hookrightarrow y(U)$ is an $E$-local epimorphism.
\end{Def}

\begin{Lemma} \label{lem system of local epis give saturated coverage}
Given a system $E$ of local epimorphims on $\Pre(\cat{C})$, the collection $j(E)$ of families on $\cat{C}$ is a saturated coverage.
\end{Lemma}

\begin{proof}
Let us show $j(E)$ is a coverage. Suppose that $r \in j(E)(U)$ for some $U \in \cat{C}$, and $g : V \to U$ is a morphism in $\cat{C}$. Then $g^*(\overline{r}) \in j(E)(V)$, since $E$ is a system of local epimorphisms. But $g^*(\overline{r}) = \overline{g^*(r)}$, so $g^*(r) \in j(E)(V)$ and $g_* g^*(r) \subseteq r$. Thus $j(E)$ is a coverage.

Clearly $j(E)$ contains all identity families, so let us show that $j(E)$ is refinement closed. Suppose that $r \in j(E)(U)$, so that $\overline{r} \hookrightarrow y(U)$ is an $E$-local epimorphism, and there exists a refinement $r \leq t$. We want to show that $\overline{t} \hookrightarrow y(U)$ is an $E$-local epi. Now $\overline{r} \subseteq \overline{t}$, so we have the following commmutative diagram
\begin{equation*}
   \begin{tikzcd}
	& {\overline{t}} \\
	{\overline{r}} && {y(U)}
	\arrow[hook, from=2-1, to=2-3]
	\arrow[hook, from=2-1, to=1-2]
	\arrow[hook, from=1-2, to=2-3]
\end{tikzcd} 
\end{equation*}
and thus $\overline{t} \hookrightarrow y(U)$ is an $E$-local epi by Definition \ref{def system of local epis}.(3).

Now let us show that $j(E)$ is composition closed. Let $r = \{r_i : U_i \to U \}_{i \in I} \in j(E)(U)$ and for each $i \in I$, suppose that $t^i \in j(E)(U_i)$. We want to show that $(r \circ t) = \bigcup_i (r_i)_*(t^i) \in j(E)(U)$. First let us prove that $r : \sum_i y(U_i) \to y(U)$ is an $E$-local epimorphism. By Proposition \ref{prop sieves are coequalizers of generating family} we have the following commutative diagram
\begin{equation*}
    \begin{tikzcd}
	{\sum_{i,j} y(U_i) \times_{y(U)} y(U_j)} & {\sum_i y(U_i)} & {\overline{r}} \\
	&& {y(U)}
	\arrow["\pi", from=1-2, to=1-3]
	\arrow[shift right=2, from=1-1, to=1-2]
	\arrow[shift left=2, from=1-1, to=1-2]
	\arrow["r"', from=1-2, to=2-3]
	\arrow[hook, from=1-3, to=2-3]
\end{tikzcd}
\end{equation*}
where $\pi$ is a coequalizer, and hence an epimorphism. Thus by Defintion \ref{def system of local epis}.(1), $\pi$ is an $E$-local epi, and $\overline{r} \hookrightarrow y(U)$ is an $E$-local epi by assumption. Thus by Definition \ref{def system of local epis}.(2), $r$ is an $E$-local epi. Now note that the following commutative diagram is a pullback
\begin{equation*}
    \begin{tikzcd}
	{\overline{t^i}} & {\bigcup_i \overline{t^i}} \\
	{y(U_i)} & {\sum_i y(U_i)}
	\arrow[hook, from=1-1, to=1-2]
	\arrow[hook, from=1-1, to=2-1]
	\arrow[hook, from=1-2, to=2-2]
	\arrow[hook, from=2-1, to=2-2]
\end{tikzcd}
\end{equation*}
Thus if $s : y(V) \to \sum_i y(U_i)$ is a section, then by the Yoneda lemma, $s$ factors through some $y(U_i)$ and taking the pullback we obtain the following pair of pullback diagrams
\begin{equation*}
\begin{tikzcd}
	{s^*(\overline{t^i})} & {\overline{t^i}} & {\bigcup_i \overline{t^i}} \\
	{y(V)} & {y(U_i)} & {\sum_i y(U_i)}
	\arrow[from=1-1, to=1-2]
	\arrow[hook, from=1-1, to=2-1]
	\arrow["\lrcorner"{anchor=center, pos=0.125}, draw=none, from=1-1, to=2-2]
	\arrow[hook, from=1-2, to=1-3]
	\arrow[hook, from=1-2, to=2-2]
	\arrow["\lrcorner"{anchor=center, pos=0.125}, draw=none, from=1-2, to=2-3]
	\arrow[hook, from=1-3, to=2-3]
	\arrow["s"', from=2-1, to=2-2]
	\arrow[hook, from=2-2, to=2-3]
\end{tikzcd}
\end{equation*}
and since $\overline{t^i} \hookrightarrow y(U_i)$ is an $E$-local epi by assumption, $s^*(\overline{t^i}) \hookrightarrow y(V)$ is an $E$-local epi. But the outer rectangle is also a pullback, and since $s$ was arbitrary, this implies that $\bigcup_i \overline{t^i} \hookrightarrow \sum_i y(U_i)$ is an $E$-local epi.

Now $\overline{(r \circ t)} = \bigcup_i (r_i)_*(\overline{t^i})$ is precisely the image of the map $\bigcup_i \overline{t^i} \to \sum_i y(U_i) \to y(U)$, so we have the following commutative diagram
\begin{equation*}
\begin{tikzcd}
	{\bigcup_i \overline{t^i}} & {\sum_i y(U_i)} \\
	{\overline{(r \circ t)}} & {y(U)}
	\arrow[hook, from=1-1, to=1-2]
	\arrow[two heads, from=1-1, to=2-1]
	\arrow["r", from=1-2, to=2-2]
	\arrow[hook, from=2-1, to=2-2]
\end{tikzcd}
\end{equation*}
But then the composite map $\bigcup_i \overline{t^i} \to y(U)$ is an $E$-local epi, and thus by (3), $\overline{(r \circ t)} \hookrightarrow y(U)$ is an $E$-local epi. Thus $j(E)$ is composition closed, and therefore saturated.
\end{proof}

\begin{Prop} \label{prop bij between saturated coverages and systems of local epis}
Given an (essentially small) category $\cat{C}$, there is a bijection 
\begin{equation*}
    \{ \text{saturated coverages on }\cat{C} \} \cong \{ \text{systems of local epis on }\Pre(\cat{C}) \}.
\end{equation*}
\end{Prop}

\begin{proof}
In Lemma \ref{lem j-local epis form a system of local epis}, we gave a map
\begin{equation*}
    \varphi : \{ \text{coverages on }\cat{C} \} \to \{ \text{systems of local epis on }\Pre(\cat{C}) \},
\end{equation*}
where $\varphi(j)$ is the system of local epimorphisms given by the class of $j$-local epimorphisms.
Conversely, in Lemma \ref{lem system of local epis give saturated coverage}, we gave a map 
\begin{equation*}
    \psi: \{ \text{systems of local epis on }\Pre(\cat{C}) \} \to \{ \text{saturated coverages on }\cat{C} \},
\end{equation*}
where $\psi(E)$ is the coverage $j(E)$. It is easy to see that if $j$ is a coverage, then $\psi \varphi (j)$ is precisely $\sat{j}$. We claim that if $E$ is a system of local epimorphisms, then $\varphi \psi(E) = E$. So suppose that $f : X \to Y$ belongs to $\varphi \psi(E)$. Thus $f$ is a $j(E)$-local epimorphism. Since $j(E)$ is saturated, by Lemma \ref{lem local epi on comp closed coverage}, this implies that for every section $s: y(U) \to Y$ there exists a $j(E)$-covering family $r$ on $U$ such that $s_*(r) \leq f$, and furthermore $s_*(\overline{r}) \leq f$. Now $r$ being a $j(E)$-covering family means that $\overline{r} \hookrightarrow y(U)$ belongs to $E$. Now $s_*(\overline{r}) \leq f$ implies that $\overline{r} \leq s^*(f)$. But this means that $\overline{r} \hookrightarrow y(U)$ factors through $s^*(f)$, and hence by Definition \ref{def system of local epis}.(3), $s^*(f) \in E$. Since this is true for every section $s$, this implies that $f \in E$. Conversely if $f: X \to Y$ is in $E$, then we want to show that for every $s: y(U) \to Y$, there exists a $j(E)$-covering family $r$ such that $s_*(f) \leq f$. But for every such $s: y(U) \to Y$, we know that $s^*(f)$ belongs to $E$. Taking the image of $s^*(f)$ we obtain a monomorphism $\im(s^*(f)) \hookrightarrow y(U)$, and by Definition \ref{def system of local epis}.(3), this belongs to $E$. Thus $r = \bigcup_{V \in \cat{C}} \im(s^*(f))(V)$ is a $j(E)$-covering family. If $r_i : y(U_i) \to y(U)$ belongs to $\im(s^*(f))(U_i)$, then by definition it factors through $s^*(f)$, and therefore $s_*(r_i) \leq f$. Therefore $s_*(r) \leq f$, so $f$ is a $j(E)$-local epimorphism. Thus $\varphi \psi(E) = E$. Therefore if we restrict $\varphi$ to saturated coverages, then we obtain a bijection.
\end{proof}

\subsection{Systems of Local Isomorphisms}

\begin{Def} \label{def 2-of-3 property}
Let $W$ be a class of morphisms in a category $\cat{C}$. We say that $W$ satisfies the \textbf{2-of-3 property} if given any pair of composeable morphisms $f: X \to Y$ and $g: Y \to Z$ if any two of the morphisms in the set $\{f, g, gf \}$ belongs to $W$, then all three of them belong to $W$.
\end{Def}

\begin{Def} \label{def system of local isos}
Let $\Pre(\cat{C})$ be a presheaf topos, and let $W$ be a class of morphisms of presheaves. We say that $W$ is a \textbf{system of local isomorphisms} if 
\begin{enumerate}
    \item every isomorphism belongs to $W$,
    \item $W$ satisfies the 2-of-3 property,
    \item a map $f: X \to Y$ belongs to $W$ if and only if for every section $s : y(U) \to Y$, the pullback morphism $s^*(f) : y(U) \times_Y X \to y(U)$ is in $W$.
\end{enumerate}
If $W$ is a system of local isomorphisms, $f: X \to Y$ is in $W$ we say that $f \in W$ or $f$ is a $W$-local isomorphism or $f$ is a $W$-local iso, interchangeably.
\end{Def}

\begin{Lemma} \label{lem alternate axiom for system of local iso}
The axiom (3) of Definition \ref{def system of local isos} is equivalent to the following (3'): a map $f: X \to Y$ of presheaves is in $W$ if and only if for every arbitrary map $g: Z \to Y$, the pullback map $g^*(f) : Z \times_Y X \to Z$ is in $W$. 
\end{Lemma}

\begin{proof}
$(3' \Rightarrow 3)$ If $f : X \to Y$ is in $W$, then for every section $s: y(U) \to Y$, the pullback morphism $s^*(f)$ is in $W$. Now conversely suppose that every $s^*(f)$ is in $W$. We want to show that $f$ is in $W$. Let $g: Z \to Y$ be a map of presheaves, and suppose that $s : y(U) \to Z$ is a section. Consider the following diagram of pullback squares
\begin{equation*}
    \begin{tikzcd}
	{y(U)\times_Z (Z \times_Y X)} & {Z \times_Y X} & X \\
	{y(U)} & Z & Y
	\arrow[from=1-1, to=1-2]
	\arrow["{s^*(g^*(f))}"', from=1-1, to=2-1]
	\arrow["\lrcorner"{anchor=center, pos=0.125}, draw=none, from=1-1, to=2-2]
	\arrow[from=1-2, to=1-3]
	\arrow["{g^*(f)}"', from=1-2, to=2-2]
	\arrow["\lrcorner"{anchor=center, pos=0.125}, draw=none, from=1-2, to=2-3]
	\arrow["f", from=1-3, to=2-3]
	\arrow["s"', from=2-1, to=2-2]
	\arrow["g"', from=2-2, to=2-3]
\end{tikzcd}
\end{equation*}
since both squares are pullbacks, the outer square is a pullback. Thus by assumption $s^*(g^*(f))$ is in $W$. But since $s$ was arbitrary, this proves that $g^*(f)$ is in $W$. But since $g$ was arbitrary, this proves that $f$ is in $W$.
$(3 \Rightarrow 3')$ If $f$ is in $W$, then by a similar argument as above, for every map $g: Z \to Y$, $g^*(f)$ is in $W$. Conversely if every $g^*(f)$ is in $W$, then taking $g$ to be the identity shows that $f$ is in $W$.
\end{proof}

\begin{Def}
Suppose that $W$ is a system of local isomorphisms. We say that a map $f: X \to Y$ is a $W$-local epimorphism if its image $\im(f) \hookrightarrow Y$ is a $W$-local isomorphism. We say that $f$ is a $W$-local mono if $\Delta_f : X \to X \times_Y X$ is a $W$-local epi. Note that a sieve $R \hookrightarrow y(U)$ is a $W$-local isomorphism if and only if it is a $W$-local epimorphism.
\end{Def}

\begin{Lemma} \label{lem covering sieves from system of local isomorphisms}
Suppose that $W$ is a system of local isomorphisms on a presheaf topos $\Pre(\cat{C})$, $R \hookrightarrow y(U)$ is a $W$-local iso, and $S \hookrightarrow y(U)$ is a sieve, then
\begin{enumerate}[(I)]
    \item if $R \subseteq S$, then $S \hookrightarrow y(U)$ is a $W$-local iso, and
    \item if for every $f : V \to U$ in $R$, the map $f^*(S) \hookrightarrow y(V)$ is a $W$-local iso, then $S \hookrightarrow y(U)$ is a $W$-local iso.
\end{enumerate}
\end{Lemma}

\begin{proof}
(I) Suppose that $R \hookrightarrow y(U)$ is a $W$-local isomorphism, $S \hookrightarrow y(U)$ is a sieve, and $R \subseteq S$. This implies that $R \cap S = R$, and therefore using Lemma \ref{lem meet of subobjects}, the following commutative diagram is a pullback square
\begin{equation*}
    \begin{tikzcd}
	R & R \\
	S & {y(U)}
	\arrow[Rightarrow, no head, from=1-1, to=1-2]
	\arrow[hook, from=1-1, to=2-1]
	\arrow["\lrcorner"{anchor=center, pos=0.125}, draw=none, from=1-1, to=2-2]
	\arrow[hook, from=1-2, to=2-2]
	\arrow[hook, from=2-1, to=2-2]
\end{tikzcd}
\end{equation*}
Thus by Definition \ref{def system of local isos}.(3'), $R \hookrightarrow S$ is a $W$-local isomorphism, and by Definition \ref{def system of local isos}.(2), this implies $S \hookrightarrow y(U)$ is a $W$-local isomorphism. 

(II) Suppose that $R \in J(W)(U)$, $S \hookrightarrow y(U)$, and for every $f : V \to U$ in $R$, the pullback $f^*(S) \hookrightarrow y(V)$ is a $W$-local iso, then we want to show that $S \hookrightarrow y(U)$ is a $W$-local iso. Suppose that $f^*(S) \hookrightarrow y(V)$ is a $W$-local iso for every $f : V \to U$ in $R$. Then for every $g : y(V) \to R$, we have a commutative diagram
\begin{equation*}
    \begin{tikzcd}
	{g^*(S)} & {R \cap S} & S \\
	{y(V)} & R & {y(U)}
	\arrow[from=1-1, to=1-2]
	\arrow[hook, from=1-1, to=2-1]
	\arrow[hook, from=1-2, to=1-3]
	\arrow[hook, from=1-2, to=2-2]
	\arrow["\lrcorner"{anchor=center, pos=0.125}, draw=none, from=1-2, to=2-3]
	\arrow[hook, from=1-3, to=2-3]
	\arrow["g"', from=2-1, to=2-2]
	\arrow[hook, from=2-2, to=2-3]
\end{tikzcd}
\end{equation*}
where the right hand square is a pullback and the outer rectangle is a pullback. Thus the left hand square is a pullback. By assumption $g^*(S) \hookrightarrow y(V)$ is a $W$-local isomorphism, and since $g: y(V) \to R$ was arbitrary, this implies that $R \cap S \hookrightarrow R$ is a $W$-local isomorphism. Thus $R \cap S \hookrightarrow y(U)$ is a $W$-local isomorphism, which implies that $S \hookrightarrow y(U)$ is a $W$-local isomorphism by (I).
\end{proof}

\begin{Rem}
Lemma \ref{lem covering sieve props} implies that if $W$ is a system of local isomorphisms, then the collection of sieve inclusions $\{R \hookrightarrow y(U) \}$ that are $W$-local isomorphisms forms a Grothendieck coverage.
\end{Rem}

\begin{Lemma}[{\cite[Exercise 16.5]{kashiwara2006}}] \label{lem sys local isos -> sys local epis}
If $W$ is a system of local isomorphisms and $E$ is the class of $W$-local epis, then $E$ is a system of local epimorphisms.
\end{Lemma}

\begin{proof}
Suppose that $f : X \to Y$ is an epimorphism. Then $\im(f) \hookrightarrow Y$ is an isomorphism and therefore a $W$-local isomorphism.

Suppose that $f : X \to Y$ and $g: X \to Y$ are morphisms such that $gf$ is a $W$-local epimorphism. We want to show that $g$ is a $W$-local epi. First note that the following diagram is a pullback
\begin{equation*}
\begin{tikzcd}
	{\text{im}(gf)} & {\text{im}(gf)} \\
	{\text{im}(g)} & Z
	\arrow[Rightarrow, no head, from=1-1, to=1-2]
	\arrow["i"', hook, from=1-1, to=2-1]
	\arrow[hook, from=1-2, to=2-2]
	\arrow[hook, from=2-1, to=2-2]
\end{tikzcd}
\end{equation*}
where $i$ is the obvious inclusion map.
Now let us show that $i$ is a $W$-local isomorphism. Suppose that $s : y(U) \to \im(g)$ is a section. Then we obtain the following pair of pullback diagrams
\begin{equation*}
\begin{tikzcd}
	{y(U) \times_{\text{im}(g)} \text{im}(gf)} & {\text{im}(gf)} & {\text{im}(gf)} \\
	{y(U)} & {\text{im}(g)} & Z
	\arrow[from=1-1, to=1-2]
	\arrow["k"', from=1-1, to=2-1]
	\arrow["\lrcorner"{anchor=center, pos=0.125}, draw=none, from=1-1, to=2-2]
	\arrow[Rightarrow, no head, from=1-2, to=1-3]
	\arrow["{i}", from=1-2, to=2-2]
	\arrow["\lrcorner"{anchor=center, pos=0.125}, draw=none, from=1-2, to=2-3]
	\arrow[hook, from=1-3, to=2-3]
	\arrow["s"', from=2-1, to=2-2]
	\arrow[hook, from=2-2, to=2-3]
\end{tikzcd}
\end{equation*}
so the outer rectangle is a pullback. Since $\im(gf) \hookrightarrow Z$ is a local isomorphism by assumption, this implies that $k$ is a local isomorphism. Since $s$ was arbitrary, this implies that $g_*$ is a $W$-local isomorphism. Thus by the 2-of-3 property of $W$-local isomorphisms, $\im(g) \hookrightarrow Z$ is a $W$-local isomorphism, and therefore $g$ is a $W$-local epi.

Now let $ f : X \to Y$ be a $W$-local epimorphism and $s : y(U) \to Y$. Then since images of presheaf maps are stable under pullback we have the following pair of pullback diagrams
\begin{equation*}
    \begin{tikzcd}
	{y(U) \times_Y X} & X \\
	{\im(s^*(f))} & {\text{im}(f)} \\
	{y(U)} & Y
	\arrow[from=1-1, to=1-2]
	\arrow[from=1-1, to=2-1]
	\arrow["\lrcorner"{anchor=center, pos=0.125}, draw=none, from=1-1, to=2-2]
	\arrow[two heads, from=1-2, to=2-2]
	\arrow[from=2-1, to=2-2]
	\arrow[from=2-1, to=3-1]
	\arrow["\lrcorner"{anchor=center, pos=0.125}, draw=none, from=2-1, to=3-2]
	\arrow[hook, from=2-2, to=3-2]
	\arrow["s"', from=3-1, to=3-2]
\end{tikzcd}
\end{equation*}
Thus $\text{im}(s^*(f)) \hookrightarrow y(U)$ is a $W$-local isomorphism, so $s^*(f)$ is a $W$-local epimorphism.

Conversely suppose that for every $s : y(U) \to Y$, the pullback map $s^*(f)$ is a $W$-local epimorphism. We want to show that $f$ is a $W$-local epimorphism. But if we pull back $f$ by $s$, then the image of $s^*(f)$ is a $W$-local isomorphism by assumption. Thus $\im(f) \hookrightarrow y(U)$ is a $W$-local isomorphism. Therefore $f$ is a $W$-local epimorphism.

Then taking $g$ to be the identity shows that $f$ is a $W$-local epi. Conversely if $f$ is a $W$-local epi, we want to show every pullback map is. This follows from the fact that $\Pre(\cat{C})$ has pullback stable image factorizations (Lemma \ref{lem presheaf topoi have pullback stable image factorization}). Namely taking the pullback of $\im(f) \hookrightarrow Y$ along any map $g: Z \to Y$ gives $\im(g^*(f)) \hookrightarrow Z$, which is a $W$-local isomorphism, since $W$-local isomorphisms are stable under pullback, thus $g^*(f)$ is a $W$-local epi.

If $f : X \to Y$ and $g: Y \to Z$ are $W$-local epis, then we want to show that $gf$ is a $W$-local epi\footnote{This part of the proof is due to \cite{doughertylocalisos}.}. Let us first suppose that $g$ is of the form $g : Y \to y(U)$. Suppose that $s : y(V) \to \im(g)$ is a section, then $s$ factors through $g$, giving a map $\widetilde{s} : y(V) \to Y$ such that $g \widetilde{s} = s$. Now taking the pullback $\widetilde{s}^*(f)$, we have the commutative diagram
\begin{equation*}
 \begin{tikzcd}
	{y(V) \times_Y X} & X \\
	{y(V)} & Y \\
	{y(V)} & {\text{im}(g)} & {y(U)}
	\arrow[from=1-1, to=1-2]
	\arrow["{\widetilde{s}^*(f)}"', from=1-1, to=2-1]
	\arrow["\lrcorner"{anchor=center, pos=0.125}, draw=none, from=1-1, to=2-2]
	\arrow["f", from=1-2, to=2-2]
	\arrow["{\widetilde{s}}"', from=2-1, to=2-2]
	\arrow[Rightarrow, no head, from=2-1, to=3-1]
	\arrow[two heads, from=2-2, to=3-2]
	\arrow["g"', from=2-2, to=3-3]
	\arrow["s"', from=3-1, to=3-2]
	\arrow[hook, from=3-2, to=3-3]
\end{tikzcd}   
\end{equation*}
Now $\widetilde{s}^*(f)$ is a $W$-local epimorphism because $f$ is. Thus $\im(\widetilde{s}^*(f)) \hookrightarrow y(V)$ is a $W$-local isomorphism. But $\im(\widetilde{s}^*(f)) \subseteq s^*(\im(gf))$. Indeed, if $h : W \to V$ is in the image of $\widetilde{s}^*(f)$, then there exists an $x \in X(W)$ such that $\widetilde{s}(h) = f(x)$. But then $s(h) = g \widetilde{s}(h) = gf(x)$. Thus by Lemma \ref{lem covering sieves from system of local isomorphisms}.(I), this implies that $s^*(\im(gf)) \hookrightarrow y(V)$ is a $W$-local isomorphism. Since this is true for every $s \in \im(g)$, by Lemma \ref{lem covering sieves from system of local isomorphisms}.(II), this implies that $\im(gf) \hookrightarrow y(U)$ is a $W$-local isomorphism.

Now suppose that $f : X \to Y$ and $g: Y \to Z$ are arbitrary $W$-local epimorphisms. Then for any section $s : y(U) \to Z$, we have the following pair of pullback diagrams
\begin{equation*}
    \begin{tikzcd}
	{y(U) \times_Z X} & X \\
	{y(U)\times_Z Y} & Y \\
	{y(U)} & Z
	\arrow[from=1-1, to=1-2]
	\arrow[from=1-1, to=2-1]
	\arrow["\lrcorner"{anchor=center, pos=0.125}, draw=none, from=1-1, to=2-2]
	\arrow["f", from=1-2, to=2-2]
	\arrow[from=2-1, to=2-2]
	\arrow[from=2-1, to=3-1]
	\arrow["\lrcorner"{anchor=center, pos=0.125}, draw=none, from=2-1, to=3-2]
	\arrow["g", from=2-2, to=3-2]
	\arrow["s"', from=3-1, to=3-2]
\end{tikzcd}
\end{equation*}
and thus as proven above, since $W$-local epimorphisms are stable under pullback, this implies that the two left vertical maps are $W$-local epimorphisms. By the argument of the preceding section, this implies their composite is a $W$-local epimorphism. Since $s: y(U) \to Z$ was an arbitrary section, and the outer rectangle is a pullback, this implies that $gf$ is a $W$-local epimorphism.
\end{proof}

\begin{Lemma} \label{lem sys local epi -> sys local iso}
Given a system $E$ of local epimorphisms, let $W$ denote the class of $E$-local isomorphisms. Then $W$ is a system of local isomorphisms.
\end{Lemma}

\begin{proof}
(1) This clearly holds because if $f: X \to Y$ is an isomorphism, then it is an $E$-local isomorphism.

(2) Since $E$ is the class of $j$-local epimorphisms for some saturated coverage $j$, this follows from Lemma \ref{lem properties of local monos and local isos}.(6).

(3')$(\Rightarrow)$ holds by Lemma \ref{lem properties of local monos and local isos}.(2) and (3')$(\Leftarrow)$ is obvious.
\end{proof}

\begin{Prop} \label{prop sys local epis bij with sys local isos}
Given a small category $\cat{C}$, there is a bijection
\begin{equation*}
    \{ \text{systems of local epimorphisms on }\Pre(\cat{C}) \} \cong \{ \text{systems of local isomorphisms on }\Pre(\cat{C}) \}.
\end{equation*}
\end{Prop}

\begin{proof}
By Lemma \ref{lem sys local epi -> sys local iso}, we have a function (between large sets)
$$\sigma : \{\text{systems of local epis on }\Pre(\cat{C})\} \to \{\text{systems of local isos on }\Pre(\cat{C})\}$$
where if $E$ is a system of local epis, then $\sigma(E)$ is the class of $E$-local isomorphisms. Namely $\sigma(E)$ is the class of morphisms that are both $E$-local epis and $E$-local monos. Conversely, by Lemma \ref{lem sys local isos -> sys local epis}, we have a function
$$\tau: \{\text{systems of local isos on }\Pre(\cat{C})\} \to \{\text{systems of local epis on }\Pre(\cat{C})\},$$
where if $W$ is a system of local isos, then $\tau(W)$ is the class of $W$-local epis. Namely $\tau(W)$ is the class of morphisms $f: X \to Y$ such that $\im(f) \hookrightarrow Y$ is a $W$-local iso. Let us show that these functions are inverse to each other.

First let us show that if $W$ is a system of local isos, and we define $E = \tau(W)$ to be the class of $W$-local epis, then if $f: X \to Y$ is in $W$, then $f$ is in $E$. First note that the following is a pullback
\begin{equation*}
    \begin{tikzcd}
	X & X \\
	{\text{im}(f)} & Y
	\arrow[Rightarrow, no head, from=1-1, to=1-2]
	\arrow[two heads, from=1-1, to=2-1]
	\arrow["\lrcorner"{anchor=center, pos=0.125}, draw=none, from=1-1, to=2-2]
	\arrow["f", from=1-2, to=2-2]
	\arrow[hook, from=2-1, to=2-2]
\end{tikzcd}
\end{equation*}
so that if $s: y(U) \to \im(f)$ is a section, then taking the pullback, we obtain pasting diagram of pullbacks
\begin{equation*}
   \begin{tikzcd}
	{y(U)\times_YX} & X & X \\
	{y(U)} & {\text{im}(f)} & Y
	\arrow[from=1-1, to=1-2]
	\arrow["k"', from=1-1, to=2-1]
	\arrow["\lrcorner"{anchor=center, pos=0.125}, draw=none, from=1-1, to=2-2]
	\arrow[Rightarrow, no head, from=1-2, to=1-3]
	\arrow[two heads, from=1-2, to=2-2]
	\arrow["\lrcorner"{anchor=center, pos=0.125}, draw=none, from=1-2, to=2-3]
	\arrow["f", from=1-3, to=2-3]
	\arrow["s"', from=2-1, to=2-2]
	\arrow[hook, from=2-2, to=2-3]
\end{tikzcd} 
\end{equation*}
but the outer rectangle is a pullback, so $k$ is in $W$. Since this is true for arbitrary $s : y(U) \to \im(f)$, this implies that $X \twoheadrightarrow \im(f)$ is in $W$. Thus by 2-of-3, $\im(f) \hookrightarrow y(U)$ is in $W$. Therefore $f \in E$.

Now we want to show that if $W$ is a system of local isos, $E = \tau(W)$ is the class of $W$-local epis, and we let $W' = \sigma(E)$ be the class of $E$-local isos, then we want to show that $W = \sigma \tau(W) = W'$. 

First let us show that $W \subseteq W'$. We know that if $f \in W$, then $f \in E$. Thus we need only to show that $\Delta_f \in E$. We have the following commutative diagram
\begin{equation*}
    \begin{tikzcd}
	X \\
	& {X \times_Y X} & X \\
	& X & Y
	\arrow["{\Delta_f}", from=1-1, to=2-2]
	\arrow["{1_X}", curve={height=-12pt}, from=1-1, to=2-3]
	\arrow["{1_X}"', curve={height=12pt}, from=1-1, to=3-2]
	\arrow["{p_1}", from=2-2, to=2-3]
	\arrow["{p_0}"', from=2-2, to=3-2]
	\arrow["\lrcorner"{anchor=center, pos=0.125}, draw=none, from=2-2, to=3-3]
	\arrow["f", from=2-3, to=3-3]
	\arrow["f"', from=3-2, to=3-3]
\end{tikzcd}
\end{equation*}
since $f$ is a $W$-local iso, then so are $p_0$ and $p_1$. But then by $2$-of-$3$, since $1_X$ is also a $W$-local iso, this implies that $\Delta_f$ is a $W$-local iso. Therefore by the above argument since $W \subseteq E$, this implies that $\Delta_f \in E$. Therefore $W \subseteq W'$.

Conversely let us show that $W' \subseteq W$. If $f: X \to Y$ in $W'$, then it is a $W$-local epi and a $W$-local mono. Thus $\im(f) \hookrightarrow Y$ is in $W$. Now note that the following diagram is a pullback
\begin{equation*}
    \begin{tikzcd}
	{X \times_YX} & X \\
	X & {\text{im}(f)}
	\arrow["{p_1}", from=1-1, to=1-2]
	\arrow["{p_0}"', from=1-1, to=2-1]
	\arrow["\lrcorner"{anchor=center, pos=0.125}, draw=none, from=1-1, to=2-2]
	\arrow[two heads, from=1-2, to=2-2]
	\arrow[two heads, from=2-1, to=2-2]
\end{tikzcd}
\end{equation*}
Now note that $\Delta_f : X \to X \times_Y X$ is a monomorphism, which implies that it is its own image, and since $\Delta_f$ is a $W$-local epi, this implies that $\Delta_f \in W$. Now since $p_0 \Delta_f = 1_X$, by 2-of-3, this implies that $p_0$ (and $p_1$) are in $W$. Now if $s : y(U) \to \im(f)$ was an arbitrary section, then it factors through $X$, and taking the pullback we obtain a pasting diagram of pullbacks
\begin{equation*}
    \begin{tikzcd}
	{y(U)\times_{\im(f)}X} & {X \times_YX} & X \\
	{y(U)} & X & {\text{im}(f)}
	\arrow[from=1-1, to=1-2]
	\arrow["k"', from=1-1, to=2-1]
	\arrow["\lrcorner"{anchor=center, pos=0.125}, draw=none, from=1-1, to=2-2]
	\arrow["{p_1}", from=1-2, to=1-3]
	\arrow["{p_0}", from=1-2, to=2-2]
	\arrow["\lrcorner"{anchor=center, pos=0.125}, draw=none, from=1-2, to=2-3]
	\arrow[two heads, from=1-3, to=2-3]
	\arrow["s"', from=2-1, to=2-2]
	\arrow[two heads, from=2-2, to=2-3]
\end{tikzcd}
\end{equation*}
thus $k$ is in $W$. But since this was true for arbitrary $s: y(U) \to \im(f)$, this implies that $X \twoheadrightarrow \im(f)$ is in $W$. Since $\im(f) \hookrightarrow Y$ is also in $W$ since $f$ is a $W$-local epi, this implies that $f \in W$. Thus $W' \subseteq W$. Therefore we have shown that $W = \sigma \tau(W) = W'$.

Now suppose that $E$ is a system of local epis, let $W = \sigma(E)$ and $E' = \tau(W)$. We want to show that $E = \tau \sigma(E) = E'$. Let us show that $E \subseteq E'$. Morphisms in $E'$ are precisely the maps $f: X \to Y$ such that $\im(f) \hookrightarrow Y$ belong to $W$, i.e. are in $E$ and are $E$-local monos. Thus if $f :X \to Y$ is in $E$, so is $\im(f) \hookrightarrow Y$. Thus we need only show it is an $E$-local mono. But it is a monomorphism, thus an $E$-local mono. Therefore $f \in E'$.

Conversely let us show that $E' \subseteq E$. If $\im(f) \hookrightarrow Y$ belongs to $W$, then it belongs to $E$, but so does $X \twoheadrightarrow \im(f)$, since it is an epi, therefore $f \in E$. Thus $E' = \sigma \tau(E) = E$. Thus $\sigma$ and $\tau$ are inverses to each other and define a bijection.
\end{proof}

Thus we have obtained the following bijections
\begin{equation} \label{eq sys local iso <-> sys local epi <-> sat coverages}
\{ \text{system of local isos} \} \underset{\sigma}{\overset{\tau}{\rightleftarrows}} \{\text{system of local epis} \} \underset{\varphi}{\overset{\psi}{\rightleftarrows}} \{\text{saturated coverages} \}
\end{equation}
The compositions are simple to interpret as well. If $j$ is a saturated coverage, then $\sigma \varphi(j)$ is the class of $j$-local isomorphisms. If $W$ is a system of local isomorphisms, then $\psi \tau(W)$ is the coverage where $r \in (\psi \tau(W))(U)$ if and only if $\overline{r} \hookrightarrow y(U)$ is in $W$.

\subsection{Lex Reflective Localizations of Presheaf Topoi} \label{section left exact localizations of presheaf topoi}
In this section we introduce left exact (lex) reflective localizations of presheaf topoi. These will turn out to correspond precisely to Grothendieck topoi. The theory of localization is a vital tool with which to study Grothendieck topoi. We collect the basics of this theory in Appendix \ref{section localizations}, and encourage the reader to head there first before returning to this section.

\begin{Def}
Let $\ncat{Pre}(\cat{C})$ be a presheaf topos. A \textbf{lex localization} of $\Pre(\cat{C})$ consists of a reflective subcategory (Definition \ref{def reflective localization}) 
\begin{equation*}
    \begin{tikzcd}
	{\cat{E} } && {\Pre(\cat{C})}
	\arrow[""{name=0, anchor=center, inner sep=0}, "i"', shift right=2, hook, from=1-1, to=1-3]
	\arrow[""{name=1, anchor=center, inner sep=0}, "L"', shift right=2, from=1-3, to=1-1]
	\arrow["\dashv"{anchor=center, rotate=-90}, draw=none, from=1, to=0]
\end{tikzcd}
\end{equation*}
such that the reflector $L$ preserves finite limits. We call $\cat{E}$ a \textbf{lex reflective subcategory} of $\Pre(\cat{C})$.
\end{Def}

\begin{Lemma} \label{lem lex loc -> sys of local isos}
Let $\cat{E} \hookrightarrow \Pre(\cat{C})$ be a lex reflective subcategory with reflector $L$ and let $W = L^{-1}(\text{iso})$. Then $W$ is a system of local isomorphisms.
\end{Lemma}

\begin{proof}
Clearly (1) and (2) of Definition \ref{def system of local isos} hold. Let us prove (3'). Suppose that $f : X \hookrightarrow Y$ is in $W$. We want to show that $g^*(f)$ is in $W$ for every $g : Z \to Y$. Since $L$ is lex, $L(g^*(f)) = L(g)^*(L(f))$. Since isomorphisms are stable under pullback, this implies that $L(g)^*(L(f))$ is an isomorphism, so $g^*(f)$ is in $W$. Conversely suppose that $g^*(f)$ is in $W$ for every $g : Z \to Y$. Then taking $g$ to be the identity shows that $f$ is in $W$.
\end{proof}

Thus by (\ref{eq sys local iso <-> sys local epi <-> sat coverages}), to every lex reflective subcategory $\cat{E} \hookrightarrow \Pre(\cat{C})$ with reflector $L$, we can associate the saturated coverage $j(L) \coloneqq \psi \tau(L^{-1}(\text{iso}))$. Call this the \textbf{coverage induced by $L$}.

\begin{Th}[The Little Giraud Theorem\footnote{We got this name from \cite[Corollary C.2.1.11]{johnstone2002sketches}. For the actual Giraud Theorem, see Section \ref{section girauds theorem}.}] \label{th lex localizations <-> grothendieck toposes}
Given a lex reflective subcategory $\cat{E} \hookrightarrow \Pre(\cat{C})$, with reflector $L$, there is an equivalence of categories
\begin{equation}
    \cat{E} \simeq \Sh(\cat{C}, j(L)).
\end{equation}
\end{Th}

\begin{proof}
By Proposition \ref{prop local objects equiv to localization}, we know that $\cat{E}$ is equivalent to the full subcategory of $L^{-1}(\text{iso})$-local presheaves. Let $W = L^{-1}(\text{iso})$, and suppose that $X$ is a $W$-local presheaf. Then if $r \in j(L)(U)$ is a covering family, then $\overline{r} \hookrightarrow y(U)$ is in $W$. Thus the function
\begin{equation*}
    \Pre(\cat{C})(y(U), X) \to \Pre(\cat{C})(\overline{r}, X)
\end{equation*}
is a bijection. But this implies that $X$ is sheaf on $\overline{j(L)}$, equivalently a sheaf on $j(L)$.

Now suppose that $Z$ is a $j(L)$-sheaf. Then by Proposition \ref{prop sheaves are local iso local} and Corollary \ref{cor sheaves are local iso local even without saturation}, $Z$ is local with respect to the $j(L)$-local isomorphisms. But the class of $j(L)$-local isomorphisms is precisely the class $W$. Thus $Z$ is $W$-local. Therefore the full subcategory of $W$-local presheaves is precisely the full subcategory of $j(L)$-sheaves.
\end{proof}

To summarize, we have obtained the following commutative diagram
{\footnotesize
\begin{equation} \label{eq cvgs <-> lex subcats}
    \begin{tikzcd}
	{\{\text{saturated coverages on }\cat{C}\}} & {\{\text{systems of local epis on }\Pre(\cat{C})\}} & {\{\text{systems of local isos on }\Pre(\cat{C})\}} \\
	{\{\text{Grothendieck coverages on }\cat{C}\}} && {\{ \text{lex reflective subcategories of }\Pre(\cat{C})\}}
	\arrow["\varphi", shift left=2, from=1-1, to=1-2]
	\arrow["{\overline{(-)}}", shift left=2, from=1-1, to=2-1]
	\arrow["\psi", shift left=2, from=1-2, to=1-1]
	\arrow["\sigma", shift left=2, from=1-2, to=1-3]
	\arrow["\tau", shift left=2, from=1-3, to=1-2]
	\arrow["{(-)^\circ}", shift left=2, from=2-1, to=1-1]
	\arrow[dashed, from=2-1, to=2-3]
	\arrow["{W}"', from=2-3, to=1-3]
\end{tikzcd}
\end{equation}
}
where $W$ is the map that takes a lex reflective subcategory $\cat{E} \hookrightarrow \Pre(\cat{C})$ with reflector $L$ and gives the class $W = L^{-1}(\text{iso})$. We wish to fill in this diagram with the dashed arrow. This will be the content of the next section.

\subsection{Sheafification} \label{section sheafification}

Given a site $(\cat{C}, j)$ and $U \in \cat{C}$, let $\sat{j}(U)$ denote the category whose objects are $j$-saturating families  and whose morphisms are refinements.

\begin{Def} \label{def general plus construction}
Given a site $(\cat{C}, j)$ and a presheaf $X$ on $\cat{C}$, let us define a new presheaf $X^+$ on $\cat{C}$ objectwise as follows. If $U \in \cat{C}$, then let 
\begin{equation} \label{eq plus construction}
    X^+(U) = \underset{r \in \sat{j}(U)^\op}{\colim} \Match(r, X), 
\end{equation}  
This construction defines a functor $(-)^+ : \Pre(\cat{C}) \to \Pre(\cat{C})$, which we call the \textbf{plus construction}.
\end{Def}

The colimit in (\ref{eq plus construction}) is over a filtered category by Lemma \ref{lem meets in saturated coverages}, but it can often be more convenient to work with a filtered poset, i.e. a directed set. Let $\Gro{j}(U)$ denote the poset of $j$-covering families ordered by the subset relation. There is a map
\begin{equation*}
    \phi: \ncolim{R \in \Gro{j}(U)^\op} \Pre(\cat{C})(R, X) \to \ncolim{r \in \sat{j}(U)^\op} \Match(r, X)
\end{equation*}
which considers an equivalence class of a matching family $m : R \to X$ as an equivalence class of matching families by refinement. This map is well-defined, since two matching families $m : R \to X$ and $n : T \to X$ are identified in the left-hand set (by Lemma \ref{lem filtered colimits description in sets}) if there exists a covering sieve $S \subseteq R \cap T$ such that $m|_{R \cap T} = n|_{R \cap T}$, and this implies that there is a common refinement, so $\phi([m]) = \phi([n])$. Now it is surjective, as a $X$-matching family $\{x_i \}$ on a $j$-saturating family $r$ can be extended (non-uniquely) to a matching family $x : \overline{r} \to X$ by choosing for each $f \in \overline{r}$ a factorization
\begin{equation*}
    \begin{tikzcd}
	V && {U_i} \\
	& U
	\arrow["s", from=1-1, to=1-3]
	\arrow["f"', from=1-1, to=2-2]
	\arrow["{r_i}", from=1-3, to=2-2]
\end{tikzcd}
\end{equation*}
and setting $x_f = X(s)(x_i)$. Then we have a refinement $\overline{r} \leq r$ given by the factorizations given above which pulls $\{x_i \}$ back to $x$ and hence identifies them in the quotient. It is also injective. For assume that we have matching families $m : R \to X$ and $n : S \to X$, and a common refinement $f : t \to R$ and $g: t \to S$ such that $f^*(m) = g^*(n)$. Then $\overline{t} \subseteq R$ and $\overline{t} \subseteq S$, so $m$ and $n$ are identified in the domain. Thus we proven the following result.

\begin{Cor} \label{cor computation of plus construction}
Given a site $(\cat{C}, j)$ and a presheaf $X$, then for every $U \in \cat{C}$, we have
\begin{equation*}
    X^+(U) \cong \underset{R \in \Gro{j}(U)^\op}{\colim} \Pre(\cat{C})(R, X).
\end{equation*}
\end{Cor}

\begin{Rem}
We also note that the above result holds just as well if we took the colimit over just $j$-trees rather than $j$-saturating families.
\end{Rem}

We can thus describe $X^+(U)$ explicitly as follows. It is the set of all $X$-matching families over some $R$, where $R$ is a covering sieve on $U$, modulo the following equivalence relation: we identify two matching families $R \xrightarrow{m} X \xleftarrow{n} R'$ if there exists a covering sieve $R'' \subseteq R \cap R'$ such that the following diagram commutes:
\begin{equation*}
\begin{tikzcd}
	R & X \\
	{R''} & {R'}
	\arrow["m", from=1-1, to=1-2]
	\arrow["n"', from=2-2, to=1-2]
	\arrow[hook', from=2-1, to=1-1]
	\arrow[hook, from=2-1, to=2-2]
\end{tikzcd}
\end{equation*}
We denote the image of a matching family $m: R \to X$ in $X^+(U)$ by $[m]$.

If $f: V \to U$ is a morphism in $\cat{C}$, then $X^+(f): X^+(U) \to X^+(V)$ is defined by 
$$X^+(f)([R \xrightarrow{m} X]) = [f^*R \to R \xrightarrow{m} X].$$ 
To fully define $(-)^+$ as a functor we must also define how it acts on presheaf maps $g: X \to Y$. This is simply defined in components as $g^+([R \xrightarrow{m} X]) = [R \xrightarrow{m} X \xrightarrow{g} Y]$. 

\begin{Rem}
We will use the form of the plus construction given in Corollary \ref{cor computation of plus construction} for the rest of this section, as it will significantly reduce the complications of the proofs. In fact, we will assume we are working over a Grothendieck site $(\cat{C}, J)$ for the rest of this section, in which case we have
\begin{equation*}
    X^+(U) \cong \underset{R \in J(U)^\op}{\colim} \Pre(\cat{C})(R,X).
\end{equation*}
\end{Rem}

\begin{Lemma} \label{lem sub-matching families have same amalgamation}
Suppose that $X$ is a presheaf, and suppose that we have two $X$-matching families $m : R \to X$ and $n : R' \to X$ on sieves that fit into the following commutative diagram:
\begin{equation*}
    \begin{tikzcd}
	{R'} && R \\
	& X
	\arrow["n"', from=1-1, to=2-2]
	\arrow["m", from=1-3, to=2-2]
	\arrow[hook, from=1-1, to=1-3]
\end{tikzcd}
\end{equation*}
in which case we call $n$ a \textbf{sub-matching family} of $m$. If a section $x: y(U) \to X$ is an amalgamation for $m$, then it is also an amalgamation for $n$.
\end{Lemma}

\begin{proof}
If $x$ is an amalgamation for $m$, then the following diagram commutes
\begin{equation*}
\begin{tikzcd}
	{R'} && R \\
	& {y(U)} \\
	& X
	\arrow[hook, from=1-1, to=1-3]
	\arrow[hook, from=1-1, to=2-2]
	\arrow["n"', curve={height=6pt}, from=1-1, to=3-2]
	\arrow[hook, from=1-3, to=2-2]
	\arrow["m", curve={height=-6pt}, from=1-3, to=3-2]
	\arrow["x"', from=2-2, to=3-2]
\end{tikzcd}
\end{equation*}
thus $x$ is an amalgamation for $n$.
\end{proof}

\begin{Def} \label{def unit map of plus construction}
Given a Grothendieck site $(\cat{C}, J)$, define the natural transformation $\eta: 1_{\Pre(C)} \Rightarrow (-)^+$ componentwise as follows. If $x : y(U) \to X$ is a section, then set $\eta_X(x) = [y(U) \xrightarrow{x} X]$. Call $\eta$ the \textbf{unit map} of the plus construction.
\end{Def}

\begin{Lemma} \label{lem eta is invertible on sheaves}
If $X$ is a sheaf on a Grothendieck site $(\cat{C}, J)$, then $\eta_X : X \to X^+$ is an isomorphism.
\end{Lemma}

\begin{proof}
Let us construct an inverse to $\eta$ in components. If $U \in \cat{C}$, then we want to define
$\eta_U^{-1}: X^+(U) \to X(U)$. If $[R \xrightarrow{m} X] \in X^+(U)$, then since $X$ is a sheaf, there is a unique amalgamation $y(U) \xrightarrow{x} X$ of $m$, so define $\eta_U^{-1}([R \to X]) = x$ to be that unique amalgamation. Let us show this is well defined. Suppose that $[R \xrightarrow{m} X] = [R' \xrightarrow{n} X]$, so that we have a commutative diagram:
\[\begin{tikzcd}
	{R''} & {R'} \\
	R & X
	\arrow["m"', from=2-1, to=2-2]
	\arrow[hook, from=1-1, to=2-1]
	\arrow[hook', from=1-1, to=1-2]
	\arrow["n", from=1-2, to=2-2]
\end{tikzcd}\]
then since $X$ is a sheaf, $m$ and $n$ amalgamate to a unique $x$ and $y$ respectively. By Lemma \ref{lem sub-matching families have same amalgamation}, this implies that $x$ and $y$ are also amalgamations for $R''$. But $R''$ is a covering sieve and $X$ is a sheaf, so $x = y$. Thus $\eta_U^{-1}$ is well defined and it is easy to see that it defines an inverse map to $\eta$.
\end{proof}

\begin{Lemma}\label{lem universal map to plus construction}
Given a Grothendieck site $(\cat{C}, J)$, and a map $f: X \to Y$ of presheaves on $\cat{C}$, if $Y$ is a sheaf, then $f$ factors uniquely:
\begin{equation*}
\begin{tikzcd}
X & {X^+} \\
{} & Y
\arrow["{\eta_X}", from=1-1, to=1-2]
\arrow["f"', from=1-1, to=2-2]
\arrow["{\exists ! \widetilde{f}}", from=1-2, to=2-2]
\end{tikzcd}
\end{equation*}
\end{Lemma}

\begin{proof}
Since $Y$ is a sheaf, the map $\eta_Y$ in the naturality square
\begin{equation*}
\begin{tikzcd}
X & {X^+} \\
Y & {Y^+}
\arrow["f"', from=1-1, to=2-1]
\arrow["{\eta_X}", from=1-1, to=1-2]
\arrow["{f^+}", from=1-2, to=2-2]
\arrow["{\eta_Y}"', from=2-1, to=2-2]
\end{tikzcd}	
\end{equation*}
is an isomorphism by Lemma \ref{lem eta is invertible on sheaves}. So let us define $\widetilde{f} = \eta_Y^{-1} f^+$. Now let us show that $\widetilde{f}$ is unique. Suppose that $\widehat{f} : X^+ \to Y$ is a map of presheaves such that $\widehat{f} \eta_X = f$. Then given a map $g: U \to V$ in $\cat{C}$, we have the following commutative diagram
\begin{equation*}
\begin{tikzcd}
	{X^+(U)} & {X^+(V)} \\
	{Y(U)} & {Y(V)}
	\arrow["{Y(g)}"', from=2-1, to=2-2]
	\arrow["{\widehat{f}_V}", from=1-2, to=2-2]
	\arrow["{{X^+(g)}}", from=1-1, to=1-2]
	\arrow["{\widehat{f}_U}"', from=1-1, to=2-1]
\end{tikzcd}
\end{equation*}
So if $[R \xrightarrow{m} X] \in X^+(U)$ for some covering sieve $R \in J(U)$, and $g: V \to U$ is any morphism, then 
$$\left( Y(g) \circ \widehat{f}_U \right) \left( [R \xrightarrow{m} X] \right) = \left( \widehat{f}_V \circ X^+(g) \right) \left( [R \xrightarrow{m} X] \right) = \widehat{f}_V \left( [g^* R \to R \xrightarrow{m} X] \right).$$ 
Now if $g \in R$, then by Lemma \ref{lem pullback of sieve by map in sieve is maximal sieve}, $g^* R = y(V)$. Thus for $g \in R$, 
$$\left( Y(g) \circ \widehat{f}_U \right) \left( [R \xrightarrow{m} X] \right) =  \widehat{f}_V \left( [y(V) \xrightarrow{g} R \xrightarrow{m} X] \right).$$
Thus the composite $y(V) \xrightarrow{g} R \xrightarrow{m} X$ is precisely the element $m_g \in X(V)$ of the matching family $m : R \to X$.

Therefore we have
$$\widehat{f}_V \left( [m_g] \right) = \widehat{f}_V \eta_X(m_g) = f(m_g).$$

Now let $\widehat{f}_U ([R \xrightarrow{m} X] ) = y$, so that $Y(g)(y) = f(m_g)$. In other words, we have shown that $y$ is an amalgamation for the $Y$-matching family $\{ f(m_g) \}$. But 
$$\widetilde{f}_U([R \xrightarrow{m} X]) = \eta^{-1}_Y f^+_U([R \xrightarrow{m} X]) = \eta^{-1}_Y([R \xrightarrow{m} X \xrightarrow{f} Y]).$$
Thus $\widetilde{f}_U([R \xrightarrow{m} X])$ is the unique amalgamation of the $Y$-matching family $\{ f(m_g) \}$. But as shown above, $y$ is such an amalgamation. Thus $\widehat{f}_U([R \xrightarrow{m} X]) = \widetilde{f}_U([R \xrightarrow{m} X])$.
\end{proof}

\begin{Prop}\label{prop plus construction is separated}
Given a presheaf $X$ on a Grothendieck site $(\cat{C}, J)$, the presheaf $X^+$ is separated (Definition \ref{def separated and sheaf}).
\end{Prop}

\begin{proof}
We wish to show that for every $R \in J(U)$ the restriction map
\begin{equation*}
    \text{res}_{R,X^+} : X^+(U) \to \Pre(\cat{C})(R,X^+)
\end{equation*}
is injective. The map $\text{res}_{R,X^+}$ takes an element $[R \xrightarrow{m} X]$ to the collection of matching families $\{ [g^*(R) \to R \xrightarrow{m} X] \}_{g \in R}$, where $g : V \to U$ is a morphism in $R$ and $[g^*(R) \to R \xrightarrow{m} X] \in X^+(V)$. So suppose that $m, n \in X^+(U)$ such that $\text{res}_{R, X^+}(m) = \text{res}_{R,X^+}(n)$. This means that for every $g : V \to U$ in $R$, there is a covering sieve $S_g \hookrightarrow y(V)$ such that the following diagram commutes
\begin{equation*}
    \begin{tikzcd}
	{g^*(R)} & R & X \\
	&& R \\
	S_g && {g^*(R)}
	\arrow[from=1-1, to=1-2]
	\arrow["m", from=1-2, to=1-3]
	\arrow["n"', from=2-3, to=1-3]
	\arrow[hook', from=3-1, to=1-1]
	\arrow[hook, from=3-1, to=3-3]
	\arrow[from=3-3, to=2-3]
\end{tikzcd}
\end{equation*}
But then $S = \bigcup_{g \in R} g_*(S_g)$ is a covering sieve over $U$ by Lemma \ref{lem grothendieck coverages are composition closed}, and $m|_S = n|_S$. Thus $[R \xrightarrow{m} X] = [R \xrightarrow{n} X]$.
\end{proof}

\begin{Th}
Given a Grothendieck site $(\cat{C}, J)$, if $X$ is a $J$-separated presheaf on $\cat{C}$, then $X^+$ is a $J$-sheaf.
\end{Th}

\begin{proof}
We wish to show that given a matching family $R \xrightarrow{m} X^+$, there exists a unique amalgamation $y(U) \xrightarrow{x} X^+$. By Proposition \ref{prop plus construction is separated}, we need only to show that it exists, since $X^+$ being separated guarantees the amalgamation is unique.

So given $m : R \to X^+$, for each $g: V \to U \in R$, we have an element $[R_g \xrightarrow{m_g} X] \in X^+(V)$, where $R_g \in J(V)$. Now consider $\widetilde{R} = \bigcup_{g \in R} g_*(R_g)$. We wish to construct a natural transformation $\widetilde{R} \xrightarrow{\sigma} X$ that will serve as the amalgamation $x = [\widetilde{R} \xrightarrow{\sigma} X]$ of $\{ [m_g] \in X^+(V) \}_{g : V \to U \in R}$.

We will construct this natural transformation in terms of its components and thus will have to show it is well defined and natural. If $h : W \to U \in \widetilde{R}$, then $h = g t$ for some $g \in R$ and some $t : W \to V \in R_g$. Define $\sigma_W: \widetilde{R}(W) \to X(W)$ by $\sigma_W(h) =  m_g(t)$.

Let us show that this is well defined. Suppose that $h = g t = g' t'$ for $g : V \to U$ and $g' : V' \to U$ in $R$, and $t : W_g \to V \in R_g$, $t' : W_{g'} \to V' \in R_{g'}$. We want to show that $m_g(t) = m_{g'}(t')$. Since $\{[m_g]\}$ is an $X^+$-matching family, we have
\begin{equation*}
    X^+(t)[m_g] = [m_{gt}] = [m_{g't'}] = X^+(t')[m_{g'}].
\end{equation*}
So there exists a covering sieve $S \subseteq t^* R_g \cap (t')^*(R_{g'})$ such that $m_g|_S = m_{g'}|_S$, or equivalently for every $s \in S$,
$$X(s)(m_g(t)) = X(s)(m_{g'}(t')).$$
But $\{X(s)(m_g(t)) = X(s)(m_{g'}(t')) \}_{s \in S}$ is an $X$-matching family over $S$. But $S$ is a covering sieve, $X$ is separated, and both $m_g(t)$ and $m_{g'}(t')$ are amalgamations. Thus $m_g(t) = m_{g'}(t')$, which implies that $\sigma_W$ is well-defined.

We must now show that the $\sigma_W$ assemble into a natural transformation, namely we wish to show that if $k: W' \to W$ is an arbitrary morphism in $\cat{C}$, then the following diagram commutes
\begin{equation*}
 \begin{tikzcd}
	{\widetilde{R}(W)} & {X(W)} \\
	{\widetilde{R}(W')} & {X(W')}
	\arrow["{\sigma_W}", from=1-1, to=1-2]
	\arrow["{\widetilde{R}(k)}"', from=1-1, to=2-1]
	\arrow["{X(k)}", from=1-2, to=2-2]
	\arrow["{\sigma_{W'}}"', from=2-1, to=2-2]
\end{tikzcd}  
\end{equation*}
So suppose that $h \in \widetilde{R}(W)$. Then $X(k) \sigma_W(h) = X(k)(m_g(t))$ for some $g \in R$ and some $t \in R_g(W)$, and $\sigma_{W'} \widetilde{R}(k)(h) = \sigma_{W'}(hk) = m_g(tk)$. But we know that $m_g (tk) = m_g(R_g(k)(t)) = X(k)(m_g(t))$ since $m_g: R_g \to X$ is a natural transformation. Thus $\sigma : \widetilde{R} \to X$ is a well defined natural transformation.

All that is left to show is that $[\widetilde{R} \xrightarrow{\sigma} X]$ is an amalgamation of $m$. In other words we need to show that $X^+(g)[\sigma] = [g^* \widetilde{R} \to \widetilde{R} \xrightarrow{\sigma} X] = [R_g \xrightarrow{m_g} X].$ 

Now $R_g \subseteq g^* \widetilde{R} = g^* \bigcup_{h \in R} h_* R_h$ since $R_g \subseteq g^* g_* R_g$, and if $t \in R_g$, then $gt \in \widetilde{R}$ and $\sigma(gt) = m_g(t)$. Thus the following diagram commutes
\begin{equation*}
    \begin{tikzcd}
	{g^* \widetilde{R}} & {\widetilde{R}} \\
	{R_g} & X
	\arrow[from=1-1, to=1-2]
	\arrow["\sigma", from=1-2, to=2-2]
	\arrow[hook', from=2-1, to=1-1]
	\arrow["{m_g}"', from=2-1, to=2-2]
\end{tikzcd}
\end{equation*}
which is what we wanted to show. Thus $[\sigma]$ is an amalgamation of $\{ [m_g] \}$.
\end{proof}

\begin{Cor}
Given a Grothendieck site $(\cat{C}, J)$ and a presheaf $X$, the presheaf $(X^+)^+$ is a $J$-sheaf.
\end{Cor}

\begin{Def} \label{def sheafification}
Given a Grothendieck site $(\cat{C}, J)$, let $a: \Pre(\cat{C}) \to \Sh(\cat{C}, J)$ denote the composite functor $a = ((-)^+)^+$. We call this the \textbf{sheafification functor}.
\end{Def}

\begin{Prop} \label{prop sheafification props}
Given a Grothendieck site $(\cat{C}, J)$, the functor $a: \Pre(\cat{C}) \to \Sh(\cat{C}, J)$
\begin{enumerate}
	\item preserves finite limits,
	\item is left adjoint to the inclusion $i: \Sh(\cat{C}) \hookrightarrow \Pre(\cat{C})$, and
	\item if $X$ is a sheaf, then the unit $\eta_X : X \to ai X$ of the adjunction $a \dashv i$, is an isomorphism.
\end{enumerate}
\end{Prop}

\begin{proof}
(1) First note that for every $U \in \cat{C}$, the poset $J(U)$ is cofiltered. This follows from Lemma \ref{lem covering sieve props}. Therefore if $X : K \to \Pre(\cat{C})$ is a finite diagram of presheaves, then $$(\lim_{k \in K} X(k))^+(U) = \underset{R \in J(U)^\op}{\colim} \Pre(\cat{C})(R, \lim_{k \in K} X(k)) \cong \underset{R \in J(U)^\op}{\colim} \lim_{k \in K} \Pre(\cat{C})(R, X(k)),$$
and since filtered colimits commute with finite limits in $\Set$ (Proposition \ref{prop filtered colimits commute with finite limits in Set}), we have
$$\underset{R \in J(U)^\op}{\colim} \lim_{k \in K} \Pre(\cat{C})(R, X(k)) \cong \lim_{k \in K} \underset{R \in J(U)^\op}{\colim} \Pre(\cat{C})(R,X(k)) \cong (\lim_{k \in K} X^+(k))(U).$$

(2) Let $X$ be a presheaf, and consider the composition:
\begin{equation*}
X \xrightarrow{\eta_X} X^+ \xrightarrow{\eta_{X^+}} X^{++} = aX
\end{equation*}
Now if we have a map $X \to iY$, where $Y$ is a sheaf, then two applications of Lemma \ref{lem universal map to plus construction} gives a unique map $aX \to Y$. It is easy to check this defines an adjunction $a \dashv i$ and the composite map $\eta_{X^+} \circ \eta_X$ is the unit of the adjunction.

(3) follows from Lemma \ref{lem eta is invertible on sheaves}.
\end{proof}

\begin{Rem} \label{rem sheafification on a general site}
Given a site $(\cat{C}, j)$ and a presheaf $X$, then $aX = (X^+)^+$ is a $j$-sheaf, because $aX$ is a $\Gro{j}$-sheaf, and by Corollary \ref{cor grothendieck closure preserves sheaves}, $a$ is a functor $a : \Pre(\cat{C}) \to \Sh(\cat{C}, j)$ and has all the same properties as in Proposition \ref{prop sheafification props}.
\end{Rem}

\begin{Cor} \label{cor sheafification gives lex localization}
Given a site $(\cat{C}, j)$, the full subcategory $\Sh(\cat{C}, j) \hookrightarrow \Pre(\cat{C})$ is a lex reflective subcategory of $\Pre(\cat{C})$.
\end{Cor}

\begin{Prop} \label{prop sheafification inverts local isos}
Given a site $(\cat{C}, j)$, the class of morphisms $W = a^{-1}(\text{iso})$ that are inverted by sheafification is precisely the class $L$ of $j$-local isomorphisms.
\end{Prop}

\begin{proof}
By Proposition \ref{prop sheaves are local iso local}, sheaves are all $L$-local. Clearly if a presheaf $X$ is $L$-local, then it is a sheaf, since all inclusions $\overline{r} \hookrightarrow y(U)$, where $r$ is a $j$-covering family are $j$-local isomorphisms, by Lemma \ref{lem saturating iff sifted closure is saturating}. So sheaves are precisely the $L$-local objects. But by Lemma \ref{lem W-local objects and equivalences of reflective localizations} and Corollary \ref{cor sheafification gives lex localization}, they are also precisely the $W$-local objects. But again by Lemma \ref{lem W-local objects and equivalences of reflective localizations}, $W$ is precisely the class of $W$-local equivalences. But this is precisely the class of $L$-local equivalences, since a presheaf is $W$-local if and only if it is $L$-local if and only if it is a sheaf. Hence $W = L$. 
\end{proof}

Now we can fill in the final piece of the commutative diagram (\ref{eq cvgs <-> lex subcats}). Note that we don't need to include the bijection between saturated coverages and Grothendieck coverages, as sheafification is defined for any general site.
\begin{equation}
    \begin{tikzcd}
	{\{\text{systems of local epis on }\Pre(\cat{C})\}} & {\{\text{systems of local isos on }\Pre(\cat{C})\}} \\
	{\{\text{saturated coverages on }\cat{C}\}} & {\{ \text{lex reflective subcategories of }\Pre(\cat{C})\}}
	\arrow["\sigma", shift left=2, from=1-1, to=1-2]
	\arrow["\psi", shift left=2, from=1-1, to=2-1]
	\arrow["\tau", shift left=2, from=1-2, to=1-1]
	\arrow["{W^{-1}}"', shift right=2, from=1-2, to=2-2]
	\arrow["\varphi", shift left=2, from=2-1, to=1-1]
	\arrow["a", shift left=2, from=2-1, to=2-2]
	\arrow["{W}"', shift right=2, from=2-2, to=1-2]
	\arrow["{a^{-1}}", shift left=2, from=2-2, to=2-1]
\end{tikzcd}
\end{equation}
where $W^{-1}$ is the map $\alpha \psi \tau$ and $a^{-1}$ is the map $\psi \tau W$. It is easy to see that each morphism above is a bijection with inverse given parallel to it, and furthermore the entire diagram commutes. The function $W \circ a$ is precisely the map that sends a saturated coverage $j$ to its class of $j$-local isomorphisms.

\begin{Cor} \label{cor sheaf topoi are bicomplete}
Given a site $(\cat{C}, j)$, the category of sheaves $\Sh(\cat{C}, j)$ has all small limits and colimits. Limits are computed pointwise, and agree with those in $\Pre(\cat{C})$. Colimits are computed by taking the colimit in $\Pre(\cat{C})$ and then sheafifying the result.
\end{Cor}

\begin{proof}
This follows from Proposition \ref{prop (co)limits in reflective subcategories}.    
\end{proof}

\begin{Cor} \label{cor sat or gro coverages are rigid}
Saturated and Grothendieck coverages are rigid in the sense that if $j$ and $j'$ are coverages on a small category $\cat{C}$ such that $\ncat{Sh}(\cat{C}, j) \simeq \ncat{Sh}(\cat{C}, j')$, then $\sat{j} = \sat{j'}$ and $\Gro{j} = \Gro{j'}$.
\end{Cor}

\begin{Cor} \label{cor j-local iso/epi/mono <-> sheafification is iso/epi/mono}
A map $f: X \to Y$ of presheaves is a $j$-local iso/epi/mono-morphism if and only if $af$ is an iso/epi/mono-morphism of sheaves.
\end{Cor}

\begin{proof}
The map $f$ is a $j$-local isomorphism if and only if $af$ is an isomorphism by Proposition \ref{prop sheafification inverts local isos}. If $f: X \to Y$ is a $j$-local epimorphism, then $\im(f) \hookrightarrow Y$ is a $j$-local isomorphism. Now $a$ preserves colimits and finite limits by Proposition \ref{prop sheafification props} and since the image of $f$ coincides with the regular coimage by Lemma \ref{lem image of presheaf map}, $\im(af) \cong a(\im(f))$. Thus $\im(af) \hookrightarrow aY$ is an isomorphism, and hence an epimorphism. Furthermore $a$ preserves epimorphisms, so $aX \twoheadrightarrow \im(af)$ is an epimorphism in $\Sh(\cat{C}, j)$. Therefore $af$ is an epimorphism.

Conversely, if $af : aX \to aY$ is an epimorphism in $\Sh(\cat{C},j)$, then $\im(af) \hookrightarrow aY$ is an isomorphism, so $\im(f) \hookrightarrow Y$ is a $j$-local isomorphism, and therefore $f$ is a $j$-local epimorphism. Similarly, $af$ is a monomorphism if and only if it is a $j$-local monomorphism.
\end{proof}

\begin{Cor} \label{cor monos of sheaves}
Let $X$ and $Y$ be sheaves on a site $(\cat{C}, j)$. Then a map $f: X \to Y$ is a $j$-local monomorphism if and only if it is a monomorphism of sheaves if and only if it a monomorphism of presheaves. 
\end{Cor}

\begin{proof}
This follows from Corollary \ref{cor j-local iso/epi/mono <-> sheafification is iso/epi/mono}, Proposition \ref{prop (co)limits in reflective subcategories} and the fact that monomorphisms can be characterized by pullback diagrams.
\end{proof}

\begin{Rem} \label{rem counterexample of strong local epis}
Here we will consider an extended example to show that we cannot replace $j$-local epimorphisms in Corollary \ref{cor j-local iso/epi/mono <-> sheafification is iso/epi/mono} with strong $j$-local epimorphisms (Remark \ref{rem epis are local epis}). We will show that there exists a $j$-local epimorphism $f : X \to Y$ that is not a strong $j$-local epimorphism and such that $af : aX \to aY$, where $a$ here is the sheafification functor of Definition \ref{def sheafification}, is an epimorphism of sheaves.

Consider the category $\cat{C}$, given as follows
\begin{equation*}
    \begin{tikzcd}
	W & V & U
	\arrow["q", from=1-1, to=1-2]
	\arrow["p", from=1-2, to=1-3]
\end{tikzcd}
\end{equation*}
equipped with the collection of families $j$, defined by
\begin{equation*}
    j(U) = \{ 1_U, p \}, \qquad j(V) = \{ 1_V, q \}, \qquad j(W) = \{ 1_W \}.
\end{equation*}
It is easy to see that this collection of families is a coverage. It is refinement closed, but not composition closed. Consider the presheaves $X$ and $Y$ on $\cat{C}$ defined by
\begin{equation*}
    X(U) = \{a \}, \; X(V) = \{b \}, \; X(W) = \{c \}, \; X(p)(a) = b, \; X(q)(b) = c,
\end{equation*}
\begin{equation*}
    Y(U) = \{\alpha, \alpha' \}, \; Y(V) = \{\beta, \beta' \}, \; Y(W) = \{ \gamma \}, \; Y(p)(\alpha) = \beta, \; Y(p)(\alpha') = \beta', \; Y(q)(\beta) = Y(q)(\beta') = \gamma.
\end{equation*}
Now let $f : X \to Y$ be the morphism of presheaves defined by
\begin{equation*}
    f_U(a) = \alpha, \qquad f_V(b) = \beta, \qquad f_W(c) = \gamma.
\end{equation*}
Now this map is not a strong $j$-local epimorphism. For instance, consider the section $\alpha' : y(U) \to Y$. There are two covering families on $U$, the identity and $p$. There is no section $s : y(U) \to X$ such that $f_U s = \alpha'$, so $f$ does not lift against the identity. Similarly, consider the covering family $p$. The composite map
\begin{equation*}
    y(V) \xrightarrow{p} y(U) \xrightarrow{\alpha'} Y
\end{equation*}
is the section $\beta' : y(V) \to Y$. There is no section $s: y(V) \to X$ such that $f_V s = \beta'$. However, note that $f$ does lift against $q$ using the identity covering family on $W$.

Now for every section $s : y(U) \to Y$, $f$ lifts against the $j$-tree $T_{pq}$ given by $T_{pq}^\circ = \{pq \}$. In other words, for $s = \alpha, \alpha'$, the following diagram commutes
\begin{equation*}
    \begin{tikzcd}
	{y(W)} & X \\
	{y(V)} \\
	{y(U)} & Y
	\arrow["c", from=1-1, to=1-2]
	\arrow["q"', from=1-1, to=2-1]
	\arrow["f", from=1-2, to=3-2]
	\arrow["p"', from=2-1, to=3-1]
	\arrow["s"', from=3-1, to=3-2]
\end{tikzcd}
\end{equation*}
because $spq = \gamma = f_W(c)$. Thus $f$ is a $j$-local epimorphism.

Now let us consider $f^+ : X^+ \to Y^+$. Clearly $X$ is a sheaf, so $X^+ \cong X$ by \ref{lem eta is invertible on sheaves}. Now $Y$ is not a sheaf, because the map
\begin{equation*}
    Y(V) \to \ncolim{r \in \sat{j}(V)^\op} \Match(r, Y)
\end{equation*}
is not a bijection. Indeed, there are only two $j$-saturating families on $V$, given by the identity and $q$, and there is a unique refinement $q \to 1_V$, making $q$ terminal in $\sat{j}(V)^\op$. Thus
\begin{equation*}
    \ncolim{r \in \sat{j}(V)^\op} \Match(r, Y) \cong \Match(q, Y) \cong Y(W).
\end{equation*}
The map $Y(V) \to Y(W)$ is thus $Y(q)$, which is not a bijection. Now let us compute $Y^+$. The category of $j$-saturating families on $U$ looks like
\begin{equation*}
   \begin{tikzcd}
	W & V & U \\
	V \\
	U & U & U
	\arrow["q", from=1-1, to=1-2]
	\arrow["q"', from=1-1, to=2-1]
	\arrow["p", from=1-2, to=1-3]
	\arrow["p", from=1-2, to=3-2]
	\arrow["{1_U}", from=1-3, to=3-3]
	\arrow["p"', from=2-1, to=3-1]
	\arrow["{1_U}"', from=3-1, to=3-2]
	\arrow["{1_U}"', from=3-2, to=3-3]
\end{tikzcd} 
\end{equation*}
so $T_{pq}$ is a terminal object in $\sat{j}(U)^\op$. Hence
\begin{equation*}
    \ncolim{r \in \sat{j}(U)^\op} \Match(r, Y) \cong \Match(T_{pq}^\circ, Y) \cong \{(y_V, y_W) \in Y(V) \times Y(W) \, | \, Y(q)(y_V) = y_W \} \cong \{(\beta, \gamma), (\beta', \gamma) \}.
\end{equation*}
So we have
\begin{equation*}
    Y^+(U) = \{(\beta, \gamma), (\beta', \gamma) \}, \; Y^+(V) = \{ \gamma \}, \; Y^+(W) = \{ \gamma \}, \; Y^+(p)(\beta, \gamma) = Y^+(p)(\beta', \gamma) = \gamma, \; Y^+(q)(\gamma) = \gamma.
\end{equation*}
This is still not a sheaf. Doing the plus construction again, and using similar arguments to the above, we find
\begin{equation*}
    aY(U) = aY(V) = aY(W) = \{ \gamma \}, \; aY(p) = aY(q) = 1_{\{\gamma \}}.
\end{equation*}
Now both $X$ and $Y$ are terminal sheaves on $(\cat{C}, j)$, and the map $af : aX \cong X \to aY$ is an isomorphism, and therefore an epimorphism of sheaves.
\end{Rem}

\subsection{Variant of the Plus Construction}
In this short section, we describe a different, but equivalent version of the plus construction. This variant is useful in the context of higher category theory and appears for example in \cite{luriesheafification} and \cite{nlab:plus_construction_on_presheaves}. Since this section is not terribly central to our tale, we'll be a little bit more terse in our descriptions than usual.

Given a Grothendieck site $(\cat{C}, J)$, let $\text{Cov}_J: \cat{C}^{\op} \to \ncat{Poset}$ denote the pseudofunctor that sends an object $U$ to $J(U)$, thought of as a poset, and sends a morphism $f: U \to V$ to the functor $f^* : J(V) \to J(U)$ of posets defined by $(R \hookrightarrow y(V)) \mapsto (f^* R \hookrightarrow y(U))$. Taking the Grothendieck construction of this pseudofunctor provides us with a category $J(\cat{C}) \coloneqq \int \text{Cov}_J$, whose objects are pairs $(U \in \cat{C}, R \in J(U))$ and morphisms $(U,R) \to (V,S)$ are maps $f: U \to V$ such that $R \subseteq f^* S$. This has a canonical functor $r: J(\cat{C}) \to \cat{C}$ defined by $(U, R) \mapsto U$. This functor $r$ has a right adjoint $s: \cat{C} \to J(\cat{C})$, defined by $U \mapsto (U, y(U))$.

The functors $r$ and $s$ induce functors on the presheaf categories by Lemma \ref{lem presheaf adjoint triple}, and since adjoints are unique up to isomorphism, we end up with the following quadruple adjunction:
\begin{equation} \label{eqn sheafification 4 adjuntion}
\begin{tikzcd}
	{\Pre(\cat{C})} &&& {\Pre(J(\cat{C}))}
	\arrow[""{name=0, anchor=center, inner sep=0}, "{\Delta_r \cong \Sigma_s}"{description}, curve={height=-12pt}, from=1-1, to=1-4]
	\arrow[""{name=1, anchor=center, inner sep=0}, "{\Sigma_r}"', curve={height=30pt}, from=1-4, to=1-1]
	\arrow[""{name=2, anchor=center, inner sep=0}, "{\Pi_r \cong \Delta_s}"{description}, curve={height=-12pt}, from=1-4, to=1-1]
	\arrow[""{name=3, anchor=center, inner sep=0}, "{\Pi_s}"{description}, curve={height=30pt}, from=1-1, to=1-4]
	\arrow["\dashv"{anchor=center, rotate=-90}, draw=none, from=1, to=0]
	\arrow["\dashv"{anchor=center, rotate=-90}, draw=none, from=0, to=2]
	\arrow["\dashv"{anchor=center, rotate=-90}, draw=none, from=2, to=3]
\end{tikzcd}
\end{equation}

Using the limit definition of right Kan extension, we see that if $X \in \Pre(\cat{C})$, then
\begin{equation*}
(\Pi_s X)(U,R) \cong \lim \left[ \left( (U,R) \downarrow s^{\op} \right) \to \cat{C}^{\op} \xrightarrow{X} \ncat{Set} \right].
\end{equation*}

\begin{Lemma}
If $X$ is a presheaf, then
$$ (\Pi_s X)(U,R) \cong \lim_{V \to U \in R} X(V) \cong \Pre(\cat{C})(R,X)$$
\end{Lemma}

\begin{proof}
Let us examine the comma category $((U,R) \downarrow s^{\op})$. This has objects $((U,R), V, f)$, where $V \in \cat{C}$ and $f$ is a map in $J(\cat{C})^{\op}$ from $(U,R)$ to $s^{\op}(V)$, which is equivalently a map $(V,yV) \to (U,R)$. Such a map is equivalent to a map $f: V \to U$ such that $y(V) \hookrightarrow f^*R$, which is only possible if $f \in R$. Some diagram chasing then proves the lemma.
\end{proof}

Similarly the formula for left Kan extension gives that if $G \in \Pre(J(\cat{C}))$, then
\begin{equation*}
(\Sigma_r G)(U) \cong \colim \left[ (r^{\op} \downarrow U) \to J(\cat{C})^{\op} \xrightarrow{G} \ncat{Set} \right]
\end{equation*}

\begin{Lemma}
If $G$ is a presheaf on $J(\cat{C})$, then
$$(\Sigma_r G)(U) \cong \underset{R \in J(U)^{\op}}{\colim} G(U,R)$$
\end{Lemma}

\begin{proof}
First note that $(r^{\op} \downarrow U) \cong (U \downarrow r)^{\op}$. Consider the functor $X: J(U) \to (U \downarrow r)$ that takes an inclusion $R \hookrightarrow R'$ to:
\begin{equation*}
    \begin{tikzcd}
	& U \\
	{r(U, R)} && {r(U,R')}
	\arrow["{r(1_U)}"', from=2-1, to=2-3]
	\arrow["{1_U}"', from=1-2, to=2-1]
	\arrow["{1_U}", from=1-2, to=2-3]
\end{tikzcd}
\end{equation*}
We want to show that this functor is initial (Definition \ref{def final functor}), thus proving the lemma. So we need to show that if $d \in (U \downarrow r)$, then we want to show that $(X \downarrow d)$ is connected. An object in $(X \downarrow d)$ is a triangle:
\begin{equation*}
    \begin{tikzcd}
	& U \\
	{r(U, R)} && {r(V,S)}
	\arrow["{r(d)}"', from=2-1, to=2-3]
	\arrow["{1_U}"', from=1-2, to=2-1]
	\arrow["d", from=1-2, to=2-3]
\end{tikzcd}
\end{equation*}
So if we have two objects, we want to show they are connected by a zig-zag of morphisms. The following commutative diagram proves this:
\begin{equation*}
    \begin{tikzcd}
	&& U \\
	&&&& {r(U,R')} \\
	&& {r(U,R \cap R')} \\
	{r(U, R)} &&& {r(V,S)}
	\arrow["{r(d)}"', from=4-1, to=4-4]
	\arrow["{1_U}"', from=1-3, to=4-1]
	\arrow["d"{pos=0.4}, from=1-3, to=4-4]
	\arrow["{r(d)}", from=2-5, to=4-4]
	\arrow["{1_U}", from=1-3, to=2-5]
	\arrow["{r(1_U)}"', from=3-3, to=4-1]
	\arrow["{r(1_U)}"{pos=0.7}, from=3-3, to=2-5]
	\arrow["{1_U}", from=1-3, to=3-3]
	\arrow["r(d)"'{pos=0.4}, from=3-3, to=4-4]
\end{tikzcd}
\end{equation*}
because if $R \hookrightarrow d^* S$ and $R' \hookrightarrow d^* S$, then $R \cap R' \hookrightarrow d^* S$.
\end{proof}

Thus if $U \in \cat{C}$, we have
\begin{equation}
X^+(U) \cong (\Sigma_r \Pi_s X)(U) \cong \ncolim{R \in J(U)^{\op}} \ncat{Pre}(\cat{C})(R,X).
\end{equation}

\subsection{Sheafification in one go}
While we can obtain a variety of useful properties of the sheafification functor $a : \Pre(\cat{C}) \to \ncat{Sh}(\cat{C}, j)$ by its description as the plus construction applied twice $a = (-)^{++}$, it is very difficult to work with in practice. In this section, we describe a more useful description of the sheafification of a presheaf. We follow the wonderful blog post \cite{kim2020sheafifyinonego}, and also mention the similar descriptions given in \cite[Page 34]{caramello2019cohomology} and \cite{yuhjtman2007bilimites}.

\begin{Def}
Let $(\cat{C}, j)$ be a site and $X$ a presheaf on $\cat{C}$. We say that two sections $s,t \in X(U)$ 
are \textbf{locally equal} if there exists a $j$-covering family $r = \{r_i : U_i \to U \}_{i \in I}$ such that $X(r_i)(s) = X(r_i)(t)$ for all $i \in I$.

Given a family of morphisms $r = \{r_i : U_i \to U \}_{i \in I}$ we say that a collection of sections $\{x_i \in X(U_i) \}_{i \in I}$ is $X$-\textbf{locally matching} if for every intersection square
\begin{equation*}
  \begin{tikzcd}
	U_{ij} & {U_j} \\
	{U_i} & U
	\arrow["{r_i}"', from=2-1, to=2-2]
	\arrow["{r_j}", from=1-2, to=2-2]
	\arrow["u_i"', from=1-1, to=2-1]
	\arrow["u_j", from=1-1, to=1-2]
\end{tikzcd}  
\end{equation*}
the sections $X(u_i)(x_i)$ and $X(u_j)(x_j)$ are locally equal.
\end{Def}

\begin{Lemma}
Suppose that $\{x_i \}$ is an $X$-locally matching family on $t = \{t_j :V_j \to U \}_{j \in J}$, and there is a refinement $f : r \to t$ where $r = \{r_i: U_i \to U \}_{i\in I}$ and with index map $\alpha : I \to J$. Then  $f^* \{x_j \} = \{X(f_i)(x_{\alpha(i)})\}_{i \in I}$ is locally matching on $r$.
\end{Lemma}

\begin{proof}
This follows from the same argument as in Lemma \ref{lem pullback of matching family by refinement is a matching family}.
\end{proof}

Thus we obtain a presheaf $\text{LocMatch}(-,X) : \text{Fam}(U)^\op \to \ncat{Set}$.

\begin{Def} \label{def hypersheafification}
Given a site $(\cat{C}, j)$, a presheaf $X$ and $U \in \cat{C}$, let $X^\dagger$ denote the presheaf defined objectwise by
\begin{equation} \label{eq hypersheafification}
    X^\dagger(U) = \ncolim{r \in \sat{j}(U)^\op} \text{LocMatch}(r,X).
\end{equation}
The action of $X^\dagger$ on morphisms is described as follows. Suppose that $f : U \to V$ is a morphism, $t = \{t_j: V_j \to V \}$ is a $j$-saturating family and $\{x_j \}$ is an $X$-locally matching family on $t$, then since $\sat{j}$ is a coverage, there exists a $j$-saturating family $r = \{U_i \to U \}$ and maps $s_i : U_i \to V_{\alpha(i)}$ such that the following diagram commutes
\begin{equation*}
    \begin{tikzcd}
	{U_i} & {V_{\alpha(i)}} \\
	U & V
	\arrow["{s_i}", from=1-1, to=1-2]
	\arrow["{r_i}"', from=1-1, to=2-1]
	\arrow["{t_\alpha(i)}", from=1-2, to=2-2]
	\arrow["f"', from=2-1, to=2-2]
\end{tikzcd}
\end{equation*}
So define $X^\dagger(f)[\{x_j \}] = [\{X(s_i)(x_{\alpha(i)})\}]$. We will show in Lemma \ref{lem hypersheafification well defined on morphisms} that this is well-defined.
\end{Def}

\begin{Def}
Given two families of morphisms $r = \{r_i : U_i \to U \}_{i \in I}$ and $t = \{t_j : V_j \to U \}_{j \in J}$ with families of local sections $\{x_i \in X(U_i) \}_{i \in I}$ and $\{y_j \in X(V_j) \}_{j \in J}$, we say that $\{x_i \}$ and $\{y_j \}$ are $X$-\textbf{locally equivalent} if for each commutative square of the form
\begin{equation*}
\begin{tikzcd}
	{W_{ij}} & {V_j} \\
	{U_i} & U
	\arrow["{v_j}", from=1-1, to=1-2]
	\arrow["{u_i}"', from=1-1, to=2-1]
	\arrow["{t_j}", from=1-2, to=2-2]
	\arrow["{r_i}"', from=2-1, to=2-2]
\end{tikzcd} 
\end{equation*}
for every $i$ and $j$, the sections $X(u_i)(x_i)$ and $X(v_j)(y_j)$ are locally equal. We write $\{x_i \} \approx_X \{y_j \}$ to denote that they are locally equivalent.
\end{Def}

\begin{Rem}
We note that a family $\{x_i \in X(U_i) \}$ of local sections is $X$-locally equivalent to itself $\{x_i \} \approx_X \{x_i \}$ if and only if $\{x_i \}$ is $X$-locally matching. Thus we see that on $X$-locally matching families, $\approx_X$ is an equivalence relation.
\end{Rem}

\begin{Lemma} \label{lem refinement implies locally equivalent}
Given a site $(\cat{C}, j)$, families of morphisms
$r = \{r_i :U_i \to U \}_{i \in I}$, $t = \{t_j : V_j \to U \}_{j \in J}$, and an $X$-locally matching family $\{x_j \in X(V_j) \}_{j \in J}$ on $t$, suppose that $f : r \to t$ is a refinement. Then $f^* \{x_j \} \approx_X \{x_j \}$.
\end{Lemma}

\begin{proof}
Suppose we have a commutative diagram of the form
\begin{equation*}
\begin{tikzcd}
	{W_{ij}} & {V_j} \\
	{U_i} & U
	\arrow["{v_j}", from=1-1, to=1-2]
	\arrow["{u_i}"', from=1-1, to=2-1]
	\arrow["{t_j}", from=1-2, to=2-2]
	\arrow["{r_i}"', from=2-1, to=2-2]
\end{tikzcd} 
\end{equation*}
Then since $f$ is a refinement with index map $\alpha : I \to J$, we can extend this to the commutative diagram
\begin{equation*}
\begin{tikzcd}
	{W_{ij}} && {V_j} \\
	& {V_{\alpha(i)}} \\
	{U_i} && U
	\arrow["{{v_j}}", from=1-1, to=1-3]
	\arrow["{v_{\alpha(i)}}"', from=1-1, to=2-2]
	\arrow["{{u_i}}"', from=1-1, to=3-1]
	\arrow["{{t_j}}", from=1-3, to=3-3]
	\arrow["{t_\alpha(i)}"', from=2-2, to=3-3]
	\arrow["{f_i}", from=3-1, to=2-2]
	\arrow["{{r_i}}"', from=3-1, to=3-3]
\end{tikzcd}    
\end{equation*}
Since $\{x_j \}$ is $X$-locally matching, this implies that $X(v_{\alpha(i)})(x_{\alpha(i)})$ and $X(v_j)(x_j)$ are locally equal. But $X(v_{\alpha(i)})(x_{\alpha(i)}) = X(u_i)X(f_i)(x_{\alpha(i)})$. Hence $f^*\{x_j \} = \{X(f_i)(x_{\alpha(i)}) \}_{i \in I}$ and $\{x_j \}_{j \in J}$ are $X$-locally equivalent.
\end{proof}

\begin{Lemma} \label{lem sections identified if locally equiv}
Given a site $(\cat{C}, j)$, a presheaf $X$, two $X$-locally matching families $\{x_i \}$ and $\{y_j\}$ are identified in $X^\dagger(U)$ if and only if they are $X$-locally equivalent.  
\end{Lemma}

\begin{proof}
Suppose that $r = \{r_i : U_i \to U \}_{i \in I}$ and $t = \{t_j : V_j \to U \}_{j \in J}$ are $j$-saturating families on $U$ and $\{x_i\}$, $\{y_j \}$ are $X$-locally matching families on $r$ and $t$ respectively.

$(\Rightarrow)$ Suppose that $x = [\{x_i \}]$ and $y = [\{y_j \}]$ are identified as elements of $$X^\dagger(U) = \ncolim{s\in \sat{j}(U)^\op} \text{LocMatch}(s, X).$$ Since $\sat{j}(U)$ is finitely cofiltered by Lemma \ref{lem meets in saturated coverages}, then by Lemma \ref{lem filtered colimits description in sets}, this means that there exists a $j$-saturating family $s$ on $U$ with refinements $f : s \to r$ and $g: s \to t$, such that $f^* \{x_i \} = g^* \{ y_j \}$. But $\{x_i \} \approx_X f^* \{x_i \} = g^* \{y_j \} \approx_X \{y_j \}$ by Lemma \ref{lem refinement implies locally equivalent}, hence $\{x_i \}$ and $\{y_j \}$ are locally equivalent.

$(\Leftarrow)$ Suppose that $\{x_i \}$ and $\{y_j \}$ are $X$-locally equivalent matching families. Then again, since $\sat{j}(U)$ is cofiltered, there exists a common refinement $f : s \to r$ and $g : s \to t$, where $s = \{s_k : W_k \to U\}_{k \in K}$. But then for every $k \in K$ we have the commutative diagram
\begin{equation*}
\begin{tikzcd}
	{W_k} & {V_{\beta(k)}} \\
	{U_{\alpha(k)}} & U
	\arrow["{g_k}", from=1-1, to=1-2]
	\arrow["{f_k}"', from=1-1, to=2-1]
	\arrow["{s_k}"{description}, from=1-1, to=2-2]
	\arrow["{t_{\beta(k)}}", from=1-2, to=2-2]
	\arrow["{r_{\alpha(k)}}"', from=2-1, to=2-2]
\end{tikzcd}
\end{equation*}
But then $r_\alpha = \{r_{\alpha(k)} : U_{\alpha(k)} \to U \}_{k \in K}$ and $t_\beta = \{t_{\beta(k)} : V_{\beta(k)} \to U \}_{k \in K}$ are $j$-saturating since they are both refined by $s$. Thus there exists a $j$-covering family $h^k = \{h^k_\ell : A^k_\ell \to W_k \}_{\ell \in L^k}$ such that $X(f_k h^k_\ell)(r_{\alpha(k)}) = X(g_k h^k_\ell)(t_{\beta(k)})$. Now there are obvious refinements $r_\alpha \to r$ and $t_\beta \to t$ given by identity maps. The composite family $h \circ s$ on $U$ is $j$-saturating, refines both $r_\alpha$ and $t_\beta$, hence refining both $r$ and $t$, and when $\{x_i \}$ and $\{y_j \}$ are pulled back along the refinements are equal. Hence they are identified in $X^\dagger(U)$. 
\end{proof}

\begin{Lemma} \label{lem hypersheafification well defined on morphisms}
The action of $X^\dagger$ on morphisms as described after Definition \ref{def hypersheafification} is well-defined.
\end{Lemma}

\begin{proof}
Suppose that $f : U \to V$ is a morphism, $t = \{t_j: V_j \to V \}$ is a $j$-saturating family and $\{x_j \}$ is an $X$-locally matching family on $t$, then since $\sat{j}$ is a coverage, there exists a $j$-saturating family $r = \{r_i : U_i \to U \}$ and maps $a_i : U_i \to V_{\alpha(i)}$ such that $t_{\alpha(i)} a_i = f r_i$. Suppose we chose a different $j$-saturating family $s = \{s_{i'} : U_{i'}\to U \}$ with maps $b_{i'} : U_{i'} \to V_{\beta(i')}$. Then for every intersection square for $s$ and $r$ we get a larger intersection square for $t$
\begin{equation*}
    \begin{tikzcd}
	& {U_{ii'}} \\
	{U_{i'}} & U & {U_i} \\
	{V_{\beta(i')}} & V & {V_{\alpha(i)}}
	\arrow["{c_{i'}}"', from=1-2, to=2-1]
	\arrow["{d_{i}}", from=1-2, to=2-3]
	\arrow["{s_{i'}}"', from=2-1, to=2-2]
	\arrow["{b_{i'}}"', from=2-1, to=3-1]
	\arrow["f"', from=2-2, to=3-2]
	\arrow["{r_i}", from=2-3, to=2-2]
	\arrow["{a_i}", from=2-3, to=3-3]
	\arrow["{t_{\beta(i')}}"', from=3-1, to=3-2]
	\arrow["{t_{\alpha(i)}}", from=3-3, to=3-2]
\end{tikzcd}
\end{equation*}
Since $\{x_j \}$ is $X$-locally matching, this implies that $\{X(a_i)(x_{\alpha(i)} \}$ and $\{X(b_{i'})(x_{\beta(i')}) \}$ are locally equivalent. Hence by Lemma \ref{lem sections identified if locally equiv}, they are identified in $X^\dagger(U)$ and thus its action on morphisms is well-defined. Choosing different representatives for $[\{x_i \}]$ is proved similarly.
\end{proof}

\begin{Prop}
Given a site $(\cat{C}, j)$ and a presheaf $X$ on $\cat{C}$, the presheaf $X^\dagger$ is a $j$-sheaf.
\end{Prop}

\begin{proof}
Given $U \in \cat{C}$, let $r = \{r_i : U_i \to U \}$ be a $j$-covering family on $U$ and let $\{x_i \in X^\dagger(U_i) \}$ be a $X^\dagger$-matching family on $r$. We want to show that there is a unique amalgamation. Each $x_i$ is an equivalence class $\{ x^i_a \}$ of an $X$-locally matching family on a $j$-saturating family $r^i = \{r^i_a : U^i_a \to U_i \}_{a \in A^i}$, which are all $X$-locally equivalent when pulled back along intersection squares. Now by composing all of the $j$-saturating families we obtain a $j$-saturating family $\{U^i_a \to U \}$ and it is not hard to see that $\cup_i \{x^i_a \}$ is now an $X$-locally matching family, and furthermore $[\cup_i \{x^i_a \}]$ is an amalgamation for the $x_i$. 

Now suppose that $x$ and $y$ are amalgamations for $\{x_i\}$. So there are $j$-saturating families $c= \{c_p : C_p \to U \}$ and $d = \{d_q : D_q \to U \}$ and $X$-locally matching families $\{x_p \}$, $\{y_q \}$ that amalgamate the $\{x^i_a \}$. So we have a commutative diagram
\begin{equation*}
    \begin{tikzcd}
	{U^i_a} & {U_i} & {U_b^{i'}} \\
	{C_p} & U & {D_q}
	\arrow["{r^i_a}", from=1-1, to=1-2]
	\arrow[from=1-1, to=2-1]
	\arrow["{r_i}", from=1-2, to=2-2]
	\arrow["{r^{i'}_b}"', from=1-3, to=1-2]
	\arrow[from=1-3, to=2-3]
	\arrow["{c_p}", from=2-1, to=2-2]
	\arrow["{d_q}"', from=2-3, to=2-2]
\end{tikzcd}
\end{equation*}
Such that when $\{x_p\}$ and $\{y_q\}$ are pulled back, they are locally equivalent to $x_i$. Hence they are locally equivalent to each other, so $x$ and $y$ are identified in $X^\dagger(U)$.
\end{proof}

For any presheaf $X$, there is a map
\begin{equation*}
    \eta : X \to X^\dagger
\end{equation*}
that sends a section $x \in X(U)$ to the equivalence class of the $X$-locally matching family $\{x \}$ on the $j$-saturating family $\{1_U : U \to U \}$. 

\begin{Lemma}
Given a site $(\cat{C}, j)$ if $X$ is a $j$-sheaf then the map $\eta$ is an isomorphism.
\end{Lemma}

\begin{proof}
Since $X$ is a $j$-sheaf, if two sections $x$ and $y$ are $X$-locally equal, then they are equal. Thus a family of sections is $X$-locally matching if and only if they are $X$-matching. In other words, $X^\dagger$ is precisely $X^+$ when $X$ is a sheaf, and hence by Lemma \ref{lem eta is invertible on sheaves}, the map $\eta$ is an isomorphism.
\end{proof}

\begin{Lemma}
Given a site $(\cat{C}, j)$, any map $f: X \to Y$ of presheaves where $Y$ is a sheaf factors uniquely through a map $\widetilde{f} : X^\dagger \to Y$.
\end{Lemma}

\begin{proof}
The proof is similar to Lemma \ref{lem universal map to plus construction}, the map $\widetilde{f}$ sends an equivalence class of an $X$-locally matching family to the unique amalgamation of its image under $f$ in $Y$.
\end{proof}

Since $X^\dagger$ is a sheaf, this implies that $(-)^\dagger$ is left adjoint to the inclusion $i : \ncat{Sh}(\cat{C},j) \hookrightarrow \ncat{Pre}(\cat{C})$. Since adjoints are unique up to isomorphism, Proposition \ref{prop sheafification props} implies the following result.

\begin{Cor}
Given a site $(\cat{C}, j)$ and a presheaf $X$, there is an isomorphism
\begin{equation*}
    X^\dagger \cong aX,
\end{equation*}
where $a$ denotes sheafification by two applications of the plus construction (Definition \ref{def sheafification}).
\end{Cor}

\section{Examples of Sites} \label{section some sites}
In this section, we finally introduce a wide variety of examples of sites from across the literature. Doing this requires delving into some set-theoretic difficulties, as we do in Section \ref{section large sites}. First, we discuss some particular properties of sites that are often satisfied by examples in the literature, such as being (sub)canonical or pullback-stable.

\subsection{More Kinds of Coverages}

\subsubsection{(Sub)canonical coverages}

\begin{Def}\label{def subcanonical site}
We say that a site $(\cat{C}, j)$ is \textbf{subcanonical} if for every $U \in \cat{C}$, the representable presheaf $y(U)$ is a $j$-sheaf.
\end{Def}

\begin{Ex}
The sites of Example \ref{ex canonical coverages}, the site $(\mathcal{O}(X), j_X)$ of a topological space $X$ from Example \ref{ex open cover coverage}, the basis coverage from Example \ref{ex basis for a topology as a coverage}, and the site $(\ncat{FinSet}, j_{\text{epi}})$ from Example \ref{ex set joint epi coverage} are all subcanonical.
\end{Ex}

\begin{Def} \label{def canonical coverage}
We say that a coverage $j$ on a small category $\cat{C}$ is \textbf{canonical} if $\Gro{j}$ is the largest subcanonical Grothendieck coverage on $\cat{C}$. Hence there is a unique canonical Grothendieck coverage on any small category.
\end{Def}

\begin{Def} \label{def colimit sieve}
Given a small category $\cat{C}$, we say that a sieve $R \hookrightarrow y(U)$ is a \textbf{colimit sieve} if $\{f : V \to U \}_{f \in R}$ forms a colimit cocone\footnote{specifically it is a colimit over the diagram $\pi : (R \downarrow U) \to \cat{C}$, where $(R \downarrow U)$ is the full subcategory of $\cat{C}/_U$ on the morphisms in $R$, and $\pi$ is defined by $\pi(f : V \to U) = V$.}. We say that $R$ is a \textbf{universal colimit sieve} if it is a colimit sieve and furthermore for every morphism $f: V \to U$, the sieve $f^*(R)$ is a colimit sieve.

Given a small category $\cat{C}$, let $J_{\text{uni}}$ denote the sifted collection of families given by the universal colimit sieves.
\end{Def}

\begin{Lemma} \label{lem colimit sieves iff representables are sheaves on them}
Given a small category $\cat{C}$ a sieve $R \hookrightarrow y(U)$ is a colimit sieve if and only if every representable $y(V)$ is a sheaf on $R$.
\end{Lemma}

\begin{proof}
By Lemma \ref{lem sheaves on sieves condition}, a representable $y(V)$ is a sheaf on $R$ if and only if
\begin{equation*}
    y(V)(U) \cong \lim_{f : W \to U \in R} y(V)(W).
\end{equation*}
But this means that
\begin{equation*}
    \cat{C}(U,V) \cong \lim_{f \in R} \cat{C}(W,V) \cong \cat{C}\left( \ncolim{f : W \to U \in R} W, V \right).
\end{equation*}
Since this is true for every $V \in \cat{C}$, by the Yoneda lemma this implies that $U \cong \ncolim{f : W \to U \in R} W$. In other words, $R$ is a colimit sieve.
\end{proof}

Hence if $(\cat{C}, J)$ is a subcanonical Grothendieck coverage, then every $J$-covering sieve $R$ is a colimit sieve. The following result gives a nice characterization of the canonical Grothendieck coverage. It is surprisingly nontrivial to prove.

\begin{Prop}[{\cite[Theorem 3.1.1 and Theorem 3.2.1]{lester2019canonical}}] \label{prop characterization of canonical grothendieck topology}
Given a small category $\cat{C}$, the collection of families $J_{\text{uni}}$ is a Grothendieck coverage, and it is the canonical Grothendieck coverage.
\end{Prop}

\begin{Prop}[{\cite[Theorem 3.2.4]{lester2019canonical}}] \label{prop characterization of canonical pretopology}
Let $\cat{C}$ be a cocomplete category with all pullbacks, and furthermore suppose that its coproducts are stable under pullback and disjoint\footnote{see Section \ref{section girauds theorem} for the definition of disjoint coproducts.}. Let $j_{\text{uniepi}}$ denote the collection of families where $\{r_i : U_i \to U \} \in j_{\text{uniepi}}(U)$ if and only if $\sum_i r_i : \sum_i U_i \to U$ is a universal effective epimorphism\footnote{i.e. it is an effective epimorphism that is stable under pullbacks}. Then $j_{\text{uniepi}}$ is a Grothendieck pretopology, and it is canonical.
\end{Prop}

\subsubsection{Stability} \label{section stability}

\begin{Def}
Given a small category $\cat{C}$, we say that a coverage $j$ on $\cat{C}$ is \textbf{pullback-stable} if for any $j$-covering family $r = \{r_i : U_i \to U \}_{i \in I}$ and for any morphism $g : V \to U$ in $\cat{C}$, the pullback
\begin{equation*}
    \begin{tikzcd}[ampersand replacement=\&]
	{V \times_U U_i} \& {U_i} \\
	V \& U
	\arrow[from=1-1, to=1-2]
	\arrow["{g^*(r_i)}"', from=1-1, to=2-1]
	\arrow["\lrcorner"{anchor=center, pos=0.125}, draw=none, from=1-1, to=2-2]
	\arrow["{r_i}", from=1-2, to=2-2]
	\arrow["g"', from=2-1, to=2-2]
\end{tikzcd}
\end{equation*}
exists in $\cat{C}$, and the family $\{g^*(r_i) : V \times_U U_i \to V \}_{i \in I}$ is a $j$-covering family of $V$.
\end{Def}

One advantage of working with pullback-stable coverages is Lemma \ref{lem coverage if pullbacks exist}. There are other technical conveniences as well.

\begin{Def}[{\cite[Definition 3.3]{roberts2012internal}}]
Let $(\cat{C}, j)$ be a pullback-stable site. We say that a covering family $r = \{r_i : U_i \to U \}_{i \in I}$ is \textbf{effective}, if $r$ is a colimit cocone of the diagram consisting of all spans $U_i \leftarrow U_i \times_U U_j \rightarrow U_j$. Note that if $r$ is a singleton, then $r$ is effective if and only if it is a regular epimorphism.
\end{Def}

\begin{Lemma} \label{lem subcanonical on pullback-stable condition}
If $(\cat{C}, j)$ is a pullback-stable site, then it is subcanonical (Definition \ref{def subcanonical site}) if and only if all of its covering families are effective.
\end{Lemma}

\begin{proof}
$(\Rightarrow)$ Suppose that $(\cat{C}, j)$ is pullback-stable and subcanonical. Then for every $U \in \cat{C}$, the representable presheaf $y(U)$ is a $j$-sheaf. Thus by Corollary \ref{cor classical sheaf condition}, if $r = \{r_i : U_i \to U \}$ is a $j$-covering family and $V \in \cat{C}$, then
\begin{equation*}
\begin{aligned}
    \cat{C}(U,V) \cong y(V)(U) & \cong \text{eq} \left( \prod_i y(V)(U_i) \rightrightarrows \prod_{i,j} y(V)(U_i \times_U U_j) \right) \\
    & \cong \text{eq} \left( \prod_i \cat{C}(U_i, V) \rightrightarrows \prod_{i,j} \cat{C}(U_i \times_U U_j, V) \right) \\
    & \cong \cat{C}\left( \text{coeq} \left( \coprod_i U_i \leftleftarrows \coprod_{i,j} U_i \times_U U_j \right), V \right). 
\end{aligned}    
\end{equation*}
But since $V$ was arbitrary, by the Yoneda lemma, we have that $U \cong \text{coeq}  \left( \coprod_i U_i \leftleftarrows \coprod_{i,j} U_i \times_U U_j \right)$. Thus $r$ is effective.
$(\Leftarrow)$ If all of the covering families are effective, then the reverse of the above argument shows that $\cat{C}$ is subcanonical.
\end{proof}

One of the most common kinds of coverages that appear in the literature are Grothendieck pretopologies. These are especially common in differential and algebraic geometry.

\begin{Def} \label{def grothendieck pretopology}
Given a small category $\cat{C}$, we say that a coverage $j$ is a \textbf{Grothendieck pretopology} if $\{f : V \to U \} \in j(U)$ for every isomorphism $f$, $j$ is composition closed, and $j$ is pullback-stable.
\end{Def}

A majority of the sites we will see in the next section are pullback-stable, but there are several important exceptions.

\subsubsection{Dense and atomic coverages}

\begin{Def}
Given a small category $\cat{C}$, let $J_{\text{dense}}$ denote the sifted collection of families where a sieve $R \hookrightarrow y(U)$ belongs to $J_{\text{dense}}$ if and only if for every map $f : V \to U$ there exists a map $g : W \to V$ such that $fg \in R$. In other words, $R \in J_{\text{dense}}(U)$ if and only if $f^*(R)$ is nonempty for every morphism $f$. It is easy to see that $J_{\text{dense}}$ is a Grothendieck coverage, which we call the \textbf{dense Grothendieck coverage} on $\cat{C}$. We say a coverage $j$ is dense if $\Gro{j}$ is the dense Grothendieck coverage.
\end{Def}

An important application of the dense Grothendieck coverage is the following result.

\begin{Prop}[{\cite[Corollary VI.1.5]{maclane2012sheaves}}]
Given a small category $\cat{C}$, the category of sheaves on the dense Grothendieck site $(\cat{C}, J_{\text{dense}})$ is equivalent to the sheaves on $\Pre(\cat{C})$ with respect to the double negation Lawvere-Tierney topology\footnote{We do not cover Lavwere-Tierney topologies in these notes, see \cite[Section V.1]{maclane2012sheaves} for more.}$j_{\neg \neg}$.
\end{Prop}

\begin{Def}[{\cite[Page 3]{caramello2012atomic}}] \label{def atomic site}
Given a small category $\cat{C}$, let $J_{\text{at}}$ denote the smallest Grothendieck coverage on $\cat{C}$ such that every nonempty sieve is $J_{\text{at}}$-covering. We call this the \textbf{atomic Grothendieck coverage}. We say a coverage $j$ is atomic if $\Gro{j} = J_{\text{at}}$.
\end{Def}

\begin{Rem}
We note that since Grothendieck coverages are closed under (small) intersections by Lemma \ref{lem intersections of grothendieck coverages}, we know that $J_{\text{at}}$ exists by taking the intersection of the set of Grothendieck coverages that have nonempty covering sieves.
\end{Rem}

For any small category it is clear that $J_{\text{dense}} \subseteq J_{\text{at}}$. We will now characterize when they are equal.

\begin{Def} \label{def ore condition}
We say that a small category $\cat{C}$ satisfies the (right) \textbf{Ore condition} if for every cospan of morphisms $A \to B \leftarrow C$ there exists a span $A \leftarrow D \to B$ such that the square commutes
\begin{equation*}
    \begin{tikzcd}
	D & C \\
	A & B
	\arrow[dashed, from=1-1, to=1-2]
	\arrow[dashed, from=1-1, to=2-1]
	\arrow[from=1-2, to=2-2]
	\arrow[from=2-1, to=2-2]
\end{tikzcd}
\end{equation*}
Given a small category $\cat{C}$ let $J_{\text{Ore}}$ denote the sifted collection families where $R \in J_{\text{Ore}}(U)$ is covering if and only if $R$ is nonempty.
\end{Def}

\begin{Rem}
Clearly any category with pullbacks satisfies the Ore condition.
\end{Rem}

\begin{Lemma}
If $\cat{C}$ is a small category that satisfies the Ore condition, then $J_{\text{Ore}}$ is a Grothendieck coverage. Furthermore it is the atomic Grothendieck coverage and the dense Grothendieck coverage.
\end{Lemma}

\begin{proof}
Clearly $J_{\text{Ore}}$ satisfies (G1). Now suppose that $R \hookrightarrow y(U)$ is nonempty and $f : V \to U$ is a morphism. Since $R$ is nonempty, there exists a map $g : V' \to U$ in $R$. Hence we have a cospan
\begin{equation*}
    \begin{tikzcd}
	W & {V'} \\
	V & U
	\arrow["k", dashed, from=1-1, to=1-2]
	\arrow["h"', dashed, from=1-1, to=2-1]
	\arrow["g", from=1-2, to=2-2]
	\arrow["f"', from=2-1, to=2-2]
\end{tikzcd}
\end{equation*}
Since $\cat{C}$ satisfies the Ore condition, there exists a map $h : W \to V$ such that $fh \in R$. Hence $h \in f^*(R)$. Thus $J_{\text{Ore}}$ satisfies (G2).

Now suppose that $R \hookrightarrow y(U)$ is nonempty, $R' \hookrightarrow y(U)$ is a sieve and for every morphism $f : V \to U$ the sieve $f^*(R')$ is nonempty. Then choosing $f = 1_U$, we see that $R'$ is nonempty and hence $J_{\text{Ore}}$-covering. Thus $J_{\text{Ore}}$ satisfies (G3) and is therefore a Grothendieck coverage.

It is clear that $J_{\text{Ore}} = J_{\text{at}}$, and it is also clear that $J_{\text{Ore}} = J_{\text{dense}}$.
\end{proof}

\begin{Cor}
If $\cat{C}$ is a small category and $J_{\text{dense}} = J_{\text{at}}$ on $\cat{C}$, then $\cat{C}$ satisfies the Ore condition.
\end{Cor}

When $\cat{C}$ satisfies the Ore condition, we will refer to $J_{\text{Ore}} = J_{\text{at}} = J_{\text{dense}}$ as the atomic Grothendieck coverage. We say that a site $(\cat{C}, j)$ is Ore-atomic if $\cat{C}$ satisfies the Ore condition and $j$ is atomic. Checking the sheaf condition is a little bit easier on Ore-atomic sites.

\begin{Lemma}[{\cite[Lemma III.4.2]{maclane2012sheaves}}]
Let $(\cat{C}, j)$ be an Ore-atomic site. Then a presheaf $X$ on $\cat{C}$ is a $j$-sheaf if and only if for every morphism $f : V \to U$, every section $x \in X(V)$, and every fork\footnote{I.e. $fg = fh$.}
\begin{equation*}
   \begin{tikzcd}
	W & V & U
	\arrow["g", shift left, from=1-1, to=1-2]
	\arrow["h"', shift right, from=1-1, to=1-2]
	\arrow["f", from=1-2, to=1-3]
\end{tikzcd} 
\end{equation*}
such that $X(g)(x) = X(h)(x)$, there exists a unique $y \in X(U)$ such that $x = X(f)(y)$.
\end{Lemma}

\subsection{Examples}

Originally the domain of algebraic topology, sheaf theory has now extended itself across mathematics. We recommend the book \cite{rosiak2022sheaf} for more on how sheaf theory is extending outwards. Here we try and give a wide array of examples of sites and their properties.

\begin{Ex}
Given a topological space $X$, recall the open cover coverage $(\mathcal{O}(X), j_X)$ from Example \ref{ex open cover coverage}. As we saw in Example \ref{ex open site is saturated}, this is a saturated site. Furthermore it is pullback-stable.
\end{Ex}

\begin{Ex}
Recall the jointly epimorphic coverage $(\ncat{FinSet}, j_{\text{epi}})$ from Example \ref{ex set joint epi coverage}. From Example \ref{ex jointly epi coverage is saturated}, we know that this coverage is saturated. It is easy to check that it is also pullback-stable, and hence also a Grothendieck pretopology. By Proposition \ref{prop characterization of canonical pretopology}, it is also canonical.

Let us also consider the singleton coverage $j_{\text{sepi}}$\footnote{sepi stands for singleton epi.} where a covering family consists of a single surjective function $r = \{f : S' \to S\}$. If $\{r_i : S_i \to S \}$ is a $j_{\text{epi}}$-covering family of a set $S$, then it can be refined by the map $\pi: \sum_i S_i \to S$. Hence $j_\text{sepi} \leq j_{\text{epi}}$ and clearly $j_\text{epi} \leq j_{\text{sepi}}$. Thus the two coverages are equivalent.

Note that every $j_\text{epi}$-sheaf $X$ is completely determined by its value on a singleton $X(*) = A$. This is because every set has a finest covering family given by the inclusion of all of its elements $\{x : * \hookrightarrow S \}_{x \in S}$, and all of its intersections are empty. So $X(S) = X( \sum_{x \in S} *) \cong \prod_{x \in S} X(*) \cong S^A$. From this it is not hard to prove that $\ncat{Sh}(\ncat{FinSet}, j_{\text{epi}}) \cong \ncat{Set}$.
\end{Ex}

If $(\cat{C}, j)$ is a site and $U \in \cat{C}$, then $j(U)$ always has a terminal object given by the family $(1_U)$. However, it is far more rare for $j(U)$ to have an initial object. The next example demonstrates when each $j(U)$ has an initial object. Having such an initial object (which we might also just call think of as the ``most refined cover of $U$''), can be very useful, especially when considering sheafification, see Section \ref{section sheafification}.

\begin{Ex} \label{ex alexandrov topology}
We say that a topology $\tau$ on a set $X$ is \textbf{Alexandrov} if arbitrary intersections of open sets are open. We say a topological space equipped with an Alexandrov topology is an Alexandrov space. Given a preorder $(P, \leq)$, we can equip $P$ with a canonical topology $\tau_{\text{Alex}}$ whose open sets are the up-sets of $P$, i.e. those subsets $U \subseteq P$ such that if $x \in U$ and $x \leq y$, then $y \in U$. This makes $(P, \tau_{\text{Alex}})$ into an Alexandrov space. In fact, every Alexandrov space can be constructed in this way, see \cite{asness2018}. This means that in particular for every $x \in P$, the set $(\uparrow x) = \{y \in P \, : \, x \leq y \}$ is open. That means that there is a ``most refined cover'' $\text{min}(P) = \{ (\uparrow x) \subseteq P \}_{x \in P}$ of $P$. In other words, given any object $U \in \mathcal{O}(P)$, $j_P(U)$ has an initial object given by $\text{min}(U) = \{ (\uparrow x) \subseteq U \}_{x \in U}$.
\end{Ex}

\begin{Ex} \label{ex cellular sheaves}
Let $P$ be a preorder equipped with the Alexandrov topology (Example \ref{ex alexandrov topology}), and consider its corresponding site $(\mathcal{O}(P), j_P)$. Let us define a functor $\varphi : \Sh(\mathcal{O}(P), j_P) \to \Fun(P, \ncat{Set})$ by setting $\varphi(X)(x) = X((\uparrow x))$. This makes sense because if $x \leq y$, then $(\uparrow y) \subseteq (\uparrow x)$. It turns out (\cite[Theorem 4.2.10]{curry2014sheaves}) that this functor $\varphi$ is an equivalence, with quasi-inverse $\psi: \Fun(P, \ncat{Set}) \to \Sh(\mathcal{O}(P), j_P)$ given by
\begin{equation*}
    \psi(F)(U) \coloneqq \lim_{x \in U} F(x).
\end{equation*}
This can be equivalently described as follows, if $F: P \to \ncat{Set}$ is a functor, then $\psi(F)$ is the right Kan extension of $F$ along the functor $\uparrow : P \to \mathcal{O}(P)^\op$ given by $x \mapsto (\uparrow x)$. In other words, for preorders $P$ we have
\begin{equation*}
   \Fun(P, \ncat{Set}) \cong \Sh(\mathcal{O}(P), j_P).
\end{equation*}
This observation is the starting point for the theory of \textbf{cellular sheaves}, see \cite{curry2014sheaves} for a detailed introduction to this subject.
\end{Ex}

\begin{Ex}
A \textbf{graph}\footnote{Often, these are called simple graphs.} $G$ consists of a finite set $V(G)$ and a set $E(G)$ of subsets of $V(G)$ of cardinality $2$. A morphism of graphs $f : G \to H$ is a function $V(f) : V(G) \to V(H)$ such that if $\{x,y\} \in E(G)$, then $\{f(x), f(y) \} \in E(H)$. Let $\ncat{Gr}$ denote the small category\footnote{In other words, let $\ncat{Gr}$ denote a small skeleton of the category of graphs.} of graphs. Let us define a collection of families $j_{\ncat{Gr}}$ on $\ncat{Gr}$ as follows. If $G$ is a graph, then we say that a family $\{r_i: G_i \to G \}_{i \in I}$ is covering if each $r_i$ is a monomorphism (it is injective on vertices), and furthermore $\cup_{i \in I} r_i(G_i) \cong G$. It is not hard to see that this coverage is saturated and pullback-stable. This site has been shown to be useful in algorithmics, see \cite{althaus2023compositional}. We note that for each graph $G$, $j_{\ncat{Gr}}(G)$ has an initial object given by the cover consisting of all the vertices and edges of $G$. Hence a $j_{\ncat{Gr}}$-sheaf is completely determined by its values on a vertex $K^1$ and an edge $K^2$\footnote{$K^n$ is the complete graph on $n$ vertices, i.e. $V(K^n) = \{1, \dots, n \}$ and every pair of distinct vertices is connected by an edge.}.
\end{Ex}

\begin{Ex} \label{ex open coverage on Man}
Let $\ncat{Man}$ denote the small category\footnote{See Example \ref{ex Man is essentially small}} whose objects are finite dimensional smooth manifolds and whose morphisms are smooth functions. Define a collection of families $j_\text{open}$ on $\ncat{Man}$ as follows: For $M \in \ncat{Man}$, let $j_\text{open}(M)$ denote the collection of open covers of $M$. In other words a family $\{r_i : U_i \to M \}$ is a $j_{\text{open}}$-covering family if and only if $U_i$ is an open subset of $M$, $r_i$ is the inclusion map, and $\cup_i \, U_i = M$. Then $j_\text{open}$ is a coverage. Indeed if $\{U_i \subseteq M \}$ is an open cover and $f: N \to M$ is a smooth map, then $f^{-1}(U_i)$ is a smooth open submanifold of $N$ (in fact it is a pullback of $f$ and the inclusion, even though not all pullbacks exist in $\ncat{Man}$) for each $i$, and $\{ f^{-1}(U_i) \subseteq N \}$ is an open cover of $N$. It is easy to see that $(\ncat{Man}, j_{\text{open}})$ is a Grothendieck pretopology. However it is not refinement closed.

We can similarly define another coverage $j_{\text{emb}}$ on $\ncat{Man}$ as follows. Say a family $\{r_i : U_i \to M \}$ of morphisms is a $j_{\text{emb}}$-covering family if and only if each $r_i : U_i \to M$ is an open embedding\footnote{This means that it is an immersion, and the underlying map of topological spaces is a homeomorphism onto its image.}, and $\cup_i \, r_i(U_i) = M$. It is easy to see that $j_{\text{emb}} \leq j_{\text{open}}$ and $j_{\text{open}} \leq j_{\text{emb}}$. Hence they define the same sheaves.
\end{Ex}

\begin{Ex} \label{ex coverages on Cart and Open}
Now consider the following full subcategories
$$\ncat{Cart} \hookrightarrow \ncat{Open} \hookrightarrow \ncat{Man},$$
where $\ncat{Cart}$ is the full subcategory of $\ncat{Man}$ whose objects are cartesian spaces, i.e. manifolds diffeomorphic to $\R^n$ for some $n \geq 0$, and $\ncat{Open}$ is the full subcategory whose objects are diffeomorphic to open subsets of a cartesian space. We obtain induced coverages (Example \ref{ex induced site}) $j_{\text{emb}}|_{\ncat{Open}}, j_\text{open}|_{\ncat{Open}}.$ 

Notice however that if we restrict $j_{\text{emb}}, j_\text{open}$ to $\ncat{Cart}$, then covers must consist of cartesian spaces. However if $\{ U_i \subseteq U \}$ is a cartesian open cover and $f: V \to U$ is a smooth map, there is no reason that $\{ f^{-1}(U_i) \subseteq V \}$ will be a cartesian open cover. However as we will see in Example \ref{ex j good coverage}, every open cover can be refined by a cartesian open cover, and thus we obtain induced coverages $j_{\text{emb}}|_{\ncat{Cart}}, j_\text{open}|_{\ncat{Cart}}$ on $\ncat{Cart}$. However, because $f^{-1}(U_i)$ does not necessarily exist in $\ncat{Cart}$, the sites $(\ncat{Cart}, j_{\text{emb}}|_{\ncat{Cart}})$ and $(\ncat{Cart}, j_{\text{open}}|_{\ncat{Cart}})$ are not Grothendieck pretopologies. Each site $(\cat{C}, j)$ with $\cat{C} \in \{ \ncat{Cart}, \ncat{Open}, \ncat{Man} \}$ and $j \in \{j_{\text{emb}}, j_{\text{open}} \}$ is composition closed, but not necessarily refinement closed, simply because arbitrary smooth maps refined by injective maps are not necessarily injective.
\end{Ex}

\begin{Ex}
Let $j_{\text{sub}}$ denote the collection of families on $\ncat{Man}$ where for $M \in \ncat{Man}$, a family $\{f : N \to M \} \in j_{\text{sub}}(M)$ consists of a single surjective submersion\footnote{This is a smooth map of manifolds that is surjective, and whose map on every tangent space is surjective.}. We note that $j_{\text{sub}}$ is not equivalent to $j_{\text{open}}$. Indeed, consider the constant presheaf $\u{\Z}$ on $\ncat{Man}$ that assigns to every manifold the set of integers and the to every morphism the identity map. We note that the empty family $\varnothing$ is a cover of each empty manifold $\varnothing_n$\footnote{There is one empty manifold of every dimension, with empty atlas.} in $j_{\text{open}}$, but it is not a cover in $j_{\text{sub}}$. By Lemma \ref{lem sheaves on empty covers}, $\u{\Z}$ is not a sheaf on $(\ncat{Man}, j_{\text{open}})$, but it is not hard to check that it is a sheaf on $(\ncat{Man}, j_{\text{sub}})$. We learned of this counterexample from \cite[Warning 5.1.4]{waldorf2024internalgeometry}\footnote{We should also mention that this paper of Waldorf's is an excellent complement to these notes, but that it uses a different notion of equivalence of coverages, and even a different notion of sheaves! However, with these different notions, Waldorf proves that $j_{\text{sub}}$ and $j_{\text{open}}$ are equivalent.}. We also note that pullbacks of surjective submersions along arbitrary maps exist in $\ncat{Man}$, and it is thus not hard to show that $(\ncat{Man}, j_{\text{sub}})$ is a Grothendieck pretopology. 
\end{Ex}

\begin{Ex} \label{ex j good coverage}
Given a finite dimensional smooth manifold $M$, we say that an open cover $\mathcal{U} = \{U_i \subseteq M \}_{i \in I}$ is \textbf{good} if every nonempty finite intersection $U_{i_1} \cap \dots U_{i_n}$ of the open subsets is diffeomorphic to some $\R^n$. Let $j_\text{good}$ denote the collection of families on $\ncat{Man}$ given by the good open covers. Let us show that the good covers form a coverage. If $\{U_i \subseteq M \}$ is a good cover and $g: N \to M$ a smooth map, then $\{ g^{-1}(U_i) \subseteq N \}$ is an open cover, but not necessarily good. By \cite[Corollary 5.2]{bott1982differential}, this open cover can be refined by a good open cover $\{W_k \subseteq N \}$ so that for every $W_k$ in the good open cover, there exists a $U_i$ such that $W_k \subseteq g^{-1}(U_i)$, and thus the following diagram commutes:
\begin{equation*}
\begin{tikzcd}
	{W_k} & {g^{-1}(U_i)} & {U_i} \\
	N && M
	\arrow["g"', from=2-1, to=2-3]
	\arrow[hook, from=1-3, to=2-3]
	\arrow[hook, from=1-1, to=2-1]
	\arrow["{g|_{g^{-1}(U_i)}}", from=1-2, to=1-3]
	\arrow[hook, from=1-1, to=1-2]
\end{tikzcd}
\end{equation*}
Thus $j_\text{good}$ is a coverage on $\ncat{Man}$. We can similarly define such induced coverages on $\ncat{Cart}$ and $\ncat{Open}$. We note that none of these coverages are composition closed. This was proven by David Roberts in \cite{roberts2024good}.

By taking $g$ to be the identity in the diagram above, we see that $j_{\text{good}} \leq j_{\text{open}}$ on each category $\cat{C} \in \{\ncat{Cart}, \ncat{Open}, \ncat{Man} \}$, and clearly $j_{\text{open}} \leq j_{\text{good}}$. Hence the three coverages $j_{\text{good}}, j_{\text{emb}}, j_{\text{open}}$ are all equivalent on each category $\cat{C} \in \{\ncat{Cart}, \ncat{Open}, \ncat{Man} \}$.
\end{Ex}

\begin{Ex}[{\cite[Section III.9]{maclane2012sheaves}}]
Let $G$ be a topological group. A (right) $G$-action on a set $S$ consists of a discrete group action such that the map $\rho : S \times G \to S$ is continuous when $S$ is given the discrete topology. Let $\text{Sub}_G$ denote the category whose objects are $G$-sets of the form $G/U$, where $U$ is an open subgroup of $G$. In other words, the objects are right cosets of the form $Ug$, and morphisms are $G$-equivariant maps $f: G/U \to G/V$. The existence of such a map turns out to be equivalent to the existence of an element $h \in G$ such that $U \subseteq hVh^{-1}$. So we can identify $\text{Sub}_G$ with the category of open subgroups $U \subseteq G$ and morphisms $h : U \to V$ are elements $h \in G$ such that $U \subseteq hVh^{-1}$, and composition is given by multiplication. Since open subgroups are closed under intersection, it is not hard to check that $\text{Sub}_G$ satisfies the Ore condition (Definition \ref{def ore condition}), and hence we can consider the atomic Grothendieck coverage $J_{\text{at}}$ on $\text{Sub}_G$. By \cite[Theorem III.9.1]{maclane2012sheaves}, the category $\ncat{Sh}(\text{Sub}_G, J_{\text{at}})$ of sheaves on the atomic site of $\text{Sub}_G$ is equivalent to the category all $G$-sets. Hence the category of $G$-sets is a Grothendieck topos.
\end{Ex}

\begin{Ex}
We present a collection of families $j_{\text{Pav}}$ on the category $\ncat{Cart}$ that we call the \textbf{Pavlov collection}. It was introduced by Dmitri Pavlov in \cite[Theorem 4.11, Remark 4.12]{pavlov2022numerable} with the motivation of finding simpler conditions to check when a simplicial presheaf on $\ncat{Cart}$ is an $\infty$-stack. The only non-identity family in $j_{\text{Pav}}(\R^n)$ is the family of subset inclusions
\begin{equation*}
    \{ (4i, 4i + 3) \times \R^{n-1} \subseteq \R^n, (4i + 2, 4i+5) \times \R^{n-1} \subseteq \R^n \}_{i \in \mathbb{Z}}.
\end{equation*}
Pavlov proves that this collection generates (in the sense of Remark \ref{rem grothendieck coverage generated by collection of families} the same Grothendieck topology as $(\ncat{Cart}, j_{\text{open}})$. While this collection might seem odd, its characterization leads to the remarkable result \cite[Proposition 4.13]{pavlov2022numerable}.
\end{Ex}

\begin{Ex} \label{ex coverages on complex manifolds}
Let $\C\ncat{Man}$ denote the category of finite dimensional complex manifolds and holomorphic maps. This is an essentially small category, using a similar argument\footnote{Now one just needs to bound the size of all } as in Example \ref{ex topman is essentially small}. Let $j_{\text{open}}$ be the collection of families given by open covers. It is easy to see that $j_{\text{open}}$ is a coverage, and in fact a Grothendieck pretopology.

There is a similar chain of full subcategory inclusions as in \ref{ex coverages on Cart and Open}. First we need some definitions. We say that a complex manifold $M$ is a Stein manifold if there exists an injective, proper, holomorphic, immersion $f : M \to \C^n$ for some $n \geq 0$. Equivalently, $M$ is Stein if and only if it is biholomorphically equivalent to a closed complex submanifold of $\C^n$. There are many other equivalent ways to define Stein manifolds, see \cite[Chapter 2]{forstnerivc2011stein}. A key property of Stein manifolds is that they have trivial Dolbeault cohomology \cite[Theorem 2.4.6]{forstnerivc2011stein}. In complex geometry, Stein manifolds are analogous to Cartesian spaces (Example \ref{ex coverages on Cart and Open}).

A polydisk is a complex manifold of the form
\begin{equation*}
    D_{k_1 \dots k_n} = \{ (z_1, \dots, z_n) \in \C^n \, : \, |z_i| < k_i\}
\end{equation*}
with each $k_i > 0$. Polydisks are Stein manifolds. 

There are full subcategory inclusions
\begin{equation*}
    \C\ncat{Disk} \hookrightarrow \ncat{Stein} \hookrightarrow \C\ncat{Man},
\end{equation*}
where $\ncat{Stein}$ is the full subcategory of $\C\ncat{Man}$ on the Stein manifolds, and $\C\ncat{Disk}$ is the full subcategory on the polydisks. 

If $f : M \to N$ is a map of complex manifolds and $U \hookrightarrow N$ is an open subset that is a Stein manifold, then by \cite[Lemma 4.1]{larusson2003excision}, $f^{-1}(U)$ is a Stein open subset of $M$\footnote{This lemma also proves that finite intersections of Stein open subsets is Stein, hence making this coverage analogous to $j_{\text{good}}$ from Example \ref{ex j good coverage}}. Hence restricting $j_{\text{open}}$ to $\ncat{Stein}$, we obtain an induced coverage $j_{\text{open}}|_{\ncat{Stein}}$.

Now suppose that $f : U \to V$ is a map of polydisks of complex dimension $n$ and $m$ respectively and $\mathcal{V} = \{V_i \subseteq V \}$ is a cover by polydisks. The preimages $U_i = f^{-1}(V_i)$, which assemble into an open cover $\mathcal{U} = \{U_i \subseteq U \}$ will not necessarily be polydisks, but they will be Stein by \cite[Lemma 4.1]{larusson2003excision}. Furthermore each $U_i$ will be an open subset of $\C^n$. Hence by \cite[Section 2.2, Theorem 2.1.3]{forstnerivc2011stein}, since $U_i$ is Stein and an open subset of $\C^n$, it is a domain of holomorphy (\cite[Page 45]{forstnerivc2011stein}), hence it is a domain in $\C^n$. Thus by the proof of \cite[Lemma II.1]{fornaess1977spreading}, there exists a refinement $\mathcal{U}' \leq \mathcal{U}$ where $\mathcal{U}'$ is an open cover by polydisks. Hence we obtain an induced coverage $j_{\text{open}}|_{\C\ncat{Disk}}$. Clearly all three of these sites are composition closed, but not refinement closed.
\end{Ex}

\begin{Ex} \label{ex weiss coverage}
Several coverages were introduced in \cite{debrito2013manifold} to study Manifold calculus, which involves Taylor series-like approximations of $\infty$-stacks. The set up is a bit different than in Example \ref{ex coverages on Cart and Open}. Given a fixed $d \geq 0$, let $\ncat{Man}^d$ denote the category of $d$-dimensional smooth manifolds and open embeddings. Given a fixed $k \geq 1$, let $j^k_\open$ denote the collection of families on $\ncat{Man}^d$ where for $M \in \ncat{Man}^d$, a family $r$ on $M$  belongs to $j^k(M)$ if it is an open cover $\mathcal{U} = \{U_i \subseteq M \}$ such that for every collection of $\ell \leq k$ distinct points $p_1, \dots, p_\ell$ in $M$, there exists a $U_i$ such that $p_1, \dots, p_\ell \in U_i$. We call this a \textbf{$k$-cover} of $M$. Note that $j^1_\open = j_\open$, and $j^n_\open \subseteq j^m_\open$ for $m \leq n$. The coverage $j^k_\open$ is called the $k$th \textbf{Weiss coverage}. For each $k \geq 1$, these coverages are composition closed. Indeed, suppose that $M$ is a $d$-manifold, and suppose that $\mathcal{U} = \{U_i \subseteq M \}_{i \in I}$ is a $k$-cover of $M$ and for each $U_i$ there is a $k$-cover $\mathcal{V}^i = \{V^i_\alpha \subseteq U_i \}$. If $p_1, \dots, p_\ell$ are a collection of distinct points with $\ell \leq k$, then there exists some $U_i$ such that $p_1, \dots, p_\ell \in U_i$. But since $\mathcal{V}^i$ is a $k$-cover, this means that there exists a $V^i_\alpha$ such that $p_1, \dots, p_\ell \in V^i_\alpha$. Hence the composite cover is also a $k$-cover. From here it is easy to check that $j^k_\open$ is a Grothendieck pretopology on $\Man$ for every $k \geq 1$.

There is an analogous coverage $j^k_\text{good}$, which is not composition closed, defined in \cite[Definition 2.9]{debrito2013manifold}. There is also a coverage $j^k_h$, defined in \cite[Definition 2.5]{debrito2013manifold} given by $k$-coverings of the form $\{M \setminus A_i \hookrightarrow M \}_{i \in 0, \dots, k }$, where $A_0, \dots, A_k \subseteq M$ are disjoint, closed subsets of $M$. It is shown in \cite[Remark 2.13, Section 5,7]{debrito2013manifold} that while $j^k_\text{good}$ and $j^k_h$ are not equivalent to $j^k_\open$, they do have the same $\infty$-stacks, and are therefore in a sense $\infty$-equivalent.
\end{Ex}

\begin{Ex} \label{ex poset coverage}
Given a meet-semilattice $P$, let $j_{\text{meet}}$ denote the collection of families where a family $ r = \{U_i \leq U \}_{i \in I} \in j_{\text{pos}}(U)$ if and only if for every $U' \leq U$, $U'$ there exists a subset $S \subseteq P$ such that for every $V \in S$, $V \leq U'$ and there is some $i \in I$ such that $V \leq U_i$.

Let us show that this collection of families is a coverage. Suppose that $\{U_i \leq U \}$ is a covering family and $U' \leq U$. Then there exists a $S \subseteq P$ such that for all $V \in S$, $V \leq U'$ and $V \leq U_i$ for some $i$. Therefore we need only show that $S \in j_{\text{pos}}(U')$. So suppose that $U'' \leq U'$. Then the family $\{U'' \wedge V \subseteq U'' \}_{V \in S}$ satisfies the desired property. Thus $j_{\text{meet}}$ is a coverage, which we call the \textbf{meet coverage}.

If $P$ is a frame\footnote{i.e. has small joins, finite meets and small joins distribute over finite meets.}, then the collection of families $j_{\text{frm}}$ on $P$ where $r = \{U_i \leq U \}_{i \in I} \in j_{\text{frm}}(U)$ if and only if $U$ is a join of $r$, is a coverage. Indeed, if $U' \leq U$, then $U'$ is a join of $\{U' \wedge U_i \leq U'\}$ since $\bigvee U' \wedge U_i = U' \wedge \bigvee U_i = U' \wedge U = U'$. We call this the \textbf{frame coverage}.
\end{Ex}

\begin{Ex} \label{ex lextensive coverage}
We say that a category $\cat{C}$ is (infinitary) \textbf{lextensive} if it has finite limits and Van Kampen (small) finite coproducts (Definition \ref{def van kampen colimit}). Let $j_{\text{lex}}$ denote the collection of families of the form $r = \{r_i : U_i \hookrightarrow U \}_{i \in I}$, where $I$ is a (small) finite set, and such that the $r_i$ define a cocone, i.e. each $r_i$ is the inclusion into a coproduct $U \cong \sum_i U_i$. Since coproducts are Van Kampen, this implies that the families in $j_{\text{lex}}$ are stable under pullback, and hence defines a coverage, called the \textbf{lextensive coverage}.
\end{Ex}

\begin{Ex}[{\cite[Definition A.4]{roberts2012internal}}]
Let $\cat{C}$ be a lextensive category. We say that a Grothendieck pretopology $j$ on $\cat{C}$ is \textbf{superextensive} if $j_{\text{lex}} \subseteq j$. We note that the (large) site $(\ncat{Top}, j_{\text{emb}})$ from Example \ref{ex top coverage} is superextensive. Given a superextensive site $(\cat{C}, j)$, let $\sum j$ denote the collection of singleton families on $\cat{C}$ of the form $\{\sum_{i \in I} r_i : \sum_{i \in I} U_i \to U \}$ where $\{r_i : U_i \to U\}_{i \in I} \in j(U)$. In \cite[Proposition A.8]{roberts2012internal} it is proved that this is a singleton pretopology on $\cat{C}$ and is subcanonical if and only if $j$ is.
\end{Ex}

\begin{Ex} \label{ex regular coverage}
We say that a category $\cat{C}$ is \textbf{regular} if 
\begin{enumerate}
    \item it is finitely complete,
    \item it has all kernel pairs (Definition \ref{def kernel pair}), and
    \item regular epimorphisms (Definition \ref{def regular epi}) are stable under pullbacks.
\end{enumerate}
For a small, regular category $\cat{C}$, let $j_{\text{reg}}$ denote the collection of families consisting of regular epimorphisms $ r= \{f : U \to V \}$. Since these are stable under pullback, they form a coverage. We call this the \textbf{regular coverage} on $\cat{C}$. Since all epimorphisms are regular on $\ncat{FinSet}$ (Lemma \ref{lem epis are regular in set}), $(\ncat{FinSet}, j_{\text{reg}}) = (\ncat{FinSet}, j_{\text{sepi}})$.

Note that a wide class of categories is regular. These include algebraic categories (like $\ncat{Grp}$, $\ncat{Ring}$, etc.), abelian categories and quasitoposes\footnote{Though of course with the usual caveat that these sites are usually essentially large, see Section \ref{section large sites}.}.

The regular coverage on a regular category is a Grothendieck pretopology, as regular epimorphisms are closed under composition in regular categories. 
\end{Ex}

\begin{Ex} \label{ex coherent coverage}
A \textbf{coherent category} $\cat{C}$ is a regular category (Example \ref{ex regular coverage}) in which the subobject poset $\text{Sub}(U)$ for all $U \in \cat{C}$ has finite joins (which in this case are called unions), and for which the pullback functor $\text{Sub}(f) : \text{Sub}(V) \to \text{Sub}(U)$ for $f : U \to V$ preserves finite joins. We say that $\cat{C}$ is a \textbf{geometric category} if $\text{Sub}(U)$ has all small joins and the pullback functors preserve small joins. Every quasitopos and topos is a coherent category.

Given a coherent category $\cat{C}$ define the (potentially large) collection of families $j_{\text{coh}}$ so that a family is $j_{\text{coh}}$-covering if it is of the form $r = \{r_i : U_i \to U \}_{i \in I}$ where $I$ is finite and $\bigcup_i \, \text{im}(r_i) = U$. In regular categories, images are pullback-stable, and in coherent categories unions of images are pullback-stable, hence $j_{\text{coh}}$ is a pullback-stable coverage. It is not hard to see that it is also composition-closed and is furthermore a Grothendieck pretopology. We call this the \textbf{coherent coverage}, and if the families are allowed to be small, then we call it the \textbf{geometric coverage}. 

Johnstone shows in \cite[Example C.2.1.12.(d)]{johnstone2002sketches} that the coherent and geometric coverages are subcanonical. 
\end{Ex}

\begin{Ex}
Let us briefly recall the idea of a first-order theory $\mathbb{T}$. First, a language $\Sigma_\mathbb{T}$ consist of sets of sorts, function symbols and relation symbols. For example, the language for the theory of groups consists of a single sort $G$, two function symbols given by multiplication $m : (G,G) \to G$ and inversion $i : G \to G$, and no relation symbols.

Given a language, one can then define the set of terms one can build. An example of a term for the language of groups is $x, y : G\, | \, m(x,m(y,i(x))) : G$. The lefthand symbols form the context of the term, they are the variables being used. Formulas are then built from relation symbols in $\Sigma_{\mathbb{T}}$ and terms. For example, let $\Sigma_{\mathbb{T}}$ be the language of posets. This has one sort $P$, no function symbols, and one relation symbol $\leq$. Then an example of a formula for this language is $x,y :P \, | \, (x \leq y) \Rightarrow (\forall z.(z \leq x) \Rightarrow (z \leq y))$.

By restricting the usage of certain first-order logical symbols like $\wedge, \vee, \Rightarrow, \neg, \exists, \forall$, we obtain particular types of theories (such as cartesian, regular, coherent, geometric, etc.), see \cite[Section D1.1]{johnstone2002sketches} for a list of different kinds of theories.

A sequent $\varphi \vdash \psi$ consists of a pair of formulas in the same context. This represents that $\varphi$ ``implies'' $\psi$. A theory $\mathbb{T}$ then consists of a language $\Sigma_{\mathbb{T}}$ and a set $S_{\mathbb{T}}$ of sequents of formulas over the language. Using first-order logic, we obtain deduction rules that say when one sequent implies another, see \cite[Section D1.3]{johnstone2002sketches}.

From a theory $\mathbb{T}$, one can construct a category $\ncat{Syn}(\mathbb{T})$, called the \textbf{syntactic category} of $\mathbb{T}$. Its objects are certain equivalence classes of formulas, and morphisms are certain kinds of formulas which act intuitively as functions between the ``sets'' defined by the formulas. Based on the type of theory $\mathbb{T}$, the syntactic category $\ncat{Syn}(\mathbb{T})$ will have the corresponding kind of structure (cartesian, regular, coherent, geometric), see \cite[Lemma D1.4.10]{johnstone2002sketches}. Each such structure has its own corresponding canonical site structure (cartesian categories are equipped with a trivial coverage) such as in Examples \ref{ex regular coverage} and \ref{ex coherent coverage}, which are called the \textbf{syntactic sites}. The category of sheaves on a syntactic site $\ncat{Sh}(\ncat{Syn}(\mathbb{T}), j)$ is called the \textbf{classifying topos} for the theory $\mathbb{T}$. Geometric morphisms $\ncat{Sh}(\ncat{Syn}(\mathbb{T}), j) \to \cat{E}$ of toposes are equivalently models of $\mathbb{T}$ in $\cat{E}$, see \cite[Section VIII]{maclane2012sheaves} for more on classifying toposes.
\end{Ex}

\begin{Ex}
We mention that the above idea of a syntactic site extends naturally to type theories, and to syntactic categories of type theories. Sheaves on such syntactic categories have been used to solve important problems in computer science, such as in \cite{altenkirch2001normalization}, where a normalization result is proved for a typed lambda calculus with coproducts. The study of sheaves for computer science in this way is sometimes referred to as \textbf{sheaf semantics}.
\end{Ex}

\begin{Ex}[{\cite[Exercise III.9.13]{maclane2012sheaves}}]
Let $\ncat{FinSet}_m$ denote the small category of finite sets and monomorphisms. Then the category $\ncat{FinSet}_m^\op$ satisfies the Ore condition (Definition \ref{def ore condition}), since $\ncat{FinSet}_m$ has pushouts. Let $(\ncat{FinSet}_m^\op, J_{\text{at}})$ denote the atomic site (Definition \ref{def atomic site}), then $\ncat{Sh}(\ncat{FinSet}_m^\op, J_{\text{at}})$ is known as the \textbf{Shanuel topos}. Its objects are called nominal sets, which give a formal framework for bound variables and alpha equivalence in programming languages, see \cite{pitts2013nominal} for more.
\end{Ex}

\begin{Ex} \label{ex johnstone topological topos}
Let $\tau$ denote the full subcategory of $\ncat{Top}$ on the topological spaces $*$, the terminal object, and $\N_\infty$, the one-point compactification of the discrete natural numbers (equivalently it is homeomorphic to the $\{1, 1/2, 1/3, \dots, 1/n, \dots, 0\} \subseteq \R$ with the subspace topology). Consider the canonical site $(\tau, J_{\text{can}})$, see Definition \ref{def canonical coverage}. Sheaves on this site $(\tau, J_{\text{can}})$ form what is called the \textbf{topological topos}, introduced by Johnstone in \cite{johnstone1979topological}. The motivation for this topos, similarly to the other convenient categories mentioned in Example \ref{ex top coverage}, is to provide a convenient category in which to do topology. The category $\ncat{Sh}(\tau, J_{\text{can}})$ contains all sequential topological spaces\footnote{A topological space $X$ is sequential if for every topological space $Y$ a map $f: X \to Y$ is continuous if and only if it is sequentially continuous, i.e. if $x_n \to x$ then $f(x_n) \to f(x)$.} as a reflective subcategory. For more about the topological topos, we recommend the series of blog posts \cite{grossack2024topological}.
\end{Ex}

\begin{Ex}
A \textbf{bornology} on a set $X$ is a cover of $X$ that is closed under subsets and finite unions. A set with a bornology is called a bornological space, and the elements of its bornology are called bounded subsets.  Their invention is attributed to Mackey \cite{mackey1942subspaces}, and their history is reviewed in the french paper \cite{hogbe1970racines}. Bornological spaces are a convenient framework for functional analysis, where all kinds of subtle problems arise from considering linear operators between locally convex topological vector spaces, with an emphasis on bounded subsets being more central to the theory than open subsets. Lawvere first described the bornological topos during unpublished talks he gave at Bogot\'a in 1983. Let $\ncat{Cnt}$ denote the small category of countable sets and functions. This category is lextensive, and the category of sheaves $\ncat{Sh}(\ncat{Cnt}, j_{\text{lex}})$ with the lextensive coverage (Example \ref{ex lextensive coverage}) is the bornological topos. There is a full subcategory of bornological spaces given as a reflective subcategory of the bornological topos \cite{espanol2002bornologies}. \end{Ex}

\begin{Ex}
The recursive topos, or the topos of recursive sets, was defined by Mulry in \cite{mulry1982generalized} with the goal of doing recursion theory in a cartesian closed category. Let $R$ denote the monoid of all total recursive functions $f : \N \to \N$. Viewing this as a one-object category, equip it with the coherent coverage 
 (Example \ref{ex coherent coverage}) $(R, j_{\text{coh}})$. This turns out to be a canonical coverage as well. Sheaves on this site $\ncat{Sh}(R, j_{\text{coh}})$ form the recursive topos.
\end{Ex}

\subsection{Concrete Sheaves}

Now let us consider several important examples in the literature that are not quite sheaf toposes. Instead they are quasitoposes of concrete sheaves.

\begin{Def} \label{def concrete site}
A site $(\cat{C}, j)$ is \textbf{concrete} if 
\begin{enumerate}
    \item it has a terminal object $*$,
    \item the functor $\cat{C}(*,-) : \cat{C} \to \ncat{Set}$ is faithful, and
    \item for every $j$-covering family $r = \{r_i : U_i \to U \}$, the function
    \begin{equation*}
        \sum_i \cat{C}(*, r_i) : \sum_i \cat{C}(*, U_i) \to \cat{C}(*, U)
    \end{equation*}
    is surjective.
\end{enumerate}
A sheaf $X$ on a concrete site $(\cat{C}, j)$ is said to be \textbf{concrete} if the canonical function
\begin{equation*}
    X(U) \to \ncat{Set}(\cat{C}(*, U), X(*))
\end{equation*}
obtained by applying $\cat{C}(*,-) \cong \ncat{Pre}(\cat{C})(y(*), -)$ to the map $y(U) \to X$ is injective.
\end{Def}

\begin{Def}
A \textbf{Grothendieck quasitopos} $\cat{E}$ is a category equivalent to one of the form $\ncat{Sep}(\ncat{Sh}(\cat{C}, j), k)$. Namely it is a category equivalent to the category of $j$-sheaves on a site $(\cat{C}, j)$ which are also $k$-separated (Definition \ref{def separated and sheaf}) with respect to another coverage $k$ on $\cat{C}$ such that $\Gro{j} \subseteq \Gro{k}$.
\end{Def}

\begin{Rem}
Grothendieck quasitoposes, introduced in \cite{borceux1991characterization}, are very nice categories, close to being Grothendieck toposes in their properties. For example, Grothendieck quasitoposes are (small) complete, cocomplete, locally cartesian closed and locally presentable (Definiiton \ref{def locally presentable category}). In fact, if a Grothendieck quasitopos $\cat{E}$ is balanced\footnote{A category is balanced if every morphism $f : U \to V$ which is a monomorphism and an epimorphism is an isomorphism.} then it is a Grothendieck topos. We recommend the following references \cite{baez2011convenient, wyler1991lecture, garner2012grothendieck} for more on Grothendieck quasitoposes. 
\end{Rem}

\begin{Lemma} \label{lem concrete sheaves form quasitopoi}
Given a concrete site $(\cat{C}, j)$, let $k$ denote the collection of families of the form $\{x : * \to U \}_{x \in \cat{C}(*, U)}$. A sheaf $X$ on $(\cat{C}, j)$ is then a concrete sheaf if and only if it is separated with respect to $k$.
\end{Lemma}

\begin{proof}
Let us note that $\Gro{j} \subseteq \Gro{k}$ by Definition \ref{def concrete site}.(3). Then being $k$-separated is precisely what it means for a sheaf to be concrete.
\end{proof}

\begin{Cor}
Given a concrete site $(\cat{C}, j)$, the full subcategory $\ncat{ConSh}(\cat{C}, j) \hookrightarrow \ncat{Sh}(\cat{C}, j)$ of concrete sheave is a Grothendieck quasitopos.
\end{Cor}

\begin{Ex} \label{ex diffeological spaces}
Diffeological spaces were introduced by Souriau in \cite{souriau1980diffeological}, and developed by his student Iglesias-Zemmour, leading to the textbook \cite{iglesias2013diffeology}. The motivation here is to have a convenient category for doing differential geometry. It is arguably one of the most successful such frameworks (with competition from synthetic differential geometry, Example \ref{ex synthetic differential geometry}), being easy to work with and understand, but also powerful enough to extend classical constructions from differential geometry. Originally the definition was given as a set with a collection of plots satisfing three simple conditions. Baez and Hoffnung proved in \cite{baez2011convenient} that the category of diffeological spaces is equivalent to a category of concrete sheaves. More precisely, consider the site $(\ncat{Open}, j_{\text{open}})$ from Example \ref{ex open coverage on Man}. Then \cite[Proposition 24]{baez2011convenient} shows that the category of diffeological spaces is equivalent to the category of concrete sheaves on $(\ncat{Open}, j_{\open})$. They prove a similar statement for so-called Chen spaces \cite{chen1977iterated}. The author's PhD thesis consists of two papers \cite{Minichiello2024bundles, minichiello2024derham} which study the cohomology of diffeological spaces from an $\infty$-stack perspective, which benefits from the sheaf-theoretic approach to differential geometry.
\end{Ex}

\begin{Ex} \label{ex simplicial complexes}
A simplicial complex $K$ consists of a set $V(K)$ along with a set $S(K)$ of nonempty finite subsets $\sigma$ of $V(K)$ whose elements are called \textbf{simplices}, containing all of the singletons, and such that if $\sigma \in S(K)$ and $\tau \subseteq \sigma$, then $\tau \in S(K)$. A map $f : K \to L$ of simplicial complexes is a function $V(f) : V(K) \to V(L)$ such that if $\sigma \in S(K)$, then $f(\sigma) \in S(L)$. Let $\ncat{Cpx}$ denote the category of simplicial complexes. Now consider the small category $\ncat{FinSet}_{\neq \varnothing}$ of nonempty finite sets and functions. The collection of families $j_{\text{inc}}$ given by singleton families $\{S \hookrightarrow S'\}$ of subset inclusions is pullback-stable and hence defines a coverage. However, this coverage is uninteresting, every presheaf on $\ncat{FinSet}_{\neq \varnothing}$ is a $j_{\text{inc}}$-sheaf. It is not hard to check that $(\ncat{FinSet}_{\neq \varnothing}, j_{\text{inc}})$ is a concrete site. By \cite[Proposition 27]{baez2011convenient}, there is an equivalence of categories $\ncat{Cpx} \simeq \ncat{ConSh}(\ncat{FinSet}_{\neq \varnothing}, j_{\text{inc}})$. Hence the category of simplicial complexes is a Grothendieck quasitopos.
\end{Ex}

\begin{Ex}
One can similarly prove that the categories $\ncat{Gr}_\ell$ of loop graphs and $\ncat{Gr}_r$ of reflexive graphs are quasitoposes by showing they are categories of concrete presheaves. See \cite[Appendix A]{bumpus2025structured} for more details.
\end{Ex}

\begin{Ex}
Recall from Example \ref{ex johnstone topological topos} the site $(\tau, J_{\text{can}})$. This site is concrete, and the concrete sheaves on this site are precisely the subsequential spaces, see \cite[Chapter IV]{harasani1980topos}.
\end{Ex}

\begin{Ex}
Quasi-topological spaces were introduced by Spanier in \cite{spanier1963quasi} as a convenient category of topological spaces. Let $\ncat{CH}$ denote the full subcategory of $\ncat{Top}$ on the compact Hausdorff spaces. This category is very well-behaved, being a pretopos \cite{marra2020characterisation}. In particular it is a coherent category. Consider the large site $(\ncat{CH}, j_{\text{coh}})$ with the coherent coverage (Example \ref{ex coherent coverage}). We note that while $\ncat{Sh}(\ncat{CH}, j_{\text{coh}})$ has large hom-sets, as we will discuss later in this section, the full subcategory $\ncat{ConSh}(\ncat{CH}, j_{\text{coh}})$ has small hom-sets, which is easy to see by the definition of concrete sheaves, and this category is equivalent to the category of quasi-topological spaces. We also recommend the paper \cite{dubuc2006concrete} for more on this example.
\end{Ex}

\begin{Ex}
Let $\ncat{Borel}$ denote the small category of standard Borel spaces and measurable functions between them, see \cite[Section II.A]{heunen2017convenient}. We say that a family $\{b_i : B_i \to B \}$ in this category is a covering family if and only if it is a countable coproduct. This gives a pullback-stable collection of families, and hence it is a coverage. It is an instance of an infinitary lextensive coverage (Example \ref{ex lextensive coverage}). The concrete sheaves on the site $(\ncat{Borel}, j_{\ncat{lex}})$ are called Quasi-Borel spaces, and are a convenient category to work in for higher-order probability theory, see \cite{heunen2017convenient}.
\end{Ex}

\subsection{Large Sites} \label{section large sites}

\subsubsection{Examples of Large Sites}
The astute reader may have noticed that we did not include any examples involving the category $\ncat{Top}$ of topological spaces. The reason for this is simple, $\ncat{Top}$ is not an essentially small category. Therefore according to Definition \ref{def coverage}, it cannot be a site. This is a very restrictive and unfortunate state of affairs, and in this section we will discuss some ways of working around the smallness condition. We recommend the reader look at Section \ref{section set theory} before reading this section.

First, let us consider what can go wrong by augmenting Definition \ref{def coverage} to allow for collections of families to consist of large sets of morphisms. Let us say a category $\cat{C}$ is \textbf{essentially large} if it is large and not essentially small.

\begin{Def} \label{def large coverage}
A \textbf{large collection of families} $j$ on a large category $\cat{C}$ consists of a large set $j(U)$ whose elements are large sets of morphisms $r = \{r_i : U_i \to U \}_{i \in I}$. The definitions of coverage, refinement closure, composition closure, sifted closure, Grothendieck pretopology, Grothendieck coverage, matching family, amalgmation and sheaf are defined in the same way using large collections of families. We call $(\cat{C}, j)$ a \textbf{large site} if $j$ is a large coverage and $\cat{C}$ is essentially large.
\end{Def}

There is no real ``problem'' with Definition \ref{def large coverage}, in the sense that we can define everything as we have before. It makes perfect sense to talk about a single sheaf on a large site. However the category $\ncat{Sh}(\cat{C}, j)$ of sheaves on a large site is less well-behaved than on small sites. Before discussing these disadvantages, let us see some examples of large sites.

\begin{Ex} \label{ex jointly epi coverage on set}
The category $\ncat{Set}$ of small sets is an essentially large category. Let $j_{\text{epi}}$ denote the large collection of families where for a set $S$, $j(S)$ denotes the large set of (small) families $r = \{r_i : S_i \to S \}_{i \in I}$ that are jointly epimorphic, i.e. $\sum_{i \in I} S_i \to S$ is surjective. This defines a saturated, pullback-stable coverage on $\ncat{Set}$.
\end{Ex}

\begin{Ex} \label{ex top coverage}
Let $\ncat{Top}$ denote the category of topological spaces. Since every set can be considered as a discrete topological space, $\ncat{Top}$ is essentially large. Similarly to Example \ref{ex open coverage on Man}, we can define the coverages $(\ncat{Top}, j_{\text{open}})$ of open covers and $(\ncat{Top}, j_{\text{emb}})$ of open covers given by topological embeddings. Both of these are Grothendieck pretopologies.

We note that the category $\ncat{Top}$ is not cartesian closed. There are other ``convenient'' categories of topological spaces that are more often considered in algebraic topology, which are cartesian closed. These include the category $k\ncat{Top}$ of $k$-spaces, the category $\ncat{CGWH}$ of compactly generated weak Hausdorff spaces\footnote{Interestingly, the canonical Grothendieck coverages on $\ncat{CGWH}$ and $\ncat{Top}$ differ, see \cite[Example 5.2.7]{lester2019canonical}.}, and the category $\Delta\ncat{Top}$ of $\Delta$-generated topological spaces. See \cite{rezk2017compactly} and \cite{nlab:convenient_category_of_topological_spaces}. The open cover coverage restricts to all of these full subcategories of $\ncat{Top}$.
\end{Ex}

\begin{Ex} \label{ex coverage on G-sets}
Let $G$ be a group, and let $G\ncat{Set}$ denote the category of $G$-sets, i.e. those sets $S$ equipped with a left action $\rho : G \times S \to S$, and where a morphism $f : S \to S'$ is an equivariant function $f(\rho(g,s)) = \rho'(g,f(s))$. We note that $G \ncat{Set}$ is equivalent to the category $\ncat{Set}^G$, where we think of $G$ as a category with one object. Let $j_{\text{G,epi}}$ denote the collection of families on $G\ncat{Set}$ where a family $r = \{r_i : S_i \to S \}$ is $j_{\text{G,epi}}$-covering if they are jointly epimorphic, i.e. if the underlying function $\sum_i r_i : \sum_i S_i \to S$ is surjective. In other words, if we let $\pi : G\ncat{Set} \to \ncat{Set}$ denote the functor that forgets the $G$-action, then $r \in j_{\text{G,epi}}(S)$ if and only if $\pi(r) \in j_{\text{epi}}(\pi(S))$. This is a canonical coverage, and it turns out that $\ncat{Sh}(G\ncat{Set}, j_{\text{G,epi}})$ is equivalent to $G \ncat{Set}$. We will see this as a consequence of Example \ref{ex canonical coverage on grothendieck toposes}. In fact, the same constructions work if $G$ is just a monoid. For an in-depth study of toposes of monoid actions see \cite{rogers2021toposes}.
\end{Ex}

\begin{Ex} \label{ex condensed and pyknotic sets}
The notion of \textbf{condensed sets} has recently become a powerful tool in algebraic geometry \cite{fargues2024geometrization, mann2022padic}. The basic idea here is another instance of ``convenience.'' Here the idea is that the category of topological abelian groups is not an abelian category. Scholze and Clausin develop a category of condensed sets, which form a kind of convenient category of topological spaces, and abelian group objects in this category do form an abelian category, see the lecture notes \cite{scholze2019condensed}. It is interesting to note that setting up this category comes with set-theoretic difficulties, for exactly the reasons we will discuss later in this section. Let $\ncat{Pro}(\ncat{FinSet})$ denote the category of pro-objects in the category of finite sets. This is the full subcategory of $(\ncat{Set}^{\ncat{FinSet}})^\op$ on those functors $P : \ncat{FinSet} \to \ncat{Set}$ that are cofiltered limits of representables. Via Stone duality, see \cite{johnstone1982stone}, this category is equivalent to the full subcategory $\ncat{St}$ of $\ncat{Top}$ on the \textbf{Stone spaces}, those topological spaces that are compact, Hausdorff and totally disconnected\footnote{A nonempty topological space is totally disconnected if its only connected components are singletons.}.

Now the category $\ncat{St}$ of Stone spaces is essentially large. So, to avoid the set-theoretic issues that we will discuss later in this section, Scholze \cite[Remark 1.3]{scholze2019condensed} fixes an uncountable strong limit cardinal\footnote{These are precisely the same as strongly inaccessible cardinals (Definition \ref{def strongly inaccessible cardinal}) without the requirement that $\kappa$ be regular.} $\kappa$, and lets $\ncat{St}_\kappa$ denote the full subcategory of Stone spaces whose underlying sets have cardinality less than $\kappa$. Thus $\ncat{St}_\kappa$ is a small category, and by \cite[Proposition 17]{lurie2018pro}, it is a coherent category (Definition \ref{ex coherent coverage}). Equipping it with coherent coverage $j_{\text{coh}}$, we obtain the site $(\ncat{St}_\kappa, J_{\text{coh}})$. The category $C_\kappa = \ncat{Sh}(\ncat{St}_\kappa, j_{\text{coh}})$ is the category of $\kappa$-condensed sets. For every pair $\kappa < \kappa'$ of uncountable strong limit cardinals there is a fully faithful functor $C_\kappa \hookrightarrow C_{\kappa'}$. The category of condensed sets is then defined to be the colimit $C = \colim_{\kappa} C_\kappa$ over the large poset of uncountable strong limit cardinals. This category is not a Grothendieck topos.

Independently, Barwick and Haine in \cite{barwick2019pyknotic} defined the closely related notion of \textbf{pyknotic set}. They deal with set-theoretic difficulties by the method of Section \ref{section very small grothendieck universe}, i.e. they fix (at least) three Grothendieck universes $\mathbbm{u} \in \mathbbm{U} \in \mathbbm{V}$, and call the sets in $\mathbbm{u}$ ``tiny,'' and the sets in $\mathbbm{U}$ small. They define $\ncat{St}_{\mathbbm{u}}$ to be the category of Stone spaces whose underlying sets are tiny, see \cite[Section 2.2]{barwick2019pyknotic}. This is a small category, and it is coherent. The category of sheaves $P = \ncat{Sh}(\ncat{St}_{\mathbbm{u}}, j_{\text{coh}})$ is called the category of pyknotic sets. The differences between $C$ and $P$ might at first seem artificial, but they do have real consequences for their resulting theory. For example the set $\{0,1\}$ with the indiscrete topology is pyknotic but not condensed \cite[Section 0.3]{barwick2019pyknotic}.
\end{Ex}

\begin{Ex} \label{ex synthetic differential geometry}
In synthetic differential geometry, the goal is to craft a framework in which to do differential geometry such that infinitesimals exist. This simplifies and clarifies many computations. For instance, it makes the tangent bundle of a manifold an exponential object. There are many different toposes that one can ``do'' synthetic differential geometry in. These are called ``well-adapted models'' of synthetic differential geometry. The book \cite{moerdijk2013models} by Moerdijk and Reyes is the standard reference for these well-adapted models. Other useful references are \cite{kock2006synthetic, lavendhomme2013basic}.

Recall the category $\ncat{Cart}$ from Example \ref{ex coverages on Cart and Open}. This category has finite products, and thus can be considered as an algebraic theory. An algebra for this theory, i.e. a finite product-preserving functor $A : \ncat{Cart} \to \ncat{Set}$ is called a \textbf{$C^\infty$-ring}. Note that if $M$ is a finite-dimensional smooth manifold, then $C^\infty(M, \R)$, the set of smooth functions to the real numbers is a $C^\infty$-ring by $\R^n \mapsto C^\infty(M,\R)^n$. In fact, this produces a fully faithful functor $\ncat{Man} \hookrightarrow C^\infty\ncat{Ring}^\op$, \cite[Theorem 2.8]{moerdijk2013models}.

Every $C^\infty$ ring is a commutative $\R$-algebra, by just considering the usual multiplication map $\R \times \R \to \R$ in $\ncat{Cart}$. Hence we can define ideals of $C^\infty$-rings as ideals of their underlying $\R$-algebra. There is then a canonical way of making quotients of $C^\infty$-rings $C$ by ideals $I$ into $C^\infty$-rings $C/I$, see \cite{joyce2013introduction}.

A Weil algebra $W$ is a finite dimensional (as an $\R$-vector space) commutative $\R$-algebra with maximal ideal $I$ such that $W/I \cong \R$ and $I^n = 0$ for some $n \geq 1$. There is a unique way of making a Weil algebra into a $C^\infty$-ring that respects the underlying $\R$-algebra structure, see the references at \cite[Example 2.12]{joyce2013introduction}. These Weil algebras are the objects that act as infinitesimal neighborhoods. Examples include $C^\infty(\R)/(x^2)$ which by Hadamard's Lemma \cite[Lemma 2.4.1]{moerdijk2013models} is isomorphic to the ring of dual numbers  $\R[\varepsilon]/(\varepsilon^2)$.

Let $\ncat{ThCart}$ (infinitesimally thickened cartesian spaces) denote the full subcategory of $C^\infty\ncat{Ring}^\op$ on objects of the form $C^\infty(\R^n) \otimes_\infty W$, where $\otimes_\infty$ denotes the coproduct in $C^\infty\ncat{Ring}$ and $W$. is a Weil algebra. We give this small category the coverage $j_{\text{Cah}}$ whose covers are of the form $\{C^\infty(U_i) \otimes_{\infty} W \xrightarrow{r_i \otimes_{\infty} 1_W} C^\infty(U) \otimes_{\infty} W \}_{i \in I}$ where $r= \{r_i : U_i \to U \}$ is a covering family in $(\ncat{Cart}, j_{\open})$, see Example \ref{ex coverages on Cart and Open}. Sheaves on this site $\ncat{Sh}(\ncat{ThCart}, j_{\text{Cah}})$ is called the \textbf{Cahiers topos}. It was originally defined by Dubuc in \cite{dubuc1979modeles}.

Now consider the category $C^\infty\ncat{Ring}_{\text{Germ}}^\op$ which is the opposite of the full subcategory of germ-determined finitely presented $C^\infty$-rings. These are $C^\infty$-rings of the form $C^\infty(M)/I$ where $M$ is a finite dimensional smooth manifold and $I$ is a germ-determined ideal, which means that for every $f \in C^\infty(M)$ if the germ of $f$ at a point $x$ in the zero locus\footnote{the set of all points $x \in M$ such that if $g \in I$, then $g(x) = 0$.}
of $I$ is equal to a germ of some $g \in I$ at $x$, then $f \in I$. This small category has a coverage given by covering families of the form $\{ A_i \to B \}$ where the corresponding maps $B \to A_i$ in $C^\infty\ncat{Ring}_{\text{Germ}}$ are localizations at some $b \in B$, in the $C^\infty$-ring sense. See \cite[Appendix 2]{moerdijk2013models} for more details and also for a large amount of other sites and toposes that give models for synthetic differential geometry.
\end{Ex}

\begin{Ex} \label{ex canonical coverage on grothendieck toposes}
Given a small site $(\cat{C}, j)$ we can consider the category $\ncat{Sh}(\cat{C}, j)$ of $j$-sheaves on $\cat{C}$. This is a large category, but it is also locally presentable (see Section \ref{section locally presentable categories}), hence there exists a small dense subcategory $\cat{D} \hookrightarrow \ncat{Sh}(\cat{C}, j)$. Then equipping $\cat{D}$ with the canonical topology, \cite{johnstone2002sketches} proves that $\ncat{Sh}(\cat{C}, j) \simeq \ncat{Sh}(\cat{D}, J_{\text{can}}) \simeq \ncat{Sh}(\ncat{Sh}(\cat{C}, j), J_{\text{can}})$.
\end{Ex}

If $(\cat{C}, j)$ is a large site, then the sheafification functor $a : \ncat{Pre}(\cat{C}) \to \ncat{Sh}(\cat{C}, j)$, which is discussed in Section \ref{section sheafification}, may not exist. Indeed, if $X$ is a presheaf on $\cat{C}$, then the diagram whose colimit is being taken over in Definition \ref{def general plus construction} may be large. Hence the colimit may not actually exist in $\ncat{Pre}(\cat{C})$. When $\cat{C}$ is small, $\ncat{Pre}(\cat{C})$ is a large category, and it has all small limits and colimits. But if $\cat{C}$ is an essentially large category, then by Lemma \ref{lem smallness of functor categories}, $\ncat{Pre}(\cat{C})$ is large, but that does not mean that it has all large colimits. In fact, by the same argument as in Freyd's Theorem, Proposition \ref{prop freyd's theorem}, if it had all large colimits, it would necessarily be a preorder. For an example of a presheaf on a large site whose sheafification does not exist see \cite{waterhouse1975basically}.

Another problem is with the categorical structure of $\ncat{Pre}(\cat{C})$. If $\cat{C}$ is an essentially large category, then $\ncat{Pre}(\cat{C})$ may not be cartesian closed. Indeed, each hom-set of $\ncat{Pre}(\cat{C})$ is essentially large, by Lemma \ref{lem smallness of functor categories}.(4). Hence the usual internal hom in $\ncat{Pre}(\cat{C})$ as given in Lemma \ref{lem presheaf topoi are locally cartesian closed} does not define a presheaf. This also implies that (pre)sheaf categories on large sites may not have subobject classifiers.

Let us now consider some workarounds to this.

\subsubsection{WISC}
When $(\cat{C}, j)$ is a large site, it is sometimes still possible to guarantee that sheafification will exist.

\begin{Def}[{\cite[Definition 3.19]{roberts2012internal}}]
We say that a large site $(\cat{C}, j)$ satisfies WISC (weakly initial set of covers), if for every $U \in \cat{C}$, there exists a small subset $j'(U) \subseteq j(U)$ such that for every covering family $r \in j(U)$ there exists a $t \in j'(U)$ and a refinement $t \to r$.
\end{Def}

\begin{Lemma}
If $(\cat{C}, j)$ is a large site that satisfies WISC, then sheafification exists.
\end{Lemma}

\begin{proof}
Given $U \in \cat{C}$, let $\Gro{j'}(U)$ denote the subposet of $\Gro{j}(U)$ on those sieves that are generated by covering families in $j'(U)$. So while each $R \in \Gro{j'}(U)$ might be a large set, $\Gro{j'}(U)$ itself is small. We have an obvious inclusion functor $i:\Gro{j'}(U) \to \Gro{j}(U)$ and we wish to show that this functor is initial (Definition \ref{def final functor}). So we want to show that for every $R \in \Gro{j}(U)$ the category $(i \downarrow R)$ is nonempty and connected. But since $R \in \Gro{j}(U)$, there exists a $r \in j(U)$ such that $R = \overline{r}$. Hence there exists a $t \in j'(U)$ and a refinement $t \to r$, so $\overline{t} \subseteq R$. Thus $(i \downarrow R)$ is nonemepty, and it is connected simply because intersections of covering sieves are covering sieves by Lemma \ref{lem covering sieve props}. Now the result follows from Corollary \ref{cor computation of plus construction}, as now $X^+(U)$ is isomorphic to a colimit of a small diagram, which therefore exists.
\end{proof}

\begin{Lemma} \label{lem Set with epi coverage satisfies WISC}
The large site $(\ncat{Set}, j_{\text{epi}})$ of Example \ref{ex jointly epi coverage on set} satisfies WISC. 
\end{Lemma}

\begin{proof}
For every set $S$, there is an initial covering family, given by inclusions of singletons $\{x : * \to S \}_{x \in S}$.
\end{proof}

The next result is due to Shulman as referenced in \cite[Example 3.21]{roberts2012internal}.

\begin{Lemma}
The large sites $(\ncat{Top}, j_{\text{open}}), (\ncat{Top}, j_{\text{emb}})$ of Example \ref{ex top coverage} satisfy WISC.
\end{Lemma}

\begin{proof}
We note that for every topological space $X$, $j_{\text{open}}(X)$ is actually a small set, since a covering family has to be equal to the collection of inclusions $\{U \subseteq X\}_{U \in \mathcal{U}}$ of an open cover $\mathcal{U}$. There are only a sets worth of open covers on $X$. Now let us show that $(\ncat{Top}, j_{\text{emb}})$ satisfies WISC. Suppose that $r \in j_{\text{emb}}(X)$ is a large covering family. For each $x \in X$, we can choose (using the axiom of choice) an open subset $U_x$ containing $x$ such that $U_x$ is the image of a map $r_i : U_i \to X$ in $r$. Then $r_x = \{U_x \subseteq X \}_{x \in X}$ is a $j_{\text{open}}$-covering family of $X$, and also a $j_{\text{emb}}$-covering family. Since $r_i$ is a homeomorphism onto its image, the $r_i^{-1}$ defines a refinement $r_x \to r$. Thus $j_{\text{open}}(X) \subseteq j_{\text{emb}}(X)$ is a small subset that refines every covering family.
\end{proof}

\begin{Rem}
We note that there are nontrivial large sites that do not satisfy WISC, see \cite[Proposition 3.23]{roberts2012internal}.
\end{Rem}

\subsubsection{Small presheaves}

Another alternative that is considered at times is to consider a full subcategory of all presheaves on a large category $\cat{C}$.

\begin{Def}
Given a large category $\cat{C}$, a presheaf $X$ on $\cat{C}$ is \textbf{small} if there exists a small diagram $d : I \to \cat{C}$ and an isomorphism
\begin{equation*}
    X \cong \colim_{i \in I} y(d(i)).
\end{equation*}
Let $P(\cat{C})$ denote the full subcategory of small presheaves.
\end{Def}

Clearly if $\cat{C}$ is small, then every presheaf is small, by the coYoneda Lemma \ref{lem coyoneda lemma}.

There are some nice categorical advantages to considering $P(\cat{C})$. For instance, the Yoneda embedding factors through $P(\cat{C})$ since we are assuming that $\cat{C}$ is locally small. The large category $P(\cat{C})$ is locally small, in contrast to $\ncat{Pre}(\cat{C})$ as maps between small presheaves are determined by a small set of maps in $\cat{C}$. It is also the case that $P(\cat{C})$ has all small colimits and is the free small cocompletion of $\cat{C}$.

A major disadvantage to $P(\cat{C})$ is that it might not have limits. Indeed, let $X$ be an essentially large set, thought of as a discrete category. Then there is an isomorphism $P(X) \cong \ncat{Set}_{/X}$, given on a small presheaf $F$ by the map $\sum_{x \in X} F(x) \to X$. But $\ncat{Set}_{/X}$ does not have a terminal object, since every object in $\ncat{Set}$ is a small set. We learned of this example from \cite{trimble2015smallpresheaves}.

However, we do know precisely when $P(\cat{C})$ has all small limits.

\begin{Prop}[{\cite{day2007limits}}]
The category $P(\cat{C})$ has all small limits if and only if for every small diagram $d : I \to \cat{C}$ and every cone $\sigma$ over $d$ there exists a small set $K$ of cones over $d$ and a map $\sigma \to \kappa$ for $\kappa \in K$. This implies that if $\cat{C}$ has all small limits then so does $P(\cat{C})$.
\end{Prop}

\subsubsection{Very Small Grothendieck Universe} \label{section very small grothendieck universe}

This last method is quite simple: rather than using two nested Grothendieck universes $\mathbb{U} \in \mathbb{V}$ for our foundations (see Section \ref{section small and large categories}), we allow for another nested Grothendieck universe $\mathbbm{u} \in \mathbb{U} \in \mathbb{V}$. We still call those sets in $\mathbb{U}$ small, and now refer to those in $\mathbbm{u}$ as \textbf{very small}. 

Now this forces the category $\ncat{Set}_{\mathbbm{u}}$ of $\mathbbm{u}$-small sets to be $\mathbb{U}$-small. From here, one can then construct whatever essentially large category (as long as its objects are built from sets) one wishes to consider using very small sets. For example, we can consider the category $\ncat{Top}_{\mathbbm{u}}$ of topological spaces whose underlying sets are $\mathbbm{u}$-small, which is a small category. This method is used by Jardine in \cite[Page 2]{jardine2007fields} and by Barwick and Haine in \cite{barwick2019pyknotic}, see Example \ref{ex condensed and pyknotic sets}. We of course do not need the full power of a very small Grothendieck universe here, but merely a consistent cardinality bound on the underlying sets being used to construct the objects in the site.

The disadvantage here is that now one must always keep track of which universe one is working with when going between the now small site and other categories. Also now the site whose objects are constructed from very small sets can no longer have all small limits or colimits unless it is a preorder by Freyd's Theorem \ref{prop freyd's theorem}.

\section{Morphisms of Sites} \label{section morphisms of sites}

The notion of a morphism between sites is not as straightforward as one might wish. We need to put conditions on a functor $F: (\cat{C}, j) \to (\cat{D}, j')$ between sites, such that we obtain a geometric morphism (Definition \ref{def geometric morphism}) between toposes $\Delta_F : \Sh(\cat{D}, j') \to \Sh(\cat{C}, j)$. However, geometric morphisms require a finite limit preserving condition, which makes the conditions we are looking for on $F$ more technical. This is mainly due to sites not necessarily having all limits or colimits. This technical condition on $F$ is called ``flatness", and we will introduce this notion in stages.

First, let us illustrate the main ideas of morphisms of sites using sites of topological spaces.

\subsection{The case of topological spaces}

If $T$ is a topological space, recall (Example \ref{ex open cover coverage}) its corresponding site $(\mathcal{O}(T), j_T)$. If $f : T \to T'$ is a map of topological spaces, then we obtain a functor $f^{-1} : \mathcal{O}(T') \to \mathcal{O}(T)$ by $U \mapsto f^{-1}(U)$. Thus we get a functor $\Delta_{f^{-1}} : \Pre(\mathcal{O}(T)) \to \Pre(\mathcal{O}(T'))$ by precomposition. For this brief section, let us denote this functor by $f_*$. So for a presheaf $X$ on $\mathcal{O}(T)$ and $U$ an open subset of $T'$, we have $f_*(X)(U) = X(f^{-1}(U))$.

If $X$ is a sheaf on $\mathcal{O}(T)$ (we'll just say $X$ is a sheaf on $T$), then $f_*(X)$ is a sheaf on $T'$. Indeed, if $U$ is an open subset of $T'$, and $\mathcal{U} = \{ U_i \subseteq U \}_{i \in I}$ is an open cover of $U$, then
\begin{equation*}
    f_*(X)(U) \to \text{eq} \left( \prod_{i \in I} f_*(X)(U_i) \rightrightarrows \prod_{i,j \in I} f_*(X)(U_i \cap U_j) \right)
\end{equation*}
is an isomorphism, because $f^{-1}(\mathcal{U})$ is an open cover of $f^{-1}(U)$, and $X$ is a sheaf on $T$.

Thus if we restrict $f_*$ to $\Sh(\mathcal{O}(T), j_T)$, which we'll denote by $\Sh(T)$, then we obtain a functor $f_*: \Sh(T) \to \Sh(T')$. This functor has a left adjoint $f^* : \Sh(T') \to \Sh(T)$, which is most easily described using the equivalence between sheaves on $T$ and local homeomorphisms (also called \'{e}tale spaces) over $T$, \cite[Section II.9]{maclane2012sheaves}. However $f^*$ is isomorphic to the composite functor 
\begin{equation*}
\Sh(T') \xhookrightarrow{i} \Pre(T') \xrightarrow{\Sigma_{f^{-1}}} \Pre(T) \xrightarrow{a} \Sh(T),
\end{equation*}
where the right hand functor denotes sheafification (Definition \ref{def sheafification}). It turns out that every continuous map $f : X \to Y$ induces a functor $\Sigma_{f^{-1}}$ that preserves finite limits. Thus $f^*$ is a left adjoint functor that preserves finite limits. 

Abstracting this common scenario leads us to the following notion, which for our purposes can be considered the default notion of a morphism between toposes.

\begin{Def} \label{def geometric morphism}
Given Grothendieck toposes $\cat{E}$ and $\cat{F}$, a \textbf{geometric morphism} $f_*: \cat{E} \to \cat{F}$ is an adjunction
$$f^*: \cat{F} \rightleftarrows \cat{E}: f_*,$$
such that the left adjoint $f^*$ preserves finite limits.
\end{Def}

\begin{Rem}
Notice that the direction of the geometric morphism matches the direction of the right adjoint. The fact that $f^*$ preserves finite limits is vitally important. One might think of this as requiring that $f^*$ be compatible with the internal logical structure of the toposes, see Remark \ref{rem categorical logic}.
\end{Rem}

\subsection{Flatness}

\subsubsection{$\Set$-valued Flatness}

\begin{Def} \label{def yoneda extension}
 Let $\cat{C}$ be a small category, and suppose we have a functor $A: \cat{C} \to \ncat{Set}$. We can take the left Kan extension of $A$ with respect to the Yoneda embedding:
\begin{equation*}
    \begin{tikzcd}
	\cat{C} & {\ncat{Set}} \\
	{\Pre(\cat{C})}
	\arrow["y"', from=1-1, to=2-1]
	\arrow["A", from=1-1, to=1-2]
	\arrow["{\text{Lan}_y A}"', from=2-1, to=1-2]
\end{tikzcd}
\end{equation*}
and let $A^* \coloneqq \text{Lan}_y A$. We call $A^*$ the \textbf{Yoneda extension} of the presheaf $A$.
\end{Def}

The idea of the Yoneda extension of $A$ is quite straightforward. By the coYoneda Lemma \ref{lem coyoneda lemma}, a presheaf $X$ is isomorphic to the colimit
\begin{equation*}
    X \cong \underset{y(U) \to X}{\colim} \, y(U).
\end{equation*}
So we simply extend $A$ in the ``obvious'' way:
\begin{equation*}
    A^*(X) = \underset{y(U) \to X}{\colim} A(U).
\end{equation*}

Using coends we can write
$$A^*(X) = (\text{Lan}_y A)(X) \cong \int^{U \in \cat{C}} \Pre(\cat{C})(y(U), X) \times A(U) \cong \int^{U \in \cat{C}} X(U) \times A(U).$$
The righthand coend is often denoted $X \otimes_\cat{C} A$ and called the \textbf{functor tensor product}. We will also adopt this notation for the Yoneda extension of $A$ in what follows.
Note that this coend can equivalently be written as the coequalizer
\begin{equation} \label{eq left kan extension of yoneda as coequalizer}
X \otimes_{\cat{C}} A \cong \text{coeq} \left(\sum_{U \xrightarrow{f} V} X(V) \times A(U) \rightrightarrows \sum_{U \in \cat{C}} X(U) \times A(U) \right)
\end{equation}
which is isomorphic to the set
$$ (X \otimes_\cat{C} A) \cong \{ (x,a) \in X(U) \times A(U) \, | \,  U \in \cat{C} \} / \sim$$
where $\sim$ denotes the smallest equivalence relation such that if $f: U \to V$, $x' \in X(V)$, then $(X(f)(x'), a) \sim (x', A(f)(a))$. We denote the equivalence class of the pair $(x,a)$ in $X \otimes_{\cat{C}} A$ by $(x \otimes a)$. 

\begin{Def} \label{def set-valued flat functor}
We say that a functor $A: \cat{C} \to \ncat{Set}$ is a $\ncat{Set}$-\textbf{flat functor} if the Yoneda extension of $A$ 
$$A^* : \Pre(\cat{C}) \to \Set$$
preserves finite limits.
\end{Def}

\begin{Rem}
The terminology comes from homological algebra, where an $R$-module $M$ is said to be flat if $(-) \otimes_R M$ is left exact, i.e. preserves finite limits. We refer to functors $A: \cat{C} \to \ncat{Set}$ as in Definition \ref{def set-valued flat functor} as $\ncat{Set}$-flat functors because we will soon introduce various other notions of flat functors whose codomain may not be $\ncat{Set}$ and even if it is, these other notions might not agree with $\ncat{Set}$-flatness.
\end{Rem}

\begin{Lemma}[{\cite[Proposition 3.8.1]{borceux1994handbook}}] \label{lem swapping arity in left kan extensions}
Let $y: \cat{C} \to \Pre(\cat{C})$ denote the usual Yoneda embedding and $y^\text{co}: \cat{C}^{\op} \to \Set^{\cat{C}}$ the covariant Yoneda embedding $y^\text{co}(c) = \cat{C}(c,-)$. Then if $X: \cat{C}^{\op} \to \Set$ is a presheaf, and $A: \cat{C} \to \Set$ a copresheaf, then there is a natural isomorphism of sets
\begin{equation*}
   (\text{Lan}_y A)(X) \cong (\text{Lan}_{y^{\text{co}}} X)(A). 
\end{equation*}
or equivalently
$$X \otimes_\cat{C} A \cong A \otimes_{\cat{C}^\op} X.$$
\end{Lemma}

Let us record the following result, which is a covariant version of the coYoneda Lemma \ref{lem coyoneda lemma}.

\begin{Lemma}[{\cite[Corollary 6.3.9]{riehl2017category}}] \label{lem yoneda extension identity on representables}
Given a functor $A : \cat{C} \to \ncat{Set}$, then for $U \in \cat{C}$,
\begin{equation*}
   y(U) \otimes_\cat{C} A \cong A(U). 
\end{equation*}
Furthermore the isomorphisms are natural in $U$.
\end{Lemma}

The following result is the key to understanding flatness.

\begin{Prop}[{\cite[Theorem VII.6.3]{maclane2012sheaves}}] \label{prop set-flat functor iff cat of elements is cofiltered}
Given a category $\cat{C}$, a functor $A: \cat{C} \to \ncat{Set}$ is $\ncat{Set}$-flat if and only if its category of elements $\smallint A$ is finitely cofiltered, i.e. if its opposite category is a finitely filtered category (Definition \ref{def filtered category}).
\end{Prop}

\begin{proof}
$(\Leftarrow)$ Note that if $X \in \Pre(\cat{C})$, then by Lemma \ref{lem swapping arity in left kan extensions} we can also write $A^*(X)$ as the colimit:
\begin{equation}
\begin{aligned}
    A^*(X) & = (\text{Lan}_y A)(X) \\ & \cong (\text{Lan}_{y^{\text{co}}} X)(A) \\ & \cong \colim \left( (y^{\text{co}} \downarrow A) \to \cat{C}^\op \xrightarrow{X} \ncat{Set} \right) \\
    & \cong \colim \left( \int A^{\op} \xrightarrow{\pi_{A^{\op}}} \cat{C}^{\op} \xrightarrow{X} \ncat{Set} \right)
\end{aligned}
\end{equation}
where $\int A^\op \coloneqq (\int A)^\op$ is the opposite of the category of elements of $A$.

By Proposition \ref{prop filtered colimits commute with finite limits in Set}, if $\int A^\op$ is a finitely filtered category then for every finite diagram $X : I \to \Pre(\cat{C})$ the $\int A^{\op}$-colimit and $I$-limit of the following functor
\begin{equation*}
    \int A^{\op} \times I \xrightarrow{\pi_{A^{\op}} \times X} \cat{C}^\op \times \Pre(\cat{C}) \xrightarrow{\text{ev}} \ncat{Set}.
\end{equation*}
commute, equivalently, $A^*$ preserves finite limits. 

$(\Rightarrow)$ Suppose that $A^*$ preserves finite limits. We will prove that $\int A^{\op}$ is finitely filtered using Lemma \ref{lem alternate def of filtered}. If we let $1$ denote the terminal object in $\Pre(\cat{C})$, then $1 \otimes_{\cat{C}} A$ is a terminal object in $\ncat{Set}$. This means that $\int A^\op$ is nonempty, satisfying Lemma \ref{lem alternate def of filtered}.(1)

Now suppose that $(U,x)$ and $(V,y)$ are objects in $\int A^\op$. We want to show that there exists $W \in \cat{C}$, $z \in A(W)$ and morphisms $f: W \to U$ and $g: W \to V$ in $\cat{C}$ such that $A(f)(z) = x$ and $A(g)(z) = y$. Since $(-) \otimes_{\cat{C}} A$ preserves finite limits, the map
\begin{equation*}
    (y(U) \times y(V)) \otimes_{\cat{C}} A \to A(U) \times A(V)
\end{equation*}
which sends an element $(f \in y(U)(W), g \in y(V)(W)) \otimes (a \in A(W))$ to $(A(f)(a), A(g)(a))$ is an isomorphism. Thus $\int A^\op$ satisfies Lemma \ref{lem alternate def of filtered}.(2).

Now suppose we are given morphisms $f,g: (V,y) \to (U,x)$ in $\int A$. We want to show that there exists a map $h : (W,z) \to (V,y)$ such that $fh = gh$. Consider the equalizer
\begin{equation*}
    P \to y(V) \underset{g}{\overset{f}{\rightrightarrows}} y(U)
\end{equation*}
in $\Pre(\cat{C})$. Thus for every object $W \in \cat{C}$, $P(W)$ is the set of maps $h : W \to V$ such that $fh = gh$. Since $A^*$ preserves finite limits, we obtain an equalizer
\begin{equation*}
    P \otimes_{\cat{C}} A \xrightarrow{k} A(V) \underset{A(g)}{\overset{A(f)}{\rightrightarrows}} A(U)
\end{equation*}
in $\ncat{Set}$, and where $k(p \otimes a) = A(p)(a)$. Thus if $y \in A(V)$ such that $A(f)(y) = A(g)(y)$, then there must exist a $(W, z) \in \int A$ and a $a \in A(W)$ such that $k(z \otimes a) = y$. Thus $\int A^\op$ satisfies Lemma \ref{lem alternate def of filtered}.(3). Thus $\int A^\op$ is a finitely filtered category.
\end{proof}

\begin{Cor} \label{cor flat iff filtered limit of coreps}
A functor $A : \cat{C} \to \ncat{Set}$ is flat if and only if it can be written as a filtered limit of corepresentables.
\end{Cor}

\begin{proof}
By Lemma \ref{lem yoneda extension identity on representables} and Lemma \ref{lem swapping arity in left kan extensions}, we have
\begin{equation*}
    A \cong A \otimes_{\cat{C}^\op} y^{\text{co}} \cong \int^{U \in \cat{C}^\op} A(U) \times y^{\text{co}}(U) \cong \ncolim{(x, U) \in \int A} y^{\text{co}}(U) \cong \lim_{(x,U) \in \int A^\op} y^\text{co}(U).
\end{equation*}
Hence the result follows from Proposition \ref{prop set-flat functor iff cat of elements is cofiltered}.
\end{proof}

\begin{Rem}
The general idea of ``flatness'' that we will pursue throughout this section is that a functor $f: \cat{C} \to \cat{D}$ is ``flat'' if it ``would preserve all finite limits if they existed''. This is not completely true for all notions of flatness we will introduce, but it is a good intuition to keep in mind.
\end{Rem}

\begin{Cor} \label{cor set-flat if and only if preserves finite limits}
Let $\cat{C}$ be a finitely complete category, then a functor $A: \cat{C} \to \ncat{Set}$ is $\ncat{Set}$-flat if and only if $A$ preserves finite limits.
\end{Cor}

\begin{proof}
$(\Rightarrow)$ By Lemma \ref{lem yoneda extension identity on representables}, $A$ is isomorphic to $A^* \circ y$. The Yoneda embedding $y$ preserves limits and $A^*$ preserves finite limits, therefore $A$ preserves finite limits.

$(\Leftarrow)$ If $A$ preserves finite limits, we wish to show that $\int A$ is cofiltered, which by Proposition \ref{prop set-flat functor iff cat of elements is cofiltered} is equivalent to $A$ being $\ncat{Set}$-flat. Now $\int \, A$ is nonempty, because $A$ preserves the terminal object $A(*) = *$. It can be easily checked that the other two conditions for being cofiltered are equivalent to $A$ preserving products and equalizers respectively.
\end{proof}

\begin{Cor} \label{cor set-flat preserves all finite limits that exist}
Let $\cat{C}$ be a small category. If $A :\cat{C} \to \Set$ is $\Set$-flat, then it preserves all finite limits that exist in $\cat{C}$.
\end{Cor}

\begin{proof}
The Yoneda embedding $y: \cat{C} \to \Pre(\cat{C})$ preserves all limits that exist in $\cat{C}$, and $A$ being $\Set$-flat implies that $A^*$ preserves all finite limits. Thus $(A^* \circ y) : \cat{C} \to \Set$ preserves all finite limits that exist in $\cat{C}$, and by Lemma \ref{lem yoneda extension identity on representables} $(A^* \circ y) \cong A$, thus $A$ preserves all finite limits that exist in $\cat{C}$.
\end{proof}

\begin{Rem}
If $\cat{C}$ is not finitely complete, then $A: \cat{C} \to \Set$ preserving all finite limits that exist in $\cat{C}$ does not imply that it is $\Set$-flat. Here is a simple counterexample, consider the category $\cat{C}$ with two objects $0$ and $1$ and no non-identity morphisms. Thus the only finite limits that exist in $\cat{C}$ are the identity cones. Thus the functor $F: \cat{C} \to \ncat{Set}$ that sends everything to the terminal set $*$ preserves all finite limits that exist in $\cat{C}$, but $\int F \cong \cat{C}$ is not cofiltered.
\end{Rem}

\subsubsection{Representable Flatness}

Now we wish to generalize this notion of flat functor to the case when the codomain is not $\Set$.

\begin{Def} \label{def rep flat functor}
We say a functor $F: \cat{C} \to \cat{D}$ between essentially small categories is \textbf{representably flat} if for every $V \in \cat{D}$, the functor
\begin{equation*}
\cat{D}(V,F(-)) : \cat{C} \to \ncat{Set}
\end{equation*}
is $\Set$-flat.
\end{Def}

\begin{Rem}
Note that in the case $\cat{D} = \Set$, it is \textbf{not} the case that $\Set$-flatness is equivalent to representable flatness. See \cite{leinster2011} for a counterexample.
\end{Rem}

\begin{Lemma} \label{lem rep flat iff comma category is cofiltered}
A functor $F: \cat{C} \to \cat{D}$ is representably flat if and only if the comma category $(V \downarrow F)$ is cofiltered for every $V \in \cat{D}$.
\end{Lemma}

\begin{proof}
This follows from the fact that $\int \cat{D}(V,F(-)) \cong (V \downarrow F)$ and Proposition \ref{prop set-flat functor iff cat of elements is cofiltered}.
\end{proof}

\begin{Lemma} \label{lem rep flat iff left kan extension preserves finite limits}
A functor $F: \cat{C} \to \cat{D}$ is representably flat if and only if the left Kan extension $\Sigma_F = \text{Lan}_{F^\op}: \Pre(\cat{C}) \to \Pre(\cat{D})$ preserves finite limits.
\end{Lemma}

\begin{proof}
$(\Rightarrow)$ For every $X \in \Pre(\cat{C})$ and $V \in \cat{D}$, the left Kan extension is isomorphic to the colimit
\begin{equation*}
    (\Sigma_F X)(V) = (\text{Lan}_{F^{\op}} X)(V) \cong \colim \left( (F^\op \downarrow V) \xrightarrow{\pi} \cat{C}^\op \xrightarrow{X} \ncat{Set} \right) \cong \colim \left( (F \downarrow V)^\op \xrightarrow{\pi} \cat{C}^{\op} \xrightarrow{X} \Set \right).
\end{equation*}
So if $d: I \to \ncat{Pre}(\cat{C})$ is a finite diagram, we can consider the functor
\begin{equation*}
    (F \downarrow V)^\op \times I \xrightarrow{\pi \times d} \cat{C}^\op \times \Pre(\cat{C}) \xrightarrow{\text{ev}} \ncat{Set}.
\end{equation*}
So by Lemma \ref{lem rep flat iff comma category is cofiltered}, if $F$ is representably flat, then $(F \downarrow V)^\op$ is filtered, so by Proposition \ref{prop filtered colimits commute with finite limits in Set}, $\Sigma_F$ preserves finite limits.

$(\Leftarrow)$ We note that $F$ is representably flat if and only if for every $V \in \cat{D}$, the functor $\cat{D}(V, F(-)) : \cat{C} \to \ncat{Set}$ is $\ncat{Set}$-flat, which is the case if and only if its Yoneda extension $\cat{D}(V,F(-))^* : \Pre(\cat{C}) \to \ncat{Set}$ preserves finite limits. But $\cat{D}(V,F(-))^*$ is the functor that sends a presheaf $X$ to $(\Sigma_FX)(V)$. Indeed we have
\begin{equation*}
\begin{aligned}
\cat{D}(V,F(-))^*(X) & = \int^{U \in \cat{C}} X(U) \times \cat{D}(V,F(U)) \\
& \cong \int^{U \in \cat{C}^\op} \cat{D}^\op(F^\op(U), V) \times X(U) \\
& \cong (\Sigma_F X)(V).
\end{aligned}
\end{equation*}
where the last isomorphism is from the proof of Lemma \ref{lem sigma on representable}.

So if $\cat{D}(V,F(-))^*$ preserves finite limits for every $V \in \cat{D}$, then so does $\Sigma_F$, since limits in presheaf toposes are computed objectwise. So if $F$ is representably flat, then $\Sigma_F$ preserves finite limits.
\end{proof}

\begin{Lemma} \label{lem rep flat iff preserves finite limits}
Let $\cat{C}$ and $\cat{D}$ be categories such that $\cat{C}$ is finitely complete. Then a functor $F: \cat{C} \to \cat{D}$ is representably flat if and only if it preserves finite limits.
\end{Lemma}

\begin{proof}
$(\Rightarrow)$ If $F$ is representably flat, then $\cat{D}(V,F(-))$ is $\Set$-flat, thus by Corollary \ref{cor set-flat if and only if preserves finite limits}, $\cat{D}(V,F(-))$ preserves finite limits. So if $d: I \to \cat{D}$ is a finite diagram, then
$$\cat{D}(V,F(\lim_{i \in I} d(i)) \cong \lim_{i \in I} \cat{D}(V,F(d(i))) \cong \cat{D}(V, \lim_{i \in I} F(d_i)).$$
Since this is true for all $V \in \cat{D}$, by the Yoneda lemma it implies that $F(\lim_i d(i)) \cong \lim_i F(d(i))$. Thus $F$ preserves finite limits. 

$(\Leftarrow)$ If $F$ preserves finite limits, then so does $\cat{D}(V,F(-))$ for every $V \in \cat{D}$, thus by Corollary \ref{cor set-flat if and only if preserves finite limits}, $\cat{D}(V,F(-))$ is $\Set$-flat for every $V \in \cat{D}$, and therefore $F$ is representably flat.
\end{proof}

\begin{Cor} \label{cor rep flat implies preserves all finite limits that exist}
If $F: \cat{C} \to \cat{D}$ is representably flat, then it preserves all finite limits that exist in $\cat{C}$.
\end{Cor}

\begin{proof}
For every $V \in \cat{D}$, the functor $\cat{D}(V,F(-))$ is $\Set$-flat, so by Corollary \ref{cor set-flat preserves all finite limits that exist}, it preserves all limits that exist in $\cat{C}$. By the same argument in the proof of Lemma \ref{lem rep flat iff preserves finite limits}, this shows that $F$ preserves all finite limits that exist in $\cat{C}$.
\end{proof}

\begin{Cor} \label{cor rep flat iff set flat}
If $\cat{C}$ is finitely complete, then a functor $A: \cat{C} \to \Set$ is $\Set$-flat if and only if it is representably flat if and only it preserves finite limits.
\end{Cor}

\subsubsection{Covering Flatness}

We now introduce a more general notion of flatness called covering flatness. This notion of flatness is specific to the case of functors between sites. It is more general than representable flatness and $\Set$-flatness. Our inspiration for this section is the blog post \cite{shulman2011flat}.

Typically in the literature, morphisms of sites are defined to be those functors that are representably flat and preserve covering families. We will instead use the notion of covering flatness as then site dense functors (Definition \ref{def site dense functor}) will be examples of morphisms of sites.

This extra generality comes at quite a high price of added complexity and technicality. The reader should feel free to skip this section on first reading. However, less is written in the literature about covering flatness, (but see \cite{shulman2012exact}, \cite{caramello2020morphism} and \cite{karazeris2004notions}) and so we include a section about it here. First we begin with some motivating observations.

\begin{Lemma} \label{lem sections of cat of elements is same as limit}
Given a functor $A: \cat{C} \to \ncat{Set}$, the set $\lim A$ is in bijection with the set of sections of the projection functor $\pi : \int A \to A$.
\end{Lemma}

\begin{proof}
This can be seen by noting that $\lim A \cong \ncat{Set}(*, \lim A) \cong \ncat{Set}^\cat{C}(\Delta(*), A)$. Natural transformations $\Delta(*) \to A$ are precisely the same thing as sections of $\pi$. In more detail, if $\sigma \in \lim A$, and $\ell_U : \lim A \to A(U)$ is the $U$-component of the limit cone, then $\ell_U(\sigma) \in A(U)$, and $U \mapsto (U, \ell_U(\sigma))$ defines a section of $\pi : \int A \to A$. Conversely such a section defines a map $\Delta(*) \to A$, which is equivalently an element of $\lim A$.
\end{proof}

We now wish to characterize the cofilteredness of $\int A$ using observations about $\lim A$ by appealing to Lemma \ref{lem sections of cat of elements is same as limit}. Given a functor $A: \cat{C} \to \ncat{Set}$ and a diagram $d : I \to \cat{C}$, we obtain a functor $d_* : \int Ad \to \int A$ defined objectwise as follows. If $(i, a) \in \int Ad$, then let $d_*(i, a) = (d(i), a)$. Furthermore the following diagram commutes
\begin{equation} \label{eq maps between cat of elements}
    \begin{tikzcd}
	{\int Ad} & {\int A} \\
	I & {\cat{C}}
	\arrow["{\pi_{Ad}}"', from=1-1, to=2-1]
	\arrow["d"', from=2-1, to=2-2]
	\arrow["{\pi_A}", from=1-2, to=2-2]
	\arrow["{d_*}", from=1-1, to=1-2]
\end{tikzcd}
\end{equation}

Given a functor $A : \cat{C} \to \ncat{Set}$ and a finite diagram $d : I \to \cat{C}$, suppose that $\lambda : \Delta(U) \to d$ is a cone in $\cat{C}$. Then $A(\lambda) : \Delta(A(U)) \to Ad$ is a cone in $\ncat{Set}$. Let $\ell : \Delta(\lim Ad) \to Ad$ denote the limit cone in $\ncat{Set}$. By the universal property of limits, there exists a unique map $h_\lambda : A(U) \to \lim Ad$ such that the following diagram commutes
\begin{equation}
    \begin{tikzcd}
	{\Delta(A(U))} && {Ad} \\
	& {\Delta(\lim Ad)}
	\arrow["{\Delta(h_\lambda)}"', from=1-1, to=2-2]
	\arrow["{A(\lambda)}", from=1-1, to=1-3]
	\arrow["\ell"', from=2-2, to=1-3]
\end{tikzcd}
\end{equation}

\begin{Prop} \label{prop set flat iff cone factorization holds}
A functor $A : \cat{C} \to \ncat{Set}$ is $\Set$-flat if and only if for every finite diagram $d: I \to \cat{C}$, the set of maps $H(A) = \{h_\lambda: A(U) \to \lim Ad \}$ is jointly epimorphic, i.e. the canonical map
$$H: \sum_{\lambda: \Delta(U) \to d} A(U) \to \lim Ad$$ 
is an epimorphism.
\end{Prop}

\begin{proof}
By Proposition \ref{prop set-flat functor iff cat of elements is cofiltered}, it is equivalent to prove that $\int A$ is cofiltered if and only if the above hypothesis holds. 

$(\Rightarrow)$ An element $\sigma \in \lim Ad$ is equivalent to a section $\sigma : I \to \int Ad$ of $\pi_{Ad}$ by Lemma \ref{lem sections of cat of elements is same as limit}. Thus given $\sigma \in \lim Ad$, from (\ref{eq maps between cat of elements}), we obtain a finite diagram $d_* \sigma$ in $\int A$. Since $\int A$ is cofiltered, there exists a cone $\Lambda : \Delta((U, a)) \to d_* \sigma$ in $\int A$, with components $\Lambda_i : (U,a) \to (d(i), a_i)$. Projecting this by $\pi_A$ produces a cone $\lambda : \Delta(U) \to d$, which in turn produces a cone $A(\lambda) : \Delta(A(U)) \to Ad$ in $\ncat{Set}$. Thus there exists a unique $h_\lambda : A(U) \to \lim Ad$ such that the following diagram commutes
\begin{equation}
    \begin{tikzcd}
	{\Delta(A(U))} && {Ad} \\
	& {\Delta(\lim Ad)}
	\arrow["{\Delta(h_\lambda)}"', from=1-1, to=2-2]
	\arrow["{A(\lambda)}", from=1-1, to=1-3]
	\arrow["\ell"', from=2-2, to=1-3]
\end{tikzcd}
\end{equation}
Thus if $f : i \to i'$ is a map in $I$, then we have the commutative diagram
\begin{equation*}
    \begin{tikzcd}
	& {A(U)} \\
	& {\lim Ad} \\
	{A(d(i))} && {A(d(i'))}
	\arrow["{h_\lambda}"', from=1-2, to=2-2]
	\arrow["{A(\lambda_i)}"', curve={height=6pt}, from=1-2, to=3-1]
	\arrow["{A(\lambda_{i'})}", curve={height=-6pt}, from=1-2, to=3-3]
	\arrow["{\ell_i}", from=2-2, to=3-1]
	\arrow["{\ell_{i'}}"', from=2-2, to=3-3]
	\arrow["{A(d(f))}"', from=3-1, to=3-3]
\end{tikzcd}
\end{equation*}
and furthermore $A(\lambda_i)(a) = a_i$. Thus $a_i = \ell_i h_\lambda(a)$. But $\ell_i(\sigma) = a_i$ by Lemma \ref{lem sections of cat of elements is same as limit}. But $h_\lambda$ is unique, so $h_\lambda(a) = \sigma$. Thus we have defined a section of $H$, which implies that it is an epimorphism.

$(\Leftarrow)$ Conversely suppose that $d: I \to \int A$ is a finite diagram. We wish to show that it has a cone in $\int A$. The composite $\pi d$ is a finite diagram $\pi d : I \to \cat{C}$. Thus by the hypothesis, the function $\Sigma_{\lambda : \Delta(U) \to \pi d} A(U) \to \lim A \pi d$ is a surjection.

Now there is a unique map $ \varphi : \lim A \to \lim A \pi d$ by the universal property of limits, and $d: I \to \int A$ is equivalently an element of $\lim A$ by Lemma \ref{lem sections of cat of elements is same as limit}. Thus $x = \varphi(d)$ defines an element $x \in \lim A \pi d$. Furthermore $d$ defines for each $i \in I$ a pair $d(i) = (U_i, a_i)$ with $U_i \in \cat{C}$ and $a_i \in A(U_i)$ and if $\ell : \Delta(\lim A \pi d) \to A \pi d$ is the limit cone, then $\ell_i(x) = a_i$ for all $i \in I$. Since $\Sigma_{\lambda : \Delta(U) \to \pi d} A(U) \to \lim A \pi d$ is surjective, there exists a cone $\lambda: \Delta(U) \to \pi d$ and an element $z \in A(U)$ such that $h_\lambda(z) = x$. Now let us consider the element $(z, U)$ in $\int A$. Applying $A$ to $\lambda$ provides a cone $A(\lambda) : A(U) \to A \pi d$, and furthermore $A(\lambda_i)(z) =  a_i$. Thus $(z, U)$ is a cone over $d$, so $\int A$ is cofiltered.
\end{proof}

Let us consider a slightly different way to think about the previous result. For a finite diagram $d : I \to \cat{C}$ consider the presheaf $\Cone(d)$ on $\cat{C}$, that sends $U \in \cat{C}$ to the set of cones $\lambda: \Delta(U) \to d$. In other words
\begin{equation*}
    \Cone(d)(U) = \cat{C}^I(\Delta(U), d).
\end{equation*}

We note that 
\begin{equation*}
    \Cone(d) \cong \lim_{i \in I} y(d(i)).
\end{equation*}

Given a small category $\cat{C}$ and a finite diagram $d : I \to \cat{C}$, consider the canonical map
\begin{equation*}
    q: \sum_{\lambda : \Delta(U) \to d} y(U) \to \lim_{i \in I} y(d(i)),
\end{equation*}
which on each component $\lambda : \Delta(U) \to d$ is defined to be the section $y(U) \to \lim_{i \in I} y(d(i)) \cong \text{Cone}(d)$ to be precisely the cone $\lambda$. This map is an epimorphism of presheaves, since if we have a cone $\sigma: \Delta(U) \to d$ then we can lift it to the identity in $y(U)(U) \cong \cat{C}(U,U)$. 

Now for any functor $A : \cat{C} \to \ncat{Set}$, we get the following commutative diagram
\begin{equation*}
    \begin{tikzcd}
	{A^*\left( \sum_{\lambda : \Delta(U) \to d} y(U) \right)} && {\lim_{i \in I} A^*(y(d(i)))} \\
	& {A^* \left( \lim_{i \in I} y(d(i)) \right)}
	\arrow["H", from=1-1, to=1-3]
	\arrow["{A^*(q)}"', from=1-1, to=2-2]
	\arrow["K"', from=2-2, to=1-3]
\end{tikzcd}
\end{equation*}
Since $A^*$ is a left adjoint, $A^*(q)$ is an epimorphism. Thus $H$ is an epimorphism if and only if $K$ is. That means that a functor $A : \cat{C} \to \ncat{Set}$ is $\ncat{Set}$-flat if and only if the comparison map
\begin{equation*}
   K :  A^*( \lim_{i \in I} y(d(i))) \to \lim_{i \in I} A^*(y(d(i))) = \lim_{i \in I} A(d(i))
\end{equation*}
is an epimorphism.

\begin{Prop} \label{prop rep flat iff cone factorization holds}
A functor $F: \cat{C} \to \cat{D}$ is representably flat if and only if for every finite diagram $d : I \to \cat{C}$, the map
\begin{equation} \label{eq map of presheaves for rep flat}
 H_d: \sum_{\lambda: \Delta(U) \to d} y(F(U)) \to \lim_{i \in I} y(F(d(i))),
\end{equation}
is an epimorphism of presheaves on $\cat{D}$.
\end{Prop}

\begin{proof}
The functor $F$ is representably flat if and only if $\cat{D}(V, F(-)): \cat{C} \to \Set$ is $\Set$-flat for every $V \in \cat{D}$. By Proposition \ref{prop set flat iff cone factorization holds}, this is equivalent to the function
\begin{equation} \label{eq rep flat in terms of surjection}
\pi_V : \sum_{\lambda: \Delta(U) \to d} \cat{D}(V,F(U)) \to \lim_{i \in I} \cat{D}(V, F(d(i)))
\end{equation}
being surjective for every $V \in \cat{D}$. This is equivalent to $H_d$ being an epimorphism of presheaves.
\end{proof}

There is a canonical map 
\begin{equation}
   K_d^\sharp: \Cone(d) \to \Delta_F \Cone(Fd),
\end{equation}
given by pushing forward a cone $\lambda : \Delta(U) \to d$ to a cone $F(\lambda) : \Delta(F(U)) \to Fd$. By the adjunction from Lemma \ref{lem presheaf adjoint triple}, we obtain a map
\begin{equation} \label{eq covering flat cone map}
 K_d : \Sigma_F \Cone(d) \to \Cone(Fd)
\end{equation}
of presheaves over $\cat{D}$.

\begin{Lemma}
The map $H_d$ of (\ref{eq map of presheaves for rep flat}) is an epimorphism of presheaves if and only if $K_d$ is an epimorphism of presheaves.
\end{Lemma}

\begin{proof}
First note that $\Cone(d) \cong \lim_{i \in I} y(d(i))$ and $\Cone(Fd) \cong \lim_{i \in I} y(Fd(i))$. Now for every $V \in \cat{D}$, by Lemma \ref{lem swapping arity in left kan extensions} we have that
\begin{equation*}
\begin{aligned}
    (\Sigma_F \Cone(d))(V) & \cong \Cone(d) \otimes_{\cat{C}} \cat{D}(V,F(-)) \\ 
    & \cong \cat{D}(V,F(-)) \otimes_{\cat{C}^\op} \Cone(d) \\
    & \cong \colim \left( \int \Cone(d)^\op \to \cat{C} \xrightarrow{\cat{D}(V,F(-))} \ncat{Set} \right) \\
    & \cong \underset{\lambda : \Delta(U) \to d}{\colim} \cat{D}(V, F(U)).
\end{aligned}
\end{equation*}
Thus there is a map $q : \sum_{\lambda : \Delta(U) \to d} y(F(U)) \to \Sigma_F \Cone(d)$ and it is objectwise an epimorphism (the map $q$ takes a pair $(\lambda : \Delta(U) \to d, f : V \to F(U))$ to its equivalence class in $(\Sigma_F \Cone(d))(V)$), and thus an epimorphism of presheaves. Furthermore the following diagram commutes
\begin{equation*}
    \begin{tikzcd}
	{\sum_{\lambda: \Delta(U) \to d} y(F(U))} && {\Cone(Fd)} \\
	& {\Sigma_F \Cone(d)}
	\arrow["{H_d}", from=1-1, to=1-3]
	\arrow["q"', two heads, from=1-1, to=2-2]
	\arrow["{K_d}"', from=2-2, to=1-3]
\end{tikzcd}
\end{equation*}
So if $K_d$ is an epimorphism, then so is $H_d$, since epimorphisms compose. However if $H_d$ is an epimorphism, then $K_d$ is as well. Thus $K_d$ is an epimorphism if and only if $H_d$ is.
\end{proof}

Now we have come to the definition of covering flatness, which will generalize the equivalent definition of representable flatness coming from Proposition \ref{prop rep flat iff cone factorization holds}.

\begin{Def} \label{def covering flat}
Let $F : \cat{C} \to (\cat{D}, j)$ be a functor to a site. We say that $F$ is \textbf{covering flat} if for every finite diagram $ d: I \to \cat{C}$, the canonical map
\begin{equation}
    K_d : \Sigma_F \, \Cone(d) \to \Cone(Fd)
\end{equation}
is a $j$-local epimorphism.
\end{Def}

\begin{Rem}
Note that $K_d$ is a $j$-local epi if and only if $H_d$ is a $j$-local epi by Lemma \ref{lem local epimorphisms closed under composition} and Lemma \ref{lem composition is local epi implies local epi}.
\end{Rem}

Let us concretely spell out what $H_d$ being a $j$-local epimorphism means: For any cone $\sigma : \Delta(V) \to Fd$, there exists a $j$-tree $T_{\sigma}$ with $T^\circ_{\sigma} = \{r_i : V_i \to V \}_{i \in I}$, such that for each $i \in I$ there is a cone $\lambda_i : \Delta(U_i) \to d$ and a map $f_i : V_i \to F(U_i)$ such that the following diagram commutes
\begin{equation} \label{eq covering flat diagram}
 \begin{tikzcd}
	{\Delta(V_i)} & {\Delta(F(U_i))} \\
	{\Delta(V)} & Fd
	\arrow["{\Delta(f_i)}", from=1-1, to=1-2]
	\arrow["{\Delta(r_i)}"', from=1-1, to=2-1]
	\arrow["{F\lambda_i}", from=1-2, to=2-2]
	\arrow["\sigma"', from=2-1, to=2-2]
\end{tikzcd}   
\end{equation}

We can paraphrase this by saying that $F$ is covering flat if for every finite diagram $d : I \to \cat{C}$, every cone over $Fd$ factors locally through the $F$-image of a cone over $d$.

We can also characterize covering flatness using a sort of ``local version'' of Lemma \ref{lem rep flat iff comma category is cofiltered}, along with Lemma \ref{lem alternate def of filtered}.

\begin{Def}[{\cite[Definition 3.2]{caramello2020morphism}}] \label{def locally cofiltered category}
Let $F : (\cat{C}, j) \to (\cat{D}, j')$ be a functor between sites, with $V \in \cat{D}$. We say that $(V \downarrow F)$ is \textbf{locally cofiltered} if
\begin{enumerate}
    \item there exists a $j'$-tree $T$ with $T^\circ = \{r_i : V_i \to V \}_{i \in I}$ such that $(V_i \downarrow F)$ is nonempty for every $i \in I$,
    \item for every pair of morphisms $f : V \to F(U)$ and $g : V \to F(U')$, there exists a $j'$-tree $T$ with $T^\circ = \{r_i : V_i \to V \}_{i \in I}$, and for each $i \in I$ maps $f_i : U_i \to U$, $g_i : U_i \to U'$ and $h_i : V_i \to F(U_i)$ such that the following diagram commutes
    \begin{equation*}
     \begin{tikzcd}
	{V_i} & {F(U_i)} \\
	V && {F(U')} \\
	& {F(U)}
	\arrow["{h_i}", from=1-1, to=1-2]
	\arrow["{r_i}", from=1-1, to=2-1]
	\arrow["{g_i}", from=1-2, to=2-3]
	\arrow["{f_i}"{pos=0.6}, from=1-2, to=3-2]
	\arrow["g"{pos=0.3}, from=2-1, to=2-3]
	\arrow["f"{pos=0.4}, from=2-1, to=3-2]
      \end{tikzcd}   
    \end{equation*}
    \item for every pair of morphisms $f,g : U \to U'$ in $\cat{C}$ and morphism $h : V \to F(U)$ in $\cat{D}$ such that $F(f)h = F(g)h$, there exists a $j'$-tree $T$ with $T^\circ = \{r_i : V_i \to V \}_{i \in I}$ such that for every $i \in I$ there are morphisms $k_i : U_i \to U$ and $h_i : V_i \to F(U_i)$ such that $f k_i = g k_i$ and the following diagram commutes
    \begin{equation*}
        \begin{tikzcd}
	{V_i} & {F(U_i)} \\
	V & {F(U)}
	\arrow["{h_i}", from=1-1, to=1-2]
	\arrow["{r_i}", from=1-1, to=2-1]
	\arrow["{F(k_i),}"{pos=0.6}, from=1-2, to=2-2]
	\arrow["h"{pos=0.4}, from=2-1, to=2-2]
\end{tikzcd}
    \end{equation*}
\end{enumerate}
\end{Def}

We leave the next lemma to the reader. It is basically just a reworking of Lemma \ref{lem alternate def of filtered}.

\begin{Lemma}
A functor $F: (\cat{C}, j) \to (\cat{D}, j')$ is covering flat if and only if for every $V \in \cat{D}$, the category $(V \downarrow F)$ is locally cofiltered.
\end{Lemma}

\begin{Rem} \label{rem special cases of covering flat}
Let us consider Definition \ref{def covering flat} when $(\cat{D}, j) = (\Set, j_{\text{epi}})$ is the category of sets with the jointly epimorphic coverage (Example \ref{ex jointly epi coverage on set}). So let $A: \cat{C} \to (\ncat{Set}, j_{\text{epi}})$ be a covering flat functor. This coverage is clearly composition closed, and therefore $K_d$ is a $j$-local epi if and only if for every cone $y(S) \to \Cone(Ad)$, there is a jointly epimorphic family $r = \{ r_i : S_i \to S \}$ over $S$ forming a commutative diagram
\begin{equation} \label{eq diagram for covering flatness in Set}
  \begin{tikzcd}
	{y(S_i)} & {\sum_{\lambda : \Delta(U) \to d} y(A(U))} \\
	{y(S)} & {\Cone(Ad)}
	\arrow[from=1-1, to=1-2]
	\arrow[from=1-1, to=2-1]
	\arrow["{H_d}", from=1-2, to=2-2]
	\arrow[from=2-1, to=2-2]
\end{tikzcd}  
\end{equation}
If we consider the jointly epimorphic family $(1_*)$, i.e. the identity function on the singleton set, then $A$ being covering-flat implies that we can lift every map $y(*) \to \Cone(Ad)$ to a map $y(*) \to \sum_{\lambda : \Delta(U) \to d } y(A(U))$. In other words it means that
$$ \sum_{\lambda : \Delta(U) \to d} \ncat{Set}(*, A(U)) \to \Cone(Ad)(*)$$
is an epimorphism of sets. But this map is isomorphic to 
\begin{equation*}
    \sum_{\lambda : \Delta(U) \to d} A(U) \to \lim Ad.
\end{equation*}
So if $A$ is covering flat, then it is $\ncat{Set}$-flat by Proposition \ref{prop set flat iff cone factorization holds}. Conversely if $A$ is $\ncat{Set}$-flat, then using the axiom of choice, we can obtain a section to the map above. Thus for any map $s : S \to \lim Ad$, we can lift it to a map $s' : S \to \sum_{\lambda: \Delta(U) \to d} A(U)$, which means that we can write $S \cong \sum_\lambda S_\lambda$, where each $S_\lambda$ is the preimage of $s'$ landing in the $\lambda$ component of $\sum_{\lambda : \Delta(U) \to d} A(U)$, and therefore obtain a commutative diagram of the form (\ref{eq diagram for covering flatness in Set}). 

When $(\cat{D}, j) = (\cat{D}, j_{\text{triv}})$ is the category $\cat{D}$ equipped with the trivial coverage (Example \ref{ex canonical coverages}), then $j_{\text{triv}}$-local epimorphisms are precisely epimorphisms of presheaves. Thus by Proposition \ref{prop rep flat iff cone factorization holds}, a functor $F: \cat{C} \to (\cat{D}, j_{\text{triv}})$ is covering flat if and only if it is representably flat. 

Thus $\ncat{Set}$-flatness and representable flatness are special cases of covering flatness.
\end{Rem}

\begin{Prop}[{\cite[Proposition 4.16]{shulman2012exact}}] \label{prop covering flat iff preserves finite limits}
A functor $F: \cat{C} \to (\cat{D}, j)$ is covering flat if and only if the composite functor
\begin{equation} \label{eq sheafification of yoneda extension for covering flatness}
    \ncat{Pre}(\cat{C}) \xrightarrow{\Sigma_F} \ncat{Pre}(\cat{D}) \xrightarrow{a} \ncat{Sh}(\cat{D}, j)
\end{equation}
preserves finite limits.
\end{Prop}

\begin{proof}
$(\Leftarrow)$ Suppose that $a \Sigma_F$ preserves finite limits, and let $d : I \to \cat{C}$ be a finite diagram. Then consider the map
\begin{equation*}
    \sum_{\lambda : \Delta(U) \to d} y(U) \to \lim_{i \in I} y(d(i)) \cong \text{Cone}(d)
\end{equation*}
which in component $\lambda$ is the section $y(U) \to \text{Cone}(d)$ defined by $\lambda$. This map is clearly an epimorphism of presheaves. Then since $a \Sigma_F$ preserves finite limits and is a composite of left adjoints, it preserves epimorphisms, so the map
\begin{equation*}
    a\Sigma_F \left( \sum_\lambda y(U) \to \lim_i y(d(i)) \right) \cong a \left( \sum_{\lambda} y(F(U)) \to \lim_i y(F(d(i))) \right)
\end{equation*}
is an epimorphism. Thus $F$ is covering flat.

$(\Rightarrow)$ Now suppose that $F$ is covering flat. Then the map
\begin{equation*}
    a\Sigma_F \lim_{i \in I} y(d(i)) \to \lim_{i \in I} a\Sigma_F y(d(i))
\end{equation*}
is an epimorphism. But if we let $A^*_V$ denote the functor $A^*_V : \Pre(\cat{C}) \to \ncat{Set}$ defined by $X \mapsto (a \Sigma_F X)(V)$, then this is equivalent to the map 
\begin{equation} \label{eq comparison map covering flat}
    A_V^*( \lim_{i \in I} y(d(i))) \to \lim_{i \in I} A_V^*(y(d(i)))
\end{equation}
being an epimorphism for every $V \in \cat{C}$.

Now let $A_V : \cat{C} \to \ncat{Set}$ denote the functor $A_V(U) = [a y(F(U))](V)$. Then $A^*_V$ is the Yoneda extension (Definition \ref{def yoneda extension}) of $A_V$. Therefore $A_V$ is $\ncat{Set}$-flat if and only if (\ref{eq comparison map covering flat}) an epimorphism. But by Proposition \ref{prop set flat iff cone factorization holds} this is the case if and only if $A^*_V$ preserves finite limits. Thus if $F$ is covering flat, then $A^*_V$ preserves finite limits for every $V \in \cat{C}$, and since limits are computed pointwise in sheaf toposes, this implies that $a \Sigma_F$ preserves finite limits.
\end{proof}

\subsection{Site Morphisms}

\begin{Def} \label{def morphism of sites}
Given sites $(\cat{C}, j)$ and $(\cat{D}, j')$, we say a functor $F: \cat{C} \to \cat{D}$ is a \textbf{morphism of sites} if it is covering flat and it sends every $j$-saturating family $r$ to a $j'$-saturating family $F(r)$.
\end{Def}

\begin{Rem} \label{rem morphism of sites sufficient conditions}
In other words $F : (\cat{C}, j) \to (\cat{D}, j')$ is a morphism of sites if and only if $F$ is $j'$-covering flat and $F$ sends $\sat{j}$-covering families to $\sat{j'}$-covering families. 

Note also that if $F$ satisfies the stronger property of sending $j$-covering families to $j'$-covering families, then it sends $j$-trees to $j'$-trees and therefore by Lemma \ref{lem family saturating iff refined by j-tree} sends $j$-saturating families to $j'$-saturating families. So if $F$ is covering flat and sends $j$-covering families to $j'$-covering families then it is a morphism of sites.
\end{Rem}

\begin{Prop} \label{prop morphism of sheaves induced by morphism of sites}
Suppose that $F: (\cat{C},j) \to (\cat{D},j')$ is a morphism of sites. Then the precomposition functor $\Delta_F: \Pre(\cat{D}) \to \Pre(\cat{C})$ takes $j'$-sheaves to $j$-sheaves, namely it descends to a functor
\begin{equation}
    \Delta_F : \Sh(\cat{D}, j') \to \Sh(\cat{C}, j)
\end{equation}
Further, $\Delta_F$ is right adjoint to the functor (\ref{eq sheafification of yoneda extension for covering flatness}) restricted to $\Sh(\cat{C}, j)$.
\end{Prop}

\begin{proof}
Suppose that $X$ is a sheaf on $\cat{D}$. We wish to show that $\Delta_F(X)$ is a sheaf on $(\cat{C},j)$. By Lemma \ref{lem sheaf on covering family iff on sieve it generates}, it is enough to show that if $U \in \cat{C}$ and $r \in j(U)$, then the canonical map
\begin{equation*}
    \Pre(\cat{C})(y(U), \Delta_F(X)) \to \Pre(\cat{C})(\overline{r}, \Delta_F(X))
\end{equation*}
is an isomorphism.

Now by the $\Sigma_F \dashv \Delta_F$ adjunction we have the following commutative diagram
\begin{equation*}
\begin{tikzcd}
	{\Pre(\cat{C})(y(U), \Delta_F(X))} & {\Pre(\cat{C})(\overline{r}, \Delta_F(X))} \\
	{\Pre(\cat{D})(\Sigma_F(y(U)), X)} & {\Pre(\cat{D})(\Sigma_F(\overline{r}), X)}
	\arrow[from=1-1, to=1-2]
	\arrow["\cong"', from=1-1, to=2-1]
	\arrow["\cong", from=1-2, to=2-2]
	\arrow[from=2-1, to=2-2]
\end{tikzcd}    
\end{equation*}
Now since $X$ is a $j'$-sheaf, by the sheafification adjunction we have
\begin{equation*}
    \begin{tikzcd}
	{\Pre(\cat{C})(y(U), \Delta_F(X))} & {\Pre(\cat{C})(\overline{r}, \Delta_F(X))} \\
	{\Sh(\cat{D}, j')(a\Sigma_F(y(U)), X)} & {\Sh(\cat{D}, j')(a\Sigma_F(\overline{r}), X)}
	\arrow[from=1-1, to=1-2]
	\arrow["\cong"', from=1-1, to=2-1]
	\arrow["\cong", from=1-2, to=2-2]
	\arrow[from=2-1, to=2-2]
\end{tikzcd}
\end{equation*}
Now by Lemma \ref{lem sigma on representable}, we have $a\Sigma_F(y(U)) \cong ay(F(U))$, and furthermore by Lemma \ref{prop sieves are coequalizers of generating family} and Proposition \ref{prop covering flat iff preserves finite limits} we have
\begin{equation*}
\begin{aligned}
    a\Sigma_F(\overline{r}) & \cong a\Sigma_F \left( \text{coeq} \left( \sum_{i,j \in I} y(U_i) \times_{y(U)} y(U_j) \rightrightarrows \sum_{i \in I} y(U_i) \right) \right) \\
    & \cong \text{coeq} \left( \sum_{i, j \in I} ay(F(U_i)) \times_{ay(F(U))} ay(F(U_j)) \rightrightarrows \sum_{i \in I} ay(F(U_i)) \right) \\
    & \cong a \, \text{coeq} \left( \sum_{i, j \in I} y(F(U_i)) \times_{y(F(U))} y(F(U_j)) \rightrightarrows \sum_{i \in I} y(F(U_i)) \right) \\
    & \cong a \overline{F(r)}.
\end{aligned}
\end{equation*}
Therefore we have 
\begin{equation*}
    \begin{tikzcd}
	{\Sh(\cat{D}, j')(a\Sigma_F(y(U)), X)} & {\Sh(\cat{D}, j')(a\Sigma_F(\overline{r}), X)} \\
	{\Sh(\cat{D}, j')(ay(F(U)), X)} & {\Sh(\cat{D},j')(a\overline{F(r)}, X)} \\
	{\Pre(\cat{D})(y(F(U)), X)} & {\Pre(\cat{D})(\overline{F(r)}, X)}
	\arrow[from=1-1, to=1-2]
	\arrow["\cong"', from=1-1, to=2-1]
	\arrow["\cong", from=1-2, to=2-2]
	\arrow[from=2-1, to=2-2]
	\arrow["\cong"', from=2-1, to=3-1]
	\arrow["\cong", from=2-2, to=3-2]
	\arrow[from=3-1, to=3-2]
\end{tikzcd}
\end{equation*}
So by Proposition \ref{prop sheaves are local iso local} $\Delta_F(X)$ is a sheaf on $r$ if and only if $\overline{F(r)} \hookrightarrow y(F(U))$ is a $j'$-local epimorphism, which is the same thing as saying that $F(r)$ is $j'$-saturating. Since $F$ is a morphism of sites, $\Delta_F(X)$ is a sheaf on $r$.

We obtain the desired adjunction $a \Sigma_F \dashv \Delta_F$ as follows. First note that if $X \in \Pre(\cat{C})$, then by the coYoneda Lemma \ref{lem coyoneda lemma} we have
\begin{equation*}
   a\Sigma_F(X) \cong a\Sigma_F \left( \ncolim{x : y(U) \to X} \, y(U) \right) \cong \ncolim{x : y(U) \to X} a \Sigma_F(y(U)) 
\end{equation*}
and by Lemma \ref{lem sigma on representable}, we have
\begin{equation*}
    a \Sigma_F(X) \cong \ncolim{x : y(U) \to X} a y(F(U)).
\end{equation*}
Thus if $X \in \Sh(\cat{C}, j)$ and $Y \in \Sh(\cat{D}, j')$, then
\begin{equation}
    \begin{aligned}
        \Sh(\cat{D}, j')(a \Sigma_F(X), Y) & \cong \Sh(\cat{D}, j')\left( \ncolim{x : y(U) \to X} a y(F(U)), Y \right) \\
        & \cong \lim_{x : y(U) \to X} \Sh(\cat{D}, j)(a y(F(U)), Y) \\
        & \cong \lim_{x : y(U) \to X} \Pre(\cat{D})(y(F(U)), iY) \\
        & \cong \lim_{x : y(U) \to X} Y(F(U)) \\
        & \cong \Sh(\cat{C}, j)\left( \ncolim{x : y(U) \to X} y(U), \Delta_F(Y) \right) \\
        & \cong \Sh(\cat{C}, j)(X, \Delta_F(Y))
    \end{aligned}
\end{equation}
\end{proof}

\begin{Not}
In what follows, if $F: (\cat{C}, j) \to (\cat{D}, j')$ is a morphism of sites, then we let $F_* : \Sh(\cat{D}, j') \to \Sh(\cat{C}, j)$ denote the functor $\Delta_F$ restricted to $\Sh(\cat{D}, j')$, and we let $F^*$ denote its left adjoint. This is in accordance with the usual notation used for geometric morphisms.
\end{Not}

\begin{Cor} \label{cor morphism of sites induces geometric morphism of topoi}
A morphism of sites $F: (\cat{C}, j) \to (\cat{D}, j')$ induces a geometric morphism of sheaf topoi:
\begin{equation}
    \begin{tikzcd}
	{\Sh(\cat{C},j)} && {\Sh(\cat{D},j')}
	\arrow[""{name=0, anchor=center, inner sep=0}, "{F^* = a \Sigma_F}", curve={height=-18pt}, from=1-1, to=1-3]
	\arrow[""{name=1, anchor=center, inner sep=0}, "{F_* = \Delta_F}", curve={height=-18pt}, from=1-3, to=1-1]
	\arrow["\dashv"{anchor=center, rotate=-90}, draw=none, from=0, to=1]
\end{tikzcd}
\end{equation}
Furthermore the following diagram commutes up to isomorphism
\begin{equation*}
    \begin{tikzcd}
	{\cat{C}} & {\cat{D}} \\
	{\Sh(\cat{C},j)} & {\Sh(\cat{D},j')}
	\arrow["F", from=1-1, to=1-2]
	\arrow["ay"', from=1-1, to=2-1]
	\arrow["ay", from=1-2, to=2-2]
	\arrow["{F^* = a\Sigma_F}"', from=2-1, to=2-2]
\end{tikzcd}
\end{equation*}
\end{Cor}

\begin{Def} \label{def morita equivalence}
We say that a morphism of sites $F : (\cat{C}, j) \to (\cat{D}, j')$ is a \textbf{Morita equivalence} if $F^*$ (equivalently $F_*$) is an equivalence. We say that two sites are Morita equivalent, if there exists a zig-zag of Morita equivalences between them. 
\end{Def}

\begin{Lemma} \label{lem morphism of sites iff morphism on saturations}
A functor $F : (\cat{C}, j) \to (\cat{D}, j')$ is a morphism of sites if and only if $F: (\cat{C}, \sat{j}) \to (\cat{D}, \sat{j'})$ is morphisms of sites.
\end{Lemma}

\begin{proof}
Firstly $F$ is $j'$-covering flat (Definition \ref{def covering flat}) if for every finite diagram $d : I \to \cat{C}$ the map $\Sigma_F \Cone(d) \to \Cone(F(d))$ is a $j'$-local epimorphism. But by Lemma \ref{lem local epi iff local epi on saturation closure}, the above map is a $j'$-local epimorphism if and only if it is a $\sat{j'}$-local epimorphism, and hence $F$ is $j'$-covering flat if and only if it is $\sat{j'}$-covering flat. Thus we need only to show that $F$ preserves $j$-saturating families if and only if it preserves $\sat{j}$-saturating families. But a family is $j$-saturating if and only if it is $\sat{j}$-covering if and only if it is $\sat{j}$-saturating. Thus $F$ is a morphism of sites if and only if it is a morphism of sites on the saturations.
\end{proof}

\subsubsection{Comorphisms of sites}

\begin{Def} \label{def comorphism of sites}
Given sites $(\cat{C}, j)$ and $(\cat{D}, j')$, a \textbf{comorphism of sites}\footnote{This is also called a cover-reflecting functor in \cite[Section C.2.3]{johnstone2002sketches}. We took the name comorphism of sites from \cite[Section 3.3]{caramello2020morphism}.} $F : (\cat{C}, j) \to (\cat{D}, j')$ is a functor $F : \cat{C} \to \cat{D}$ such that for every $U \in \cat{C}$ and every $j'$-saturating family $r$ on $F(U)$, there exists a $j$-saturating family $t$ on $U$ and a refinement $F(t) \leq r$.
\end{Def}

\begin{Lemma} \label{lem comorphism iff comorphism on saturations}
A functor $F: (\cat{C}, j) \to (\cat{D}, j')$ is a comorphism of sites if and only if $F : (\cat{C}, \sat{j}) \to (\cat{D}, \sat{j'})$ is a comorphism of sites.
\end{Lemma}

\begin{proof}
This follows from the simple observation that a family $r$ is $j$-saturating if and only if it is $\sat{j}$-covering if and only if it is $\sat{j}$-saturating.
\end{proof}

\begin{Prop} \label{prop geometric morphism from comorphism of sites}
If $F :(\cat{C}, j) \to (\cat{D}, j')$ is a comorphism of sites, then it induces a geometric morphism
\begin{equation}
\begin{tikzcd}
	{\Sh(\cat{C},j)} && {\Sh(\cat{D},j')}
	\arrow[""{name=0, anchor=center, inner sep=0}, "{\Pi_F}"', curve={height=18pt}, from=1-1, to=1-3]
	\arrow[""{name=1, anchor=center, inner sep=0}, "{a \Delta_F}"', curve={height=18pt}, from=1-3, to=1-1]
	\arrow["\dashv"{anchor=center, rotate=-90}, draw=none, from=1, to=0]
\end{tikzcd} 
\end{equation}
\end{Prop}

\begin{proof}
Since limits in $\Sh(\cat{D}, j')$ are computed objectwise (Proposition \ref{prop (co)limits in reflective subcategories}), and $\Delta_F : \Pre(\cat{D}) \to \Pre(\cat{C})$ has a left adjoint, it preserves limits, and $a$ preserves finite limits by Proposition \ref{prop sheafification props}. Thus $a \Delta_F$ preserves finite limits. 

Now we need to show that if $X$ is a sheaf on $(\cat{C}, j)$, then $\Pi_F(X)$ is a sheaf on $(\cat{D}, j')$. Given $V \in \cat{D}$ and $r = \{ r_k : V_k \to V \} \in j(V)$, with $i : \overline{r} \hookrightarrow y(V)$ then by Lemma \ref{cor sheaves on sifted closure}, $\Pi_F(X)$ is a sheaf on $r$ if the canonical map
\begin{equation*}
 \Pre(\cat{D})(i, \Pi_F(X)):  \Pre(\cat{D})(y(V), \Pi_F(X)) \to \Pre(\cat{D})(\overline{r}, \Pi_F(X))
\end{equation*}
is an isomorphism.

Now by the adjunction $\Pi_F : \Pre(\cat{C}) \rightleftarrows \Pre(\cat{D}) : \Delta_F$, we have the following commutative diagram
\begin{equation*}
\begin{tikzcd}
	{\Pre(\cat{D})(y(V), \Pi_F(X))} & {\Pre(\cat{D})(\overline{r}, \Pi_F(X))} \\
	{\Pre(\cat{C})(\Delta_F(y(V)), X)} & {\Pre(\cat{C})(\Delta_F(\overline{r}), X)}
	\arrow[from=1-1, to=1-2]
	\arrow["\cong"', from=1-1, to=2-1]
	\arrow["\cong", from=1-2, to=2-2]
	\arrow[from=2-1, to=2-2]
\end{tikzcd}    
\end{equation*}
Thus by Proposition \ref{prop sheaves are local iso local}, it is enough to show that $\Delta_F(\overline{r}) \to \Delta_F(y(V))$ is a $j$-local isomorphism, because $X$ is a $j$-sheaf. 

Now $\Delta_F$ has a left adjoint and therefore preserves monomorphisms. Thus $\Delta_F(i)$ is a monomorphism, and hence a $j$-local monomorphism.

Now let us show that $\Delta_F(i): \Delta_F(\overline{r}) \to \Delta_F(y(V))$ is a $j$-local epimorphism. Suppose that we have a section $x : y(U) \to \Delta_F(y(V))$. This is the same thing as a map $x : F(U) \to V$. We want to find a $j$-tree $T$ on $U$ such that $F(T^\circ)$ refines $r$. 

Firstly, since $j$ is a coverage, there exists a $j'$-covering family $s = \{ s_\ell : W_\ell \to F(U) \}$ with $s \in j(F(U))$ such that $x_*(s) \leq r$. But since $F$ is a comorphism of sites, and $s$ being $j'$-covering implies that it is $j'$-saturating, there exists a $j$-saturating family $t = \{t_\ell : B_\ell \to U \}$ on $U$ such that $F(t) \leq s$. But by Lemma \ref{lem family saturating iff refined by j-tree}, since $t$ is $j$-saturating, there exists a $j$-tree $T$ on $U$ and a refinement $T^\circ \leq t$. Thus if $T^\circ = \{ g_i : U_i \to U \}_{i \in I}$, then we have the following commutative diagram
\begin{equation*}
\begin{tikzcd}
	{F(U_i)} & {F(B_\ell)} & {W_k} & {V_j} \\
	{F(U)} & {F(U)} & {F(U)} & V
	\arrow[from=1-1, to=1-2]
	\arrow["{F(g_i)}"', from=1-1, to=2-1]
	\arrow[from=1-2, to=1-3]
	\arrow["{{F(t_\ell)}}"', from=1-2, to=2-2]
	\arrow[from=1-3, to=1-4]
	\arrow["s"', from=1-3, to=2-3]
	\arrow["{{r_j}}"', from=1-4, to=2-4]
	\arrow[Rightarrow, no head, from=2-1, to=2-2]
	\arrow[Rightarrow, no head, from=2-2, to=2-3]
	\arrow["x"', from=2-3, to=2-4]
\end{tikzcd}
\end{equation*}
so that $x_*(F(T^\circ)) \leq r$.
Which implies that there are morphisms $y(U_i) \to \Delta_F(\overline{r})$ for every $i \in I$ making the following diagram commute
\begin{equation*}
    \begin{tikzcd}
	{y(U_i)} & {\Delta_F(\overline{r})} \\
	{y(U)} & {\Delta_F(y(V))}
	\arrow[from=1-1, to=1-2]
	\arrow["{t_i}"', from=1-1, to=2-1]
	\arrow["{\Delta_F(i)}", hook, from=1-2, to=2-2]
	\arrow["x", from=2-1, to=2-2]
\end{tikzcd}
\end{equation*}
In other words $\Delta_F(i)$ is a $j$-local epimorphism, and hence a $j$-local isomorphism. Thus $\Pi_F(X)$ is a $j'$-sheaf.

Now suppose that $X \in \Sh(\cat{D}, j')$ and $Y \in \Sh(\cat{C}, j)$, then we have
\begin{equation*}
\begin{aligned}
    \Sh(\cat{C}, j)(a \Delta_F(X), Y) & \cong \Pre(\cat{C})(\Delta_F(X), iY) \\
    & \cong \Pre(\cat{D})(X, \Pi_F(iY)) \\
    & \cong \Sh(\cat{D}, j')(X, \Pi_F(Y)).
\end{aligned}
\end{equation*}
Where the last isomorphism holds because $\Pi_F$ sends sheaves to sheaves, $X$ is a sheaf, and $i : \Pre(\cat{D}) \hookrightarrow \Sh(\cat{D}, j')$ is fully faithful, though we abuse notation above.
\end{proof}

\begin{Lemma} \label{lem slice site is comorphism of sites}
Given a site $(\cat{C}, j)$ with $U \in \cat{C}$, the projection functor $\pi_{/U} : \cat{C}_{/U} \to \cat{C}$ is a comorphism of sites $\pi_{/U} : (\cat{C}_{/U}, j_{/U}) \to (\cat{C}, j)$, see Example \ref{ex slice site}.
\end{Lemma}

\subsubsection{Examples of Morphisms of Sites} \label{section examples of morphisms of sites}

\begin{Lemma}
Given a site $(\cat{C}, j)$, the identity functor $1_\cat{C}$ induces a morphism of sites $1_\cat{C} : (\cat{C}, j_{\text{triv}}) \to (\cat{C}, j)$, where $j_{\text{triv}}$ is the trivial coverage on $\cat{C}$, Example \ref{ex canonical coverages}. The induced geometric morphism is precisely the inclusion-sheafification adjunction
\begin{equation*}
  \begin{tikzcd}
	{\Sh(\cat{C},j)} && {\Pre(\cat{C})}
	\arrow["i"', shift right=2, hook, from=1-1, to=1-3]
	\arrow["a"', shift right=2, from=1-3, to=1-1]
\end{tikzcd}  
\end{equation*}
\end{Lemma}

\begin{Lemma}
Given a site $(\cat{C}, j)$, let $(*, j_{\text{triv}})$ denote the terminal category equipped with the trivial coverage. Then the unique functor $\cat{C} \to *$ is a morphism of sites, and the induced geometric morphism is
\begin{equation*}
    \begin{tikzcd}
	{\Sh(\cat{C},j)} && {\ncat{Set}}
	\arrow[""{name=0, anchor=center, inner sep=0}, "\Gamma"', curve={height=16pt}, from=1-1, to=1-3]
	\arrow[""{name=1, anchor=center, inner sep=0}, "{a(-)_c}"', curve={height=16pt}, from=1-3, to=1-1]
	\arrow["\dashv"{anchor=center, rotate=-90}, draw=none, from=1, to=0]
\end{tikzcd}
\end{equation*}
where $a(-)_c$ sends a set $S$ to the sheafification of the constant presheaf $S_c$, and $\Gamma$ is the functor of global sections, namely $\Gamma(X) = \ncat{Sh}(\cat{C}, j)(\ncat{1}, X)$, where $\ncat{1}$ is the terminal sheaf.
\end{Lemma}

\begin{Rem}
It is an unfortunate fact of life that if $(\cat{C}, j)$ is a site with $U \in \cat{C}$, then $\pi_{/U} : \cat{C}_{/U} \to \cat{C}$ is not in general a morphism of sites. For example, if $\cat{C}$ is the discrete category on two objects $0$ and $1$, equipped with the trivial coverage, then $\pi_{/0}$ is a morphism of sites if and only if it is representably flat. But $(1 \downarrow \pi_{/0})$ is empty, and therefore not cofiltered. However it is a comorphism of sites always by Lemma \ref{lem slice site is comorphism of sites}. See the discussion after \cite[Lemma C.2.3.3]{johnstone2002sketches} for more on this. 
\end{Rem}



\subsection{Dense Morphisms of Sites}

\begin{Def}[{\cite[Definition 11.1]{shulman2012exact}}] \label{def site dense functor}
Given sites $(\cat{C}, j)$ and $(\cat{D}, j')$, we say a functor $F : \cat{C} \to \cat{D}$ is \textbf{site dense} if the following conditions hold:
\begin{enumerate}
    \item[(D1)] a family of morphisms $r$ is a $j$-saturating family if and only if $F(r)$ is a $j'$-saturating family,
    \item[(D2)] for every $V \in \cat{D}$, there exists a $j'$-saturating family of $V$ of the form $\{F(U_i) \to V \}$,
    \item[(D3)] for every pair $U, U' \in \cat{C}$, and morphism $g: F(U) \to F(U')$, there exists a $j$-saturating family $a = \{a_i : U_i \to U \}_{i \in I}$ and a family $b = \{ b_i : U_i \to U' \}_{i \in I}$ (notice that $a$ and $b$ have the same domains) over $U'$ such that $g \circ F(a) = F(b)$, and
    \item[(D4)] for every $U, U' \in \cat{C}$ and morphisms $f,g : U \to U'$ such that $F(f) = F(g)$, there exists a $j$-saturating family $r = \{r_i : U_i \to U \}_{i \in I}$ such that $f r_i = g r_i$ for all $i \in I$.
\end{enumerate}
\end{Def}

\begin{Rem} \label{rem usual cases for dense functors}
Usually $\cat{C}$ is taken to be a subcategory of $\cat{D}$, and $F$ is the inclusion functor. In this case condition (D4) holds trivially. Furthermore if $\cat{C}$ is a full subcategory of $\cat{D}$, then condition (D3) holds trivially. By \cite[Remark 5.2]{caramello2020morphism}, if $F$ is a morphism of sites, then (D4) holds. Note that in general site dense functors are not representably flat, hence the necessity of using covering flatness.
\end{Rem}

\begin{Lemma}
Given sites $(\cat{C}, j)$, $(\cat{D}, j')$ and a functor $F: \cat{C} \to \cat{D}$, if $F$ satisfies the condition that a family of morphisms $r$ is $j$-covering if and only if $F(r)$ is $j'$-covering, then $F$ satisfies (D1).
\end{Lemma}

\begin{proof}
Suppose that $r$ is $j$-saturating. Then there exists a $j$-tree $T$ such that $T^\circ \leq r$. Thus $F(T^\circ) \leq F(r)$. Since $F$ takes $j$-covering families to $j'$-covering families, it takes $j$-trees to $j'$-trees. Hence $F(r)$ is $j'$-saturating. 

Conversely suppose that $F(r)$ is $j'$-saturating. Then by 
\end{proof}

\begin{Lemma} \label{lem site dense functor iff site dense on saturation}
A functor $F : (\cat{C}, j) \to (\cat{D}, j')$ is a site dense functor if and only if $F: (\cat{C}, \sat{j}) \to (\cat{D}, \sat{j'})$ is a site dense functor.
\end{Lemma}

\begin{proof}
This follows from the simple observation that a family $r$ is $j$-saturating if and only if it is $\sat{j}$-covering if and only if it is $\sat{j}$-saturating.
\end{proof}

\begin{Lemma}[{\cite[Theorem 11.2]{shulman2012exact}}] \label{lem site dense functors are morphisms of sites}
If $(\cat{C}, j)$ and $(\cat{D}, j')$ are sites and $F: (\cat{C}, j) \to (\cat{D}, j')$ is a site dense functor, then it is a morphism of sites.
\end{Lemma}

\begin{proof}
By Lemma \ref{lem site dense functor iff site dense on saturation}, we can assume that both $(\cat{C}, j)$ and $(\cat{D}, j')$ are saturated sites. Then we can replace all occurrences of ``saturating'' in Definition \ref{def site dense functor} with ``covering''.

Now (D1) implies that $F$ sends $j$-covering families to $j'$-covering families, so we need only to show that $F$ is covering flat.

Suppose that $d : I \to \cat{C}$ is a finite diagram and $\sigma : \Delta(V) \to Fd$ is a cone. We want to show that $\sigma$ locally factors through a cone over $d$. By (D2), we know there is a $j'$-covering family $r = \{r_k : F(U_k) \to V \}_{k \in K}$, and thus for every $k \in K$ and $i \in I$, we have a map
\begin{equation*}
    F(U_k) \xrightarrow{r_k} V \xrightarrow{\sigma_i} F(d(i)).
\end{equation*}
Thus by (D3), there exists a $j$-covering family $a^{k,i} = \{ W^{k,i}_\ell \to U_k \}_{\ell \in L^{k,i}}$ and a family $b^{k,i} = \{b^{k,i}_\ell : W^{k,i}_\ell \to d(i) \}_{\ell \in L^{k,i}}$ such that $\sigma_i r_k F(a^{k,i}) = F(b^{k,i})$.

In other words, for every $\ell$, $k$ and $i$, we have the following commutative diagram
\begin{equation*}
    \begin{tikzcd}
	{F(U_k)} & {F(W^{k,i}_\ell)} \\
	V & {F(d(i))}
	\arrow["{r_k}"', from=1-1, to=2-1]
	\arrow["{F(a^{k,i}_\ell)}"', from=1-2, to=1-1]
	\arrow["{F(b^{k,i}_\ell)}", from=1-2, to=2-2]
	\arrow["{\sigma_i}"', from=2-1, to=2-2]
\end{tikzcd}
\end{equation*}
Note that $r \circ F(a^{k,i})$ is a $j'$-covering family.

Now since $I$ is finite, we can consider the meet $a^k = \bigwedge_{i \in I} a^{k,i}$ by Lemma \ref{lem meets in saturated coverages}. Let us write $a^k = \{ a^k_n : W^k_n \to U_k \}_{n \in N^k}$. So for every $k \in K$ and $i \in I$, there exists some $n \in N^k$ and $\ell \in L^{k,i}$ and a commutative diagram
\begin{equation*}
    \begin{tikzcd}
	{W^k_n} & {W^{k,i}_\ell} & {d(i)} \\
	& {U_k}
	\arrow["{s^{k,i}_\ell}", from=1-1, to=1-2]
	\arrow["{a^k_n}"', from=1-1, to=2-2]
	\arrow["{b^{k,i}_\ell}", from=1-2, to=1-3]
	\arrow["{a^{k,i}_\ell}", from=1-2, to=2-2]
\end{tikzcd}
\end{equation*}
Let us write $b^{k,i} = b^{k,i}_\ell s^{k,i}_\ell$. Then for every fixed $k \in K$, the family $\{ b^{k,i} : W^k_n \to d(i)\}$ forms a cone over the discrete diagram $\{ d(i) \}_{i \in I}$. Now we wish to modify the maps above to form a cone over $d$. 

So suppose that $f : i \to i'$ is a morphism in $I$. Then for each $k \in K$, applying (D4) to the maps $d(f) \circ b^{k,i}$ and $b^{k,i'}$, we obtain a $j$-covering family $g_f = \{ g^f_\alpha : B^f_\alpha \to W^k_n \}_{\alpha \in A_f}$ such that $d(f) b^{k,i} g^f_\alpha = b^{k,i'} g^f_\alpha$ for all $\alpha \in A^{k,n}$. In other words we have
\begin{equation*}
    \begin{tikzcd}
	&& {d(i)} \\
	{B^f_\alpha} & {W^k_n} \\
	&& {d(i')}
	\arrow["{d(f)}", from=1-3, to=3-3]
	\arrow["{g^f_\alpha}", from=2-1, to=2-2]
	\arrow["{b^{k,i}}", from=2-2, to=1-3]
	\arrow["{b^{k,i'}}"', from=2-2, to=3-3]
\end{tikzcd}
\end{equation*}
Now again since $I$ is finite, we can consider the meet $g = \bigwedge_{f \in \text{Mor}(I)} g^f$ which is a $j$-covering family on $W^k_n$. Then if we let $g = \{ g_\beta : B_\beta \to W^k_n \}$, then for each fixed $\beta$, the composite maps $\{ b^{k,i} g_\beta \}_{i \in I}$ form a cone over $d$. Composing $g_\beta$ with $a^k_n$ gives us the following commutative diagram for every $k$, $i$ and $\beta$.
\begin{equation*}
    \begin{tikzcd}
	{F(B_\beta)} & {F(B_\beta)} \\
	{F(W^k_n)} \\
	{F(U_k)} \\
	V & {F(d(i))}
	\arrow[Rightarrow, no head, from=1-1, to=1-2]
	\arrow["{F(g_\beta)}"', from=1-1, to=2-1]
	\arrow["{F(b^{k,i}g_\beta)}", from=1-2, to=4-2]
	\arrow["{F(a^k_n)}"', from=2-1, to=3-1]
	\arrow["{r_k}"', from=3-1, to=4-1]
	\arrow["{\sigma_i}"', from=4-1, to=4-2]
\end{tikzcd}
\end{equation*}
Furthermore the left hand families compose to give a $j$-covering family $(r \circ a^k \circ g)$ over $V$. Thus $F$ is covering flat.
\end{proof}

\begin{Rem}
Thanks to Lemma \ref{lem site dense functors are morphisms of sites}, we may also refer to a site dense functor simply as a \textbf{dense morphism of sites}.
\end{Rem}

\begin{Lemma} \label{lem site dense functors are comorphisms of sites}
If $F : (\cat{C}, j) \to (\cat{D}, j')$ is a site dense functor, then $F$ is a comorphism of sites.
\end{Lemma}

\begin{proof}
By Lemma \ref{lem site dense functor iff site dense on saturation} and Lemma \ref{lem comorphism iff comorphism on saturations}, we can assume that $(\cat{C}, j)$ and $(\cat{D}, j')$ are saturated sites.
Suppose that $U \in \cat{C}$ and $r = \{r_k : V_k \to F(U)\}_{k \in K} \in j'(F(U))$. By (D2), there is a $j'$-covering family $s^k = \{s^k_i : F(U^k_i) \to V_k \}$ for every $k \in K$. Thus for every $i$ and $k$ we obtain a map 
$$F(U^k_i) \xrightarrow{s^k_i} V_k \xrightarrow{r_k} F(U).$$
So by (D3), we have a $j$-covering family $a = \{ a^{k,i}_\ell : W^{k,i}_\ell \to U^k_i\}$ and a family $b =\{b^{k,i}_\ell : W^{k,i}_\ell \to U \}$ such that the following diagram commutes
\begin{equation*}
   \begin{tikzcd}
	{F(U^k_i)} & {F(W^{k,i}_\ell)} \\
	{V_k} & {F(U)}
	\arrow["{s^k_i}"', from=1-1, to=2-1]
	\arrow["{F(a^{k,i}_\ell)}"', from=1-2, to=1-1]
	\arrow["{F(b^{k,i}_\ell)}", from=1-2, to=2-2]
	\arrow["{r_k}"', from=2-1, to=2-2]
\end{tikzcd} 
\end{equation*}
But note that since $F$ sends $j$-covering families to $j'$-covering families, the composite family $F(b) = (r \circ s^k \circ F(a))$ is a $j'$-covering family of $F(U)$. By (D1), $F$ reflects covering families, which implies that $b$ is a $j$-covering family of $U$. But we also have that $F(b) \leq r$, so $F$ is a comorphism of sites.
\end{proof}

\begin{Th}[The Comparison Lemma, {\cite[Theorem C.2.2.3]{johnstone2002sketches}, \cite[Theorem 11.8]{shulman2012exact}}] \label{th comparison lemma}
If $F: (\cat{C},j) \to (\cat{D},j')$ is a site dense functor, then the induced geometric morphism
\begin{equation*}
    \Delta_F = F_* : \Sh(\cat{D}, j') \to \Sh(\cat{C}, j)
\end{equation*}
is an equivalence of categories, i.e. $F$ is a Morita equivalence.
\end{Th}

\begin{proof}
By Proposition \ref{prop sheaf iff sheaf on saturation closure} and Lemma \ref{lem site dense functor iff site dense on saturation}, we can assume that $(\cat{C}, j)$ and $(\cat{D}, j')$ are saturated sites. By Lemma \ref{lem site dense functors are morphisms of sites} and Lemma \ref{lem site dense functors are comorphisms of sites} we know that $F$ is both a morphism and comorphism of sites. Thus the functors $\Delta_F$ and $\Pi_F$ take sheaves to sheaves.

Now we want to show that if $X \in \Sh(\cat{C}, j)$, then the counit map
\begin{equation*}
    \varepsilon_X : \Delta_F \Pi_F(X) \to X
\end{equation*}
is an isomorphism.

Suppose that $U \in \cat{C}$, then by Lemma \ref{lem presheaf adjoint triple} we have
\begin{equation*}
    \Delta_F\Pi_F(X)(U) \cong \Pi_F(X)(F(U)) \cong \lim_{F(V) \to F(U)} X(V),
\end{equation*}
and the counit $\varepsilon_X : \Delta_F \Pi_F(X)(U) \to X(U)$ sends an element $x \in \lim_{F(V) \to F(U)} X(V)$ to the component $x_{1_{F(U)}} \in X(U)$ given by the identity map $1_{F(U)} : F(U) \to F(U)$. 

Let us define a map $\varphi_U : X(U) \to \Delta_F \Pi_F(X)(U)$ as follows. Given $x \in X(U)$ and a map $f : F(V) \to F(U)$ in $\cat{D}$, since $F$ is site dense, by (D3) there is a $j$-covering family $r = \{r_i : W_i \to V \}_{i \in I}$ and a family $t = \{t_i : W_i \to U \}_{i \in I}$ such that $f F(r) = F(t)$. So we obtain for every $i \in I$ an element $X(t_i)(x) \in X(W_i)$. Now let us show that the family $\{ X(t_i)(x) \}$ forms an $X$-matching family on $r$.

So suppose we have an intersection square for $r$
\begin{equation*}
    \begin{tikzcd}
	{W_{ij}} & {W_j} & U \\
	{W_i} & V \\
	U
	\arrow["b", from=1-1, to=1-2]
	\arrow["a"', from=1-1, to=2-1]
	\arrow["{t_j}", from=1-2, to=1-3]
	\arrow["{r_j}", from=1-2, to=2-2]
	\arrow["{r_i}"', from=2-1, to=2-2]
	\arrow["{t_i}"', from=2-1, to=3-1]
\end{tikzcd}
\end{equation*}
We want to show that $X(b)X(t_j)(x) = X(a)X(t_i)(x)$. But note that applying $F$ to the above diagram we obtain a commutative diagram
\begin{equation*}
    \begin{tikzcd}
	{F(W_{ij})} & {F(W_j)} & {F(U)} \\
	{F(W_i)} & {F(V)} \\
	{F(U)}
	\arrow["{F(b)}", from=1-1, to=1-2]
	\arrow["{F(a)}"', from=1-1, to=2-1]
	\arrow["{F(t_j)}", from=1-2, to=1-3]
	\arrow["{F(r_j)}"', from=1-2, to=2-2]
	\arrow["{F(r_i)}", from=2-1, to=2-2]
	\arrow["{F(t_i)}"', from=2-1, to=3-1]
	\arrow["f"', from=2-2, to=1-3]
	\arrow["f", from=2-2, to=3-1]
\end{tikzcd}
\end{equation*}
So
\begin{equation*}
    F(t_j)F(b) = f F(r_j)F(b) = fF(r_i)F(a) = F(t_i)F(a).
\end{equation*}
Now by (D4), there exists a $j$-covering family $s = \{ s_k : B_k \to W_{ij} \}_{k \in K_{ij}}$ such that $t_j b s_k = t_i a s_k$ for every $k \in K_{ij}$. Thus we obtain an $X$-matching family $\{ X(s_k)X(t_j b)(x) = X(s_k)X(t_i a)(x) \}$ on $s$. Since $X$ is a $j$-sheaf, this implies that there exists a unique amalgamation $w \in X(W_{ij})$. Thus
\begin{equation*}
    X(t_jb)(x) = w = X(t_ia)(x)
\end{equation*}
for every $i, j \in I$. Thus $\{ X(t_i)(x) \}_{i \in I}$ is an $X$-matching family on $r$, and since $X$ is a $j$-sheaf, this has a unique amalgamation $y \in X(V)$.

Thus having fixed a covering family $r$ and family $t$, we have obtained for every map $f: F(V) \to F(U)$ a map $X(U) \to X(V)$. By construction we know that this map forms a cone and thus we obtain a map $X(U) \to \lim_{F(V) \to F(U)} X(V)$. Now it is not too hard to show that this map is inverse to the $U$-component of the counit map $\varepsilon_X(U)$. Indeed, by considering the above construction for the identity $F(U) \to F(U)$, we can always use the trivial covering family by the identity map. Since inverses are unique, we see that the map above does not depend on $r$ and $t$, and therefore is well-defined. Thus $\varepsilon_X$ is a natural isomorphism.

Now we wish to show that if $X \in \Sh(\cat{D}, j')$, then the unit map
\begin{equation*}
    \eta_X : X(V) \to \Pi_F \Delta_F(X)(V)
\end{equation*}
is an isomorphism. This is equivalently the map
\begin{equation*}
    X(V) \to \lim_{F(U) \to V} X(F(U))
\end{equation*}
that sends an element $x \in X(V)$ to the collection of elements $\{X(f)(x) \}_{f : F(U) \to V}$.

Now by (D2) there exists a $j'$-saturating family $r = \{r_i: F(U_i) \to V \}$ and hence $\eta_X(x)$ restricted to $r$ is an $X$-matching family on $r$ $\{X(r_i)(x) \}$. Now the collection of all morphisms $\{F(U) \to V \}$ is refined by $r$, and since $j'$ is saturated, this implies that $\{F(U) \to V \}$ is $j'$-covering. Since $X$ is a $j'$-sheaf, this implies that $\eta_X$ is an isomorphism.
\end{proof}

\subsubsection{Examples}

\begin{Ex}
Consider the sequence of full subcategory inclusions from Example \ref{ex coverages on Cart and Open}, where each category is equipped with the open cover coverage
\begin{equation*}
    (\ncat{Cart}, j_{\open}) \hookrightarrow (\ncat{Open}, j_{\open}) \hookrightarrow (\ncat{Man}, j_{\open})
\end{equation*}
Since each of these functors is a full subcategory inclusion, (D3) and (D4) hold trivially. Since each of these sites is composition closed, if $r$ is a family of morphisms in say $( \ncat{Open}, j_\text{open})$, then it is a covering family if and only if it is a covering family when included into $(\ncat{Man}, j_\text{open})$. Now if $r$ is a family in $\ncat{Open}$ such that its image in $\ \ncat{Man}$ is refined by a $j_{\text{open}}$-covering family, then by restricting/pulling back each open subset of a cartesian space along the refinement (which are then open subsets of a cartesian space), we obtain a cover by  open subsets of cartesian spaces. Hence (D1) holds. A similar argument holds for $\ncat{Cart}$. Now as discussed in Example \ref{ex j good coverage}, we can refine any open cover of a manifold by a good open cover, and hence (D2) holds for both inclusions. Hence both inclusions are dense site morphisms, so all three sites are Morita equivalent.
\end{Ex}

\begin{Ex}
Consider the sequence of full subcategory inclusions from Example \ref{ex coverages on complex manifolds}
\begin{equation*}
    (\C\ncat{Disk},j_{\text{open}}) \hookrightarrow (\ncat{Stein}, j_{\text{open}}) \hookrightarrow (\C\ncat{Man}, j_{\text{open}})
\end{equation*}
where we are abusing notation and letting $j_{\text{open}}$ also denote the induced coverages. Since each of these functors is a full subcategory inclusion, (D3) and (D4) hold trivially. Since each of these sites is composition closed, if $r$ is a family of morphisms in say $( \ncat{Stein}, j_\text{open})$, then it is a covering family if and only if it is a covering family when included into $(\C\ncat{Man}, j_\text{open})$. Now if $r$ is a family in $\ncat{Stein}$ such that its image in $\C \ncat{Man}$ is refined by a $j_{\text{open}}$-covering family, then by restricting/pulling back each Stein manifold along the refinement (which are then Stein manifolds by \cite[Lemma 4.1]{larusson2003excision}), we obtain a cover by Stein open subsets. Hence (D1) holds. The argument for $(\C\ncat{Disk}, j_{\text{open}})$ being a coverage in Example \ref{ex coverages on complex manifolds} also proves (D1) holds for the inclusion $(\C\ncat{Disk}, j_{\text{open}}) \hookrightarrow (\ncat{Stein}, j_{\text{open}})$. Now if $M$ is any complex manifold, let $\mathcal{U}$ be an open cover of $M$ by coordinate charts, i.e. domains of $\C^n$. Then by \cite[Lemma II.1]{fornaess1977spreading}, we can refine it by a covering by polydisks. Hence both of the above site inclusions satisfy (D2). Hence both inclusions are dense site morphisms, so all three sites are Morita equivalent.
\end{Ex}

One might wish that for every manifold the obvious map $i : \mathcal{O}(M) \to \ncat{Man}_{/M}$ to be site dense, but unfortunately this is not the case.

\begin{Ex} \label{ex gros and petit topos}
From Example \ref{ex slice site}, we know that given any site $(\cat{C}, j)$, and any $U \in \cat{C}$, there is a canonical site structure on the slice category $(\cat{C}_{/U}, j_{/U})$. Consider a manifold $M \in \ncat{Man}$. We obtain two sites $(\mathcal{O}(M), j_M)$ and $(\ncat{Man}_{/M}, j_{/M})$.

There is an obvious functor $i : \mathcal{O}(M) \to \ncat{Man}_{/M}$ that sends an open subset to its inclusion $U \hookrightarrow M$. This functor clearly preserves covering families, and since both sites are saturated, it therefore preserves saturating families. Now $\mathcal{O}(M)$ is finitely complete, with all of its finite limits equal to finite products, which are given by intersections. These are sent to finite products in $\ncat{Man}_{/M}$, so by Lemma \ref{lem rep flat iff preserves finite limits}, $i$ is representably flat, and therefore it is a morphism of sites. 

Unfortunately, $i$ is not a dense morphism of sites in general. For example, take $M = \R^0 = *$. Then $\Sh(\mathcal{O}(*), j_*) \cong \ncat{Set}$, and $\Sh(\ncat{Man}_{/*}, j_{/*}) \cong \Sh(\ncat{Man}, j_{/*}) \cong \Pre(\ncat{Man})$.

We call $(\mathcal{O}(M), j_M)$ the \textbf{petit site} and $(\ncat{Man}_{/M}, j_{/M})$ the \textbf{gros site} of $M$.
\end{Ex}

\begin{Rem}
While the petit and gros sites do not have equivalent sheaf topoi, they are ``homotopy equivalent'' in a certain sense, see \cite[Page 416]{maclane2012sheaves}.
\end{Rem}

\section{Points of a Site} \label{section points of a site}
In this section we discuss points of sites. When a site has a set of enough points, this provides us with another way to test when a morphism of presheaves is a local isomorphism.

\begin{Def} \label{def point of a topos}
If $\cat{E}$ is a Grothendieck topos, then a \textbf{point} of $\cat{E}$ is a geometric morphism $p_* : \ncat{Set} \to \cat{E}$. 
\end{Def}

Let us specialize first to the case of presheaf topoi. We will show that points of a presheaf topos are given precisely by $\ncat{Set}$-flat functors. As we've seen in Section \ref{section morphisms of sites}, if $\cat{C}$ is a small category, a functor $A : \cat{C} \to \ncat{Set}$ is $\ncat{Set}$-flat (Definition \ref{def set-valued flat functor}) if and only if $A^* = (-) \otimes_{\cat{C}} A : \Pre(\cat{C}) \to \ncat{Set}$ preserves finite limits. Let us for the moment just consider arbitrary functors $A : \cat{C} \to \ncat{Set}$. Then $(-) \otimes_{\cat{C}} A$ clearly preserves colimits and has a right adjoint, which we call the \textbf{skyscraper construction}. Let $A_* : \ncat{Set} \to \Pre(\cat{C})$ be defined objectwise for any $S \in \ncat{Set}$ and $U \in \cat{C}$ by
\begin{equation*}
(A_*(S))(U) = \ncat{Set}(A(U),S).
\end{equation*}

Now $A_*$ is right adjoint to $(-) \otimes_{\cat{C}} A$ by the following computation. For every $X \in \Pre(\cat{C})$ and $S \in \ncat{Set}$ we have
\begin{equation*}
    \begin{aligned}
        \ncat{Set}(X \otimes_\cat{C} A, S) & \cong \ncat{Set} \left( \int^{U \in \cat{C}} X(U) \times A(U), S \right) \\
        & \cong \int_{U \in \cat{C}} \ncat{Set}(X(U) \times A(U), S) \\
        & \cong \int_{U \in \cat{C}} \ncat{Set}(X(U), \ncat{Set}(A(U),S)) \\
        & \cong \ncat{Pre}(\cat{C})(X, \ncat{Set}(A(-),S)) \\
        & \cong \ncat{Pre}(\cat{C})(X, A_*(S)). 
    \end{aligned}
\end{equation*}
Thus from any functor $A : \cat{C} \to \ncat{Set}$, we obtain an adjunction $A^* = (-)\otimes_\cat{C} A \dashv A_*$. Now let us show the converse, i.e. that if we have an adjunction $L : \Pre(\cat{C}) \rightleftarrows \ncat{Set} : R$, we obtain a functor $A : \cat{C} \to \ncat{Set}$.

Let $\ncat{Adj}(\ncat{Pre}(\cat{C}), \ncat{Set})$ denote the category whose objects are adjunctions $L : \ncat{Pre}(\cat{C}) \rightleftarrows \ncat{Set} : R$ and whose morphisms are natural transformations between the left adjoints (equivalently right adjoints) of the adjunctions.

Given a functor $A: \cat{C} \to \ncat{Set}$, let $\phi(A)$ denote the adjunction $A^* \dashv A_*$ constructed above. This construction extends to a functor $\phi: \ncat{Set}^{\cat{C}} \to \ncat{Adj}(\ncat{Pre}(\cat{C}), \ncat{Set})$. Indeed, if $f : A \to B$ is a natural transformation, then since $\text{Lan}_y : \ncat{Set}^\cat{C} \to \ncat{Set}^{\ncat{Pre}(\cat{C})}$ is a functor, one obtains a natural transformation $(-)\otimes_\cat{C} f: A^* = (-)\otimes_\cat{C} A \to (-)\otimes_\cat{C} B = B^*$.

Conversely, given an adjunction $L : \ncat{Pre}(\cat{C}) \rightleftarrows \ncat{Set} : R$, let $\psi(L \dashv R)$ denote the composite functor
$$\cat{C} \xrightarrow{y} \ncat{Pre}(\cat{C}) \xrightarrow{L} \ncat{Set}.$$
This also clearly extends to a functor $\psi: \ncat{Adj}(\ncat{Pre}(\cat{C}), \ncat{Set}) \to \ncat{Set}^\cat{C}$. 

\begin{Lemma} \label{lem equivalence between set functors and adjunctions}
The functors $\phi$ and $\psi$ defined above form an equivalence of categories
\begin{equation*}
    \begin{tikzcd}
	{\ncat{Set}^{\cat{C}}} && {\ncat{Adj}(\ncat{Pre}(\cat{C}), \ncat{Set})}
	\arrow[""{name=0, anchor=center, inner sep=0}, "\phi", curve={height=-18pt}, from=1-1, to=1-3]
	\arrow[""{name=1, anchor=center, inner sep=0}, "\psi", curve={height=-18pt}, from=1-3, to=1-1]
	\arrow["\simeq"{description}, draw=none, from=0, to=1]
\end{tikzcd}
\end{equation*}
\end{Lemma}

\begin{proof}
Suppose that $A$ is a functor $A: \cat{C} \to \Set$. By the coYoneda Lemma (Lemma \ref{lem coyoneda lemma}), every presheaf is a colimit of representables, so if $X \in \ncat{Pre}(\cat{C})$, then
$$X \otimes_{\cat{C}} A \cong \left( \ncolim{y(U) \to X} y(U) \right) \otimes_{\cat{C}} A \cong \ncolim{y(U) \to X} A(U).$$
Thus $A^* = (-) \otimes_{\cat{C}} A$ is completely determined by what it does to representables. So if we consider the functor $\psi(\phi(A))$, which is the composite
$$ \cat{C} \xrightarrow{y} \ncat{Pre}(\cat{C}) \xrightarrow{A^*} \ncat{Set},$$
then by Lemma \ref{lem yoneda extension identity on representables}, we have an isomorphism $\psi(\phi(A)) = (A^* \circ y) \cong A$, natural in $A$.

Conversely, given an adjunction $L : \Pre(\cat{C}) \rightleftarrows \ncat{Set}: R$, then $\psi(L \dashv R) = (L \circ y)$, and we have $\phi(\psi(L \dashv R)) = \text{Lan}_y (L \circ y)$. But by Lemma \ref{lem yoneda extension identity on representables}, $\text{Lan}_y (L \circ y)$ and $L$ are isomorphic on representables, and thus by the coYoneda Lemma are naturally isomorphic functors. Thus $\phi(\psi(L \dashv R)) \cong L \dashv R$.
\end{proof}

Now if $\cat{C}$ is a small category, let $\ncat{Geo}(\ncat{Set}, \ncat{Pre}(\cat{C}))$ denote the full subcategory of $\ncat{Adj}(\ncat{Pre}(\cat{C}), \Set)$ whose objects are the geometric morphisms. Let $\ncat{Flat}(\cat{C})$ denote the category whose objects are the $\Set$-flat functors $A: \cat{C} \to \ncat{Set}$ and whose morphisms are natural transformations. The functor $\psi$ from Lemma \ref{lem equivalence between set functors and adjunctions} restricts to a functor $\psi': \ncat{Flat}(\cat{C}) \to \ncat{Geo}(\Set, \Pre(\cat{C}))$. Similarly the functor $\phi$ from Lemma \ref{lem equivalence between set functors and adjunctions} also restricts to a functor $\phi': \ncat{Geo}(\Set, \Pre(\cat{C})) \to \ncat{Flat}(\cat{C})$ because if $A_* : \Set \to \Pre(\cat{C})$ is a geometric morphism, then $\phi(A^* \vdash A_*) = (A^* \circ y)$ is $\Set$-flat, since $\text{Lan}_y (A^* \circ y) \cong A^*$ and $A^*$ preserves finite limits.

\begin{Prop}[{\cite[Theorem VII.5.2]{maclane2012sheaves}}] \label{prop flat functors equiv to geometric morphisms}
Given a small category $\cat{C}$, the functors $\psi', \phi'$ defined above form an equivalence of categories
\begin{equation} \label{eqn equiv points and flat functors}
  \begin{tikzcd}
	{\ncat{Flat}(\cat{C})} && {\ncat{Geo}(\Set, \Pre(\cat{C}))}
	\arrow[""{name=0, anchor=center, inner sep=0}, "\phi'", curve={height=-18pt}, from=1-1, to=1-3]
	\arrow[""{name=1, anchor=center, inner sep=0}, "\psi'", curve={height=-18pt}, from=1-3, to=1-1]
	\arrow["\simeq"{description}, draw=none, from=0, to=1]
\end{tikzcd}
\end{equation}
\end{Prop}

So we have now shown that points of presheaf toposes $\ncat{Set} \to \Pre(\cat{C})$ are equivalent to $\ncat{Set}$-flat functors $\cat{C} \to \ncat{Set}$. We would like to generalize this result to Grothendieck toposes. For that, we need to put more conditions on the functor $\cat{C} \to \ncat{Set}$.

\begin{Def}
Given a site $(\cat{C}, j)$, we say a functor $A: \cat{C} \to \ncat{Set}$ is \textbf{$j$-continuous} if for every $U \in \cat{C}$ and every covering family $r = \{r_i : U_i \to U \}_{i \in I} \in j(U)$, the corresponding map 
\begin{equation*}
\sum_{i \in I} A(U_i) \xrightarrow{\sum_{i \in I} A(r_i)}A(U).
\end{equation*}
is an epimorphism of sets.
\end{Def}

\begin{Rem} \label{rem equiv def of j-continuous functor to set}
 Note that this is equivalent to $A^*$ sending the map $\sum_{i \in I} r_i : \sum_{i \in I} y(U_i) \to y(U)$ to an epimorphism. By the proof of Lemma \ref{lem family is saturating iff sum of maps is local epi}, this is equivalent to $A^*$ sending the map $i : \overline{r} \hookrightarrow y(U)$ of presheaves to an epimorphism. Note that if $A$ is furthermore $\ncat{Set}$-flat, then $A^*$ preserves monomorphisms and epimorphisms, and therefore sends $i$ to an isomorphism.
\end{Rem}

\begin{Lemma}[{\cite[Lemma VII.5.3]{maclane2012sheaves}}] \label{lem continuous functors give sheaves}
Given a site $(\cat{C}, j)$ and a point $A : \cat{C} \to \ncat{Set}$, the functor $A_* : \ncat{Set} \to \Pre(\cat{C})$ factors through $\Sh(\cat{C}, j)$ if and only if $A$ is $j$-continuous.
\end{Lemma}

\begin{proof}
We want to show that if $S$ is a set, then $A_*(S)$ is a sheaf if and only if $A$ is $j$-continuous. By Lemma \ref{lem sheaf on covering family iff on sieve it generates}, it is enough to show that for every $U \in \cat{C}$ and every covering family $r \in j(U)$ the canonical map
\begin{equation*}
    \Pre(\cat{C})(y(U), A_*(S)) \to \Pre(\cat{C})(\overline{r}, A_*(S))
\end{equation*}
is an isomorphism. Using the adjunction $A^* \dashv A_*$ adjunction, we have the following commutative diagram
\begin{equation*}
\begin{tikzcd}
	{\Pre(\cat{C})(y(U), A_*(S))} && {\Pre(\cat{C})(\overline{r}, A_*(S))} \\
	{\ncat{Set}(A^*(U), S)} && {\ncat{Set}(A^*(\overline{r}), S)}
	\arrow["{\Pre(\cat{C})(i, A_*(S))}", from=1-1, to=1-3]
	\arrow["\cong"', from=1-1, to=2-1]
	\arrow["\cong", from=1-3, to=2-3]
	\arrow["{\ncat{Set}(A^*(i),S)}"', from=2-1, to=2-3]
\end{tikzcd}  
\end{equation*}
but by Remark \ref{rem equiv def of j-continuous functor to set}, if $A$ is $j$-continuous, then $A^*(i)$ is an isomorphism, so the bottom horizontal map is an isomorphism, so the top one is as well. Thus $A_*(S)$ is a sheaf. Conversely if the top horizontal map is an isomorphism for every $S$, then by the Yoneda lemma $A^*(i)$ is an isomorphism and hence $A$ is $j$-continuous.
\end{proof}

Now given a site $(\cat{C}, j)$, let $\ncat{ConFlat}(\cat{C}, j)$ denote the full subcategory of $\ncat{Flat}(\cat{C})$ on those flat functors $A : \cat{C} \to \ncat{Set}$ that are also $j$-continuous.

\begin{Cor}[{\cite[Corollary VII.5.4]{maclane2012sheaves}}] \label{cor points of sheaf topoi are equivalent to continuous flat functors}
The equivalence of Proposition \ref{prop flat functors equiv to geometric morphisms} restricts to an equivalence
$$ \ncat{ConFlat}(\cat{C}, j) \simeq \ncat{Geo}(\ncat{Set}, \ncat{Sh}(\cat{C})).$$
\end{Cor}

\begin{Def}
By a \textbf{point of a site} $(\cat{C}, j)$, we mean a $j$-continuous flat functor $A : (\cat{C},j) \to \ncat{Set}$ or equivalently a point $A_* : \ncat{Set} \to \Sh(\cat{C}, j)$.
\end{Def}

\begin{Rem}
Note that if $(\cat{C}, j)$ is a site, then a functor $A : \cat{C} \to \ncat{Set}$ is $\ncat{Set}$-flat if and only if it is covering flat by Remark \ref{rem special cases of covering flat}. Therefore if we consider $\ncat{Set}$ with the jointly epimorphic coverage $j_{\text{epi}}$, then since it is a saturated coverage, the functor $A : \cat{C} \to \ncat{Set}$ is a point of $(\cat{C}, j)$ if and only if $A$ is a morphism of sites.
\end{Rem}

\begin{Rem}
Note that any point $A_* : \ncat{Set} \to \Sh(\cat{C}, j)$ of a site $(\cat{C}, j)$ also determines a point $i A_* : \ncat{Set} \to \Pre(\cat{C})$ of the presheaf topos, by composition with the geometric morphism $i : \Sh(\cat{C}, j) \to \Pre(\cat{C})$, and if $a$ denotes sheafification, then its left adjoint is $A^* a$.
\end{Rem}

\begin{Prop} \label{prop sheaf epi implies epi on points}
Given a site $(\cat{C}, j)$, if $f: X \to Y$ is a epi/monomorphism of sheaves and $p$ is a point of $(\cat{C}, j)$, then $p^* f: p^*X \to p^*Y$ is an epi/monomorphism of sets.
\end{Prop}

\begin{proof}
This just follows from the fact that $p^* : \Sh(\cat{C}, j) \to \ncat{Set}$ is a left adjoint and hence preserves epimorphisms, and since it preserves finite limits it preserves monomorphisms.
\end{proof}

\begin{Cor} \label{cor local mono implies mono on points}
Given a site $(\cat{C}, j)$, if $f: X \to Y$ is a $j$-local epi/monomorphism of presheaves, $p$ is a point and $a$ denotes sheafification, then $p^*a f : p^*a X \to p^*a Y$ is a epi/monomorphism of sets.
\end{Cor}

\begin{proof}
This follows from Corollary \ref{cor j-local iso/epi/mono <-> sheafification is iso/epi/mono} and Proposition \ref{prop sheaf epi implies epi on points}.
\end{proof}

\begin{Def} \label{def set of enough points}
Let $(\cat{C}, j)$ be a site. Let $P = \{ p_i \}_{i \in I}$ be a set of points of $(\cat{C}, j)$. We say that $P$ is a \textbf{set of enough points} for $(\cat{C}, j)$ if for every map $f: X \to Y$ of sheaves the following property holds: if $p_i^*(f): p^*_i(X) \to p^*_i(Y)$ is an isomorphism for every $i \in I$, then $f$ is an isomorphism of sheaves.
\end{Def}

If $P = \{p_i \}_{i \in I}$ is a set of points on a site $(\cat{C}, j)$, then we obtain a map
\begin{equation*}
p^* = \langle p_i^* \rangle_{i \in I} : \ncat{Sh}(\cat{C}) \to \prod_{i \in I} \ncat{Set} \cong \ncat{Set}^I,
\end{equation*}
induced by the left adjoints $p_i^*: \ncat{Sh}(\cat{C}) \to \ncat{Set}$ of each of the points. Thus $P$ is a set of enough points for $(\cat{C}, j)$ if and only if $p^*$ reflects isomorphisms. 

Note that $p^*$ preserves finite limits since each $p_i^*$ does. Furthermore, $p^*$ has a right adjoint $p_* : \ncat{Set}^I \to \ncat{Sh}(\cat{C}, j)$ defined for $X \in \ncat{Set}^I$ by
\begin{equation*}
    p_*(X) = \prod_{i \in I} (p_i)_*(X_i).
\end{equation*}

In other words, a set of points $P = \{ p_i \}$ on a site $(\cat{C}, j)$ induces a geometric morphism $p_* : \ncat{Set}^I \to \ncat{Sh}(\cat{C}, j)$.

The following result provides a convenient criterion for such a functor to reflect isomorphisms.

\begin{Lemma}[{\cite[Lemma 3.26]{jardine2015local}}] \label{lem jardine's characterization of points}
Let $f_*: \cat{E} \to \cat{F}$ be a geometric morphism of Grothendieck topoi. Then the following are equivalent:
\begin{enumerate}
    \item The left adjoint $f^*: \cat{F} \to \cat{E}$ is faithful,
    \item $f^*$ reflects isomorphisms,
    \item $f^*$ reflects epimorphisms,
    \item $f^*$ reflects monomorphisms.
\end{enumerate}
\end{Lemma}

\begin{proof}
$(1) \Rightarrow (4)$ Suppose that $f^*$ is faithful, and $g: X \to Y$ is a morphism in $\cat{F}$ such that $f^*(g)$ is a monomorphism. If $h,h': Y \to Z$ are maps such that $h g = h' g$, then $f^*(hg) = f^*(h) f^*(g) = f^*(h') f^*(g) = f^*(h'g)$ implies that $f^*(h) = f^*(h)'$, and since $f^*$ is faithful, $h = h'$. 

A similar argument shows $f^*$ reflects epimorphisms and therefore also isomorphisms. Thus $(1) \Rightarrow (3)$ and $(1) \Rightarrow (2)$.

$(3) \Rightarrow (1)$ Suppose that $f^*$ reflects epimorphisms, and suppose that $g,g': Y \to Z$ are maps such that $f^*(g) = f^* (g')$. Now in any category with equalizers, two maps $g, g'$ are equal if and only if for an equalizer diagram
$$\begin{tikzcd}
	X & Y & Z
	\arrow["{g'}"', shift right=1, from=1-2, to=1-3]
	\arrow["g", shift left=1, from=1-2, to=1-3]
	\arrow["{h}", from=1-1, to=1-2]
\end{tikzcd}$$
the limit map $h$ is an epimorphism. Since $f^*$ preserves finite limits, $f^*(h)$ is an equalizer of $f^*(g)$ and $f^*(g')$. Thus $f^*(h)$ is an epimorphism in $\cat{E}$. But $f^*$ reflects epimorphisms, thus $g = g'$. A similar argument proves $(4) \Rightarrow (1)$ and thus $(2) \Rightarrow (1)$.
\end{proof}

If $P = \{ p_i \}$ is a set of enough points for a site $(\cat{C}, j)$, and $f : X \to Y$ is a map of sheaves, then the following are equivalent:
\begin{enumerate}
    \item $p^*(f) = \langle p^*_i \rangle(f)$ is an isomorphism,
    \item $p^*_i(f)$ is an isomorphism of sets for each $i \in I$, and
    \item $f$ is an isomorphism of sheaves.
\end{enumerate}

It will be useful to characterize when a set of points $\{p_i \}$ is a set of enough points using the actual functors $p_i : \Pre(\cat{C}) \to \ncat{Set}$. For this we must restrict the coverages we consider.

Now suppose that $(\cat{C}, j)$ is a subcanonical site (Definition \ref{def subcanonical site}, so that the Yoneda embedding factors through $\Sh(\cat{C}, j)$, and suppose that $P = \{ p_i \}$ is a set of points of $(\cat{C}, j)$. If $P$ is a set of enough points, then the induced functor
\begin{equation*}
    p^* : \ncat{Sh}(\cat{C}, j) \to \ncat{Set}^I
\end{equation*}
reflects isomorphisms. Since $y : \cat{C} \hookrightarrow \ncat{Sh}(\cat{C}, j)$ is fully faithful, it also reflects isomorphisms, so this implies that the composite functor
\begin{equation*}
    \cat{C} \xhookrightarrow{y} \ncat{Sh}(\cat{C}, j) \xrightarrow{p^*} \ncat{Set}^I
\end{equation*}
reflects isomorphisms. Let us denote this composite functor by $p: \cat{C} \to \ncat{Set}^I$, which is defined objectwise by
\begin{equation*}
    p(U) = ( p_i(U) )_{i \in I}.
\end{equation*}
We've shown that if $P$ is a set of enough points, then $p$ reflects isomorphisms. However the converse is also true. If $p$ reflects isomorphisms, then since every sheaf $X \in \ncat{Sh}(\cat{C}, j)$ is a colimit of representables and $p^*$ is a left adjoint, this implies that $p^*$ will reflect isomorphisms. Thus we have proven the following result.

\begin{Lemma} \label{lem characterizatino for when set of points is enough}
If $(\cat{C}, j)$ is a subcanonical site, and $P = \{ p_i \}_{i \in I}$ is a set of points of $(\cat{C}, j)$, then $P$ is a set of enough points for $(\cat{C}, j)$ if and only if the functor $p : \cat{C} \to \ncat{Set}^I$ given by the product $\langle p_i \rangle_{i \in I}$, reflects isomorphisms.
\end{Lemma}

\begin{Lemma} \label{lem maps of sheaves mono/epi/iso iff on points}
Given a site $(\cat{C}, j)$, and a set $P = \{ p_i \}_{i \in I}$ of enough points for $(\cat{C}, j)$, a map $f: X \to Y$ of sheaves on $(\cat{C}, j)$ is an epi/monomorphism of sheaves if and only if $p_i^* (f)$ is a epi/monomorphism of sets for all $i \in I$.
\end{Lemma}

\begin{proof}
$(\Rightarrow)$ Since each $p_i^*$ preserves finite colimits and limits, they each preserves epimorphisms and monomorphisms.

$(\Leftarrow)$ This follows from Lemma \ref{lem jardine's characterization of points}.
\end{proof}

\begin{Cor} \label{cor maps are local epi/mono iff epi/mono on points}
Given a site $(\cat{C}, j)$ with a set $P = \{p_i \}_{i \in I}$ of enough points, and a map $f: X \to Y$ of presheaves on $\cat{C}$, the following are equivalent:
\begin{enumerate}
    \item $f$ is a $j$-local epi/monomorphism of presheaves,
    \item $p_i^*(af)$ is an epi/monomorphism of sets for every $i \in I$,
\end{enumerate}
\end{Cor}

\begin{proof}
$(1) \Rightarrow (2)$ is Corollary \ref{cor local mono implies mono on points}.

$(2) \Rightarrow (1)$ follows from Lemma \ref{lem maps of sheaves mono/epi/iso iff on points} 
 and Corollary \ref{cor j-local iso/epi/mono <-> sheafification is iso/epi/mono}.
\end{proof}

\subsection{Examples}

\subsubsection{Points of the site of a topological space} \label{section site of a topological space has enough points} 

In this section, we will show that if $X$ is a topological space, then the site $(\mathcal{O}(X), j_X)$ as introduced in Example \ref{ex open cover coverage}, is a site with enough points.

For every $x \in X$ consider the full subcategory $\mathcal{O}(X,x)$ of $\mathcal{O}(X)$ whose objects are those open subsets containing $x$. This category is finitely cofiltered. Indeed, $X \in \mathcal{O}(X,x)$, so $\mathcal{O}(X,x)$ is nonempty. Suppose $U,V \in \mathcal{O}(X,x)$, then $x \in U \cap V$ and $U \cap V \leq U$, $U \cap V \leq V$, and since $\mathcal{O}(X, x)$ is a poset, this is sufficient to be finitely cofiltered. Now given $F \in \ncat{Sh}(X) = \ncat{Sh}(\mathcal{O}(X), j_X)$, consider the set:
$$ p_{x}^*F = \ncolim{U \in \mathcal{O}(X,x)^{\op}} \, F(U).$$
In other words, $p_{x}^*F$ is the set of equivalence classes of local sections $s \in F(U)$ where $x \in U$ and where if $x \in V$ and $t \in F(V)$, then $s \sim t$ if there exists a $W \subseteq U \cap V$ with $x \in W$ such that $s|_W = t|_W$. In other words, $p_x^*F$ computes the set of \textbf{germs} of $F$ at $x \in X$. 

It is easy to see that this extends to a functor $p_x^*: \ncat{Sh}(X) \to \ncat{Set}$ by the functoriality of colimits. Now suppose that $d: I \to \ncat{Sh}(X)$ is a finite diagram. Then we obtain a diagram $\widetilde{d} : \mathcal{O}(X, x)^\op \times I \to \ncat{Set}$ with $\widetilde{d}(U, i) = d(i)(U)$. Since $\mathcal{O}(X, x)^\op$ is finitely filtered, by Proposition \ref{prop filtered colimits commute with finite limits in Set}, taking the $I$-limit and $\mathcal{O}(X,x)^\op$-colimit of $\widetilde{d}$ is commutative, and therefore $p_x^*$ preserves finite limits in $\ncat{Sh}(X)$. It is easy to check directly that $p_x^*$ has a right adjoint given using the usual skyscraper sheaf construction, and hence defines a point $p_{x,*} : \ncat{Set} \to \ncat{Sh}(X)$.

Now if $U \in \mathcal{O}(X)$, then it is not hard to see that $y(U)$ is a $j_X$-sheaf. So we can compute the $j_X$-continuous, $\ncat{Set}$-flat functor $\mathcal{O}(X) \to \ncat{Set}$ corresponding to $p_x^*$ by the composite functor
\begin{equation*}
    \mathcal{O}(X) \xhookrightarrow{y} \ncat{Sh}(X) \xrightarrow{p_x^*} \ncat{Set}.
\end{equation*}
So $p_x: \mathcal{O}(X) \to \ncat{Set}$ sends an open subset $V \subseteq X$ to the set 
$$p_x^*(y(V)) \cong \ncolim{U \in \mathcal{O}(X,x)^\op} y(V)(U).$$
Now $y(V)(U)$ is empty unless $U \subseteq V$, in which case its a singleton. In other words, the functor $p_x : \mathcal{O}(X) \to \ncat{Set}$ sends an open subset $V \subseteq X$ to the singleton set if $x \in V,$ and to the empty set otherwise.

Given an element $s \in F(U)$, with $x \in U$, denote the image of $s$ in $p_x^*F$ by $[s]_x$. If $f: F \to G$ is a map of sheaves on $X$, then it is easy to see that $p_x^*f : p_x^*F \to p_x^*G$ is given by $p_x^*f[s]_x = [f(s)]_x$.

\begin{Lemma} \label{lem points of a top space are faithful}
Given a topological space $X$, the map
$$\langle p_x^* \rangle : \ncat{Sh}(X) \to \prod_{x \in X} \ncat{Set}$$
constructed from the geometric morphisms arising from every point $x \in X$ is a faithful functor.
\end{Lemma}

\begin{proof}
Suppose that $f,g: F \to G$ are maps of sheaves on $X$ such that $p_x^*f = p_x^* g$ for all $x \in X$. Then if $s \in F(U)$, and $x \in U$, then we have $(p_x^*f)[s]_x = [f(s)]_x = [g(s)]_x = (p_x^*g)[s]_x$. Thus $f(s)|_{U_x} = g(s)|_{U_x}$ for some $U_x \subseteq U$ with $x \in U_x$. Thus if we vary over all $x \in U$, the collection $\{U_x \}$ forms an open cover of $U$. Now $F|_{\mathcal{O}(U)}$ and $G|_{\mathcal{O}(U)}$ are sheaves on $U$, and $\{ f(s)|_{U_x} = g(s)|_{U_x} \}_{x \in U}$ forms a matching family on $U$. Thus $f(s)$ and $g(s)$ are both amalgamations of this matching family, and since amalgamations are unique for sheaves, $f(s) = g(s)$.
\end{proof}

\begin{Cor} \label{Cor sheaves on a topological space have enough points}
Given a topological space $X$, the set $\{ p_x^* \}_{x \in X}$ is a set of enough points for $\ncat{Sh}(X)$.
\end{Cor}

\begin{proof}
This follows from Lemma \ref{lem points of a top space are faithful} and Lemma \ref{lem jardine's characterization of points}.
\end{proof}

An alternative way we could have proven the above result is by showing that the functor $\langle p_x \rangle_{x \in X} : \mathcal{O}(X) \to \ncat{Set}$ reflects isomorphisms, by Lemma \ref{lem characterizatino for when set of points is enough}. The only way that $\langle p_x \rangle_{x \in X}(U \subseteq V)$ is an isomorphism is if $U$ and $V$ contain precisely the same points, i.e. $U = V$. Thus $\langle p_x \rangle_{x \in X}$ reflects isomorphisms, and so $\{ p_x \}_{x \in X}$ is a set of enough points for $\mathcal{O}(X)$.

\subsubsection{$(\ncat{Set}, j_{\text{epi}})$ has enough points}

Recall the coverage $j_{\text{epi}}$ on $\ncat{Set}$ given by jointly epimorphic families from Example \ref{ex set joint epi coverage}. This site is subcanonical, and the identity functor $1_{\ncat{Set}} : \ncat{Set} \to \ncat{Set}$ is clearly $j_{\text{epi}}$-continuous and $\ncat{Set}$-flat, and therefore defines a point. Since $1_{\ncat{Set}}$ reflects isomorphisms, by Lemma \ref{lem characterizatino for when set of points is enough}, $\{ 1_{\ncat{Set}} \}$ is a set of enough points for $(\ncat{Set}, j_{\text{epi}})$.

\subsubsection{$(\ncat{Man}, j_\text{open})$ has enough points}

Consider the category $\ncat{Man}$ of finite dimensional smooth manifolds, equipped with the open cover coverage $j_{\text{open}}$ of Example \ref{ex open cover coverage}. We wish to show that this site has enough points. Let $p_n : \ncat{Man} \to \ncat{Set}$ denote the functor defined objectwise by
\begin{equation*}
    p_n(M) = \ncolim{r \to \infty} \, \ncat{Man}(B^n(1/r), M),
\end{equation*}
where $B^n(1/r)$ is the $n$-dimensional ball of radius $1/r$, and for $r \leq r'$ the map $B^n(1/r') \to B^n(1/r)$ is just inclusion. The idea here being that we want to probe manifolds with vanishingly small $n$-dimensional balls. We want to show that $\{ p_n \}_{n \geq 0}$ is a set of enough points for $(\ncat{Man}, j_{\text{open}})$. Consider the corresponding functor $p_n^* : \ncat{Sh}(\ncat{Man}, j_{\text{open}}) \to \ncat{Set}$ given by
\begin{equation*}
\begin{aligned}
    p_n^*(X) = X \otimes_{\ncat{Man}} p_n & \cong \int^{y(M) \to X} X(M) \times p_n(M) \\
    & \cong \int^{y(M) \to X} X(M) \times \left( \ncolim{r \to \infty} \ncat{Man}(B^n(1/r), M) \right) \\
    & \cong \ncolim{r \to \infty} \int^{y(M) \to X} X(M) \times y(M)(B^n(1/r)) \\
    & \cong \ncolim{r \to \infty} X(B^n(1/r)).
\end{aligned}
\end{equation*}
Now since the colimit is obviously filtered, it is easy to see that $p_n^*$ preserves finite limits. It also clearly has a right adjoint using the skyscraper construction. Thus we know that each $p_n$ is $\ncat{Set}$-flat and $j_{\text{open}}$-continuous. We now only need to show that $\{ p_n \}_{n \geq 0}$ jointly reflect isomorphisms. 

To do this, we follow Schreiber's proof \cite[Proposition 4.3.1.7]{schreiber2013dcct} and introduce an auxillary site for each manifold $M \in \ncat{Man}$. Let $\ncat{Man}^\text{open}_{/ M}$ denote the full subcategory of $\ncat{Man}_{/M}$ on those maps $f : U \to M$ that are open embeddings. Let $\pi : \ncat{Man}^{\text{open}}_{/ M} \to \mathcal{O}(M)$ be the functor that takes the image of the open embeddings. Let $i : \ncat{Man}_{/M}^\text{open} \to \ncat{Man}_{/M}$ be the inclusion functor. Equip $\ncat{Man}_{/M}^\text{open}$ with the restriction $j^\text{open}_{/M}$ of the coverage $j_{/M}$. Let us show that $i : (\ncat{Man}_{/M}^\text{open}, j^\text{open}_{/M}) \to (\ncat{Man}_{/M}, j_{/M})$ is a dense morphism of sites. Since $i$ is the inclusion of a full subcategory, conditions (D3) and (D4) of Definition \ref{def site dense functor} hold automatically. We need only check (D1) and (D2), but it is easy to see that these hold in this example. Thus by Theorem \ref{th comparison lemma}, $i$ induces an equivalence of sheaf topoi.

For each point $x \in M$ and $r > 0$, let $\varphi_x : B^n_x(1/r) \hookrightarrow M$ denote a fixed open embedding of the $n$-dimensional ball of radius $1/r$ into $M$ such that $\varphi_x(0) = x$. Consider the functor $p_{n,x}^* : \ncat{Sh}(\ncat{Man}^\text{open}_{/M}, j^\text{open}_{/M}) \to \ncat{Set}$ defined objectwise by
\begin{equation*}
    p_{n,x}^*(X) = \ncolim{r \to \infty} X(B^n_x(1/r)).
\end{equation*}
It is easy to see that each $p^*_{n,x}$ is a left adjoint that preserves finite limits. By practically the same argument as in Lemma \ref{lem points of a top space are faithful}, it is not hard to see that $\{ p^*_{n,x} \}_{n \geq 0, x \in M}$ defines a set of enough points for $(\ncat{Man}^\text{open}_{/M}, j^\text{open}_{/M})$, equivalently $(\ncat{Man}_{/M}, j_{/M})$. By Lemma \ref{lem slice site is comorphism of sites}, the projection map $\pi_{/M} : (\ncat{Man}_{/M}, j_{/M}) \to (\ncat{Man}, j_{\text{open}})$ is a comorphism of sites, and furthermore $\Delta_{\pi_{/M}}$ sends sheaves to sheaves.

So if $X \in \ncat{Sh}(\ncat{Man}, j_{\text{open}})$, then $\Delta_{\pi_{/M}}(X) \in \Sh(\ncat{Man}_{/M}, j_{/M})$ and clearly 
\begin{equation*}
    p_{n,x}^*(\Delta_{\pi_{/M}}(X)) \cong p_{n,x'}^*(\Delta_{\pi_{/M}}(X)) \cong p^*_n(X),
\end{equation*}
for all pairs of points $x, x' \in M$, simply because $B^n_{x}(1/r) \cong B^n_{x'}(1/r)$.

So if $f: X \to Y$ is a map of sheaves on $(\ncat{Man}, j_{\text{open}})$ such that $p^*_{n,x}(\Delta_{\pi_{/M}})(f)$ is an isomorphism for all $n \geq 0$ and $x \in M$, then $f$ is an isomorphism. But this is equivalent to $p_n^*(f)$ being an isomorphism for every $n \geq 0$. Therefore $\{p_n^* \}$ is also a set of enough points for $(\ncat{Man}, j_{\text{open}})$.

\sloppy
\begin{Rem}
In fact, Schreiber proves \cite[Proposition 4.3.1.7]{schreiber2013dcct} that the functor 
$p^*_\infty: \Sh(\ncat{Man}, j_{\text{open}}) \to \ncat{Set}$ defined objectwise by
\begin{equation*}
    p_\infty^*(X) = \ncolim{n \to \infty} \ncolim{r \to \infty} X(B^n(1/r))
\end{equation*}
by itself is a set of enough points for $(\ncat{Man}, j_{\text{open}})$.
\end{Rem}

\section{Giraud's Theorem} \label{section girauds theorem}
In this section, we will prove Giraud's Theorem, which characterizes Grothendieck toposes entirely by exactness properties. We will also discuss how Giraud's axioms are equivalent to what Rezk calls \textbf{weak descent} \cite[Section 2]{rezk2010toposes}. We end this section with a discussion of what satisfying descent means, and how $\infty$-toposes repair this ``weakness'' of $1$-categorical Grothendieck toposes.

More precisely, we will be focused on proving the following meta-theorem:
Let $\cat{E}$ be a category, then:
$$(D)' \implies (G) \implies (T) \implies (D)'$$
where
\begin{enumerate}
	\item $(D)'$ is the statement that $\cat{E}$ is locally presentable and satisfies weak descent,
	\item $(G)$ is the statement that $\cat{E}$ is a Giraud category and
	\item $(T)$ is the statement that $\cat{E}$ is a Grothendieck topos.
\end{enumerate}
Thus this will prove that $(D)' \Leftrightarrow (G) \Leftrightarrow (T)$.
This will take a good deal of work, and in places we refer to the literature, but it is the basis for the theory of Grothendieck toposes. Namely giving three equivalent definitions for what a Grothendieck topos is. $(D)'$ provides a perspective that generalizes most efficiently to model topoi and $\infty$-topoi, and makes explicit the difference between Grothendieck topoi and model topoi, while $(G)$ is an internal perspective, it gives a characterization of Grothendieck topoi based on their internal categorical structure.

\subsection{Giraud's Axioms}

In this section we briefly introduce Giraud's axioms. These are a list of conditions whose conjunction is necessary and sufficient for a category to be equivalent to a Grothendieck topos. The first two are easy to state, while the third will need some preliminary definitions.

Let $\cat{E}$ be a locally presentable category:

\begin{itemize}
	\item $(G1):$ (Disjoint coproducts) For any objects $X,Y \in \cat{E}$, the following diagram is a pullback:
\begin{equation}
	\begin{tikzcd}
	\varnothing & Y \\
	X & {X + Y}
	\arrow[from=1-1, to=2-1]
	\arrow[from=1-1, to=1-2]
	\arrow[from=2-1, to=2-2]
	\arrow[from=1-2, to=2-2]
	\arrow["\lrcorner"{anchor=center, pos=0.125}, draw=none, from=1-1, to=2-2]
\end{tikzcd}
\end{equation}
\item $(G2):$ (Universal colimits) Given a morphism $f: X \to Y$ in $\cat{E}$, the pullback functor
$$f^*: \cat{E}_{/Y} \to \cat{E}_{/X}$$
preserves small colimits,
\item $(G3):$ (Effective equivalence relations) equivalence relations in $\cat{E}$ are effective.
\end{itemize}

\begin{Def} \label{def giraud's axioms}
We refer to the conditions $(G1) - (G3)$ as \textbf{Giraud's Axioms}. We call a category $\cat{E}$ a \textbf{Giraud category} if it is locally presentable and satisfies Giraud's axioms. Let $(G)$ be the statement $\cat{E}$ is a Giraud category.
\end{Def}

Now we will explain what we mean by $(G3)$. Because equivalence relations are important in what follows, we dedicate a small section to them.

\subsubsection{Equivalence Relations}

\begin{Def} \label{def equivalence relation}
Let $\cat{C}$ be a category with finite limits and $X \in \cat{C}$ an object. An \textbf{equivalence relation} on $X$ is a pair of morphisms $s,t: R \to X$ such that:
\begin{enumerate}
	\item (Relation) the map $R \xhookrightarrow{(s,t)} X \times X$ is a monomorphism,
	\item (Reflexivity) The diagonal factors through $R$: 
\begin{equation*}
	\begin{tikzcd}
	X && {X \times X} \\
	& R
	\arrow["{\Delta_X}", from=1-1, to=1-3]
	\arrow[hook', from=2-2, to=1-3]
	\arrow["r"', from=1-1, to=2-2]
\end{tikzcd}
\end{equation*}
\item (Symmetry) There exists a map $i: R \to R$ such that $s = t i$ and $t = s i$,
\item (Transitivity) There exists a map $c: R \times_X R \to R$ such that if 
\begin{equation*}
	\begin{tikzcd}
	{R \times_X R} & R \\
	R & X
	\arrow["{s}"', from=2-1, to=2-2]
	\arrow["{q_1}"', from=1-1, to=2-1]
	\arrow["{t}", from=1-2, to=2-2]
	\arrow["{q_2}", from=1-1, to=1-2]
\end{tikzcd}
\end{equation*}
is a pullback, then the following diagram commutes:
\begin{equation*}
\begin{tikzcd}
	{R \times_X R} && {X \times X} \\
	& R
	\arrow["{{(s q_2, t q_1)}}", from=1-1, to=1-3]
	\arrow["c"', from=1-1, to=2-2]
	\arrow["{{(s,t)}}"', hook, from=2-2, to=1-3]
\end{tikzcd}
\end{equation*}
\end{enumerate}
\end{Def}

The definition of an equivalence relation is a bit unwieldy, so it will be helpful to obtain another method to manipulate them. By a relation $R \subseteq X \times X$ on an object $X$ in a category $\cat{C}$, we mean a pair of maps $s,t : R \to X$ such that $R \xhookrightarrow{(s,t)} X \times X$ is a monomorphism.

Given a relation $R \subseteq X \times X$, we obtain a monomorphism of presheaves $y(R) \hookrightarrow y(X \times X) \cong y(X) \times y(X)$.

\begin{Lemma} \label{lem equiv def for equivalence relation}
Given a finitely complete category $\cat{C}$, a relation $R \subseteq X \times X$ is an equivalence relation on $X$ if and only if for every $U \in \cat{C}$, the relation of sets
\begin{equation*}
    y(R)(U) = \cat{C}(U,R) \subseteq \cat{C}(U,X)^2 \cong \cat{C}(U, X \times X) = y(X \times X)(U)
\end{equation*}
is an equivalence relation in the classical sense.
\end{Lemma}

\begin{proof}
$(\Rightarrow)$ Given an equivalence relation $R \subseteq X \times X$, we want to show that for every $U \in \cat{C}$, the relation $\sim$ on morphisms $f, g: U \to X$ where $f \sim g$ if the map $(f,g) : U \to X \times X$ factors through $R$, is an equivalence relation.

It is now easy to read off the properties of $\sim$ being an equivalence relation diagrammatically from Definition \ref{def equivalence relation}. For example, if $f : U \to X$, then $f \sim f$, since the following diagram commutes by Definition \ref{def equivalence relation}.(1)
\begin{equation*}
 \begin{tikzcd}
	U & X & {X \times X} \\
	&& R
	\arrow["f", from=1-1, to=1-2]
	\arrow["{(f,f)}", curve={height=-24pt}, from=1-1, to=1-3]
	\arrow["{\Delta_X}", from=1-2, to=1-3]
	\arrow["r"', from=1-2, to=2-3]
	\arrow[hook, from=2-3, to=1-3]
\end{tikzcd}   
\end{equation*}
The other conditions follow from similar arguments.

$(\Leftarrow)$ Suppose that $R \subseteq X \times X$ is a relation, and suppose that for every $U \in \cat{C}$, the relation $\sim$ on morphisms $U \to X$ introduced above is an equivalence relation.

Since $\sim$ is reflexive, this means that $1_X \sim 1_X$. In other words, the diagonal map $(1_X, 1_X) : X \to X \times X$ factors through $R$. This defines the map $r : X \to R$ satisfying Definition \ref{def equivalence relation}.(2).

Since $\sim$ is symmetric, that means that the map $(t,s) : R \to X \times X$ factors through $(s,t) : R \hookrightarrow X \times X$. This defines the map $i: R \to R$ satisfying Definition \ref{def equivalence relation}.(3).

Now first note that $s q_2 \sim s q_1$, because $s q_1 = t q_2$, so the following diagram commutes
\begin{equation*}
    \begin{tikzcd}
	{R \times_X R} && {X \times X} \\
	& R
	\arrow["{(sq_2,sq_1)}", from=1-1, to=1-3]
	\arrow["{q_2}"', from=1-1, to=2-2]
	\arrow["{(s,t)}"', hook, from=2-2, to=1-3]
\end{tikzcd}
\end{equation*}
Similarly $tq_2 \sim t q_1$, so the following diagram commutes
\begin{equation*}
    \begin{tikzcd}
	{R \times_X R} && {X \times X} \\
	& R
	\arrow["{(tq_2, tq_1)}", from=1-1, to=1-3]
	\arrow["{q_1}"', from=1-1, to=2-2]
	\arrow["{(s,t)}"', hook, from=2-2, to=1-3]
\end{tikzcd}
\end{equation*}
But since $\sim$ is transitive, this means that $s q_2 \sim t q_1$, so there is a map $c$ satisfying Definition \ref{def equivalence relation}.(4).
\end{proof}

\begin{Ex} \label{ex kernel pair as equiv relation}
Recall the notion of kernel pair from Definition \ref{def kernel pair}. Every kernel pair provides an equivalence relation. Indeed, given $f : X \to Y$ with kernel pair $p_0, p_1 : X \times_Y X \to X$, then for maps $g, h: U \to X$, we have $g \sim h$ if and only if $fg = fh$. This is easily seen to be an equivalence relation.
\end{Ex}

\begin{Ex} \label{ex diagonal as equiv relation}
Given any $X \in \cat{C}$, the diagonal map $\Delta_X : X \to X \times X$ is an equivalence relation. Indeed for maps $f, g: U \to X$, $f \sim g$ if and only if $f = g$.
\end{Ex}

\begin{Lemma}
Given equivalence relations $R \subseteq X \times X$ and $S \subseteq X \times X$, on $X$, the union subobject $R \cup S \subseteq X \times X$ is also an equivalence relation on $X$.
\end{Lemma}

\begin{Ex} \label{ex smallest equiv relation in special case}
Given any monomorphisms $f, g : A \hookrightarrow B$ in $\cat{C}$, the subobject $A \xhookrightarrow{(f,g)} B \times B$ is a relation. Suppose further that $A$ satisfies Definition \ref{def equivalence relation}.(4), i.e. it is transitive. Let $(g,f)$ denote the subobject $A \xhookrightarrow{(g,f)} B \times B$, and let $(f,g)^*$ denote the subobject
\begin{equation*}
    (f,g)^* = \Delta_B \cup (f,g) \cup (g,f).
\end{equation*}
Then $(f,g)^*$ is an equivalence relation. This is really a special case of constructing the smallest equivalence relation containing $(f,g)$, see \cite{Henryequivclosure2024}.
\end{Ex}

\begin{Def}
Recall the notion of quotient object from Definition \ref{def quotient object}. Let $X$ be an object in a finitely bicomplete category $\cat{C}$, and let $R$ be an equivalence relation on $X$. Let $X / R$ denote the coequalizer in $\cat{C}$:
$$ R \overset{s}{\underset{t}{\rightrightarrows}} X \xrightarrow{q} X/R$$
i.e $X / R$ is the quotient object of $X$ by $R$.

An equivalence relation $R$ on an object $X$ is said to be \textbf{effective} if the following diagram is a pullback:
\begin{equation*}
	\begin{tikzcd}
	R & X \\
	X & {X/R}
	\arrow["{s}"', from=1-1, to=2-1]
	\arrow["{t}", from=1-1, to=1-2]
	\arrow["q"', from=2-1, to=2-2]
	\arrow["q", from=1-2, to=2-2]
	\arrow["\lrcorner"{anchor=center, pos=0.125}, draw=none, from=1-1, to=2-2]
\end{tikzcd}
\end{equation*}
\end{Def}

\subsubsection{Technical Results}

Now let us introduce a couple of technical results about categories satisfying Giraud's axioms (G) that will be helpful in what follows. We owe much to Simon Henry and Michael Shulman's guidance in \cite{Henrybalanced2024}. For the next result recall the definition of intersections and unions of subobjects from Lemmas \ref{lem meet of subobjects} and \ref{lem join of subobjects}.

\begin{Lemma}[{\cite[Lemma A.1.4.8]{johnstone2002sketches}}] \label{lem technical johnstone lemma pullback squares}
Given a Giraud category $\cat{E}$ (Definition \ref{def giraud's axioms}), monomorphisms $f : A \hookrightarrow B$, $g: A \hookrightarrow C$ and a pushout square
\begin{equation*}
    \begin{tikzcd}
	A & C \\
	B & P
	\arrow["g", hook', from=1-1, to=1-2]
	\arrow["f"', hook, from=1-1, to=2-1]
	\arrow["k", from=1-2, to=2-2]
	\arrow["h"', from=2-1, to=2-2]
	\arrow["\lrcorner"{anchor=center, pos=0.125, rotate=180}, draw=none, from=2-2, to=1-1]
\end{tikzcd}
\end{equation*}
then the morphisms $h$ and $k$ are monomorphisms. Furthermore the square is a pullback square.
\end{Lemma}

\begin{proof}
Let $D = B + C$, $b : B \hookrightarrow B + C$ and $c : C \hookrightarrow B + C$ be the inclusion maps. Let $R = B + A + A + C$. Following \cite[Lemma A.1.4.8]{johnstone2002sketches} we will write this as $R = B + A_1 + A_2 + C$ to help keep track of which $A$ we are talking about in the course of the proof. Let $r$ and $r'$ denote the morphisms
\begin{equation*}
    \begin{tikzcd}
	{(B + C) + A_1 + A_2} && {B + C}
	\arrow["{r = (1_{B + C}, bf, cg)}", shift left=2, from=1-1, to=1-3]
	\arrow["{r' = (1_{B + C}, cg, bf}"', shift right=2, from=1-1, to=1-3]
\end{tikzcd}
\end{equation*}
Now the relation $A \xhookrightarrow{(bf, cg)} (B + C) \times (B + C)$ is transitive. Indeed if $p,q,r : U \to (B + C)$ are maps such that $p \sim q$ and $q \sim r$, then there exist maps $\ell, \ell' : U \to A$ such that $p = bf \ell$, $q = cg \ell$, $q = bf \ell'$ and $r = cg \ell'$. In other words we get an induced map
\begin{equation*}
    \begin{tikzcd}
	U && A \\
	& \varnothing & C \\
	A & B & {B + C}
	\arrow["{\ell'}", from=1-1, to=1-3]
	\arrow[dashed, from=1-1, to=2-2]
	\arrow["\ell"', from=1-1, to=3-1]
	\arrow["g", from=1-3, to=2-3]
	\arrow[from=2-2, to=2-3]
	\arrow[from=2-2, to=3-2]
	\arrow["\lrcorner"{anchor=center, pos=0.125}, draw=none, from=2-2, to=3-3]
	\arrow["c", from=2-3, to=3-3]
	\arrow["f"', from=3-1, to=3-2]
	\arrow["b"', from=3-2, to=3-3]
\end{tikzcd}
\end{equation*}
where $\varnothing$ is initial by $\cat{E}$ having disjoint coproducts (G1). Thus $U \cong \varnothing$, and it then follows that $p \sim r$. Thus $R \xhookrightarrow{(r,r')} D$ is the relation $(bf,cg)^*$ from Example \ref{ex smallest equiv relation in special case} and hence an equivalence relation.

Now let $P$ be the coequalizer
\begin{equation*}
    \begin{tikzcd}
	R & B+C & P
	\arrow["r", shift left=2, from=1-1, to=1-2]
	\arrow["{r'}"', shift right=2, from=1-1, to=1-2]
	\arrow["q", from=1-2, to=1-3]
\end{tikzcd}
\end{equation*}
and let us consider the following commutative diagram 
\begin{equation*}
\begin{tikzcd}
	{A_1} & {B + A_1} & B \\
	{A_1 + C} & R & {B + C} \\
	C & {B+C} & P
	\arrow[hook, from=1-1, to=1-2]
	\arrow[hook', from=1-1, to=2-1]
	\arrow["{(1_B, f)}", from=1-2, to=1-3]
	\arrow[hook', from=1-2, to=2-2]
	\arrow["b", hook', from=1-3, to=2-3]
	\arrow[hook, from=2-1, to=2-2]
	\arrow["{(g, 1_C)}"', from=2-1, to=3-1]
	\arrow["{r'}", from=2-2, to=2-3]
	\arrow["r", from=2-2, to=3-2]
	\arrow["q", from=2-3, to=3-3]
	\arrow["c"', hook, from=3-1, to=3-2]
	\arrow["q"', from=3-2, to=3-3]
\end{tikzcd}
\end{equation*}
where the unlabelled maps are inclusions. The bottom right hand square is a pullback by (G3), i.e. because $R$ is an equivalence relation and therefore effective. All of the other squares can be verified to be pullbacks as well.

Now we wish to show that $qc$ and $qb$ are monomorphisms. So consider the commutative diagram
\begin{equation*}
    \begin{tikzcd}
	B & {B + A_1} & B \\
	{B + A_2} & R & {B + C} \\
	B & {B+C} & P
	\arrow[hook, from=1-1, to=1-2]
	\arrow[hook', from=1-1, to=2-1]
	\arrow["{(1_B, f)}", from=1-2, to=1-3]
	\arrow[hook', from=1-2, to=2-2]
	\arrow["b", hook', from=1-3, to=2-3]
	\arrow[hook, from=2-1, to=2-2]
	\arrow["{(1_B, f)}"', from=2-1, to=3-1]
	\arrow["{r'}", from=2-2, to=2-3]
	\arrow["r", from=2-2, to=3-2]
	\arrow["q", from=2-3, to=3-3]
	\arrow["b"', hook, from=3-1, to=3-2]
	\arrow["q"', from=3-2, to=3-3]
\end{tikzcd}
\end{equation*}
It can again be checked that each square is a pullback, so that the whole outer square is a pullback. But this is the same thing as saying that the square
\begin{equation*}
    \begin{tikzcd}
	B & B \\
	B & P
	\arrow["{1_B}", from=1-1, to=1-2]
	\arrow["{1_B}"', from=1-1, to=2-1]
	\arrow["qb", from=1-2, to=2-2]
	\arrow["qb"', from=2-1, to=2-2]
\end{tikzcd}
\end{equation*}
is a pullback, which is equivalent to $qb$ being a monomorphism. Using the same trick again proves that $qc$ is a monomorphism as well.

Now notice that this also shows that the intersection $B \cap C$ of the subobjects $B \xhookrightarrow{qb} P$ and $C \xhookrightarrow{qc} P$ is $A$. Thus it is then not hard to see that the outermost square is also a pushout.
\end{proof}

\begin{Lemma}[{\cite[Lemma A.1.4.9]{johnstone2002sketches}}] \label{lem giraud categories monos are effective}
Given a Giraud category $\cat{E}$, every monomorphism in $\cat{E}$ is an effective monomorphism (Definition \ref{def effective mono}). 
\end{Lemma}

\begin{proof}
Given a monomorphism $f: X \to Y$ in $\cat{E}$, by Lemma \ref{lem technical johnstone lemma pullback squares}, the cokernel pair (Definition \ref{def cokernel pair})
\begin{equation*}
\begin{tikzcd}
	X & Y \\
	Y & {Y +_X Y}
	\arrow["f", from=1-1, to=1-2]
	\arrow["f"', from=1-1, to=2-1]
	\arrow["{i_1}", from=1-2, to=2-2]
	\arrow["{i_0}"', from=2-1, to=2-2]
	\arrow["\lrcorner"{anchor=center, pos=0.125, rotate=180}, draw=none, from=2-2, to=1-1]
\end{tikzcd}
\end{equation*}
is also a pullback, which implies that 
\begin{equation*}
\begin{tikzcd}
	X & Y & {Y +_X Y}
	\arrow["f", from=1-1, to=1-2]
	\arrow["{i_0}", shift left=2, from=1-2, to=1-3]
	\arrow["{i_1}"', shift right=2, from=1-2, to=1-3]
\end{tikzcd}
\end{equation*}
is an equalizer. Thus $f$ is an effective monomorphism.
\end{proof}

\begin{Lemma} \label{lem giraud categories are balanced}
If $\cat{E}$ is a Giraud category, then $\cat{E}$ is balanced, i.e. if $f: X \to Y$ is a morphism in $\cat{E}$ that is both a monomorphism and an epimorphism, then it is an isomorphism.
\end{Lemma}

\begin{proof}
If $f : X \to Y$ is a monomorphism and an epimorphism, then by Lemma \ref{lem giraud categories monos are effective}, $f$ is an effective monomorphism, so it is the equalizer of its cokernel pair,
\begin{equation*}
\begin{tikzcd}
	X & Y & {Y +_X Y}
	\arrow["f", from=1-1, to=1-2]
	\arrow["{i_0}", shift left=2, from=1-2, to=1-3]
	\arrow["{i_1}"', shift right=2, from=1-2, to=1-3]
\end{tikzcd}
\end{equation*}
But $f$ is also an epimorphism, and $i_0 f = i_1 f$, so $i_0 = i_1$. But the equalizer of $i_0$ with $i_0$ is just the identity map
\begin{equation*}
    \begin{tikzcd}
	Y & Y & {Y +_X Y}
	\arrow["{1_Y}", from=1-1, to=1-2]
	\arrow["{i_0}", shift left=2, from=1-2, to=1-3]
	\arrow["{i_0}"', shift right=2, from=1-2, to=1-3]
\end{tikzcd}
\end{equation*}
Since equalizers are unique up to isomorphism, this shows that $f : X \to Y$ is an isomorphism.
\end{proof}

\begin{Lemma}[{\cite[Appendix Lemma 2.2]{maclane2012sheaves}}] \label{lem epimono factorization for giraud's axioms}
If $\cat{E}$ is a Giraud category, then every morphism $f : X \to Y$ in $\cat{E}$ can be factored as
\begin{equation*}
    X \overset{e_f}{\twoheadrightarrow} Z_f \xhookrightarrow{\iota_f} Y
\end{equation*}
where $e_f$ is an effective epimorphism and $\iota_f$ is a monomorphism. 
\end{Lemma}

\begin{proof}
Let $Z_f = \coim_{\text{reg}}(f)$ denote the regular coimage (Definition \ref{def regular coimage}) of $f$, i.e. the coequalizer of the kernel pair of $f$
\begin{equation*}
    \begin{tikzcd}
	{X \times_YX} & X & {Z_f} \\
	&& Y
	\arrow["{p_0}", shift left, from=1-1, to=1-2]
	\arrow["{p_1}"', shift right, from=1-1, to=1-2]
	\arrow["{e_f}", two heads, from=1-2, to=1-3]
	\arrow["f"', from=1-2, to=2-3]
	\arrow["\iota_f", dashed, from=1-3, to=2-3]
\end{tikzcd}
\end{equation*}
Now by definition $e_f$ is an effective epimorphism. We wish to show that $\iota_f$ is a monomorphism.

So suppose there are maps $n, m: A \to Z_f$ such that $\iota_f n = \iota_f m$. Then by the universal property of pullbacks we have a unique induced map
\begin{equation*}
\begin{tikzcd}
	A \\
	& {Z_f \times_Y Z_f} & {Z_f} \\
	& {Z_f} & Y
	\arrow["h"{description}, dashed, from=1-1, to=2-2]
	\arrow["m", curve={height=-12pt}, from=1-1, to=2-3]
	\arrow["n"', curve={height=12pt}, from=1-1, to=3-2]
	\arrow[from=2-2, to=2-3]
	\arrow[from=2-2, to=3-2]
	\arrow["\lrcorner"{anchor=center, pos=0.125}, draw=none, from=2-2, to=3-3]
	\arrow[from=2-3, to=3-3]
	\arrow[from=3-2, to=3-3]
\end{tikzcd}
\end{equation*}

Now by the pasting law for pullback diagrams we have a commutative diagram where each square is a pullback
\begin{equation*}
    \begin{tikzcd}
	{X \times_Y X} & {Z_f \times_YX} & X \\
	{X \times_Y Z_f} & {Z_f \times_Y Z_f} & {Z_f} \\
	X & {Z_f} & Y
	\arrow[two heads, from=1-1, to=1-2]
	\arrow[two heads, from=1-1, to=2-1]
	\arrow["\lrcorner"{anchor=center, pos=0.125}, draw=none, from=1-1, to=2-2]
	\arrow[from=1-2, to=1-3]
	\arrow[two heads, from=1-2, to=2-2]
	\arrow["\lrcorner"{anchor=center, pos=0.125}, draw=none, from=1-2, to=2-3]
	\arrow["{e_f}", two heads, from=1-3, to=2-3]
	\arrow[two heads, from=2-1, to=2-2]
	\arrow[from=2-1, to=3-1]
	\arrow["\lrcorner"{anchor=center, pos=0.125}, draw=none, from=2-1, to=3-2]
	\arrow[from=2-2, to=2-3]
	\arrow[from=2-2, to=3-2]
	\arrow["\lrcorner"{anchor=center, pos=0.125}, draw=none, from=2-2, to=3-3]
	\arrow[from=2-3, to=3-3]
	\arrow["{e_f}"', two heads, from=3-1, to=3-2]
	\arrow[from=3-2, to=3-3]
\end{tikzcd}
\end{equation*}
Furthermore, since epimorphisms are stable in $\cat{E}$ by Corollary \ref{cor epis in giraud categories are effective}, the composite map above which we will denote by $q : X \times_Y X \to Z_f \times_Y Z_f$ is an epimorphism.

Therefore the pullback map
\begin{equation*}
\begin{tikzcd}
	{A \times_{Z_f \times_Y Z_f} (X \times_Y X)} & {X \times_Y X} \\
	A & {Z_f \times_Y Z_f}
	\arrow["{(n',m')}", from=1-1, to=1-2]
	\arrow["{h^*(q)}"', two heads, from=1-1, to=2-1]
	\arrow["\lrcorner"{anchor=center, pos=0.125}, draw=none, from=1-1, to=2-2]
	\arrow["q", two heads, from=1-2, to=2-2]
	\arrow["{(n,m)}"', from=2-1, to=2-2]
\end{tikzcd}
\end{equation*}
is an epimorphism as well.

Now note that $p_0 (n',m') = n'$ and $p_1 (n',m') = m'$. Since $e_f p_0 = e_f p_1$ we have 
$$e_f n' = e_f p_0 (n', m') = e_f p_1 (n', m') = e_f m'$$
Therefore
$$n h^*(q) = e_f n' = e_f m' = m h^*(q).$$
But $h^*(q)$ is an epimorphism, so $n = m$. Thus $\iota_f$ is a monomorphism.
\end{proof}

\begin{Lemma} \label{lem giraud categories epi iff coimage is iso}
Given a Giraud category $\cat{E}$, a morphism $f : X \to Y$ in $\cat{E}$ is an epimorphism if and only if for the regular coimage factorization
\begin{equation*}
    \begin{tikzcd}
	X & {\coim_{\text{reg}}(f)} & Y
	\arrow["e", two heads, from=1-1, to=1-2]
	\arrow["i", hook, from=1-2, to=1-3]
\end{tikzcd}
\end{equation*}
the map $i$ is an isomorphism.
\end{Lemma}

\begin{proof}
$(\Leftarrow)$ if $i$ is an isomorphism, then it is an epimorphism, so $i e = f$ is an epimorphism.
$(\Rightarrow)$ if $f = i e$ is an epimorphism, then $i$ is an epimorphism and a monomorphism, so by Lemma \ref{lem giraud categories are balanced}, $i$ is an isomorphism.
\end{proof}

\begin{Cor} \label{cor epis in giraud categories are effective}
If $\cat{E}$ is a Giraud category, then all epimorphisms in $\cat{E}$ are effective epimorphisms.  
\end{Cor}

\begin{proof}
By Lemma \ref{lem giraud categories epi iff coimage is iso}, if $f : X \to Y$ is an epimorphism, then $f$ is isomorphic to $e_f : X \to \coim_{\text{reg}}(f)$. But coequalizers are unique up to isomorphism, so $f$ is therefore the coequalizer of its own kernel pair and hence an effective epimorphism.
\end{proof}

\begin{Lemma} \label{lem giraud categories epimorphisms are stable under pullback}
If $\cat{E}$ is a Giraud category, then epimorphisms in $\cat{E}$ are stable under pullback. In other words if $f: X \to Y$ is an epimorphism in $\cat{E}$ and $g : Z \to Y$ is an arbitrary morphism and the following commutative diagram is a pullback
\begin{equation*}
    \begin{tikzcd}
	{Z \times_Y X} & X \\
	Z & Y
	\arrow[from=1-1, to=1-2]
	\arrow["{g^*(f)}"', from=1-1, to=2-1]
	\arrow["\lrcorner"{anchor=center, pos=0.125}, draw=none, from=1-1, to=2-2]
	\arrow["f", from=1-2, to=2-2]
	\arrow["g"', from=2-1, to=2-2]
\end{tikzcd}
\end{equation*}
then $g^*(f)$ is an epimorphism.
\end{Lemma}

\begin{proof}
Suppose that $f : X \to Y$ is an epimorphism and $g : Z \to Y$ is an arbitrary morphism in $\cat{E}$. By Corollary \ref{cor epis in giraud categories are effective}, $f$ is an effective epimorphism, hence $f$ is the coequalizer of its kernel pair. Thus 
\begin{equation*}
    \begin{tikzcd}
	{X \times_Y X} & X & Y
	\arrow["{p_0}", shift left=2, from=1-1, to=1-2]
	\arrow["{p_1}"', shift right=2, from=1-1, to=1-2]
	\arrow["f", from=1-2, to=1-3]
\end{tikzcd}
\end{equation*}
is a coequalizer. Now by the dual of \cite[Proposition 3.3.8]{riehl2017category}, the projection functor $\pi_Y : \cat{E}_{/Y} \to \cat{E}$ strictly creates all colimits. In other words, the diagram 
\begin{equation*}
    \begin{tikzcd}
	{X \times_Y X} & X & Y \\
	& Y
	\arrow["{p_0}", shift left=2, from=1-1, to=1-2]
	\arrow["{p_1}"', shift right=2, from=1-1, to=1-2]
	\arrow[from=1-1, to=2-2]
	\arrow["f", from=1-2, to=1-3]
	\arrow["f"', from=1-2, to=2-2]
	\arrow[Rightarrow, no head, from=1-3, to=2-2]
\end{tikzcd}
\end{equation*}
is a coequalizer in $\cat{E}_{/Y}$. Now by (G2), the pullback functor $g^* : \cat{E}_{/Y} \to \cat{E}_{/Z}$ preserves colimits, so the diagram
\begin{equation*}
    \begin{tikzcd}
	{g^*(X) \times_Z g^*(X)} & {g^*(X)} & Z \\
	& Z
	\arrow["{q_0}", shift left=2, from=1-1, to=1-2]
	\arrow["{q_1}"', shift right=2, from=1-1, to=1-2]
	\arrow[from=1-1, to=2-2]
	\arrow["{g^*(f)}", from=1-2, to=1-3]
	\arrow["{g^*(f)}"', from=1-2, to=2-2]
	\arrow[Rightarrow, no head, from=1-3, to=2-2]
\end{tikzcd}
\end{equation*}
is a coequalizer in $\cat{E}_{/Z}$. By the dual of \cite[Proposition 3.3.3]{riehl2017category} the functor $\pi_Z : \cat{E}_{/Z} \to \cat{E}$ preserves colimits, and hence $g^*(f)$ is a coequalizer and hence an epimorphism.
\end{proof}

\subsection{Giraud's Axioms implies Grothendieck Topos}

The goal for this section is to prove that if $\cat{E}$ is a locally presentable category then $(G) \implies (T)$. This is a classical theorem of Giraud \cite{SGA1972}, who actually proves $(G) \Leftrightarrow (T)$, but we prefer to prove the one direction so that we may prove all of the equivalences at once. Our method of proof was inspired by \cite{giraudtheorem}.

For the rest of this section, assume that $\cat{E}$ is a Giraud category, so in particular $\cat{E}$ is locally presentable. This implies that there exists a reflective localization:
\begin{equation}
	\begin{tikzcd}
	{\cat{E}} && {\Pre(\cat{C})}
	\arrow[""{name=0, anchor=center, inner sep=0}, "\iota"', shift right=2, hook, from=1-1, to=1-3]
	\arrow[""{name=1, anchor=center, inner sep=0}, "L"', shift right=2, from=1-3, to=1-1]
	\arrow["\dashv"{anchor=center, rotate=-90}, draw=none, from=1, to=0]
\end{tikzcd}
\end{equation}
where $\cat{C}$ is a small full subcategory of $\cat{E}$ whose objects are $\kappa$-presentable for some regular cardinal $\kappa$, and $\iota$ is the yoneda embedding restricted to $\cat{C}$, i.e. if $U \in \cat{E}$, then if $j: \cat{C} \hookrightarrow \cat{E}$ is the inclusion of the full subcategory, then $\iota$ is the functor $\iota(U) = y(U)|_{\cat{C}} = \cat{E}(j(-), U)$ and furthermore preserves $\kappa$-filtered colimits.

Let us define a coverage $j_\Gir$ on $\cat{C}$ so that we can obtain an equivalence $\Sh(\cat{C}, j_\Gir) \simeq \cat{E}$.

Given $U \in \cat{C}$, let $j_\text{Gir}(U)$ be the set of families $r = \{ r_i : U_i \to U \}_{i \in I}$ in $\cat{C}$ such that
\begin{equation*}
 \sum_{i \in I} r_i : \sum_{i \in I} U_i \to U
\end{equation*}
is an epimorphism in $\cat{E}$.

\begin{Lemma} \label{lem forming the Giraud coverage}
The collection of families $j_{\text{Gir}}$ forms a coverage on $\cat{C}$.
\end{Lemma}

\begin{proof}
Suppose that $U \in \cat{C}$, $r = \{ r_i :U_i \to U \} \in j_\Gir(U)$ and $g : V \to U$ is a morphism in $\cat{C}$. Then we can form the pullback
\begin{equation*}
    \begin{tikzcd}
	{g^*(\sum_i U_i)} & {\sum_i U_i} \\
	V & U
	\arrow[from=1-1, to=1-2]
	\arrow["{g^*(\sum_i r_i)}"', from=1-1, to=2-1]
	\arrow["{\sum_i r_i}", from=1-2, to=2-2]
	\arrow["g"', from=2-1, to=2-2]
\end{tikzcd}
\end{equation*}
in $\cat{E}$. Now by Lemma \ref{lem giraud categories epimorphisms are stable under pullback}, epimorphisms in $\cat{E}$ are stable under pullback. Thus $g^*(\sum_i r_i)$ is an epimorphism in $\cat{E}$ and furthermore by (G2) is isomorphic to
\begin{equation*}
    \begin{tikzcd}
	{\sum_i g^*(U_i)} & {\sum_i U_i} \\
	V & U
	\arrow[from=1-1, to=1-2]
	\arrow["{\sum_i g^*(r_i)}"', from=1-1, to=2-1]
	\arrow["{\sum_i r_i}", from=1-2, to=2-2]
	\arrow["g"', from=2-1, to=2-2]
\end{tikzcd}
\end{equation*}
Thus $g^*(r) = \{g^*(U_i) \to V \}$ (note we are abusing notation here, $g^*(r)$ just denotes the family of pullbacks) is a $j_\Gir$ covering family of $V$. Thus $j_\Gir$ is a coverage.
\end{proof}

Due to Lemma \ref{lem forming the Giraud coverage}, we call $j_{\text{Gir}}$ the \textbf{Giraud Coverage} on $\cat{C}$.

\begin{Lemma} \label{lem giraud coverage is saturated}
The Giraud coverage $j_{\Gir}$ is saturated.
\end{Lemma}

\begin{proof}
Let us show that $j_\Gir$ is refinement closed. Suppose that $U \in \cat{C}$, $t = \{ V_j \to U \}$ is a family on $U$, $r = \{ U_i \to U \}$ is a $j_\Gir$-covering family on $U$ and there is a refinement $g : t \to r$. That means that there is a map $\sum_j V_j \to \sum_i U_i$ such that the following diagram commutes
\begin{equation*}
\begin{tikzcd}
	{\sum_j V_j} && {\sum_i U_i} \\
	& V
	\arrow["g", from=1-1, to=1-3]
	\arrow["{\sum_j t_j}"', from=1-1, to=2-2]
	\arrow["{\sum_i r_i}", from=1-3, to=2-2]
\end{tikzcd}
\end{equation*}
But $\sum_i r_i$ is an epimorphism, which implies that $\sum_i r_i \circ g = \sum_j t_j$ is an epimorphism. Thus $t$ is a $j_\Gir$-covering family on $U$.

Let us show that $j_\Gir$ is composition-closed. Suppose that $r = \{r_i:  U_i \to U \}_{i \in I} \in j_\Gir(U)$ and $t^i = \{ t^i_j: V^i_j \to U_i \}_{j \in J_i} \in j_\Gir(U_i)$, then $t^i : \sum_j V^i_j \to U_i$ is an epimorphism, and since coproducts of epimorphisms are epimorphisms, the map
\begin{equation*}
    \sum_i \sum_j V^i_j \xrightarrow{\sum_i t^i} \sum_i U_i
\end{equation*}
is an epimorphism. But the map $r : \sum_i U_i \to U$ is also an epimorphism, thus the composite
$$\sum_i \sum_j V^i_j \xrightarrow{\sum_i t^i} \sum_i U_i \xrightarrow{r} U$$ 
is an epimorphism. Thus the composite family $\cup_i (t^i \circ r)$ is an epimorphism and hence in $j_\Gir(U)$.
\end{proof}

We now wish to show that the functor $\iota : \cat{E} \to \Pre(\cat{C})$ factors through $\Sh(\cat{C}, j_\Gir)$. In other words we want to show the following.

\begin{Prop} \label{prop restricted yoneda is giraud sheaf}
For every $V \in \cat{E}$ the presheaf $y(V)|_{\cat{C}} = \cat{E}(j(-), V)$ on $\cat{C}$ is a $j_\Gir$-sheaf.
\end{Prop}

\begin{proof}
Suppose that $U, V \in \cat{C}$ and $r = \{U_i \to U \} \in j_\Gir(U)$. Then $r : \sum_i U_i \to U$ is an epimorphism in $\cat{E}$. By Corollary \ref{cor epis in giraud categories are effective}, $r$ is the coequalizer of its kernel pair, and by (G2), we can write this as the coequalizer
\begin{equation*}
\begin{tikzcd}
	{\sum_{i, j \in I} U_i \times_U U_j} & {\sum_i U_i} & U
	\arrow["{p_0}", shift left=2, from=1-1, to=1-2]
	\arrow["{p_1}"', shift right=2, from=1-1, to=1-2]
	\arrow["r", from=1-2, to=1-3]
\end{tikzcd} 
\end{equation*}
Now each $U_i \times_U U_j$ may not be an object in $\cat{C}$, but it can be written as a colimit $U_i \times_U U_j \cong \colim_k B_k^{ij}$ where each $B^{ij}_k \in \cat{C}$. Since we can write any colimit as a coequalizer, then the cocone map $\sum_k B^{ij}_k \to U_i \times_U U_j$ is an epimorphism. Therefore if we let 
\begin{equation*}
   h: \sum_{i,j} \sum_k B^{ij}_k \to \sum_{i,j} U_i \times_U U_j
\end{equation*}
denote the resulting epimorphism and let $q_0 = p_0 h$, $q_1 = p_1 h$, then the following diagram is still a coequalizer
\begin{equation*}
    \begin{tikzcd}
	{\sum_{i, j} \sum_k B^{ij}_k} & {\sum_i U_i} & U
	\arrow["{q_0}", shift left=2, from=1-1, to=1-2]
	\arrow["{q_1}"', shift right=2, from=1-1, to=1-2]
	\arrow["r", from=1-2, to=1-3]
\end{tikzcd}
\end{equation*}
Indeed if $\ell: \sum_i U_i \to Q$ is a map such that $\ell q_0 = \ell q_1$, then $\ell q_0 = \ell p_0 h = \ell p_1 h = \ell q_1$, but since $h$ is an epimorphism this implies that $\ell p_0 = \ell p_1$. Since $r$ is a coequalizer of $p_0$ and $p_1$ this gives a unique map $U \to Q$ making the diagram commute. Thus $r$ is also a coequalizer of $q_0$ and $q_1$.

Now we apply $y(V)|_{\cat{C}}$ to the above coequalizer diagram and we obtain an equalizer diagram
\begin{equation*}
    y(V)|_{\cat{C}}(U) \to \text{eq} \left( \prod_i y(V)|_{\cat{C}}(U_i) \underset{p_1^*}{\overset{p_0^*}{\rightrightarrows}} \prod_{i,j} y(V)|_{\cat{C}}(U_i) \times_{y(V)|_{\cat{C}}(U)} y(V)|_{\cat{C}}(U_j) \xrightarrow{h^*} \prod_{i,j} \prod_k y(V)|_{\cat{C}}(B^{ij}_k) \right)
\end{equation*}
but since $h$ is an epimorphism, $h^*$ is a monomorphism, which implies that the following diagram is an equalizer
\begin{equation*}
    y(V)|_{\cat{C}}(U) \to \text{eq} \left( \prod_i y(V)|_{\cat{C}}(U_i) \underset{p_1^*}{\overset{p_0^*}{\rightrightarrows}} \prod_{i,j} y(V)|_{\cat{C}}(U_i) \times_{y(V)|_{\cat{C}}(U)} y(V)|_{\cat{C}}(U_j) \right)
\end{equation*}
thus $y(V)|_{\cat{C}}$ is a $j_\Gir$-sheaf.
\end{proof}

So we have managed to factor the reflective localization $L \dashv \iota$ through $\Sh(\cat{C}, j_\Gir)$
\begin{equation*}
    \begin{tikzcd}
	{\ncat{Sh}(\cat{C},j_\Gir)} && {\Pre(\cat{C})} \\
	\\
	& {\cat{E}}
	\arrow["i"', shift right=2, hook, from=1-1, to=1-3]
	\arrow["{L'}", shift left=2, from=1-1, to=3-2]
	\arrow["a"', shift right=2, from=1-3, to=1-1]
	\arrow["L"', shift right=2, from=1-3, to=3-2]
	\arrow["{\iota'}", shift left=2, from=3-2, to=1-1]
	\arrow["\iota"', shift right=2, from=3-2, to=1-3]
\end{tikzcd}
\end{equation*}
So $\iota = i \iota'$ and $L = L' a$ for adjoint functors $L' \dashv \iota'$.

\begin{Prop}
The adjunction $L' \dashv \iota'$ is an adjoint equivalence of categories.
\end{Prop}

\begin{proof}
Since $\iota$ and $i$ are fully faithful, so is $\iota'$. So for every $V \in \cat{E}$, the counit $\varepsilon_V : L'(\iota'(V)) \to V$ is an isomorphism. 

Now if $U \in \cat{C}$, then by Proposition \ref{prop restricted yoneda is giraud sheaf}, $y(U) \in \Sh(\cat{C}, j_\Gir)$ and $L'(y(U)) \cong L'(a y(U)) = L(y(U)) = U$. Thus the unit is isomorphic to
\begin{equation*}
y(U) \to \iota'(L'(ay(U))) \cong \iota'(U) \cong \cat{E}(j(-), U) \cong \cat{C}(-,U) \cong y(U)
\end{equation*}
and hence is an isomorphism.

Now for $X \in \Pre(\cat{C})$, by Lemma \ref{lem presheaf topoi are locally presentable}, $X$ can be written as a filtered colimit of representables $X \cong \ncolim{i \in I} y(U_i)$. Thus if $X$ is a sheaf then 
$$X \cong aX \cong a \ncolim{i \in I} y(U_i) \cong \ncolim{i \in I} ay(U_i) \cong \ncolim{i \in I} y(U_i)$$
where the last isomorphism holds by Proposition \ref{prop restricted yoneda is giraud sheaf}. So the unit is isomorphic to
\begin{equation*}
    \begin{aligned}
       X \xrightarrow{\eta_X} & \iota'(L'(X)) \\
       & \cong \iota'(L'(\ncolim{i\in I} y(U_i))) \\
       & \cong \iota'(\ncolim{i \in I} U_i) \\
       & \cong \cat{E}(j(-), \ncolim{i \in I} U_i) \\
       & \cong \ncolim{i \in I} \cat{E}(j(-), U_i) \\
       & \cong \ncolim{i \in I} y(U_i) \\
       & \cong X
    \end{aligned}
\end{equation*}
where the third to last isomorphism holds because $I$ is filtered and $j : \cat{C} \hookrightarrow \cat{E}$ is the inclusion of subcategory of presentable objects. Thus the unit is an isomorphism. Thus $L' \dashv \iota'$ is an adjoint equivalence of categories.
\end{proof}

Thus we have proven (G) $\Rightarrow$ (T).

\begin{Th} \label{th first half of Giraud's Theorem}
If $\cat{E}$ is a Giraud category $(G)$, then $\cat{E}$ is a Grothendieck topos (T).
\end{Th}

\subsection{Descent} In this section we define what it means for a category to have descent. This property is also known as having Van Kampen colimits. We show that $\ncat{Set}$ does not have all Van Kampen colimits, i.e. does not have descent, but satisfies a weaker property, identified by Rezk \cite{rezk2010toposes}, and which he appropriately refers to as weak descent.

For the first part of this section we will be a little sketchier than usual in motivating descent, and we will use some terms that we will not define, as following the details would quickly lead us down quite an advanced rabbit hole, but we will return to a much more detailed discussion in Section \ref{section weak descent implies giraud axioms}. We recommend the reader consult \cite{anel2019descent} for more details on this advanced topic, which was one of our inspirations for this section.

Let $\cat{C}$ be a category with all small colimits and finite limits (which we will refer to as a cocomplete lex category in what follows), and consider the following pseudo-functor $\mathbb{U}: \cat{C}^{\op} \to \ncat{CAT}$, called the \textbf{universe of $\cat{C}$}, defined on objects by
$$X \mapsto \cat{C}_{/X}$$
and morphisms by
$$ \left( f: X \to Y \right) \mapsto \left( f^*: \cat{C}_{/Y} \to \cat{C}_{/X} \right)$$
clearly for a composable pair $X \xrightarrow{f} Y \xrightarrow{g} Z$ in $\cat{C}$, there exists a natural isomorphism $(g \circ f)^* \cong g^* \circ f^*$, and $(1_X)^* \cong 1_{\cat{C}_{/X}}$, thus defining a pseudofunctor.

Given a cocomplete lex category $\cat{C}$, we can ask whether its universe $\mathbb{U}: \cat{C}^{\op} \to \ncat{CAT}$ sends limits to (pseudo)limits. An equivalent way of asking this is to ask whether for a small diagram $X: I \to \cat{C}$ there is an equivalence of categories
$$\cat{C}_{/{\ncolim{i \in I} X_i}} \simeq \lim_{i \in I} \cat{C}_{/X_i}$$
where the right hand side is a pseudolimit in $\ncat{CAT}$.

There is a nice way of getting a handle on the above question. First let $\ncat{Cart}(\cat{C}^I)$ denote the category whose
\begin{enumerate}
	\item objects are functors $X: I \to \cat{C}$, and
	\item whose morphisms are natural transformations $\alpha: Y \Rightarrow X$ such that every component square is a pullback:
\begin{equation*}
 	\begin{tikzcd}
	{Y_i} & {X_i} \\
	{Y_j} & {X_j}
	\arrow[from=1-1, to=2-1]
	\arrow["{\alpha_i}", from=1-1, to=1-2]
	\arrow[from=1-2, to=2-2]
	\arrow["{\alpha_j}"', from=2-1, to=2-2]
	\arrow["\lrcorner"{anchor=center, pos=0.125}, draw=none, from=1-1, to=2-2]
\end{tikzcd}
 \end{equation*}
\end{enumerate}
We call natural transformations with this property \textbf{Cartesian}.

\begin{Lemma}
Given a cocomplete lex category $\cat{C}$ and a small diagram $X : I \to \cat{C}$, there is an equivalence of categories
$$\lim_{i \in I} \cat{C}_{/X_i} \simeq \ncat{Cart}(\cat{C}^I)_{/X}.$$
\end{Lemma}

Thanks to the above lemma, we can now re-express our desire for $\mathbb{U}$ to preserves limits as asking that the following adjunction be an adjoint equivalence:

\begin{equation} \label{eq van kampen colimit}
\begin{tikzcd}
	{\cat{C}_{/ \colim_i X_i}} && {\ncat{Cart}(\cat{C}^I)_{/X}}
	\arrow[""{name=0, anchor=center, inner sep=0}, "\colim"', shift right=2, from=1-3, to=1-1]
	\arrow[""{name=1, anchor=center, inner sep=0}, "{\text{cst}}"', shift right=2, from=1-1, to=1-3]
	\arrow["\dashv"{anchor=center, rotate=-90}, draw=none, from=0, to=1]
\end{tikzcd}	
\end{equation}
where $\alpha = \text{cst}(Y \to \colim_i \, X_i)$ is the natural transformation given objectwise by pullback:
\begin{equation*}
	\begin{tikzcd}
	{Y_i} & Y \\
	{X_i} & {\colim_i \, X_i}
	\arrow[from=2-1, to=2-2]
	\arrow[from=1-2, to=2-2]
	\arrow["{\alpha_i}"', from=1-1, to=2-1]
	\arrow[from=1-1, to=1-2]
	\arrow["\lrcorner"{anchor=center, pos=0.125}, draw=none, from=1-1, to=2-2]
\end{tikzcd}
\end{equation*}
 and $\colim ( \alpha: Y \Rightarrow X)$ is simply the induced map of colimits $\colim_i \, Y_i \to \colim_i \, X_i$.

\begin{Def} \label{def van kampen colimit}
Given a cocomplete lex category $\cat{C}$, we say that a small diagram $X : I \to \cat{C}$ is \textbf{Van Kampen} if (\ref{eq van kampen colimit}) exists and is an adjoint equivalence. If $S$ is a class of diagrams in $\cat{C}$, then we say that $\cat{C}$ has Van Kampen $S$-colimits, if for each $s \in S$, $s$ is a Van Kampen diagram. We may also say that $\cat{C}$ satisfies \textbf{descent} with respect to $S$.
\end{Def}

\begin{Rem}
Many kinds of categories that appear in applications are those that have Van Kampen $S$-colimits for various classes $S$. Such examples include lextensive categories, adhesive categories, exhaustive categories and exact categories. See \cite{nlab:van_kampen_colimit} for more.
\end{Rem}

Now if both $\colim$ and $\text{cst}$ are fully faithful, then the unit and counit will be isomorphisms and therefore they will form an adjoint equivalence. 

\begin{Lemma} \label{lem universal and effective colimits description}
Asking that $\colim$ and $\text{cst}$ are fully faithful is equivalent to the following two conditions holding:
 \begin{enumerate}
 	\item[(D1)] (Colimits are Universal) Given a map $Y \to \colim_i X_i$, then
 	$$Y \cong \colim_i \, \left( Y \times_{\colim_i \, X_i} X_i \right),$$ and
 	\item[(D2)] (Colimits are Effective) Given a Cartesian natural transformation with components $\alpha_i: Y_i \to X_i$, then
 	$$Y_i \cong \left( \colim_i \, Y_i \right) \times_{\colim_i \, X_i} X_i.$$
 \end{enumerate}    
\end{Lemma}

 \begin{Ex}
In $\ncat{Set}$, all colimits are universal. Similarly all colimits are universal in locally cartesian closed categories $\cat{C}$, since for every map $f: U \to V$ the pullback functor $f^* : \cat{C}_{/V} \to \cat{C}_{/U}$ has a right adjoint and therefore preserves colimits. In other words, all locally cartesian closed categories satisfy $(D1)$.
\end{Ex}

However $\ncat{Set}$ does \textit{\textbf{not}} satisfy (D2), i.e. not all of its colimits are effective!

\begin{Ex}[{\cite[Example 2.3]{rezk2010toposes}}] \label{ex Set doesn't have effective pushouts}
Let $I$ be the category $\{ a \leftarrow b \to c \}$, and $\cat{C} = \ncat{Set}$. Let $X = \{ x, x' \}$, and fix $\sigma: X \to X$ given by $\sigma(x) = x'$ and $\sigma(x') = x$. Let the functor $A: I \to \ncat{Set}$ be given by 
$$X \xleftarrow{(1_X, 1_X)} X + X \xrightarrow{(1_X, \sigma)} X. $$
We let $X+X = \{(x, 0), (x', 0), (x, 1), (x', 1) \}$.
Similarly define the functor $B : I \to \ncat{Set}$ given by 
$$ * \leftarrow * + * \to * $$
where $*$ is the singleton set. There is a cartesian natural transformation $\alpha : A \Rightarrow B$ given by:
\begin{equation}
\begin{tikzcd}
	X & {X + X} & X \\
	{*} & {* + *} & {*}
	\arrow["{(1_X,1_X)}"', from=1-2, to=1-1]
	\arrow["{(1_X,\sigma)}", from=1-2, to=1-3]
	\arrow[from=1-1, to=2-1]
	\arrow[from=2-2, to=2-1]
	\arrow[from=2-2, to=2-3]
	\arrow[from=1-3, to=2-3]
	\arrow[from=1-2, to=2-2]
	\arrow["\lrcorner"{anchor=center, pos=0.125, rotate=-90}, draw=none, from=1-2, to=2-1]
	\arrow["\lrcorner"{anchor=center, pos=0.125}, draw=none, from=1-2, to=2-3]
\end{tikzcd}	
\end{equation}
Indeed, given maps $q_0 : Q \to X$ and $q_1 : Q \to * + *$, the composite 
$$Q \xrightarrow{q_0} X \xhookrightarrow{i_\ell} X + X$$
where $i_\ell$ is the inclusion into the left component, gives a unique map making the obvious diagram commute, and hence both squares are pullbacks.
The natural transformation is an object $\alpha: A \Rightarrow B$ in $\ncat{Cart}(\cat{C}^I)_{/B}$. But now consider taking the pushouts of $A$ and $B$. The pushout of $B$ is clearly just a singleton $\colim \, B = *$. Let us compute the pushout of $A$. If $q, q' :X \to Q$ are maps making the following diagram commute
\begin{equation*}
    \begin{tikzcd}
	{X+X} & X \\
	X & {*} \\
	&& Q
	\arrow["{(1_X, \sigma)}", from=1-1, to=1-2]
	\arrow["\nabla"', from=1-1, to=2-1]
	\arrow[from=1-2, to=2-2]
	\arrow["{q'}", curve={height=-12pt}, from=1-2, to=3-3]
	\arrow[from=2-1, to=2-2]
	\arrow["q"', curve={height=12pt}, from=2-1, to=3-3]
	\arrow["\lrcorner"{anchor=center, pos=0.125, rotate=180}, draw=none, from=2-2, to=1-1]
	\arrow["h"', dashed, from=2-2, to=3-3]
\end{tikzcd}
\end{equation*}
then 
$$q \nabla(x, 0) = q(x) = q'(x) = q'(1_X, \sigma)(x, 0), \qquad \text{and}, \qquad q\nabla(x, 1) = q(x) = q'(x') = q'(1_X, \sigma)(x, 1),$$
so $q'$ is constant, and therefore $q$ is constant since $q \nabla (x', 0) = q(x') = q'(x') = q'(1_X, \sigma)(x', 0)$.
But now notice that the square
\begin{equation*}
\begin{tikzcd}
	A(a) = X & \colim A \cong * \\
	{B(a) = *} & {\colim B \cong *}
	\arrow[from=1-1, to=2-1]
	\arrow[from=2-1, to=2-2]
	\arrow[from=1-2, to=2-2]
	\arrow[from=1-1, to=1-2]
\end{tikzcd}	
\end{equation*}
is not a pullback. In other words, we have a commutative cube
\begin{equation*}
    \begin{tikzcd}
	&& X \\
	{X+X} &&& {*} \\
	& X & {*} \\
	{*+*} &&& {*} \\
	& {*}
	\arrow[from=1-3, to=2-4]
	\arrow[from=1-3, to=3-3]
	\arrow[""{name=0, anchor=center, inner sep=0}, from=2-1, to=1-3]
	\arrow[from=2-1, to=3-2]
	\arrow["\lrcorner"{anchor=center, pos=0.125}, draw=none, from=2-1, to=3-3]
	\arrow[from=2-1, to=4-1]
	\arrow[from=2-4, to=4-4]
	\arrow[from=3-2, to=2-4]
	\arrow[""{name=1, anchor=center, inner sep=0}, from=3-2, to=5-2]
	\arrow[from=3-3, to=4-4]
	\arrow[""{name=2, anchor=center, inner sep=0}, from=4-1, to=3-3]
	\arrow[from=4-1, to=5-2]
	\arrow[from=5-2, to=4-4]
	\arrow["\lrcorner"{anchor=center, pos=0.125}, draw=none, from=2-1, to=1]
	\arrow["\lrcorner"{anchor=center, pos=0.125, rotate=225}, draw=none, from=2-4, to=0]
	\arrow["\lrcorner"{anchor=center, pos=0.125, rotate=225}, draw=none, from=4-4, to=2]
\end{tikzcd}
\end{equation*}
where the two back faces are pullbacks, and the top and bottom faces are pushouts, but the two front faces are not pullbacks\footnote{This is often the way which effective pushouts are described in the literature, especially for adhesive categories \cite[Definition 2]{lack2004adhesive}.} This implies that not all pushouts in $\ncat{Set}$, and hence not all colimits, are effective.
\end{Ex}

\begin{Rem} \label{rem Van Kampen pushouts in set}
In fact, the precise class of pushouts in $\ncat{Set}$ that are effective, i.e. Van Kampen, (which contains the class of pushouts along monomorphisms) has been classified in \cite{lowe2010van}. 
\end{Rem}

\begin{Rem} \label{rem why set doesn't have effective colimits}
The (very vague) idea here is the following: colimits are universal if given a map $Y \to \colim_i X_i$, then we can ``break apart'' $Y$ into pieces $Y \times_{\colim_i X_i} X_i$ and glue them back together to get $Y$. 

Colimits are effective if given maps $Y_i \to X_i$, we can ``glue the pieces together'' $\colim_i Y_i \to \colim_i X_i$, and then break them apart again to get the individual pieces.

In $\ncat{Set}$, colimits aren't effective because when we glue together sets, we lose information about the original sets that we cannot retrieve.
\end{Rem}

However, as one would hope given the Giraud axiom (G3), there are certain colimits that are effective in $\ncat{Set}$, namely coequalizers by equivalence relations. Let $I = (a \rightrightarrows b)$, and let $A, B : I \to \ncat{Set}$ be defined as the following coequalizers, with their respective coequalizers shown,
\begin{equation*}
\begin{tikzcd}
	R & Y & {Y/R,} & S & X & {X/S}
	\arrow["{p_0}", shift left=2, from=1-1, to=1-2]
	\arrow["{p_1}"', shift right=2, from=1-1, to=1-2]
	\arrow["h", from=1-2, to=1-3]
	\arrow["{q_0}", shift left=2, from=1-4, to=1-5]
	\arrow["{q_1}"', shift right=2, from=1-4, to=1-5]
	\arrow["k", from=1-5, to=1-6]
\end{tikzcd}
\end{equation*}
and where we assume that $R$ and $S$ are equivalence relations. Suppose we have a cartesian natural transformation $\alpha : A \Rightarrow B$, i.e. a diagram of the form
\begin{equation} \label{eq effective equiv rels in Set}
    \begin{tikzcd}
	Y & R & Y \\
	X & S & X
	\arrow["{\alpha}"', from=1-1, to=2-1]
	\arrow["{p_0}"', from=1-2, to=1-1]
	\arrow["{p_1}", from=1-2, to=1-3]
	\arrow["\lrcorner"{anchor=center, pos=0.125, rotate=-90}, draw=none, from=1-2, to=2-1]
	\arrow["{\alpha'}"', from=1-2, to=2-2]
	\arrow["\lrcorner"{anchor=center, pos=0.125}, draw=none, from=1-2, to=2-3]
	\arrow["{\alpha}", from=1-3, to=2-3]
	\arrow["{q_0}", from=2-2, to=2-1]
	\arrow["{q_1}"', from=2-2, to=2-3]
\end{tikzcd}
\end{equation}
We wish to show that the resulting commutative diagram
\begin{equation*}\label{eq effective equiv rels in Set 2}
    \begin{tikzcd}
	Y & {Y/R} \\
	X & {X/S}
	\arrow["{h}", from=1-1, to=1-2]
	\arrow["{\alpha}"', from=1-1, to=2-1]
	\arrow["{\alpha''}", from=1-2, to=2-2]
	\arrow["{k}"', from=2-1, to=2-2]
\end{tikzcd}
\end{equation*}
is a pullback. Now the two squares commuting in (\ref{eq effective equiv rels in Set}) implies that $\alpha : Y \to X$ is a map such that if $yRy'$ then $\alpha(y)S\alpha(y')$. The squares being pullbacks imply that $yRy'$ if and only if $\alpha(y)S \alpha(y')$. Furthermore, it implies that there are bijections
\begin{equation*}
    \alpha^{-1}(x) \cong (\alpha')^{-1}(x,x') \cong \alpha^{-1}(x')
\end{equation*}
of the fibers\footnote{This is an important property of pullbacks. They induce isomorphisms between the fibers of the two vertical (horizontal) maps.} for every $(x,x') \in S$. So if $y \in Y$ and $\alpha(y)Sx$, then there exists a unique $y' \in Y$ such that $\alpha(y') = x$. 

Now to show that (\ref{eq effective equiv rels in Set 2}) is a pullback, it is enough to show that the map
\begin{equation*}
Y \to \left \{ (x, [y]_R) \, : \, \alpha''([y]_R) = [\alpha(y)]_S = [x]_S \right \}, \qquad \left( y \mapsto (\alpha(y), [y]_R) \right)
\end{equation*}
is a bijection. Let us show it is injective. Suppose that $\alpha(y) = \alpha(y')$ and $[y]_R = [y']_R$. We want to show that $y = y'$. Let $\alpha(y') = x$. We have $\alpha(y) S x$, so there exists a unique $y'' \in Y$ such that $\alpha(y'') = x$. Thus $y = y' = y''$. Now let us show that it is surjective. If $x \in X$ and $y \in Y$ such that $\alpha(y) S x$, then there exists a unique $y' \in Y$ such that $\alpha(y') = x$. Thus $(x, [y]_R) = (\alpha(y'), [y']_R)$. Thus we have a bijection. In summary we have proven the following.

\begin{Prop} \label{prop Set has effective equivalence relations}
The category $\ncat{Set}$ satisfies (G3). In other words $\ncat{Set}$ has effective equivalence relations.
\end{Prop}

In fact, there is a full characterization of effective colimits in $\ncat{Set}$.

\begin{Prop}[{\cite{Hoyois2015vankampen}}]
A colimit of a diagram $d : I \to \ncat{Set}$ is effective if and only if it is Van Kampen if and only if its category of elements $\int d$ has the property that each of its connected components is simply connected.
\end{Prop}

Notice that conditions $(D1)$ and $(D2)$ from Lemma \ref{lem universal and effective colimits description} can be reformulated as follows. Consider the following conditions on a cocomplete lex category $\cat{C}$:
\begin{itemize}
	\item $(D1a)$-(Universal coproducts): Given a collection of objects $\{ Y_i \}_{i \in I}$, let $Y = \sum_i Y_i$. Let $f: X \to Y$ be a morphism, and let $X_i = Y_i \times_Y X$. Then the induced map $\sum_i X_i \to X$ is an isomorphism,
	\item $(D1b)$-(Universal pushouts): Given a span $Y_0 \leftarrow Y_1 \to Y_2$, let $Y = Y_0 +_{Y_1} Y_2$. Let $f: X \to Y$ be a morphism and let $X_i = Y_i \times_{Y} X$. Then the induced map $X_0 +_{X_1} X_2 \to X$ is an isomorphism.
	\item $(D2a)$-(Effective coproducts): Given a collection of maps $\{ f_i: X_i \to Y_i \}$, let $X = \sum_i X_i$, and $Y = \sum_i Y_i$, and let $f: X \to Y$ be the coproduct $\sum_i f_i$. Then the natural maps $X_i \to Y_i \times_Y X$ are isomorphisms for each $i$.
	\item $(D2b)$-(Effective pushouts): Given a cartesian map of spans:
	\begin{equation*}
\begin{tikzcd}
	{X_0} & {X_1} & {X_2} \\
	{Y_0} & {Y_1} & {Y_2}
	\arrow["{f_0}"', from=1-1, to=2-1]
	\arrow[from=1-2, to=1-1]
	\arrow[from=2-2, to=2-1]
	\arrow["{f_1}", from=1-2, to=2-2]
	\arrow[from=2-2, to=2-3]
	\arrow[from=1-2, to=1-3]
	\arrow["{f_2}", from=1-3, to=2-3]
	\arrow["\lrcorner"{anchor=center, pos=0.125, rotate=-90}, draw=none, from=1-2, to=2-1]
	\arrow["\lrcorner"{anchor=center, pos=0.125}, draw=none, from=1-2, to=2-3]
\end{tikzcd}
	\end{equation*}
	let $X = X_0 +_{X_1} X_2$ and $Y = Y_0 +_{Y_1} Y_2$, and let $f:X \to Y$ denote the induced map of pushouts. Then the natural maps $X_i \to Y_i \times_Y X$ are isomorphisms.
\end{itemize}
since all small colimits can be computed as coequalizers (and therefore from pushouts and an initial object) of small coproducts, $(D1) \Leftrightarrow (D1a) + (D1b)$ and $(D2) \Leftrightarrow (D2a) + (D2b)$. It is often easier to work with these other conditions.

Now since $\ncat{Set}$ does not satisfy $(D2b)$ by Example \ref{ex Set doesn't have effective pushouts}, we weaken this condition to the following:
\begin{itemize}
    \item $(D2b)'$-(Weak Effective pushouts) Given a Cartesian map of spans:
\begin{equation*}
\begin{tikzcd}
	{X_0} & {X_1} & {X_2} \\
	{Y_0} & {Y_1} & {Y_2}
	\arrow["{f_0}"', from=1-1, to=2-1]
	\arrow[from=1-2, to=1-1]
	\arrow[from=2-2, to=2-1]
	\arrow["{f_1}", from=1-2, to=2-2]
	\arrow[from=2-2, to=2-3]
	\arrow[from=1-2, to=1-3]
	\arrow["{f_2}", from=1-3, to=2-3]
	\arrow["\lrcorner"{anchor=center, pos=0.125, rotate=-90}, draw=none, from=1-2, to=2-1]
	\arrow["\lrcorner"{anchor=center, pos=0.125}, draw=none, from=1-2, to=2-3]
\end{tikzcd}	
\end{equation*}
let $X = X_0 +_{X_1} X_2$ and $Y = Y_0 +_{Y_1} Y_2$, and let $f:X \to Y$ denote the induced map of pushouts. Then the natural maps $X_i \to Y_i \times_Y X$ are effective\footnote{In \cite{rezk2010toposes}, Rezk defines $(D2b)'$ with effective epimorphisms replaced with regular epimorphisms. But by Lemma \ref{lem sufficient condition for reg epi to be effective}, this will make no difference in what follows.} epimorphisms.
\end{itemize}

\begin{Def}
We will call denote the logical conjunction by $(D2)' = (D2a) + (D2b)'$. We let $(D)'$ be the statement that $\cat{E}$ is locally presentable and satisfies $(D1)$ and $(D2)'$. We also say that $\cat{E}$ has \textbf{weak Van Kampen colimits} or \textbf{weak descent}.
\end{Def}

\subsection{Weak Descent implies Giraud's Axioms} \label{section weak descent implies giraud axioms}
In this section we will show that weak descent implies Giraud's axioms, i.e. $(D)' \implies (G)$.

\begin{Prop}\label{prop D1 equiv to G1}
Let $\cat{C}$ be a cocomplete lex category. Then $\cat{C}$ satisfies $(D1)$ if and only if for all maps $f: T \to S$ in $\cat{C}$, the pullback functor $f^*: \cat{C}_{/S} \to \cat{C}_{/T}$ preserves colimits ($G1)$.
\end{Prop}

\begin{proof}
$(\Rightarrow)$ Suppose that $x : I \to \cat{C}_{/S}$ is a small diagram. Since we can compute colimits in $\cat{C}_{/S}$ by looking at their projections, let $X \xrightarrow{\hat{x}} S = \colim_i \, ( X(i) \xrightarrow{x(i)} S )$ be the colimit in $\cat{C}_{/S}$. This implies that $X \cong \colim_i \, X(i)$ in $\cat{C}$. Consider the pullback
\begin{equation*}
    \begin{tikzcd}
	Y & X \\
	T & S
	\arrow[from=1-1, to=1-2]
	\arrow["{f^*(\hat{x})}"', from=1-1, to=2-1]
	\arrow["\lrcorner"{anchor=center, pos=0.125}, draw=none, from=1-1, to=2-2]
	\arrow["{\hat{x}}", from=1-2, to=2-2]
	\arrow["f"', from=2-1, to=2-2]
\end{tikzcd}
\end{equation*}
We want to show that $f^*(\hat{x}) \cong \colim_i f^*(x(i))$. Now by since $\cat{C}$ satisfies $(D1)$, if we set $Y(i) = X(i) \times_X Y$, then $\colim_i \, Y(i) \cong Y$ in $\cat{C}$. Since colimits in $\cat{C}_{/T}$ are computed by their projections in $\cat{C}$, the object $f^*(\hat{x}) : Y \to T$ in $\cat{C}_{/T}$ is the colimit $\colim_i f^*(x(i)) : Y(i) \to T$. Thus $\colim_i f^*( x(i) ) \cong f^* (\colim_i x(i)) \cong f^*(\hat{x})$. Thus $f^*$ preserves colimits, so $\cat{C}$ satisfies $(G1)$.

$(\Leftarrow)$ Suppose that $\cat{C}$ satisfies $(G1)$. Then if $X = \colim_i \, X(i)$, $g : Y \to X$ is an arbitrary morphism and $Y(i) = X(i) \times_X Y$, then 
\begin{equation*}
\colim_i \, Y(i) = \colim_i \, \left( X(i) \times_X Y \right) \cong \left( \left( \colim_i \, X(i) \right) \times_X Y \right) \cong \left(X \times_X Y \right) \cong Y.
\end{equation*}
Thus $\cat{C}$ satisfies $(D1)$.
\end{proof}

\begin{Prop} \label{prop effective coproducts imply disjoint coproducts}
If a cocomplete lex category $\cat{C}$ satisfies $(D2a)$, i.e. has effective coproducts, then it has disjoint coproducts $(G2)$.
\end{Prop}

\begin{proof}
Let $I$ be a small, discrete category, and $\{ X(i) \}_{i \in I}$ a collection of objects in $\cat{C}$, then we wish to show that the following diagram 
\begin{equation*}
	\begin{tikzcd}
	\varnothing & {X(k)} \\
	{X(j)} & {\sum_{i \in I} X(i)}
	\arrow[from=1-2, to=2-2]
	\arrow[from=2-1, to=2-2]
	\arrow[from=1-1, to=2-1]
	\arrow[from=1-1, to=1-2]
\end{tikzcd}
\end{equation*}
is a pullback square whenever $j \neq k$.

First, for every $j \in I$, define a functor $\delta^j: I \to \cat{C}$ by 
$$
\delta^j(i) = \begin{cases}
 \varnothing, & \text{ if } 	i \neq j, \\
 X_j, & \text{ if } i = j.
 \end{cases}
 $$
 There is an obvious natural transformation $\delta^j \Rightarrow X$ given by inclusion or the identity. This natural transformation is cartesian vacuously. Thus by $(D2a)$ the following diagram is a pullback:
 \begin{equation*}
 \begin{tikzcd}
	{\delta^j(i)} & {\sum_{i \in I} \delta^j(i)} \\
	{X(i)} & {\sum_{i \in I} X(i)}
	\arrow[from=1-1, to=2-1]
	\arrow[from=1-1, to=1-2]
	\arrow[from=1-2, to=2-2]
	\arrow[from=2-1, to=2-2]
	\arrow["\lrcorner"{anchor=center, pos=0.125}, draw=none, from=1-1, to=2-2]
\end{tikzcd}	
 \end{equation*}
for every $i$ and $j$. If $i \neq j$, then the above pullback diagram implies disjoint coproducts $(G2)$. 
\end{proof}

Now let us prove the last implication $(D1) + (D2b)' \Rightarrow (G3)$. First, let us prove a few technical results that we will need. 

\begin{Lemma}[{\cite{Henryweakdescent2021}}] \label{lem weak descent for coequalizers}
If $\cat{C}$ is a cocomplete lex category satisfying $(D1) + (D2b)'$, then it has weak descent with respect to coequalizers.
\end{Lemma}

\begin{proof}
This is morally true because $(D1) + (D2b)'$ is equivalent to $\cat{C}$ having Van Kampen coproducts and having weak Van Kampen pushouts. Since coequalizers can be built from pushouts and coproducts, the result follows. Suppose we have a commutative diagram
\begin{equation*}
    \begin{tikzcd}
	{X_0} & {X_1} & X \\
	{Y_0} & {Y_1} & Y
	\arrow["{p_0}", shift left=2, from=1-1, to=1-2]
	\arrow["{p_1}"', shift right=2, from=1-1, to=1-2]
	\arrow["{\alpha_0}"', from=1-1, to=2-1]
	\arrow["h", from=1-2, to=1-3]
	\arrow["{\alpha_1}"', from=1-2, to=2-2]
	\arrow["{\alpha_2}"', from=1-3, to=2-3]
	\arrow["{q_0}", shift left=2, from=2-1, to=2-2]
	\arrow["{q_1}"', shift right=2, from=2-1, to=2-2]
	\arrow["k"', from=2-2, to=2-3]
\end{tikzcd}
\end{equation*}
where the two squares
\begin{equation*}
    \begin{tikzcd}
	{X_1} & {X_0} & {X_1} \\
	{Y_1} & {Y_0} & {Y_1}
	\arrow["{\alpha_1}"', from=1-1, to=2-1]
	\arrow["{p_0}"', from=1-2, to=1-1]
	\arrow["{p_1}", from=1-2, to=1-3]
	\arrow["\lrcorner"{anchor=center, pos=0.125, rotate=-90}, draw=none, from=1-2, to=2-1]
	\arrow["{\alpha_0}"', from=1-2, to=2-2]
	\arrow["\lrcorner"{anchor=center, pos=0.125}, draw=none, from=1-2, to=2-3]
	\arrow["{\alpha_1}", from=1-3, to=2-3]
	\arrow["{q_0}"', from=2-2, to=2-1]
	\arrow["{q_1}", from=2-2, to=2-3]
\end{tikzcd}
\end{equation*}
are pullbacks. We wish to prove that the induced map $\ell$
\begin{equation*}
    \begin{tikzcd}
	{X_1} \\
	& {Y_1 \times_Y X} & X \\
	& {Y_1} & Y
	\arrow["\ell"{description}, dashed, from=1-1, to=2-2]
	\arrow["h", curve={height=-12pt}, from=1-1, to=2-3]
	\arrow["{\alpha_1}"', curve={height=12pt}, from=1-1, to=3-2]
	\arrow[from=2-2, to=2-3]
	\arrow[from=2-2, to=3-2]
	\arrow["{\alpha_2}", from=2-3, to=3-3]
	\arrow["k"', from=3-2, to=3-3]
\end{tikzcd}
\end{equation*}
is an effective epimorphism.

First, let us note that $Y_0 \rightrightarrows Y_1 \to Y$ being a coequalizer is equivalent to the commutative square
\begin{equation*}
    \begin{tikzcd}
	{Y_0 + Y_1} & {Y_1} \\
	{Y_1} & Y
	\arrow["{(q_1, 1_{Y_1})}", from=1-1, to=1-2]
	\arrow["{(q_0, 1_{Y_1})}"', from=1-1, to=2-1]
	\arrow["k", from=1-2, to=2-2]
	\arrow["k"', from=2-1, to=2-2]
\end{tikzcd}
\end{equation*}
being a pushout, and similarly for $X$. Now we wish to show that the natural transformation $\alpha'$ given in components as
\begin{equation*}
\begin{tikzcd}
	{X_1} & {X_0 + X_1} & {X_1} \\
	{Y_1} & {Y_0 + Y_1} & {Y_1}
	\arrow["{{\alpha_1}}"', from=1-1, to=2-1]
	\arrow["{{(p_0, 1_{X_1})}}"', from=1-2, to=1-1]
	\arrow["{{(p_1, 1_{X_1})}}", from=1-2, to=1-3]
	\arrow["{{\alpha_0 + \alpha_1}}"', from=1-2, to=2-2]
	\arrow["{{\alpha_1}}", from=1-3, to=2-3]
	\arrow["{{(q_0, 1_{Y_1})}}", from=2-2, to=2-1]
	\arrow["{{(q_1, 1_{Y_1})}}"', from=2-2, to=2-3]
\end{tikzcd}
\end{equation*}
is cartesian, i.e. we wish to show that the two squares above are pullbacks. But since $\cat{C}$ satisfies $(D1)$, and since $\alpha$ is cartesian, we have
\begin{equation*}
    (Y_0 + Y_1) \times_{Y_1} X_1 \cong (Y_0 \times_{Y_1} X_1) + (Y_1 \times_{Y_1} X_1) \cong X_0 + X_1.
\end{equation*}
Thus the natural transformation $\alpha'$ is cartesian, and since $\cat{C}$ has weak Van Kampen pushouts $(D2b)'$, this implies that the resulting map $\ell$ above is an effective epimorphism.
\end{proof}

\begin{Prop}[{\cite{Henryweakdescent2021}}]
If $\cat{C}$ is a cocomplete lex category that satisfies $(D1)$ and $(D2b)'$, then it has effective equivalence relations $(G3)$.
\end{Prop}

\begin{proof}
Let
\begin{equation*}
    \begin{tikzcd}
	R & X & {X/R}
	\arrow["{p_0}", shift left=2, from=1-1, to=1-2]
	\arrow["{p_1}"', shift right=2, from=1-1, to=1-2]
	\arrow["p", from=1-2, to=1-3]
\end{tikzcd}
\end{equation*}
be an equivalence relation on $X$. Consider the equalizer
\begin{equation*}
\begin{tikzcd}
	{R^{(2)}} & {R \times R} & X
	\arrow["{\iota = (\iota_0, \iota_1)}", hook, from=1-1, to=1-2]
	\arrow["{{\pi_1}}", shift left=2, from=1-2, to=1-3]
	\arrow["{{\pi_2}}"', shift right=2, from=1-2, to=1-3]
\end{tikzcd}
\end{equation*}
where $\pi_1$ and $\pi_2$ are the composite maps
\begin{equation*}
    R \times R \xrightarrow{\text{proj}_0} R \xrightarrow{p_0} X, \qquad R \times R \xrightarrow{\text{proj}_1} R \xrightarrow{p_0} X.
\end{equation*}
If $\cat{C}$ were $\ncat{Set}$, then 
\begin{equation*}
    R^{(2)} \cong \{ ((x_0, x_1), (y_0, y_1)) \in R^2 \, : \, x_0 = y_0 \} \cong \{ (x_0, x_1, x_2) \in X^3 \, : \, x_0 R x_1, \text{ and } x_1 R x_2 \}.
\end{equation*} 
So $R^{(2)}$ is a relation on $R$, and we want to show that it is an equivalence relation. Set-theoretically, this is easy to see. $R^{(2)}$ is the relation $\sim$ on $R$ where $(x_0, x_1) \sim (y_0, y_1)$ if $x_0 = y_0$. The diagonal map $\Delta : R \to R \times R$ induces a map $r : R \to R^{(2)}$, by the universal property of the equalizer, which is basically the diagonal map. Thus $R^{(2)}$ is reflexive. The other properties follows similarly, so we will just show them set-theoretically, and trust the reader can translate them to diagrammatic proofs. $R^{(2)}$ is symmetric, because if $(x_0, x_1) \sim (y_0, y_1)$, then $x_0 = y_0$, and thus $(y_0, y_1) \sim (x_0, x_1)$. It is also transitive, because if $(x_0, x_1) \sim (y_0, y_1)$ and $(y_0, y_1) \sim (z_0, z_1)$, then $x_0 = y_0$ and $y_0 = z_0$. So $x_0 = z_0$. Thus $R^{(2)}$ is an equivalence relation on $R$.

Now we wish to show that the following commutative diagram
\begin{equation} \label{eq split coeq for equiv rel}
\begin{tikzcd}
	{R^{(2)}} & R & X
	\arrow["{\iota_0}", shift left=2, from=1-1, to=1-2]
	\arrow["{\iota_1}"', shift right=2, from=1-1, to=1-2]
	\arrow["{p_0}", from=1-2, to=1-3]
\end{tikzcd}
\end{equation}
is a coequalizer. First, consider the map $s : X \to R$, such that $p_0 s = p_1 s = 1_X$, which exists since $R$ is reflexive. Then, since $\iota : R^{(2)} \to R \times R$ is an equalizer, we obtain a map $s' : R \to R^{(2)}$ induced by the map $(1_R, s p_0) : R \to R \times R$. But note that the following relations hold
\begin{equation*}
    \iota_0 s' = 1_R, \qquad \iota_1 s' = s p_0, \qquad r p_0 = 1_X.
\end{equation*}
Thus (\ref{eq split coeq for equiv rel}) is actually a split coequalizer diagram, and hence a coequalizer.

We have a cartesian natural transformation given as follows
\begin{equation*}
    \begin{tikzcd}
	R & {R^{(2)}} & R \\
	X & R & X
	\arrow["{p_1}"', from=1-1, to=2-1]
	\arrow["{\iota_0}"', from=1-2, to=1-1]
	\arrow["{\iota_1}", from=1-2, to=1-3]
	\arrow["\lrcorner"{anchor=center, pos=0.125, rotate=-90}, draw=none, from=1-2, to=2-1]
	\arrow["\ell", from=1-2, to=2-2]
	\arrow["\lrcorner"{anchor=center, pos=0.125}, draw=none, from=1-2, to=2-3]
	\arrow["{p_1}", from=1-3, to=2-3]
	\arrow["{p_0}", from=2-2, to=2-1]
	\arrow["{p_1}"', from=2-2, to=2-3]
\end{tikzcd}
\end{equation*}
where $\ell$ is the map defined as follows. Let us define it set-theoretically first as $\ell(x_0, x_1, x_2) = (x_1, x_2)$. To describe this diagrammatically, we consider the induced map
\begin{equation*}
    \begin{tikzcd}
	{R^{(2)}} \\
	& {R \times_X R} & R \\
	& R & X
	\arrow["h", dashed, from=1-1, to=2-2]
	\arrow["{\iota_1}", curve={height=-12pt}, from=1-1, to=2-3]
	\arrow["{i\iota_0}"', curve={height=12pt}, from=1-1, to=3-2]
	\arrow[from=2-2, to=2-3]
	\arrow[from=2-2, to=3-2]
	\arrow["\lrcorner"{anchor=center, pos=0.125}, draw=none, from=2-2, to=3-3]
	\arrow["{p_0}", from=2-3, to=3-3]
	\arrow["{p_1}"', from=3-2, to=3-3]
\end{tikzcd}
\end{equation*}
Then by using the composition map $ c: R \times_X R \to R$, we define $\ell = c h$.

Thus since $\cat{C}$ satisfies $(D1) + (D2b)'$, by Lemma \ref{lem weak descent for coequalizers}, $\cat{C}$ satisfies weak descent with respect to coequalizers, hence the following induced map
\begin{equation*}
\begin{tikzcd}
	R \\
	& {X \times_{X/R}X} & X \\
	& X & {X/R}
	\arrow["i"{description}, dashed, from=1-1, to=2-2]
	\arrow["{p_1}", curve={height=-12pt}, from=1-1, to=2-3]
	\arrow["{p_0}"', curve={height=12pt}, from=1-1, to=3-2]
	\arrow[from=2-2, to=2-3]
	\arrow[from=2-2, to=3-2]
	\arrow["\lrcorner"{anchor=center, pos=0.125}, draw=none, from=2-2, to=3-3]
	\arrow["q", from=2-3, to=3-3]
	\arrow["q"', from=3-2, to=3-3]
\end{tikzcd}
\end{equation*}
is an effective epimorphism. But since $X \times_{X/R} X$ is a pullback, the following diagram is an equalizer
\begin{equation*}
    \begin{tikzcd}
	{X \times_{X/R} X} & {X \times X} & {X/R}
	\arrow["{i'}", hook, from=1-1, to=1-2]
	\arrow["{q\,\text{proj}_0}", shift left=2, from=1-2, to=1-3]
	\arrow["{q \, \text{proj}_1}"', shift right=2, from=1-2, to=1-3]
\end{tikzcd}
\end{equation*}
Furthermore, the following diagram commutes
\begin{equation*}
    \begin{tikzcd}
	R && {X \times_{X/R} X} \\
	& {X \times X}
	\arrow["i", from=1-1, to=1-3]
	\arrow["{(p_0,p_1)}"', hook, from=1-1, to=2-2]
	\arrow["{i'}", hook', from=1-3, to=2-2]
\end{tikzcd}
\end{equation*}
So therefore $i$ is a monomorphism. But then $i$ is a monomorphism and an effective epimorphism, and therefore an isomorphism. Thus $R$ is an effective equivalence relation. Thus $\cat{C}$ satisfies $(G3)$.
\end{proof}

Thus we have managed to prove that $(D)' \Rightarrow (G)$.

\begin{Prop} \label{prop weak descent implies giraud's axioms}
If $\cat{E}$ is a locally presentable category and satisfies weak descent $(D)'$, then it satisfies Giraud's axioms $(G)$. 
\end{Prop}

\subsection{Grothendieck Topoi Satisfy Weak Descent}
In this section we prove that $(T) \Rightarrow (D)'$. We will do this by first proving that $\ncat{Set}$ satisfies weak descent, which then immediately implies that presheaf topoi satisfy weak descent, and then using lex localization, show that this implies it for all Grothendieck topoi.

\begin{Lemma} \label{lem set has universal colimits}
$\ncat{Set}$ satisfies $(D1)$, i.e. it has universal colimits.
\end{Lemma}

\begin{proof}
By Proposition \ref{prop D1 equiv to G1}, $(D1)$ is equivalent to $(G1)$, i.e. for every function $f: S \to T$, the functor $f^* : \ncat{Set}_{/T} \to \ncat{Set}_{/S}$ preserves colimits. This is true of every cartesian closed category, hence $\ncat{Set}$ satisfies $(D1)$.
\end{proof}

\begin{Lemma} \label{lem set has effective coproducts}
$\ncat{Set}$ satisfies $(D2a)$, i.e. it has effective coproducts. 
\end{Lemma}

\begin{proof}
Suppose we have a collection $\{ \alpha_i : X_i \to Y_i \}$ of morphisms in $\ncat{Set}$, and let $X = \sum_i X_i$ and $Y = \sum_i Y_i$. We want to show that the induced map
\begin{equation*}
    \begin{tikzcd}
	{X_i} \\
	& {Y_i \times_Y X} & X \\
	& {Y_i} & Y
	\arrow["h", dashed, from=1-1, to=2-2]
	\arrow["{\text{in}^X_i}", curve={height=-12pt}, from=1-1, to=2-3]
	\arrow["{\alpha_i}"', curve={height=12pt}, from=1-1, to=3-2]
	\arrow[from=2-2, to=2-3]
	\arrow[from=2-2, to=3-2]
	\arrow["\lrcorner"{anchor=center, pos=0.125}, draw=none, from=2-2, to=3-3]
	\arrow["\alpha"', from=2-3, to=3-3]
	\arrow["{\text{in}^Y_i}"', from=3-2, to=3-3]
\end{tikzcd}
\end{equation*}
is an isomorphism. But $Y_i \times_Y X \cong \{ (y, x) \in Y_i \times X \, : \, \alpha(x) = y \}$, and of course $\alpha(x) = y$ if and only if $x \in X_i$. Thus $h$ is an isomorphism, with inverse given by $h^{-1}(y,x) = x$.
\end{proof}

\begin{Lemma}
$\ncat{Set}$ satisfies $(D2b)'$, i.e. it has weak effective pushouts. 
\end{Lemma}

\begin{proof}
In $\ncat{Set}$, all epimorphisms are effective epimorphisms. So we need to show that if we have a cartesian natural transformation of pushouts
\begin{equation} \label{eq cartesian map of pushouts}
    \begin{tikzcd}
	{X_0} & {X_1} & {X_2} \\
	{Y_0} & {Y_1} & {Y_2}
	\arrow["{\alpha_0}"', from=1-1, to=2-1]
	\arrow["{f_0}"', from=1-2, to=1-1]
	\arrow["{f_1}", from=1-2, to=1-3]
	\arrow["\lrcorner"{anchor=center, pos=0.125, rotate=-90}, draw=none, from=1-2, to=2-1]
	\arrow["{\alpha_1}"', from=1-2, to=2-2]
	\arrow["\lrcorner"{anchor=center, pos=0.125}, draw=none, from=1-2, to=2-3]
	\arrow["{\alpha_2}", from=1-3, to=2-3]
	\arrow["{g_0}", from=2-2, to=2-1]
	\arrow["{g_1}"', from=2-2, to=2-3]
\end{tikzcd}
\end{equation}
and if we let $X = X_0 +_{X_1} X_2$ and $Y = Y_0 +_{Y_1} Y_2$, then the induced maps
\begin{equation*}
   \begin{tikzcd}
	{X_0} &&& {X_2} \\
	& {Y_0 \times_Y X} & X && {Y_2 \times_Y X} & X \\
	& {Y_0} & Y && {Y_2} & Y
	\arrow["{h_0}", dashed, from=1-1, to=2-2]
	\arrow["{\sigma_0}", curve={height=-12pt}, from=1-1, to=2-3]
	\arrow["{\alpha_0}"', curve={height=12pt}, from=1-1, to=3-2]
	\arrow["{h_2}", dashed, from=1-4, to=2-5]
	\arrow["{\sigma_2}", curve={height=-12pt}, from=1-4, to=2-6]
	\arrow["{\alpha_2}"', curve={height=12pt}, from=1-4, to=3-5]
	\arrow[from=2-2, to=2-3]
	\arrow[from=2-2, to=3-2]
	\arrow["\lrcorner"{anchor=center, pos=0.125}, draw=none, from=2-2, to=3-3]
	\arrow["\alpha", from=2-3, to=3-3]
	\arrow[from=2-5, to=2-6]
	\arrow[from=2-5, to=3-5]
	\arrow["\lrcorner"{anchor=center, pos=0.125}, draw=none, from=2-5, to=3-6]
	\arrow["\alpha", from=2-6, to=3-6]
	\arrow["{\lambda_0}"', from=3-2, to=3-3]
	\arrow["{\lambda_2}"', from=3-5, to=3-6]
\end{tikzcd} 
\end{equation*}
are epimorphisms. Let us prove that $h_0$ is an epimorphism, as showing $h_2$ is an epimorphism follows similarly. 

First, let us consider the set $Y = Y_0 +_{Y_1} Y_2$. This is precisely the set $Y_0 + Y_1/{\sim_{Y_1}}$, where $\sim_{Y_1}$ is the smallest equivalence relation containing the relation $\sim_{Y_1}'$, where for $y, y' \in Y_0 + Y_2$, we have $y \sim_{Y_1}' y'$ if and only if there exists a $y'' \in Y_1$ such that $g_0(y'') = y$ and $g_1(y'') = y'$.

Note that $h_0$ is the map $h_0(x_0) = (\alpha_0(x_0), [x_0])$, which belongs to $Y_0 \times_Y X \cong \{(y_0, [x]) \in Y_0 \times (X_0 +_{X_1} X_2) \, : \, [y_0] = \alpha[x] \},$ since $\alpha[x_0] = [\alpha_0(x_0)]$. Now given $(y_0, [x]) \in Y_0 \times_Y X$, we want to show that there exists a $x_0 \in X_0$ such that $h(x_0) = (y_0, [x])$. So choose some $x_0 \in [x]$. This means that $x_0 \in X_0 + X_2$. Suppose that $x_0 \in X_0$, for if $x_0 \in X_2$, the argument is similar. Then $[y_0] = \alpha[x_0]$ implies that $\alpha_0(x_0) \sim_{Y_1} y_0$. Assume that $\alpha_0(x_0) \neq y_0$, otherwise $h_0$ is surjective. Then since $\alpha_0(x_0) \sim_{Y_1} y_0$, and there exists a zig-zag of the form
\begin{equation*}
   \begin{tikzcd}
	& {Y_1} && {Y_1} && \dots \\
	{Y_0} && {Y_2} && {Y_0} && {Y_0}
	\arrow["{g_0}"', from=1-2, to=2-1]
	\arrow["{g_1}", from=1-2, to=2-3]
	\arrow["{g_1}"', from=1-4, to=2-3]
	\arrow["{g_0}", from=1-4, to=2-5]
	\arrow[from=1-6, to=2-5]
	\arrow["{g_0}", from=1-6, to=2-7]
\end{tikzcd} 
\end{equation*}
and elements $y^1_0, \dots, y^1_n$ such that $g_0(y^1_0) = \alpha_0(x_0)$, $g_1(y^1_0) = g_1(y^1_1)$, $g_0(y^1_1) = g_0(y^1_2)$, $\dots$, $g_1(y^1_{n-1}) = g_1(y^1_n)$ and $g_0(y^1_n) = y_0$. But if $g_0(y^1_0) = \alpha_0(x_0)$, then since the left-hand square in (\ref{eq cartesian map of pushouts}) is a pullback, there exists a $x_1 \in X_1$ such that $f_0(x_1) = x_0$ and $\alpha_1(x_1) = y^1_0$. But then $\alpha_2(f_1(x_1)) = g_1(\alpha_1(x_1)) = g_1(y^1_0) = g_1(y^1_1)$. But since the right-hand square in (\ref{eq cartesian map of pushouts}) is a pullback, that means that there is a $x_1' \in X_1$ such that $\alpha_1(x'_1) = y^1_1$ and $f_1(x'_1) = f_1(x_1)$. So $g_0(y^1_1) = g_0(\alpha_1(x'_1)) = \alpha_0(f_0(x'_1))$. Furthermore $x_0 \sim_{X_1} f_0(x'_1)$ because $x_0 = f_0(x_1)$, and $f_1(x_1) = f_1(x'_1)$. Continuing on in this way, we can show that there exists a $\widetilde{x} \in X_0$ such that $y_0 = \alpha_0(\widetilde{x})$ and such that $\widetilde{x} \sim_{X_1} \alpha_0(x_0)$. In other words, we have shown that $h_0$ is surjective. Thus $\ncat{Set}$ has weak effective pushouts.
\end{proof}

We leave the next result to the reader, it follows straightforwardly from the fact that colimits and limits are computed objectwise in presheaf topoi.

\begin{Lemma} \label{lem presheaf topoi satisfy weak descent}
If $\Pre(\cat{C})$ is a presheaf topos, then it satisfies weak descent.
\end{Lemma}

\begin{Prop}
If $\cat{E}$ is a Grothendieck topos, then $\cat{E}$ satisfies weak descent, i.e. $(T) \Rightarrow (D)'$.
\end{Prop}

\begin{proof}
If $\cat{E}$ is a Grothendieck topos, then by the Little Giraud Theorem \ref{th lex localizations <-> grothendieck toposes}, there exists a site $(\cat{C}, j)$ and a lex localization
\begin{equation*}
    \begin{tikzcd}
	{\Sh(\cat{C},j) \simeq \cat{E}} && {\Pre(\cat{C})}
	\arrow[""{name=0, anchor=center, inner sep=0}, "{\iota}"', shift right=2, hook, from=1-1, to=1-3]
	\arrow[""{name=1, anchor=center, inner sep=0}, "a"', shift right=2, from=1-3, to=1-1]
	\arrow["\dashv"{anchor=center, rotate=-90}, draw=none, from=1, to=0]
\end{tikzcd}
\end{equation*}
Let us show that $\Sh(\cat{C}, j)$ has universal colimits. Suppose we have a small diagram $X: I \to \Sh(\cat{C},j)$, and we abuse notation to let $X = \colim_i X(i)$. Suppose there is a map $f : Y \to X$, and consider the pullback 
\begin{equation*}
   \begin{tikzcd}
	{Y(i)} & Y \\
	{X(i)} & X
	\arrow["{\sigma_i}", from=1-1, to=1-2]
	\arrow["{f(i)}"', from=1-1, to=2-1]
	\arrow["\lrcorner"{anchor=center, pos=0.125}, draw=none, from=1-1, to=2-2]
	\arrow["f", from=1-2, to=2-2]
	\arrow["{\lambda_i}"', from=2-1, to=2-2]
\end{tikzcd} 
\end{equation*}
for every $i \in I$. We want to show that $\colim_i Y(i) \cong Y$. Now $\iota$ preserves the above pullback for every $i \in I$, and since presheaf topoi have universal colimits, we have $\colim_i \iota Y(i) \cong \iota Y$. Then 
\begin{equation*}
    \colim_i Y(i) \cong \colim_i a \iota Y(i) \cong a \left( \colim_i \iota Y(i) \right) \cong  a \iota Y \cong Y.
\end{equation*}
Thus $\Sh(\cat{C}, j) \simeq \cat{E}$ has universal colimits.

The proof for effective coproducts follows similarly.

Now let us show that $(T) \Rightarrow (D2b)'$. So suppose we have a cartesian natural transformation of pushouts in $\Sh(\cat{C}, j)$
\begin{equation*}
       \begin{tikzcd}
	{X_0} & {X_1} & {X_2} \\
	{Y_0} & {Y_1} & {Y_2}
	\arrow["{\alpha_0}"', from=1-1, to=2-1]
	\arrow["{f_0}"', from=1-2, to=1-1]
	\arrow["{f_1}", from=1-2, to=1-3]
	\arrow["\lrcorner"{anchor=center, pos=0.125, rotate=-90}, draw=none, from=1-2, to=2-1]
	\arrow["{\alpha_1}"', from=1-2, to=2-2]
	\arrow["\lrcorner"{anchor=center, pos=0.125}, draw=none, from=1-2, to=2-3]
	\arrow["{\alpha_2}", from=1-3, to=2-3]
	\arrow["{g_0}", from=2-2, to=2-1]
	\arrow["{g_1}"', from=2-2, to=2-3]
\end{tikzcd} 
\end{equation*}
Then applying $\iota$ to the diagrams
\begin{equation*}
   \begin{tikzcd}
	{X_0} &&& {X_2} \\
	& {Y_0 \times_Y X} & X && {Y_2 \times_Y X} & X \\
	& {Y_0} & Y && {Y_2} & Y
	\arrow["{h_0}", dashed, from=1-1, to=2-2]
	\arrow["{\sigma_0}", curve={height=-12pt}, from=1-1, to=2-3]
	\arrow["{\alpha_0}"', curve={height=12pt}, from=1-1, to=3-2]
	\arrow["{h_2}", dashed, from=1-4, to=2-5]
	\arrow["{\sigma_2}", curve={height=-12pt}, from=1-4, to=2-6]
	\arrow["{\alpha_2}"', curve={height=12pt}, from=1-4, to=3-5]
	\arrow[from=2-2, to=2-3]
	\arrow[from=2-2, to=3-2]
	\arrow["\lrcorner"{anchor=center, pos=0.125}, draw=none, from=2-2, to=3-3]
	\arrow["\alpha", from=2-3, to=3-3]
	\arrow[from=2-5, to=2-6]
	\arrow[from=2-5, to=3-5]
	\arrow["\lrcorner"{anchor=center, pos=0.125}, draw=none, from=2-5, to=3-6]
	\arrow["\alpha", from=2-6, to=3-6]
	\arrow["{\lambda_0}"', from=3-2, to=3-3]
	\arrow["{\lambda_2}"', from=3-5, to=3-6]
\end{tikzcd} 
\end{equation*}
we see that $\iota(h_0)$ and $\iota(h_2)$ are effective epimorphisms, since $\Pre(\cat{C})$ has weak effective pushouts. Since $\alpha$ preserves finite limits and colimits, it preserves effective epimorphisms. Thus $\alpha \iota(h_0) \cong h_0$ and $\alpha \iota(h_2) \cong h_2$ are effective epimorphisms.
\end{proof}

Thus we have finally arrived at the major goal of this section.

\begin{Th}[Giraud's Theorem] \label{th Giraud's Theorem}
Let $\cat{E}$ be a locally presentable category, then the following are equivalent:
\begin{enumerate}
    \item $\cat{E}$ satisfies weak descent $(D)'$,
    \item $\cat{E}$ satisfies Giraud's axioms $(G)$, and
    \item $\cat{E}$ is a Grothendieck topos $(T)$.
\end{enumerate}
\end{Th}

\subsection{Infinity Toposes}
In this very brief section, we discuss the major difference between Grothendieck $1$-toposes and $\infty$-toposes.

\begin{Rem}
By a Grothendieck $1$-topos, we mean what we have been calling a Grothendieck topos this whole time. For this section, we will reference $\infty$-categories, but will not define them, or motivate them. For that we recommend \cite{rezk2022introduction}, \cite{land2021introduction}, \cite{riehl2020homotopical}.
\end{Rem}

In the landmark book \cite{lurie2009higher}, Lurie defines $\infty$-topoi as follows: an $\infty$-category $\cat{E}$ is an $\infty$-topos if there is a small $\infty$-category $\cat{C}$ and an accessible\footnote{i.e. the right adjoint preserves filtered $\infty$-colimits} lex localization (in the $\infty$-category sense) $\Pre^\infty(\cat{C}) \to \cat{E}$, where $\Pre^\infty(\cat{C}) = \cat{S}^{\cat{C}^\op}$ is the $\infty$-category of $\infty$-functors from $\cat{C}^\op$ to the $\infty$-category $\cat{S}$ of spaces (up-to-homotopy equivalence).

When Lurie writes ``$\infty$-topos,'' he means Grothendieck $\infty$-topos. So he takes the Little Giraud Theorem \ref{th lex localizations <-> grothendieck toposes} as the starting point for his definition. He then proves the following theorem\footnote{building off of very similar work from \cite{rezk2010toposes} and \cite{toen2005homotopical} in the context of model categories and model topoi}

\begin{Th}[{\cite[Theorem 6.1.0.6]{lurie2009higher}}]
If $\cat{E}$ is an $\infty$-category, then the following statements are equivalent:
\begin{itemize}
    \item $\cat{E}$ is an $\infty$-topos,
    \item $\cat{E}$ satisfies the $\infty$-categorical analogues of Giraud's axioms:
    \begin{enumerate}
        \item $\cat{E}$ is a (locally) presentable $\infty$-category,
        \item $\cat{E}$ has universal $\infty$-colimits,
        \item $\cat{E}$ has disjoint $\infty$-coproducts, and
        \item $\cat{E}$ has effective groupoid objects.
    \end{enumerate}
\end{itemize}
\end{Th}

Now the $\infty$-categorical analogues of Giraud's axioms are straightforward except for (4). The idea being that in the $\infty$-category context, when one ``quotients,'' one needs to ``remember how one quotiented.'' In other words, for sets if we glue together a bunch of points, we have lost the information of \textit{how} the points were identified in the first place. Since equivalence relations are precisely internal groupoids with no automorphisms, suitably interpreted \cite{minichiello2021internalgroupoid}, in the $\infty$-category context, internal groupoids are understand as certain kinds of simplicial objects. So in the $\infty$-category context, quotienting is a much more sophisticated operation, and hence effective equivalence relations become effective groupoid objects.

Now we can restate the conditions for descent, i.e. Van Kampen colimits, and here, using $\infty$-categories actually makes everything easier to understand. We say that an $\infty$-category $\cat{E}$ satisfies descent if for every $\infty$-diagram the canonical $\infty$-adjunction
\begin{equation*}
 \begin{tikzcd}
	{\cat{E}_{/ \colim_i X(i)}} && {\lim_i \cat{E}_{/X(i)}}
	\arrow[""{name=0, anchor=center, inner sep=0}, "{{\text{cst}}}"', shift right=2, from=1-1, to=1-3]
	\arrow[""{name=1, anchor=center, inner sep=0}, "\colim"', shift right=2, from=1-3, to=1-1]
	\arrow["\dashv"{anchor=center, rotate=-90}, draw=none, from=1, to=0]
\end{tikzcd}	
\end{equation*}
is an equivalence of $\infty$-categories, where now the (co)limits shown above are $\infty$-(co)limits in $\cat{E}$ or in the (very large) $\infty$-category $\ncat{Cat}^\infty$ of large infinity categories.

\begin{Th}
An $\infty$-category $\cat{E}$ is an $\infty$-topos if and only if it is presentable and satisfies descent.
\end{Th}

Notice the difference here with Grothendieck $1$-toposes. The analogue of the above adjunction is not an equivalence in that case. Only weak descent holds.

\appendix

\section{Category Theory Background} \label{section background}

In this section, we provide basic material from category theory that will be important for these notes. We try to be detailed here without being overly pedantic. Much of this material is well-known, but we feel that it could be useful to the reader to gather all these results in one place.

\subsection{Set Theory} \label{section set theory}

Set theory is a subtle and oftentimes difficult subject for mathematicians to wrap their heads around. Much of these notes will not need to be concerned with the subtleties of set theory, however there are several areas where a concrete understanding of certain notions of set theory is absolutely vital, namely the theory of localizations (Section \ref{section localizations}), functor categories (Section \ref{section small and large categories}), large sites (Section \ref{section large sites}) and locally presentable categories (Section \ref{section locally presentable categories}).

In this section we include a bare-bones account of those concepts of set theory that we will need in what follows. We fix here the foundations of set theory using the axiomatic approach \cite{welch2020axiomatic}, \cite{kunen2014set}. We do not define what a set is, rather we use a formal system to manipulate symbols in a way that satisfies the Zermelo-Frankel axioms with choice (ZFC).

Let us say a few informal words on this. We build set theory using first order logic, which is a system for manipulating symbols. We introduce variable symbols $x_1, x_2, \dots$, logical symbols $\vee, \wedge, \to, \leftrightarrow, \neg, \exists, \forall$\footnote{Typically we take a subcollection of such symbols and define the rest of them in terms of this subcollection, for example we can define $\forall = \neg \exists \neg$.}, and two relation symbols $=, \in$. With these we can construct formulae, such as $\exists x_1.\forall x_2.((x_1 = x_2) \vee (x_1 \in x_2))$. Formulae have a notion of free variables, and those without free variables are called sentences. From such formulae we can construct what we'll call \textbf{collections}, but are more often called classes\footnote{For our purposes, we want to fix what it means to be a set, then use Grothendieck universes to distinguish between small and large sets, and we'd rather reserve the terminology for class for large sets.}, expressions of the form $\{ x_1 | \phi \}$, but these are also merely abbreviations, for example, the formula $x_2 \in \{x_1 | \phi \}$ is just the formula $\phi(x_2/x_1)$, where we have replaced all occurences of $x_1$ in $\phi$ with $x_2$. We can use the abbreviation $\varnothing = \{ x \, | \, \neg( x =  x) \}$ and $V = \{x \, | \, x = x \}$, which are meant to be the empty set and the collection of all sets, respectively. We state again that these are purely syntactic constructs, i.e. they are just symbols. From here we use a formal deduction system for first order logic, see \cite[Chapter 3]{johnstone1987notes}, which provides us with axioms and inference rules like modus ponens, and this allows us to make formal deductions of sentences from collections of sentences. Logicians typically write $S \vdash \phi$ to show that there exists a formal deduction of the sentence $\phi$ from the collection $S$ of sentences. We then add the axioms of ZFC to this deductive system, which say things like  $\forall x_1.\forall x_2.(\forall x_3.(x_3 \in x_1 \leftrightarrow x_2 \in x_1) \to (x_1 = x_2))$, which is called the axiom of extensionality.

This kind of description is overly pedantic for our purposes. The reader must merely accept that there is a completely formal way of manipulating collections $\{ x | \phi \}$, and that a collection is called a \textbf{set} if we can prove, using our deductive system, that $\{x | \phi \} \in V$. A collection that is not a set is called a \textbf{proper collection}.

We refer to \cite[Chapter 1]{kunen2014set} for a more in-depth and readable account of the foundations of set theory and the ZFC axioms. The notes \cite{shulman2008set} provides a readable account for other choices for foundations of set theory. The ZFC axioms guarantee that there exists an empty set $\varnothing$, that we have at our disposal the usual operations of set theory like union $\cup$, power set $P$, cartesian product $\times$, and intersection $\cap$. Note that from the axioms of ZFC, one can prove that the collection $V$ of all sets is not a set, and hence is a proper collection. We also work in classical mathematics, namely the law of excluded middle holds.

\subsubsection{Cardinals and Set Theoretic Considerations}
In what follows, we assume that we have fixed the foundations of set theory using ZFC as in the discussion above and know what a set is. This implies that all elements of a set are themselves sets.

\begin{Def} \label{def total order}
A \textbf{total order} on a set $S$ is a binary relation $\leq \, \subseteq S \times S$ satisfying the following conditions:
\begin{itemize}
    \item (\textbf{reflexivity}) for all $x \in S$, $x \leq x$,
    \item (\textbf{transitivity}) for all $x,y,z \in S$, if $x \leq y$ and $y \leq z$, then $x \leq z$,
    \item (\textbf{antisymmetry}) for all $x,y \in S$, if $x \leq y$ and $y \leq z$, then $x = y$,
    \item (\textbf{totality}) for all $x,y \in S$, if $x, y \in S$, then either $x \leq y$ or $y \leq x$.
\end{itemize}
We say that $S$ is a \textbf{preorder} if $\leq$ satisfies reflexivity and transitivity, and a \textbf{partial order} or \textbf{poset} if it is a preorder that also satisfies antisymmetry.
\end{Def}

\begin{Def} \label{def linear order}
A \textbf{linear order} on a set $S$ is a binary relation $< \, \subseteq S \times S$ satisfying the following conditions:
\begin{itemize}
    \item (\textbf{irreflexivity}) for all $x \in S$, $x \nless x$,
    \item (\textbf{trichotomy}) for all $x,y \in S$, exactly one of the following holds: $x < y$, $y < x$ or $x = y$. 
\end{itemize}
\end{Def}

There is a bijection between linear orders on a set $S$ and total orders. Indeed, if $\leq$ is a total order on $S$, then we can define a linear order $<$ on $S$ by declaring $x < y$ if and only if $y \nleq x$, and conversely $x \leq y$ if and only if $y \nless x$.

\begin{Ex}
Every set $S$ has a linear order given by the membership relation $\in \, \subseteq S \times S$. 
\end{Ex}

\begin{Def} \label{def well-order}
A \textbf{well-order} on a set $S$ is a linear order $<$ with the property that every nonempty subset $T \subseteq S$ has a \textbf{least element}, namely an element $t \in T$ such that there exists no $x \in T$ such that $x < t$.
\end{Def}

\begin{Rem}
A well-ordered set could equivalently be defined as a totally-ordered set such that every nonempty subset has a least element.
\end{Rem}

\begin{Def} \label{def transitive set}
We say that a set $S$ is \textbf{transitive} if whenever $y \in S$ and $x \in y$, then $x \in S$.
\end{Def}

\begin{Def} \label{def ordinal number}
An \textbf{ordinal} is a transitive set $S$ such that the membership relation $\in \, \subseteq S \times S$ is a well order.
\end{Def}

The collection of ordinal numbers can be obtained by \textbf{transfinite recursion}. Namely
\begin{itemize}
    \item the empty set $\varnothing$ is an ordinal,
    \item if $\alpha$ is an ordinal, then the \textbf{successor ordinal}, $\alpha + 1$, defined by $\alpha + 1 = \alpha \cup \{ \alpha \}$, is an ordinal, and
    \item given a set $I$ of ordinals, the set $\bigcup_{\alpha \in I} \alpha$ is an ordinal.
\end{itemize}
We say that an ordinal $\alpha$ is a successor ordinal if it is of the form $\alpha = \beta + 1$ for some ordinal $\beta$. We say that $\alpha$ is a \textbf{limit ordinal} if it is nonempty and not a successor ordinal.

We use Von Neumann's convention of labelling
\begin{equation*}
    0 = \varnothing, \qquad 1 = \{ 0 \}, \qquad 2 = \{0, 1 \}, \qquad 3 = \{0, 1, 2 \}, \dots
\end{equation*}
Clearly for $n > 0$, $n$ is a successor ordinal. We let $\omega = \N = \{0, 1, 2, \dots \}$ denote the first limit ordinal.

We say that two sets $A$ and $B$ have the same \textbf{cardinality} if there exists a bijection between them. Notice that there are many infinite ordinals with the same cardinality $$\omega, \; \omega + 1, \; \omega + 2, \; \dots, \; 2 \omega = \bigcup_{n \in \omega} \omega + n.$$
If $\alpha, \beta$ are ordinals, we say that $\alpha < \beta$ if $\alpha \in \beta$.

\begin{Def} \label{def cardinal number}
A \textbf{cardinal number} is an ordinal number $\kappa$ with the property that it is not in bijection with any ordinal number $\lambda$ such that $\lambda < \kappa$. 
\end{Def}

All finite ordinal numbers are cardinal numbers, and the first infinite cardinal number is $\omega$. Since the collection of cardinal numbers is also well-ordered, we label them as $\aleph_\alpha$ for an ordinal number $\alpha$. Thus we write $\aleph_0 = \omega$ for the first infinite cardinal number. If $\kappa$ is a cardinal number, then we write $\kappa^+$ to mean the smallest cardinal number that is greater than $\kappa$, and we call it the \textbf{successor cardinal}. Note that this is very different from the notion of successor ordinal. Indeed, $\omega^+ \neq \omega + 1$. However it is a tautology that $\aleph_{\alpha + 1} = \aleph_{\alpha}^+$. 

The question of the truth of the statement $\omega^+ = 2^\omega$, where $2^\omega$ denotes the powerset of $\omega$, is precisely the Continuum Hypothesis. The question of the truth of the statement of $\aleph_\alpha^+ = 2^{\aleph_\alpha}$ is called the Generalized Continuum Hypothesis. Note that $\R$ is in bijection with $2^{\aleph_0}$. 

The following result goes under the name of the \textbf{Well-Ordering Theorem}, it was proven by Zermelo in 1904, and can be found in any introductory text on set theory.

\begin{Th}
In ZFC, every set can be well-ordered.
\end{Th}

It can be shown by transfinite induction that every well-ordered set is order-isomorphic to a unique ordinal number.

\begin{Cor}
Every set is in bijection with a unique cardinal number.
\end{Cor}

\begin{Rem}
Let $S$ be a set, and $\kappa$ a cardinal. We will write $|S| < \kappa$ to mean that if $\lambda$ is the unique cardinal number that is in bijection with $S$, then $\lambda < \kappa$. 
\end{Rem}

\begin{Def} \label{def regular cardinal}
An infinite cardinal $\kappa$ is said to be \textbf{regular} if the following condition holds: If $I$ is a set such that $|I| < \kappa$ and $\{ S_i \}_{i \in I}$ is an $I$-indexed family of sets such that $|S_i| < \kappa$ for all $i \in I$, then for $S = \bigcup_{i \in I} S_i$, it follows that $|S| < \kappa$. We say a cardinal $\kappa$ is \textbf{singular} if it is not regular. In other words, a cardinal $\kappa$ is singular if it can be written as $\kappa = \sum_{i < \alpha} \kappa_i$ where $\kappa_i < \kappa$ and $\alpha < \kappa$.
\end{Def}

\begin{Rem} \label{rem regular cardinals}
Note that $\omega = \aleph_0$ is a regular cardinal, as a union of finitely many finite sets is always finite. Every successor cardinal $\aleph_0, \, \aleph_1, \, \dots$ is regular. The cardinal $\aleph_{\omega} = \sum_{i < \omega} \aleph_i$ however is singular.

Regular cardinals are well-behaved in the following sense. Let $\ncat{Set}_\kappa$ denote the category of all sets with cardinality strictly less than $\kappa$. Then this category is closed under unions indexed by sets in $\ncat{Set}_\kappa$. This is not true for singular cardinals. In other words, if $\kappa$ is singular, then $\ncat{Set}_{\kappa}$ is not $\kappa$-cocomplete (Definition \ref{def small diagrams and completeness})! Hence for purposes of category theory, it is best that we stick to regular cardinals.
\end{Rem}

\begin{Def} \label{def strongly inaccessible cardinal}
A cardinal number $\kappa$ is $\textbf{strongly inaccessible}$ if it is uncountable, regular and for every $\lambda < \kappa$ it follows that $2^\lambda < \kappa$.
\end{Def}

The name strongly inaccessible cardinal is justified by the above definition, as it implies that if $\kappa$ is strongly inaccessible, then it cannot be proved to exist using the ZFC axioms \cite[Section 2.6]{welch2020axiomatic}.

\begin{Def} \label{def grothendieck universe}
A \textbf{Grothendieck universe} is a set $\mathbb{U}$ with the following properties:
\begin{enumerate}
    \item it is a transitive set (Definition \ref{def transitive set}),
    \item if $x \in \mathbb{U}$ and $y \in \mathbb{U}$, then $\{ x, y \} \in \mathbb{U}$,
    \item if $x \in \mathbb{U}$, then its powerset, $P(x) \in \mathbb{U}$,
    \item if $x \in \mathbb{U}$, and $f: x \to \mathbb{U}$ is a function, then $\bigcup_{i \in x} f(i) \in \mathbb{U}$, and
    \item the first limit ordinal $\omega \in \mathbb{U}$.
\end{enumerate}
\end{Def}

We take the following statement as an axiom. It was first postulated by Grothendieck and Verdier.
\begin{equation}
    \text{For every set $x$, there exists a Grothendieck universe $\mathbb{U}$ such that $x \in \mathbb{U}$.}
\end{equation}
This is often called the \textbf{universe axiom}. The use of Grothendieck universes is pervasive in modern mathematics, as it is one of the quickest ways to couch one's mathematical discussions in rigor without having to worry much about set-theoretic details. It is convenient to note that Grothendieck universes are equivalent to strongly inaccessible cardinals in the following sense.

\begin{Def} \label{def cumulative hierarchy and rank}
Define a set $V_\alpha$ for every ordinal number $\alpha$ by transfinite induction as follows.
\begin{equation}
\begin{aligned}
    V_0 &= \varnothing \\
    V_{\alpha + 1} &= 2^{V_{\alpha}} \\
    V_{\beta} & = \bigcup_{\alpha < \beta} V_{\alpha}, \qquad \qquad (\beta \text{ a limit ordinal}).
\end{aligned}
\end{equation}
The collection of sets $V_\alpha$ is referred to as the \textbf{cumulative hierarchy} or \textbf{Von Neumann universe}. One of the axioms of ZFC, the axiom of foundation, is equivalent to the statement that every set is an element of some $V_\alpha$. 

The \textbf{rank} of a set $S$ is the smallest $\alpha$ for which $S \in V_{\alpha + 1}$. It can be proved by transfinite induction that the rank of an ordinal $\alpha$ is $\alpha$. Thus $V_\kappa$ is the set of all sets with rank $< \kappa$.
\end{Def}

\begin{Th}[\cite{Williams_1969}]
If $\mathbb{U}$ is a Grothendieck universe, then it is of the form $\mathbb{U} = V_\kappa$ for $\kappa$ a strongly inaccessible cardinal number.
\end{Th}

Thus the universe axiom is equivalent to the statement: for every cardinal $\lambda$ there exists a strongly inaccessible cardinal $\kappa$ with $\lambda < \kappa$.

\begin{Rem}
It can be shown that for $\kappa$ inaccessible, $V_\kappa$ is a model of ZFC. Note that this means that all sets that are elements of $V_\kappa$ have bounded rank, and not cardinality. This is an important distinction, see \cite[Section 8]{shulman2008set} for more information.
\end{Rem}

\begin{Def}
Given a Grothendieck universe $\mathbb{U}$, let $\kappa(\mathbb{U})$ denote the set of ordinal numbers that are elements of $\mathbb{U}$.
\end{Def}

\begin{Lemma}[{\cite[Proposition 0.1.16]{Low_homotopical_2015}}]
If $\mathbb{U}$ is a Grothendieck universe, then $\kappa(\mathbb{U})$ is an ordinal number such that $\kappa(\mathbb{U}) \notin \mathbb{U}$. Further, if $\mathbb{U} = V_\kappa$ for a strongly inaccessible cardinal $\kappa$, then $\kappa(\mathbb{U}) = \kappa$.
\end{Lemma}

\begin{Def}
Fix two Grothendieck universes $\mathbb{U}$ and $\mathbb{V}$ such that $\mathbb{U} \in \mathbb{V}$. We say that a set $x$ is \textbf{small} if it is in bijection with a set $y$ such that $y \in \mathbb{U}$, \textbf{large} if it is in bijection with a set $y'$ such that $y' \in \mathbb{V}$, and \textbf{very large} otherwise. We will sometimes refer to a large set as a \textbf{class}.
\end{Def}

\begin{Rem}
Some authors prefer to use a single Grothendieck universe $\mathbb{U}$, and say that elements of $\mathbb{U}$ are small and \textit{subsets} of $\mathbb{U}$ are large. The following lemma shows that such large sets are still large in our context with multiple Grothendieck universes.
\end{Rem}

\begin{Lemma} \label{lem small classes are large sets}
If $x \subseteq \mathbb{U}$, then $x \in \mathbb{V}$.
\end{Lemma}

\begin{proof}
First note that $x \in P(\mathbb{U})$ since $x \subseteq \mathbb{U}$. But since $\mathbb{U} \in \mathbb{V}$, and $\mathbb{V}$ is a Grothendieck universe, this implies that $P(\mathbb{U}) \in \mathbb{V}$ by Definition \ref{def grothendieck universe}.(3). Therefore $x \in \mathbb{V}$ by Definition \ref{def grothendieck universe}.(1).
\end{proof}


\subsubsection{Small and Large Categories} \label{section small and large categories}
Now with the convenient framework of Grothendieck universes, we can now turn to set theory used in category theory. We fix two nested Grothendieck universes $\mathbb{U} \in \mathbb{V}$. Any set mentioned below without disambiguation will always mean a small set, as will any cardinal.

\begin{Def} \label{def small and large categories}
A \textbf{small category} $\cat{C}$ consists of a small set $\Obj(\cat{C})$ of objects and a small set $\Mor(\cat{C})$ of morphisms, equipped with the usual structure defining a category. A \textbf{large category} consists of a large set of objects and a large set of morphisms. When we say a category in these notes, we mean a large category, and use both terms interchangeably. A \textbf{very large category} consists of very large sets of objects and morphisms. A \textbf{locally small category} is a large category $\cat{C}$ such that for every pair of objects $x,y \in \cat{C}$, the set $\cat{C}(x,y)$ of morphisms from $x$ to $y$ is small. We say a category $\cat{C}$ is \textbf{essentially small} if it is equivalent to a small category. In what follows we will treat essentially small categories as if they were small.
\end{Def}

\begin{Def}
Given a regular cardinal $\kappa$, we say a set $S$ is $\kappa$-small if $|S| < \kappa$. If $\kappa$ is a regular cardinal, we say that a category $\cat{C}$ is $\kappa$-small if its set of objects and each hom-set is $\kappa$-small. Note that $\kappa$ being regular implies that $\Mor(\cat{C})$ is $\kappa$-small, which provides another advantage to working with regular cardinals.
\end{Def}

\begin{Ex}
Let $\ncat{Set}$ denote the category of small sets. This is a large category that is locally small, as the set of functions between two small sets $A$ and $B$ is small since it can be encoded as a certain subset of $A \times B$. However $\ncat{Set}$ is not essentially small. Indeed, suppose it was, then there would be a small set $S \in \mathbb{U}$ such that every set in $\mathbb{U}$ is in bijection with a set in $S$. This would imply that $\mathbb{U} \in \mathbb{U}$, which is a contradiction. 

Similarly, the category $\ncat{Cat}$ of small categories is locally small. However, the category $\ncat{CAT}$ of large categories is a very large category whose hom-sets are large sets.
\end{Ex}

\begin{Ex} \label{ex Man is essentially small}
The category $\ncat{Man}$ whose objects are finite dimensional smooth manifolds and whose morphisms are smooth maps is a locally small category, but it is also essentially small. Indeed, every smooth manifold can be embedded into $\R^\infty$ by the Whitney Embedding Theorem. Thus we can consider the category whose objects are finite dimensional smooth manifolds embedded in $\R^\infty$, and this will be a small category equivalent to $\ncat{Man}$. Similarly the category $\ncat{Cart}$ whose objects are finite dimensional smooth manifolds diffeomorphic to $\R^n$ for some $n \geq 0$ and whose morphisms are smooth maps, is a large category that is essentially small. In fact it is equivalent to the category whose objects are $\R^n$ for $n \geq 0$ with smooth maps between them.
\end{Ex}

While the above example is convenient, and the proof straightforward, it does not generalize well to other kinds of manifolds. Let us remedy that.

\begin{Ex} \label{ex topman is essentially small}
Let $\ncat{TopMan}$ denote the category whose objects are finite dimensional topological manifolds and whose morphisms are continuous functions. Let us show that $\ncat{TopMan}$ is essentially small.

For us a topological manifold $M$ is a topological space that is Hausdorff, paracompact and locally Euclidean. Topological manifolds can only have countably many connected components, and each component is path connected.

By \cite{stephens2011cardinalitymanifolds}, every connected topological manifold has cardinality at most $\mathfrak{c} = 2^{\aleph_0} \cong |\R|$. Hence by \cite{karagila2017cardinalityofunion}, assuming the axiom of choice, the cardinality of $M$ is also at most $\mathfrak{c}= 2^{\aleph_0}$.

For each set $X$ of cardinality $\mathfrak{c}$, there are at most $2^{\mathfrak{c}}$ topologies on $X$. Hence there can be at most $2^{\mathfrak{c}}$ many non-isomorphic topological manifolds. Hence $\ncat{TopMan}$ is essentially small, by using the axiom of choice to choose one topological manifold from each isomorphism class.

Now a finite dimensional smooth manifold $M$ is a topological manifold equipped with a smooth structure. We want to show that there is a set of isomorphism classes of smooth manifolds. A smooth atlas $\mathcal{A}$ for $M$ consists of a set of topological embeddings $\{\varphi_i : \R^n \to M \}_{i \in I}$ such that the transition functions $\varphi_j^{-1}\varphi_i : \R^n \to \R^n$ are smooth. We say that $\varphi_i$ and $\varphi_j$ are compatible. Two atlases $\mathcal{A}, \mathcal{A}'$ are deemed equivalent if every pair $\varphi \in \mathcal{A}$ and $\psi \in \mathcal{A}'$ are compatible. A smooth structure on $M$ is an equivalence class of atlases. Using the fact that $M$ is second countable, every atlas is equivalent to a countable atlas. Hence the isomorphism type of a smooth manifold is determined by an underlying topological manifold and a countable atlas. There are precisely $\mathfrak{c}^{\mathfrak{c}}$ functions from $\R^n$ to $M$. Hence there are at most $\sum_{\omega} \mathfrak{c}^{\mathfrak{c}}$ distinct smooth atlases on a topological manifold and hence there are at most$\sum_\omega \mathfrak{c}^{\mathfrak{c}} \cdot \mathfrak{c} = \sum_\omega \mathfrak{c}^{\mathfrak{c}}$ distinct isomorphism classes of finite dimensional smooth manifolds, which is the cardinality of a small set. Hence the category $\ncat{Man}$ is essentially small.\footnote{We thank Kevin Carlson and John Baez for discussion around this example.}
\end{Ex}

\begin{Ex}
We will see in Section \ref{section localizations} that the localization of a locally small category $\cat{C}$ at a class of morphisms $W$ can lead to a large category $\cat{D}$ that is not locally small. Finding ways to ensure that $\cat{D}$ is locally small is a huge motivation for studying set-theoretical techniques in category theory.
\end{Ex}

\begin{Lemma} \label{lem smallness of functor categories}
Let $\cat{C}$ and $\cat{D}$ be categories, if
\begin{enumerate}
    \item $\cat{C}$ and $\cat{D}$ are small, then the category $\cons{Fun}(\cat{C}, \cat{D})$ whose objects are functors and whose morphisms are natural transformations, is a small category,
    \item $\cat{C}$ is small and $\cat{D}$ is locally small, then $\cons{Fun}(\cat{C}, \cat{D})$ is locally small,
    \item $\cat{C}$ and $\cat{D}$ are large categories, then $\Fun(\cat{C}, \cat{D})$ is a large category,
    \item $\cat{C}$ and $\cons{Fun}(\cat{C}, \ncat{Set})$ are locally small, then $\cat{C}$ is essentially small.
\end{enumerate}
\end{Lemma}

\begin{proof}
(1) The collection of functors $F : \cat{C} \to \cat{D}$ can be encoded as a relation on $\Obj(\cat{C}) \times \Obj(\cat{D})$, which is a small set. Similarly the collection of natural transformations can be encoded as a relation on $\Obj(\cat{C}) \times \Mor(\cat{D})$, which is a small set.
(2) Using the same argument as (1), we see $\Fun(\cat{C}, \cat{D})$ is locally small.
(3) Using the same argument as (1), we see $\Fun(\cat{C}, \cat{D})$ is large.
(4) This is \cite{freyd1995size}.
\end{proof}

\begin{Rem}
If $\cat{C}$ and $\cat{D}$ are very large categories, then there is no guarantee that the functor category $\Fun(\cat{C},\cat{D})$ exists in the sense that even if we can construct its collections of objects and morphisms, we may not be able to prove that these collections are sets.
\end{Rem}

\begin{Def} \label{def small diagrams and completeness}
Given a category $\cat{C}$, a \textbf{small diagram} in $\cat{C}$ is a functor $F: I \to \cat{C}$ from a small category $I$. Similarly for a regular cardinal $\kappa$, a $\kappa$-small diagram is a functor whose domain is a category that is $\kappa$-small. We say a category $\cat{C}$ is \textbf{$\kappa$-(co)complete} if it admits (co)limits over all $\kappa$-small diagrams. We say $\cat{C}$ is \textbf{(co)complete} if it admits (co)limits over all small diagrams.
\end{Def}

Being careful about set-theoretic issues in category theory is necessary, as the following result shows.

\begin{Prop}[Freyd's Theorem] \label{prop freyd's theorem}
Let $\cat{C}$ be a small, complete category. Then $\cat{C}$ is a preorder.
\end{Prop}

\begin{proof}
Suppose that $\cat{C}$ is small and complete. Let $\kappa$ be the unique cardinal such that $| \Mor \cat{C} | = \kappa$. Suppose that $\cat{C}$ is not a preorder, namely given two objects $c$ and $d$ in $\cat{C}$, there exist two distinct morphisms $f, g: c \to d$. Consider the product $\prod_{\kappa} d$. Then there are $2^\kappa$ many morphisms $c \to \prod_{\kappa} d$, since there are $2^\kappa$ many $\kappa$-indexed sequences of the form $(a_0, a_1, \dots)$ where $a_i \in \{f, g \}$. Since $2^\kappa > \kappa$, this implies that $| \Mor \cat{C} | > \kappa$, contradiction.  
\end{proof}

\subsection{The Zoo of Monos and Epis}

\subsubsection{Monos}

\begin{Def} \label{def monomorphism}
A monomorphism $f : X \to Y$ in a category $\cat{C}$ is a morphism with the property that if $g,h : Z \to X$ are morphisms such that $fg = fh$, then $g = h$.
\end{Def}

\begin{Def} \label{def subobject}
Let $\cat{C}$ be a category and $U \in \cat{C}$. A \textbf{subobject} of $U$ is an equivalence class of monomorphisms $i: V \hookrightarrow U$, where identify two monomorphisms $i: V \hookrightarrow U$ and $i': V' \hookrightarrow U$ if there is an isomorphism $\varphi : V \to V'$ such that $i' \varphi = i$. We typically denote an equivalence class $[i: V \hookrightarrow U]$ by $V \subseteq U$.
\end{Def}

\begin{Def}
We say that a category $\cat{C}$ is \textbf{well-powered} if for every $U \in \cat{C}$, the subobjects of $U$ form a small set. Note that every small category is well-powered.
\end{Def}

\begin{Def} \label{def subobject poset}
Suppose that $\cat{C}$ is well-powered, with $U \in \cat{C}$. Let $\cons{Sub}(U)$ denote the (small) poset (with order denoted by $\subseteq$) of subobjects\footnote{We can equivalently consider it as the posetal reflection of the category of monomorphisms over $U$.} of $U$, where for subobjects $V, W \subseteq U$, then $V \subseteq W$ if some monomorphism $i: V \hookrightarrow U$ representing $V$ factors through some monomorphism $j: W \hookrightarrow U$ representing $W$. It is not hard to see this is well defined and makes $\cons{Sub}(U)$ into a poset. 

This extends to a functor $\cons{Sub} : \cat{C}^\op \to \ncat{Set}$\footnote{we can also put the codomain to be the category of posets}, if $f: W \to U$ is a morphism in $\cat{C}$, then set $\cons{Sub}(f)(V \subseteq U) = (f^*V \subseteq W)$ to be the subobject given by the pullback
\begin{equation*}
\begin{tikzcd}
	{f^*(V)} & V \\
	W & U
	\arrow[from=1-1, to=1-2]
	\arrow[hook', from=1-1, to=2-1]
	\arrow["\lrcorner"{anchor=center, pos=0.125}, draw=none, from=1-1, to=2-2]
	\arrow[hook', from=1-2, to=2-2]
	\arrow["f", from=2-1, to=2-2]
\end{tikzcd}
\end{equation*}
Note what we have written above is not technically correct, since we should have chosen a representative monomorphism for $V \subseteq U$ in order for the pullback to make sense. This is something we will do often, as it can be easily checked to be well defined, because pullbacks are only defined up to isomorphism anyway, and pullbacks of monomorphisms are monomorphisms.
\end{Def}

\begin{Lemma} \label{lem meet of subobjects}
If $\cat{C}$ is a well-powered category with finite limits, then for every $U \in \cat{C}$, $\cons{Sub}(U)$ is a meet-semilattice. Namely for every pair $V, W \hookrightarrow U$ of subobjects of $U$, the pullback
\begin{equation*}
\begin{tikzcd}
	{V \cap W} & W \\
	V & U
	\arrow[hook, from=2-1, to=2-2]
	\arrow[hook', from=1-2, to=2-2]
	\arrow[hook', from=1-1, to=2-1]
	\arrow[hook, from=1-1, to=1-2]
	\arrow["\lrcorner"{anchor=center, pos=0.125}, draw=none, from=1-1, to=2-2]
\end{tikzcd}
\end{equation*}
is a subobject of $U$ such that $V \cap W \subseteq V$, $V \cap W \subseteq W$ and for any other subobject $X$ such that $X \subseteq V$, $X \subseteq W$, then $V \cap W \subseteq X$. In other words $V \cap W$ is the meet of $V$ and $W$ in the poset $\cons{Sub}(U)$. We call $V \cap W$ the \textbf{intersection} of $V$ and $W$ as subobjects.
\end{Lemma}

\begin{Lemma}[{\cite[Proposition 0.4]{nlab:subobject}}] \label{lem join of subobjects}
If $\cat{C}$ is a well-powered category with finite limits and universal finite colimits (see Section \ref{section girauds theorem}), then for every $U \in \cat{C}$, the subobject poset $\cons{Sub}(U)$ also has binary joins given as follows. If $V, W \subseteq U$ are subobjects of $U$, then consider the pushout
\begin{equation*}
    \begin{tikzcd}
	{V \cap W} & W \\
	V & {V \cup W}
	\arrow[from=1-1, to=1-2]
	\arrow[from=1-1, to=2-1]
	\arrow[from=1-2, to=2-2]
	\arrow[from=2-1, to=2-2]
	\arrow["\lrcorner"{anchor=center, pos=0.125, rotate=180}, draw=none, from=2-2, to=1-1]
\end{tikzcd}
\end{equation*}
Then $V \cup W$ is the join of $V$ and $W$ in $\cons{Sub}(U)$, and we call it the \textbf{union} of $V$ and $W$ as subobjects.
\end{Lemma}

\begin{Rem} \label{rem categorical logic}
Note that $\cat{C}$ having finite limits and colimits is not enough for $\cons{Sub}$ to take values in lattices. For that to be the case we must also demand that images be preserved by pullbacks. As we ask for $\cons{Sub}$ to take values in more structured lattices (like distributive lattices, Heyting algebras, Boolean algebras), we must demand more structure on the category $\cat{C}$ (like being a coherent category, Heyting category and Boolean category). The study of this correspondence between structure on subobjects and structure on the category is a part of \textbf{categorical logic}. We will not delve too far into this subject, but the reader should be aware of its deep connections with Grothendieck topos theory. See \cite{awodey2009introduction}, \cite[Section D1]{johnstone2002sketches}.
\end{Rem}

\begin{Def} \label{def image}
Given a category $\cat{C}$ and a morphism $f: U \to V$, an \textbf{image} of $f$ is a monomorphism $m : A \hookrightarrow V$ that $f$ factors through such that if $m' : B \hookrightarrow V$ is another monomorphism through which $f$ factors, then there exists a unique monomorphism $i : A \hookrightarrow B$ such that the following diagram commutes
\begin{equation*}
    \begin{tikzcd}
	& X \\
	A && B \\
	& Y
	\arrow[from=1-2, to=2-1]
	\arrow[from=1-2, to=2-3]
	\arrow["i", hook, from=2-1, to=2-3]
	\arrow["m"', hook, from=2-1, to=3-2]
	\arrow["{m'}", hook', from=2-3, to=3-2]
\end{tikzcd}
\end{equation*}
In other words, $[m]$ is the smallest subobject of $V$ through which $f$ factors. If it exists, then it is clearly unique up to unique isomorphism.
\end{Def}

\begin{Def} \label{def regular mono}
We say that a monomorphism $f : U \to V$ in a category $\cat{C}$ is \textbf{regular} if there exist maps $g,h: V \to W$ in $\cat{C}$ such that
\begin{equation*}
\begin{tikzcd}
	U & V & W
	\arrow["f", from=1-1, to=1-2]
	\arrow["g", shift left, from=1-2, to=1-3]
	\arrow["h"', shift right, from=1-2, to=1-3]
\end{tikzcd}
\end{equation*}
is an equalizer.
\end{Def}

\begin{Def} \label{def cokernel pair}
Given a category $\cat{C}$ and a morphism $f : U \to V$ in $\cat{C}$, we call a pushout of the form
\begin{equation*}
    \begin{tikzcd}
	U & V \\
	V & {V +_U V}
	\arrow["f"', from=1-1, to=2-1]
	\arrow["f", from=1-1, to=1-2]
	\arrow["{i_0}"', from=2-1, to=2-2]
	\arrow["{i_1}", from=1-2, to=2-2]
	\arrow["\lrcorner"{anchor=center, pos=0.125, rotate=180}, draw=none, from=2-2, to=1-1]
\end{tikzcd}
\end{equation*}
a \textbf{cokernel pair} of $f$, if it exists.
\end{Def}

\begin{Def} \label{def regular image}
Given a category $\cat{C}$ with cokernel pairs and equalizers and a morphism $f: U \to V$, we call an equalizer of the form
\begin{equation*}
\begin{tikzcd}
	{\im_{\text{reg}}(f)} & V & {V +_U V}
	\arrow["\iota_f", hook, from=1-1, to=1-2]
	\arrow["{i_0}", shift left, from=1-2, to=1-3]
	\arrow["{i_1}"', shift right, from=1-2, to=1-3]
\end{tikzcd}
\end{equation*}
a \textbf{regular image} of $f$.
\end{Def}

\begin{Def} \label{def effective mono}
We say that a monomorphism $f: U \to V$ in a category $\cat{C}$ is \textbf{effective} if the following diagram
\begin{equation*}
\begin{tikzcd}
	U & V & {V +_U V}
	\arrow["f", from=1-1, to=1-2]
	\arrow["{{i_0}}", shift left, from=1-2, to=1-3]
	\arrow["{{i_1}}"', shift right, from=1-2, to=1-3]
\end{tikzcd}   
\end{equation*}
is an equalizer, i.e. $f$ is a regular image of itself.
\end{Def}

\begin{Rem} \label{rem motivation for regular and effective monos}
The differences between monos, regular monos and effective monos might seem abstract and arbitrary at first glance, but there are important differences. Regular and effective monomorphisms behave better in general than arbitrary monomorphisms. To not get ourselves too distracted on this point, we defer to \cite{Yuan2012} and \cite{LowQuotients}.
\end{Rem}

\begin{Lemma} \label{lem sufficient condition for reg mono to be effective}
If $\cat{C}$ is a category with equalizers and cokernel pairs, then a monomorphism in $\cat{C}$ is regular if and only if it is effective.    
\end{Lemma}

\begin{proof}
$(\Leftarrow)$ This is clear, all effective epimorphisms are regular.

$(\Rightarrow)$ Suppose that $f : U \to V$ is a regular monomorphism, given as an equalizer
\begin{equation*}
\begin{tikzcd}
	U & V & W
	\arrow["f", from=1-1, to=1-2]
	\arrow["p", shift left, from=1-2, to=1-3]
	\arrow["q"', shift right, from=1-2, to=1-3]
\end{tikzcd}    
\end{equation*}
Then by the universal property of pushouts we obtain a unique map
\begin{equation*}
\begin{tikzcd}
	U & V \\
	V & {V +_U V} \\
	&& W
	\arrow["f", from=1-1, to=1-2]
	\arrow["f"', from=1-1, to=2-1]
	\arrow["{i_1}", from=1-2, to=2-2]
	\arrow["q", curve={height=-12pt}, from=1-2, to=3-3]
	\arrow["{i_0}", from=2-1, to=2-2]
	\arrow["p"', curve={height=12pt}, from=2-1, to=3-3]
	\arrow["{{\exists! h}}"{description}, dashed, from=2-2, to=3-3]
\end{tikzcd}
\end{equation*}
But if we let
\begin{equation*}
\begin{tikzcd}
	{\im_{\text{reg}}(f)} & V & {V+_UV}
	\arrow["{\iota_f}", from=1-1, to=1-2]
	\arrow["{i_0}", shift left, from=1-2, to=1-3]
	\arrow["{i_1}"', shift right, from=1-2, to=1-3]
\end{tikzcd}
\end{equation*}
denote the regular image, then
\begin{equation*}
   q \iota_f = h i_1 \iota_f = h i_0 \iota_f = p \iota_f
\end{equation*}
so $\iota_f$ factors uniquely through $f$ since it is an equalizer of $p$ and $q$, and it is easy to see that this map provides an inverse to $h$. Thus $f$ is an effective monomorphism.
\end{proof}

\begin{Lemma} \label{lem regular image is image in cat with pushouts and cokernal pairs}
If $\cat{C}$ is a category with pushouts and cokernel pairs, $f : U \to V$ is a morphism in $\cat{C}$, and all monomorphisms in $\cat{C}$ are regular monomorphisms, then the regular image of $f$ is an image of $f$.
\end{Lemma}

\begin{proof}
If $f = me$ for some monomorphism $m : A \hookrightarrow V$, then since all monomorphisms in $\cat{C}$ are regular, $m$ is regular, and by Lemma \ref{lem sufficient condition for reg mono to be effective}, $m$ is effective. So the following diagram is an equalizer
\begin{equation*}
    \begin{tikzcd}[ampersand replacement=\&]
	A \& V \& {V+_AV}
	\arrow["m", hook, from=1-1, to=1-2]
	\arrow["s", shift left=2, from=1-2, to=1-3]
	\arrow["t"', shift right=2, from=1-2, to=1-3]
\end{tikzcd}
\end{equation*}
But by the universal property of pushouts, we get a unique map $h : V+_U V \to V+_A V$ making the following solid diagram commute
\begin{equation*}
 \begin{tikzcd}[ampersand replacement=\&]
	{\im_{\reg}(f)} \& V \& {V +_U V} \\
	A \& V \& {V+_AV}
	\arrow["i", hook, from=1-1, to=1-2]
	\arrow["k"', dashed, hook, from=1-1, to=2-1]
	\arrow["p", shift left=2, from=1-2, to=1-3]
	\arrow["q"', shift right=2, from=1-2, to=1-3]
	\arrow[equals, from=1-2, to=2-2]
	\arrow["h", from=1-3, to=2-3]
	\arrow["m", hook, from=2-1, to=2-2]
	\arrow["s", shift left=2, from=2-2, to=2-3]
	\arrow["t"', shift right=2, from=2-2, to=2-3]
\end{tikzcd}
\end{equation*}
Now $m$ is an equalizer, and $si = hpi = hqi = ti$, so by the universal property of the equalizer we obtain a uniqued dotted map $k$ making the above diagram commute. Since $i$ is a monomorphism, so is $k$. Thus, every monomorphism factoring through $f$ must factor through $i$, and hence $\im_\reg(f) \hookrightarrow V$ is an image of $f$.
\end{proof}

\begin{Lemma} \label{lem monos are regular in Set}
In $\ncat{Set}$ all monomorphisms are regular and effective.
\end{Lemma}

\begin{proof}
Let $i : S \hookrightarrow T$ be an injective function. We want to show that the following diagram is an equalizer
\begin{equation*}
    \begin{tikzcd}[ampersand replacement=\&]
	S \& T \& {T+_ST}
	\arrow["i", hook, from=1-1, to=1-2]
	\arrow["p", shift left, from=1-2, to=1-3]
	\arrow["q"', shift right, from=1-2, to=1-3]
\end{tikzcd}
\end{equation*}
where $p, q: T \to T+_S T$ are just the quotient maps. Suppose that $f : A \to T$ is a function such that $pf = qf$. Then for every $a \in A$, there exists an $x_a \in S$ such that $i(x_a) = f(a)$, and since $i$ is injective, this $x_a$ is unique. Thus the map $x : A \to S$ given by $x(a) = x_a$ defines a unique map such that $ix = a$, and hence $i$ is an equalizer.
\end{proof}

\begin{Cor} \label{cor monos in presheaf toposes}
Monomorphisms in any presheaf topos $\Pre(\cat{C})$ are regular and effective.
\end{Cor}

\begin{Def} \label{def subobject classifier}
Given a category $\cat{C}$ with finite limits, we say a monomorphism $t: 1 \hookrightarrow \Omega$ out of the terminal object of $\cat{C}$ is a \textbf{subobject classifier} for $\cat{C}$ if given any object $U \in \cat{C}$ and any subobject $V \hookrightarrow U$, there is a unique morphism $\chi_V: U \to \Omega$ such that the following diagram commutes and is a pullback:
\begin{equation*}
    \begin{tikzcd}
	V & 1 \\
	U & \Omega
	\arrow[hook, from=1-1, to=2-1]
	\arrow["{\chi_V}"', from=2-1, to=2-2]
	\arrow["t", hook, from=1-2, to=2-2]
	\arrow[from=1-1, to=1-2]
	\arrow["\lrcorner"{anchor=center, pos=0.125}, draw=none, from=1-1, to=2-2]
\end{tikzcd}
\end{equation*}
\end{Def}

\begin{Lemma}
Given a (locally small) category $\cat{C}$ with finite limits, then $\Omega \in \cat{C}$ is a subobject classifier for $\cat{C}$ if and only if it is a representing object for the subobject presheaf defined above, namely there is an isomorphism
\begin{equation}
    \cons{Sub}(U) \cong \cat{C}(U, \Omega), 
\end{equation}
natural in $U$.
\end{Lemma}

\begin{Ex}
In the category $\cat{C} = \ncat{Set}$, it is easy to see that the subobject classifier is the set $2 = \{0, 1 \}$, with $t: 1 \to 2$ picking out $0$. Given a subset $V \subseteq U$, the map $\chi_V: U \to 2$ is given by $$\chi_V(u) = \begin{cases}
0 & u \notin V \\
1 & u \in V
\end{cases}.$$
It is then easy to check that $t: 1 \to 2$ is a subobject classifier for $\ncat{Set}$.
\end{Ex}

\begin{Ex}
Every presheaf topos $\Pre(\cat{C})$ has a subobject classifier $\Omega$, defined as the presheaf which sends $U \in \cat{C}$ to the set of sieves $R \hookrightarrow y(U)$. The map $* \to \Omega$ is given by picking out the maximal sieve $y(U)$.
\end{Ex}

\begin{Ex}
Given a Grothendieck site $(\cat{C}, J)$, we say that a sieve $R \hookrightarrow y(U)$ is $J$-closed if it has the property that if for every map $f : V \to U$, the sieve $f^*R$ is $J$-covering, then $R$ is $J$-covering.  The sheaf topos $\ncat{Sh}(\cat{C}, j)$ has a subobject classifier $\Omega$, defined as the sheaf which sends $U \in \cat{C}$ to the set of $J$-closed sieves on $U$. Again $* \to \Omega$ picks out the maximal sieve.
\end{Ex}

\subsubsection{Epis}

\begin{Def} \label{def quotient object}
Let $\cat{C}$ be a category and $U \in \cat{C}$. A \textbf{quotient object} of $U$ is an equivalence class of epimorphisms $q : U \twoheadrightarrow V$, where we identify two epimorphisms $q$ and $q' : U \twoheadrightarrow V'$ if there exists an isomorphism $\varphi : V \to V'$ such that $\varphi q = q'$.
\end{Def}

\begin{Def}
Given a category $\cat{C}$ and a morphism $f: U \to V$, the \textbf{coimage} of $f$ is the largest quotient object $U \twoheadrightarrow \text{coim}(f)$ through which $f$ factors.
\end{Def}

\begin{Def} \label{def regular epi}
We say that an epimorphism $f : V \to U$ in a category $\cat{C}$ is \textbf{regular} if there exist maps $g,h: W \to V$ in $\cat{C}$ such that
\begin{equation*}
\begin{tikzcd}[ampersand replacement=\&]
	W \& V \& U
	\arrow["g", shift left=2, from=1-1, to=1-2]
	\arrow["h"', shift right=2, from=1-1, to=1-2]
	\arrow["f", two heads, from=1-2, to=1-3]
\end{tikzcd}
\end{equation*}
is a coequalizer.
\end{Def}

\begin{Def} \label{def kernel pair}
Given a category $\cat{C}$ and a morphism $f: U \to V$ in $\cat{C}$, we call a pullback square of the form
\begin{equation*}
\begin{tikzcd}
	{U \times_V U} & U \\
	U & V
	\arrow["{p_1}", from=1-1, to=1-2]
	\arrow["{p_0}"', from=1-1, to=2-1]
	\arrow["\lrcorner"{anchor=center, pos=0.125}, draw=none, from=1-1, to=2-2]
	\arrow["f", from=1-2, to=2-2]
	\arrow["f"', from=2-1, to=2-2]
\end{tikzcd}
\end{equation*}
a \textbf{kernel pair} of $f$, if it exists.
\end{Def}

\begin{Def} \label{def regular coimage}
Given a category $\cat{C}$ with pullbacks and coequalizers and a morphism $f: U \to V$, we call a coequalizer of the form
\begin{equation*}
\begin{tikzcd}
	{U \times_V U} & U & {\text{coim}_{\text{reg}}(f)}
	\arrow["{p_1}"', shift right=2, from=1-1, to=1-2]
	\arrow["{p_0}", shift left=2, from=1-1, to=1-2]
	\arrow["{e_f}", two heads, from=1-2, to=1-3]
\end{tikzcd}
\end{equation*}
a \textbf{regular coimage} of $f$.
\end{Def}

\begin{Def} \label{def effective epimorphism}
We say that an epimorphism $f: V \to U$ in a category $\cat{C}$ is \textbf{effective} if the following diagram
\begin{equation*}
\begin{tikzcd}
	V & U & {U +_V U}
	\arrow["f", from=1-1, to=1-2]
	\arrow["{{i_0}}", shift left, from=1-2, to=1-3]
	\arrow["{{i_1}}"', shift right, from=1-2, to=1-3]
\end{tikzcd}   
\end{equation*}
is an equalizer, i.e. $f$ is a regular coimage of itself.
\end{Def}

\begin{Lemma} \label{lem sufficient condition for reg epi to be effective}
If $\cat{C}$ is a category with coequalizers and kernel pairs, then an epimorphism in $\cat{C}$ is regular if and only if it is effective. 
\end{Lemma}

\begin{proof}
By the dual of the proof of Lemma \ref{lem sufficient condition for reg mono to be effective}.
\end{proof}

\begin{Lemma} \label{lem regular coimage is image for cats with pullbacks and kernel pairs}
If $\cat{C}$ is a category with pullbacks and kernel pairs, $f : U \to V$ is a morphism in $\cat{C}$, and all epimorphisms in $\cat{C}$ are regular epimorphisms, then the regular coimage of $f$ is a coimage of $f$.
\end{Lemma}

\begin{proof}
By the dual of the proof of Lemma \ref{lem regular image is image in cat with pushouts and cokernal pairs}.
\end{proof}

\begin{Lemma} \label{lem epis are regular in set}
In $\ncat{Set}$, all epimorphisms are regular and effective.
\end{Lemma}

\begin{proof}
By the dual of the proof of Lemma \ref{lem monos are regular in Set}.
\end{proof}

\begin{Cor}
In any presheaf topos $\Pre(\cat{C})$, all epimorphisms are regular and effective.
\end{Cor}

\begin{Rem} \label{rem im and coim iso}
Note that given a morphism $f: U \to V$ in a category $\cat{C}$ with finite limits and colimits, by universal properties, there is a canonical morphism $\text{coim}(f) \to \im(f)$. In $\Set$, and for presheaf toposes $\Pre(\cat{C})$, this map is always an isomorphism, as can be checked directly.
\end{Rem}

\subsection{Presheaf Toposes} \label{section presheaf topoi}

\begin{Lemma} \label{lem presheaves (co)limits computed objectwise}
Given a small category $\cat{C}$, its presheaf topos $\Pre(\cat{C})$ is a locally small category with all small limits and colimits. Furthermore limits and colimits are computed objectwise in the following sense: given a (small) diagram $d : I \to \Pre(\cat{C})$ and $U \in \cat{C}$, there is an isomorphism
\begin{equation*}
   \left( \ncolim{i \in I} d(i) \right)(U) \cong \ncolim{i \in I} d(i)(U), \qquad \left( \lim_{i \in I} d(i) \right)(U) \cong \lim_{i \in I} d(i)(U).
\end{equation*}
\end{Lemma}

\begin{proof}
That $\Pre(\cat{C})$ is locally small is Lemma \ref{lem smallness of functor categories}. The fact that (co)limits are computed objectwise is a standard fact about functor categories, see \cite[Proposition 3.3.9]{riehl2017category} for instance.
\end{proof}

\begin{Rem} \label{rem epis of presheaves computed objectwise}
The condition for a map being an (monomorphism) epimorphism can be characterized using (pullbacks) pushouts, thus Lemma \ref{lem presheaves (co)limits computed objectwise} implies that a map $f : X \to Y$ of presheaves is an (mono) epi if and only if it is objectwise an (mono) epi.  
\end{Rem}

\begin{Lemma} \label{lem presheaf adjoint triple}
Given a functor $F : \cat{C} \to \cat{D}$ between (small) categories, there is a pair of adjunctions
\begin{equation}
    \begin{tikzcd}
	{\Pre(\cat{D})} && {\Pre(\cat{C})}
	\arrow[""{name=0, anchor=center, inner sep=0}, "{\Delta_F}"{description}, from=1-1, to=1-3]
	\arrow[""{name=1, anchor=center, inner sep=0}, "{\Sigma_F}"', curve={height=18pt}, from=1-3, to=1-1]
	\arrow[""{name=2, anchor=center, inner sep=0}, "{\Pi_F}", curve={height=-18pt}, from=1-3, to=1-1]
	\arrow["\dashv"{anchor=center, rotate=-91}, draw=none, from=0, to=2]
	\arrow["\dashv"{anchor=center, rotate=-89}, draw=none, from=1, to=0]
\end{tikzcd}
\end{equation}
where if $Y$ is a presheaf on $\cat{D}$ and $X$ is a presheaf on $\cat{C}$, then 
\begin{enumerate}
    \item for $U \in \cat{C}$, $\Delta_F(Y)$ is the presheaf defined objectwise by $(\Delta_F(Y))(U) = Y(F(U))$,
    \item $\Sigma_F(X)$ is the left Kan extension of $X$ along $F^\op$,
    \item $\Pi_F(X)$ is the right Kan extension of $X$ along $F^\op$.
\end{enumerate}
Thus for $V \in \cat{D}$ we have
\begin{equation*}
    \Sigma_F(X)(V) \cong \ncolim{F^{\op}(U) \to V} X(U) \cong \ncolim{V \to F(U)} X(U),
\end{equation*}
and
\begin{equation*}
    \Pi_F(X)(V) \cong \lim_{V \to F^\op(U)} X(U) \cong \lim_{F(U) \to V} X(U).
\end{equation*}
\end{Lemma}

\begin{Lemma} \label{lem sigma on representable}
Given a functor $F: \cat{C} \to \cat{D}$ between small categories and $U \in \cat{C}$, we have
\begin{equation*}
    \Sigma_F(y(U)) \cong y(F(U)).
\end{equation*}
\end{Lemma}

\begin{proof}
Given $V \in \cat{D}$, we can write
\begin{equation*}
    \Sigma_F(X)(V) \cong (\text{Lan}_{F^\op} X)(V) \cong \int^{U \in \cat{C}^\op} \cat{D}^\op(F^\op(U), V) \times X(U)
\end{equation*}
equivalently
\begin{equation*}
    \Sigma_F(X)(V) \cong \cat{D}^\op(F^\op(-),V) \otimes_{\cat{C}^\op} X
\end{equation*}
and by Lemma \ref{lem swapping arity in left kan extensions} we have
\begin{equation*}
    \Sigma_F(X)(V) \cong X \otimes_{\cat{C}} \cat{D}(V, F(-)).
\end{equation*}
Now if $U \in \cat{C}$, and $X = y(U)$, then by Lemma \ref{lem yoneda extension identity on representables}, we have
\begin{equation*}
    \Sigma_F(y(U))(V) \cong y(U) \otimes_{\cat{C}} \cat{D}(V, F(-)) \cong \cat{D}(V, F(U)).
\end{equation*}
Thus for every $V \in \cat{D}$, $\Sigma_F(y(U))(V) \cong \cat{D}(V, F(U)) \cong y(F(U))(V)$. Thus we have $\Sigma_F(X) \cong y(F(U))$.
\end{proof}

\begin{Def} \label{def category of elements}
Given a presheaf $X$ on a category $\cat{C}$, let $\int X$, denote its \textbf{category of elements}, namely $\int X$ is the category whose objects are pairs $(U, x)$ with $U \in \cat{C}$ and $x \in X(U)$, and whose morphisms $f : (U, x) \to (V, y)$ are those morphisms $f: U \to V$ such that $X(f)(y) = x$. There is a canonical functor $\pi_X : \int X \to \cat{C}$ given by $\pi_X(U,x) = U$. Note that $\int X$ is isomorphic to the comma category $( y \downarrow X)$, where $y : \cat{C} \to \Pre(\cat{C})$ is the Yoneda embedding.
\end{Def}

\begin{Lemma} \label{lem presheaf slices are presheaf topoi}
Given a (small) category $\cat{C}$ and a presheaf $X$ on $\cat{C}$, there is an equivalence of categories
\begin{equation*}
    \Pre(\cat{C})/X \simeq \Pre(\smallint X).
\end{equation*}
\end{Lemma}

\begin{proof}
Let us construct a functor $\varphi : \Pre(\cat{C})/X \to \Pre(\smallint X)$. Given a map $f: Y \to X$ of presheaves, let $\varphi(f)$ denote the presheaf on $\int X$ where $\varphi(f)(x,U)$ is the set of sections $s \in Y(U)$ such that $f_U(s) = x$. If $g : (x,U) \to (y,V)$ is a morphism in $\int X$, then let $\varphi(f)(g) : \varphi(f)(y,V) \to \varphi(f)(x,U)$ take an element $s \in f_V^{-1}(y)$ to $f(g)(s) \in f_U^{-1}(x)$. Conversely, let $\psi : \Pre(\smallint X) \to \Pre(\cat{C})/X$ take a presheaf $A$ on $\smallint X$ to the presheaf $\pi_*(A)$ defined objectwise by $\pi_*(A)(U) = \sum_{x \in X(U)} A(U,x)$. If $g : A \to B$ is a map of presheaves on $\int X$, then let $\pi_*(g) : \pi_*(A) \to \pi_*(B)$ be defined in the obvious way. It is not hard to show that these functors are quasi-inverses to each other and hence define an equivalence of categories.
\end{proof}

\begin{Rem}
We note that the above result extends to sheaves, see \cite[Page 157]{maclane2012sheaves}.
\end{Rem}
 
\begin{Rem}
The following important result is known both as the \textbf{Density Theorem} and the \textbf{coYoneda Lemma}. We will just refer to it as the coYoneda Lemma.
\end{Rem}

\begin{Lemma}[coYoneda Lemma] \label{lem coyoneda lemma}
Given a small category $\cat{C}$ and $X \in \Pre(\cat{C})$, there is an isomorphism
\begin{equation*}
    X \cong \colim \left( \int X \xrightarrow{\pi} \cat{C} \xrightarrow{y} \Pre(\cat{C}) \right) \cong \int^{(x,U) \in \int X} X(U) \otimes y(U) \cong X \otimes_{\cat{C}} y(U) \cong \ncolim{y(U) \to X} y(U),
\end{equation*}
where $X(U) \otimes y(U) = \sum_{X(U)} y(U)$ is the tensoring of $y(U)$ by the set $X(U)$.
\end{Lemma}

\begin{proof}
This can be proven in a variety of ways, here we give a coend calculus proof. Let $Y \in \Pre(\cat{C})$, then
\begin{equation*}
\begin{aligned}
    \Pre(\cat{C})\left( \int^{(x,U) \in \int X} X(U) \otimes y(U), Y \right) & \cong \int_{(x,U) \in \int X} \Pre(\cat{C})( X(U) \otimes y(U), Y) \\
    & \cong \int_{(x,U) \in \int X} \ncat{Set}(X(U),\Pre(\cat{C})(y(U),Y)) \\
    & \cong \int_{(x,U) \in \int X} \ncat{Set}(X(U), Y(U)) \\
    & \cong \Pre(\cat{C})(X, Y),
\end{aligned}
\end{equation*}
where the second isomorphism is the defining property of tensoring, and the third is the Yoneda Lemma.
\end{proof}

\begin{Lemma} \label{lem presheaf topoi are locally cartesian closed}
Given a small category $\cat{C}$, its presheaf topos $\Pre(\cat{C})$ is Cartesian closed, with internal hom defined as follows. Given presheaves $X$ and $Y$ on $\cat{C}$, let $[X,Y]$ denote the presheaf defined objectwise by
\begin{equation} \label{eq internal hom of presheaf topoi}
[X,Y](U) \coloneqq \Pre(\cat{C})(X \times y(U), Y).
\end{equation}
For a fixed presheaf $X$, the functor $[X,-]$ defines a right adjoint to the functor $- \times X$. Furthermore $\Pre(\cat{C})$ is locally cartesian closed, i.e. all of its slice categories $\Pre(\cat{C})/X$ are Cartesian closed.
\end{Lemma}

\begin{proof}
We leave it to the reader to check that the internal hom defined by (\ref{eq internal hom of presheaf topoi}) makes $\Pre(\cat{C})$ cartesian closed. Now the fact that $\Pre(\cat{C})$ is locally cartesian closed follows Lemma \ref{lem presheaf slices are presheaf topoi}.
\end{proof}

\begin{Def} \label{def image presheaf}
Given a map $f: X \to Y$ of presheaves on a small category $\cat{C}$, let $\im(f)$ be the presheaf defined as follows. For $U \in \cat{C}$, let $\im(f)(U)$ be the set of sections $s: y(U) \to Y$ such that there exists a dotted map (called a lift of $s$) making the following diagram commute
\begin{equation*}
    \begin{tikzcd}
	& X \\
	{y(U)} & Y
	\arrow["s"', from=2-1, to=2-2]
	\arrow["f", from=1-2, to=2-2]
	\arrow["{s'}", dashed, from=2-1, to=1-2]
\end{tikzcd}
\end{equation*}
Clearly $\im(f) \subseteq Y$. Equivalently, $\im(f)(U) = \im(f_U)$ where $f_U : X(U) \to Y(U)$ is the component map.
\end{Def}

\begin{Lemma} \label{lem image of presheaf map}
Given a map $f : X \to Y$ of presheaves, the presheaf $\im(f)$ is the image of $f$, and is isomorphic to both the regular image and coimage of $f$ in $\Pre(\cat{C})$.
\end{Lemma}

\begin{proof}
This follows from Lemma \ref{lem presheaves (co)limits computed objectwise}, Corollary \ref{cor monos in presheaf toposes} and Lemma \ref{lem regular image is image in cat with pushouts and cokernal pairs}.
\end{proof}

\begin{Lemma} \label{lem presheaf topoi have pullback stable image factorization}
Suppose that $f : X \to Y$ is a map of presheaves on a small category $\cat{C}$, such that the right hand vertical map is its image factorization, and such that both squares are pullbacks
\begin{equation*}
\begin{tikzcd}[ampersand replacement=\&]
	{Z \times_Y X} \& X \\
	{g^*(\text{im}(f))} \& {\text{im}(f)} \\
	Z \& Y
	\arrow[from=1-1, to=1-2]
	\arrow[two heads, from=1-1, to=2-1]
	\arrow["\lrcorner"{anchor=center, pos=0.125}, draw=none, from=1-1, to=2-2]
	\arrow["{g^*(f)}"', curve={height=30pt}, from=1-1, to=3-1]
	\arrow[two heads, from=1-2, to=2-2]
	\arrow["f", curve={height=-30pt}, from=1-2, to=3-2]
	\arrow[from=2-1, to=2-2]
	\arrow[hook', from=2-1, to=3-1]
	\arrow["\lrcorner"{anchor=center, pos=0.125}, draw=none, from=2-1, to=3-2]
	\arrow[hook', from=2-2, to=3-2]
	\arrow["g"', from=3-1, to=3-2]
\end{tikzcd}
\end{equation*}
Then $g^*(\text{im}(f))$ is an image of the map $g^*(f) : Z \times_Y X \to Z$. We say that $\ncat{Pre}(\cat{C})$ has pullback-stable image factorizations.
\end{Lemma}

\begin{proof}
This can be proven pointwise, and thus follows from the same result holding in $\ncat{Set}$, which is easy to see.
\end{proof}

\section{Localization} \label{section localizations}
In this section we deal in detail with the important concept of localization. This is a powerful way of ``inverting'' a class of morphisms $W$ in a category $\cat{C}$ to obtain a new category $\cat{C}[W^{-1}]$.

\begin{Def} \label{def inverting morphisms}
Let $\cat{C}$ be a locally small category and $W$ a class (i.e. large set) of morphisms in $\cat{C}$. We say that a functor $F: \cat{C} \to \cat{D}$ \textbf{inverts} $W$ if $F(w)$ is an isomorphism for every $w \in W$. Given a category $\cat{D}$, let $\Fun(\cat{C}, \cat{D})$ denote the category whose objects are functors and morphisms natural transformations, and let $\Fun_W(\cat{C},\cat{D})$ denote the full subcategory on those functors that invert $W$.
\end{Def}

\begin{Def} \label{def localization of category}
 A \textbf{localization}\footnote{This is sometimes called weak localization. See \cite{govzmann2022comparison} for a comparison of different kinds of localization} of $\cat{C}$ at $W$ consists of a functor $L : \cat{C} \to \cat{D}$ such that
\begin{enumerate}
    \item $L$ inverts $W$,
    \item for any functor $F : \cat{C} \to \cat{E}$ that inverts $W$, there exists a functor $\widetilde{F} : \cat{D} \to \cat{E}$ along with a natural isomorphism $\eta : F \Rightarrow \widetilde{F} L$, and
    \item for every category $\cat{E}$, the functor
    \begin{equation*}
        \Fun(L, \cat{E}) : \Fun(\cat{D}, \cat{E}) \to \Fun_W(\cat{C}, \cat{E})
    \end{equation*}
    is fully faithful.
\end{enumerate}
\end{Def}

\begin{Lemma}
Given a locally small category $\cat{C}$ and a class $W$ of morphisms, a functor $L : \cat{C} \to \cat{D}$ is a localization of $\cat{C}$ by $W$ if and only if $L$ inverts $W$ and for every category $\cat{E}$ the functor
\begin{equation} \label{eq universal prop of localization}
    \Fun(L, \cat{E}) : \Fun(\cat{D}, \cat{E}) \to \Fun_W(\cat{C}, \cat{E})
\end{equation}
is an equivalence of categories.
\end{Lemma}

\begin{Lemma}[{\cite[Proposition 4]{govzmann2022comparison}}] \label{lem uniqueness of localizations}
Let $\cat{C}$ be a locally small category and $W$ a class of morphisms of $\cat{C}$. Suppose that $L : \cat{C} \to \cat{D}$ and $L' : \cat{C} \to \cat{D}'$ are localizations of $\cat{C}$ at $W$, equipped with natural isomorphisms $\widetilde{L} L \cong L'$ and $\widetilde{L}' L' \cong L$, then there exist unique natural isomorphisms $\varphi : \widetilde{L} \widetilde{L}' \cong 1_{\cat{D}'}$ and $\psi: \widetilde{L}' \widetilde{L} \cong 1_{\cat{D}}$, defining an equivalence of categories $\phi: \cat{D} \to \cat{D}'$ such that $\phi L \cong L'$.
\end{Lemma}

\begin{Rem}
The universal property of a localization of $\cat{C}$ at $W$ can be stated much more succinctly as a pushout
\begin{equation*}
\begin{tikzcd}
	{W \times I} & {\cat{C}} \\
	{W \times J} & {\cat{D}}
	\arrow[hook, from=1-1, to=2-1]
	\arrow["i", from=1-1, to=1-2]
	\arrow["L", from=1-2, to=2-2]
	\arrow[from=2-1, to=2-2]
	\arrow["\lrcorner"{anchor=center, pos=0.125, rotate=180}, draw=none, from=2-2, to=1-1]
\end{tikzcd}
\end{equation*}
in the very large category $\ncat{CAT}$ of large (and not necessarily locally small) categories, where $W$ is thought of as a discrete category, $I$ is the poset $\{ 0 \leq 1 \}$, and $J$ is the category with two objects $0$ and $1$ and a single isomorphism between them. The functor $i$ takes a pair $(w, \leq )$ to the actual morphism $w$ in $\cat{C}$. See \cite{simpson2005explaining} for more details.
\end{Rem}

Given a category $\cat{C}$ and a class of morphisms $W$ in $\cat{C}$, we can construct a (possibly not locally small) category $\cat{C}[W^{-1}]$, which we call the \textbf{Gabriel-Zisman localization} of $\cat{C}$ at $W$. Its objects are the same as $\cat{C}$, but its morphisms $f: U \to V$ consist of equivalence classes of finite zig-zags of morphisms in $\cat{C}$ of the form 
$$ U \xrightarrow{f_0} V_0 \xleftarrow{w_0} V_1 \xrightarrow{f_1} V_2 \leftarrow \dots \rightarrow V_n \xleftarrow{w_k} V$$
where each $w_i \in W$. See \cite[Chapter 1]{gabriel1967calculus}, \cite[Section 2]{simpson2005explaining} or \cite{nlab:category_of_fractions} for more details. This implies the following result.

\begin{Prop}[{\cite[Lemma 1.2]{gabriel1967calculus}}] \label{prop localizations of categories exist}
For any category $\cat{C}$ and class of morphisms $W$, the (possibly not locally small) category $\cat{C}[W^{-1}]$ is a localization of $\cat{C}$ at $W$.
\end{Prop}

Unfortunately, Proposition \ref{prop localizations of categories exist} is not as useful as it might seem. The explicit description of $\cat{C}[W^{-1}]$ is not amenable to concrete computations.

\subsubsection{Reflective Localization}
In this section, we introduce reflective localizations, a very important and useful class of localizations.

\begin{Def} \label{def reflective functor}
We say that a functor $L : \cat{C} \to \cat{D}$ is \textbf{reflective} if it has a fully faithful right adjoint $i$:
\begin{equation*}
\begin{tikzcd}
	{\cat{D}} && {\cat{C}}
	\arrow[""{name=0, anchor=center, inner sep=0}, "L"', shift right=2, from=1-3, to=1-1]
	\arrow[""{name=1, anchor=center, inner sep=0}, "i"', shift right=2, hook, from=1-1, to=1-3]
	\arrow["\dashv"{anchor=center, rotate=-90}, draw=none, from=0, to=1]
\end{tikzcd}
\end{equation*}
If $i: \cat{E} \hookrightarrow \cat{C}$ is a full subcategory with a left adjoint $L: \cat{C} \to \cat{E}$, we call $\cat{E}$ a \textbf{reflective subcategory}, and we call $L$ the \textbf{reflector}. Let $\cat{L} \coloneqq iL$, we call this the \textbf{localizer} functor.
\end{Def}

In what follows we will be interested in reflective functors $L: \cat{C} \to \cat{D}$ that are furthermore localizations, which we call reflective localizations. These are very important kinds of localizations because the category $\cat{D}$ is equivalent to a certain full subcategory of $\cat{C}$ on the $W$-local objects (Proposition \ref{prop local objects equiv to localization}) and because reflective localizations inherit the (co)limits of $\cat{C}$ in a certain sense (Proposition \ref{prop (co)limits in reflective subcategories}).

\begin{Def} \label{def W local object and local equivalence}
Let $\cat{C}$ be a category and $W$ a class of morphisms in $\cat{C}$. We say that an object $X \in \cat{C}$ is \textbf{$W$-local}\footnote{This terminology is really unfortunate considering how over-used the term local is used in these notes, but this terminology is by now very standard.}if for every $w: A \to B$ in $W$, the map
\begin{equation*}
    w^* : \cat{C}(B,X) \to \cat{C}(A,X)
\end{equation*}
is a bijection. In other words $X$ is $W$-local if and only if $y(X)$ inverts $W^\op$.

We say that a map $f: A \to B$ is a \textbf{$W$-local equivalence} if for every $W$-local object $Z$, the map
\begin{equation*}
    f^* : \cat{C}(B,Z) \to \cat{C}(A,Z)
\end{equation*}
is a bijection.
\end{Def}

\begin{Rem} \label{rem counit of reflective adjunction is iso}
Since $i$ is fully faithful, we know that the counit of the adjunction $L \dashv i$ is an isomorphism. Thus for every $d \in \cat{D}$, the component $\varepsilon_d: Li(d) \to d$ of the counit is an isomorphism.
\end{Rem}


\begin{Lemma} \label{lem objects in reflective localizations}
Let $L: \cat{C} \to \cat{D}$ be a reflective functor with right adjoint $i: \cat{D} \hookrightarrow \cat{C}$. Let $W$ denote the class of morphisms in $\cat{C}$ inverted by $L$. If $X \in \cat{C}$, then the following are equivalent:
\begin{enumerate}
	\item $X$ is isomorphic to an object $i Y$, with $Y \in \cat{D}$,
	\item $X$ is $W$-local, and
	\item the unit $\eta_X : X \to iLX = \cat{L}X$ is an isomorphism.
\end{enumerate}
\end{Lemma}

\begin{proof}
$(1 \Rightarrow 2)$ Suppose that $X \cong iY$ for some $Y \in \cat{D}$. If $w: A \to B$ is a morphism in $W$, then naturality of the adjunction $L \dashv i$ implies that the following diagram commutes
\begin{equation*}
\begin{tikzcd}
	{\cat{C}(B,X)} & {\cat{C}(A,X)} \\
	{\cat{C}(B,iY)} & {\cat{C}(A,iY)} \\
	{\cat{D}(LB,Y)} & {\cat{D}(LA,Y)}
	\arrow["{{w^*}}", from=1-1, to=1-2]
	\arrow["{{(Lw)^*}}"', from=3-1, to=3-2]
	\arrow["\cong"', from=1-1, to=2-1]
	\arrow["\cong", from=1-2, to=2-2]
	\arrow["\cong"', from=2-1, to=3-1]
	\arrow["\cong", from=2-2, to=3-2]
	\arrow["{w^*}", from=2-1, to=2-2]
\end{tikzcd}	
\end{equation*}
and since $w \in W$, $L(w)$ is an isomorphism, so $(Lw)^*$ is also an isomorphism. Thus $w^*$ is an isomorphism. Thus $X$ is $W$-local.

$(2 \Rightarrow 3)$ The triangle identities of the adjunction imply that the composition
$$LX \xrightarrow{L(\eta_X)} LiLX \xrightarrow{\varepsilon_{LX}} LX$$
is the identity on $X$, but $i$ is fully faithful so $\varepsilon_{LX}$ is an isomorphism. Therefore $L(\eta_X)$ is an isomorphism, and thus $\eta_X \in W$.

Now since $X$ is $W$-local, the map 
$$\eta^*_X: \cat{C}(iLX, X) \to \cat{C}(X,X)$$
is a bijection. Thus there exists a unique map $g : iLX \to X$ such that $(\eta^*_X)(g) = g \eta_X= 1_X$. But by (1) above, $iLX$ is $W$-local as well, and $L(g \eta_X) = L(g) L(\eta_X) = 1_{LX}$, thus $L(g)$ is an isomorphism, so $g \in W$. Thus the map
\begin{equation*}
    g^*: \cat{C}(X, iLX) \to \cat{C}(iLX, iLX)
\end{equation*}
is also a bijection. Thus there exists a unique $h: X \to iLX$ such that $hg = 1_{iLX}$. But then $\eta_X = (hg) \eta_X = h (g \eta_{X}) = h$. Thus $\eta_X$ is an isomorphism.

$(3 \Rightarrow 1)$ This is clear.
\end{proof}

\begin{Lemma} \label{lem W-local objects and equivalences of reflective localizations}
Let $L: \cat{C} \to \cat{D}$ be a reflective functor and let $W$ denote the class of morphisms in $\cat{C}$ that are inverted by $L$. Then
\begin{enumerate}
\item the essential image of $iL$ consists precisely of the $W$-local objects, and
\item the $W$-local equivalences are precisely the maps in $W$.
\end{enumerate}
\end{Lemma}

\begin{proof}
(1) This follows from Lemma \ref{lem objects in reflective localizations}. 

(2) Suppose $f: A \to B$ is a $W$-local equivalence. We want to show that $L(f)$ is an isomorphism. Suppose that $Y \in \cat{D}$, then by the naturality of the adjunction $L \vdash i$, the following diagram commutes
\begin{equation*}
    \begin{tikzcd}
	{\cat{D}(LB,Y)} & {\cat{D}(LA,Y)} \\
	{\cat{C}(B,iY)} & {\cat{C}(A,iY)}
	\arrow["{(Lf)^*}", from=1-1, to=1-2]
	\arrow["\cong"', from=1-1, to=2-1]
	\arrow["\cong", from=1-2, to=2-2]
	\arrow["{f^*}", from=2-1, to=2-2]
\end{tikzcd}
\end{equation*}
Now since $f$ is a $W$-local equivalence, and $iY$ is $W$-local by Lemma \ref{lem objects in reflective localizations}, then $f^*$ is an isomorphism, which implies that $(Lf)^*$ is an isomorphism. Since $Y$ was arbitrary, by the Yoneda lemma this implies that $Lf$ is an isomorphism. Thus $f \in W$.

Now suppose $f \in W$, then $Lf$ is an isomorphism, so $(Lf)^*$ is an isomorphism so $f^*$ is an isomorphism in the above commuting diagram. Since every $W$-local object is isomorphic to an object of the form $iY$ by Lemma \ref{lem objects in reflective localizations}, this implies that $f$ is a $W$-local equivalence.
\end{proof}

\begin{Lemma}[{\cite[Proposition 0.8]{nlab:reflective_localization}}]
Let $L: \cat{C} \to \cat{D}$ be a reflective functor, and let $W$ denote the class of morphisms in $\cat{C}$ that are inverted by $L$. Then $L : \cat{C} \to \cat{D}$ is a weak localization of $\cat{C}$ at $W$.
\end{Lemma}

\begin{proof}
Let $F: \cat{C} \to \cat{E}$ be a functor that inverts all the morphisms in $W$. We wish to show that it factors through $L$ up to natural isomorphism. Consider the following diagram
\begin{equation*}
\begin{tikzcd}
	{\cat{C}} && {\cat{C}} & {\cat{E}} \\
	& {\cat{D}}
	\arrow[""{name=0, anchor=center, inner sep=0}, "{1_{\cat{C}}}", from=1-1, to=1-3]
	\arrow["L"', from=1-1, to=2-2]
	\arrow["i"', from=2-2, to=1-3]
	\arrow["F", from=1-3, to=1-4]
	\arrow["\eta", shorten <=3pt, Rightarrow, from=0, to=2-2]
\end{tikzcd}	
\end{equation*}
where $\eta$ denotes the unit of the adjunction $L \dashv i$. Let $\widetilde{F} \coloneqq F i$. We want to show that $F \eta : F \Rightarrow \widetilde{F}L$ is a natural isomorphism. Notice that $L(\eta_X)$ is an isomorphism for every object $X \in \cat{C}$ as argued in the proof of Lemma \ref{lem objects in reflective localizations}. That means that $\eta_X \in W$, so $F(\eta_X)$ is also an isomorphism. This implies that $F \eta$ is a natural isomorphism.

Now suppose there is another functor $\widetilde{F}' : \cat{D} \to \cat{E}$ and a natural isomorphism $\alpha: F \Rightarrow \widetilde{F}'L$. Consider the pasting diagram:
\begin{equation*}
\begin{tikzcd}
	& {\cat{C}} && {\cat{E}} \\
	{\cat{D}} && {\cat{D}}
	\arrow["i", from=2-1, to=1-2]
	\arrow["L"{description}, from=1-2, to=2-3]
	\arrow[""{name=0, anchor=center, inner sep=0}, "{1_\cat{D}}"', from=2-1, to=2-3]
	\arrow[""{name=1, anchor=center, inner sep=0}, "F", from=1-2, to=1-4]
	\arrow["{\widetilde{F}'}"', from=2-3, to=1-4]
	\arrow["\varepsilon"', shorten >=3pt, Rightarrow, from=1-2, to=0]
	\arrow["\alpha", shorten <=3pt, Rightarrow, from=1, to=2-3]
\end{tikzcd}	
\end{equation*}
this provides a natural isomorphism $Fi \cong \widetilde{F}'$. If we now paste the above diagram to the unit $\eta$ we will get $\alpha$ back by the triangle identities:
\begin{equation*}
\begin{tikzcd}
	{\cat{C}} && {\cat{C}} && {\cat{E}} & {=} & {\cat{C}} && {\cat{E}} \\
	& {\cat{D}} && {\cat{D}} &&&& {\cat{D}}
	\arrow["i"{description}, from=2-2, to=1-3]
	\arrow["L"{description}, from=1-3, to=2-4]
	\arrow[""{name=0, anchor=center, inner sep=0}, "{1_\cat{D}}"', from=2-2, to=2-4]
	\arrow[""{name=1, anchor=center, inner sep=0}, "F", from=1-3, to=1-5]
	\arrow["{\widetilde{F}'}"', from=2-4, to=1-5]
	\arrow["L"', from=1-1, to=2-2]
	\arrow[""{name=2, anchor=center, inner sep=0}, "{1_\cat{C}}", from=1-1, to=1-3]
	\arrow["L"', from=1-7, to=2-8]
	\arrow["{\widetilde{F}'}"', from=2-8, to=1-9]
	\arrow[""{name=3, anchor=center, inner sep=0}, "F", from=1-7, to=1-9]
	\arrow["\varepsilon"', shorten >=3pt, Rightarrow, from=1-3, to=0]
	\arrow["\alpha", shorten <=3pt, Rightarrow, from=1, to=2-4]
	\arrow["\eta"', shorten <=3pt, Rightarrow, from=2, to=2-2]
	\arrow["\alpha", shorten <=3pt, Rightarrow, from=3, to=2-8]
\end{tikzcd}	
\end{equation*}
Thus $\alpha$ factors through $F \eta$ uniquely, showing that (\ref{eq universal prop of localization}) is fully faithful, and therefore $L : \cat{C} \to \cat{D}$ is a localization of $\cat{C}$ at $W$.
\end{proof}

\begin{Def} \label{def reflective localization}
Suppose that $L: \cat{C} \to \cat{D}$ is a localization of $\cat{C}$ at a class of morphisms $W$. We say that $L$ is a \textbf{reflective localization} if $L$ is furthermore a reflective functor.
\end{Def}

For plenty of examples of reflective localizations, see \cite[Example 4.5.14]{riehl2017category}.

\begin{Cor}
Every reflective functor $L : \cat{C} \to \cat{D}$ is a reflective localization of $\cat{C}$ at the class $L^{-1}(\text{iso})$ of morphisms inverted by $L$.
\end{Cor}

\begin{Prop} \label{prop local objects equiv to localization}
Suppose that $L : \cat{C} \to \cat{D}$ is a reflective localization of $\cat{C}$ at $W$, with right adjoint $i: \cat{D} \hookrightarrow \cat{C}$. Then if we let $\iota : \cat{C}_W \hookrightarrow \cat{C}$ denote the full subcategory of $\cat{C}$ on the $W$-local objects, there is an equivalence $\phi: \cat{D} \to \cat{C}_W$ making the following diagram commutes
\begin{equation*}
\begin{tikzcd}
	{\cat{D}} && {\cat{C}} \\
	& {\cat{C}_W}
	\arrow["i", hook, from=1-1, to=1-3]
	\arrow["{{{\iota}}}"', hook, from=2-2, to=1-3]
	\arrow["\phi"', from=1-1, to=2-2]
\end{tikzcd}
\end{equation*}
\end{Prop}

\begin{proof}
First note that $W \subseteq L^{-1}(\text{iso})$ since $L$ inverts $W$. Let us show that $i$ factors through $\cat{C}_W$. In other words if $Y \in \cat{D}$, then we wish to show that $iY$ is $W$-local. But by Lemma \ref{lem objects in reflective localizations}, $iY$ is $L^{-1}(\text{iso})$-local, and therefore $W$-local. Therefore $i$ factors through the inclusion $\iota$, let us denote the corestriction of $i$ by $\phi$. Let us show that $\phi$ is essentially surjective.

Suppose that $X \in \cat{C}$ is $W$-local. We want to show that $X$ is also $L^{-1}(\text{iso})$-local. Now since $X$ is $W$-local, that implies that $y(X)$ inverts $W$. Therefore by the universal property of localization, we obtain a factorization of $y(X)$ up to natural isomorphism 
\begin{equation*}
\begin{tikzcd}
	{\cat{C}} && {\Set^{\op}} \\
	{\cat{D}}
	\arrow[""{name=0, anchor=center, inner sep=0}, "{y(X)}", from=1-1, to=1-3]
	\arrow["L"', from=1-1, to=2-1]
	\arrow[""{name=1, anchor=center, inner sep=0}, "G"', from=2-1, to=1-3]
	\arrow["\ell"', shift right=5, shorten <=2pt, shorten >=2pt, Rightarrow, from=0, to=1]
\end{tikzcd} 
\end{equation*}
But this implies that $y(X)$ inverts $L^{-1}(\text{iso})$. Thus $X$ is $L^{-1}(\text{iso})$-local. Thus by Lemma \ref{lem objects in reflective localizations}, this implies that there exists a $Y \in \cat{D}$ such that $X \cong iY$. Therefore $\phi$ is essentially surjective. Furthermore, $\phi$ is fully faithful because $i$ is. Thus $\phi$ is an equivalence and clearly $\iota \phi = i$.
\end{proof}

\begin{Cor}
Given a reflective functor $L : \cat{C} \to \cat{D}$, then $L$ is a reflective localization at $L^{-1}(\text{iso})$ and $\cat{D}$ is equivalent to the full subcategory $\cat{C}_{L^{-1}(\text{iso})}$ of $L^{-1}(\text{iso})$-local objects in $\cat{C}$.
\end{Cor}

The following result gives a big motivation for understanding reflective localizations. It can help us obtain (co)completeness results just by knowing a category is a reflective localization.

\begin{Prop}[{\cite[Proposition 4.5.15]{riehl2017category}}] \label{prop (co)limits in reflective subcategories}
Given a reflective subcategory $i: \cat{D} \hookrightarrow \cat{C}$ with left adjoint $L$ then
\begin{itemize}
    \item if $d : I \to \cat{D}$ is a (small) diagram, then a cone $\lambda : \Delta(V) \Rightarrow d$ in $\cat{D}$ is a limit cone if and only if $i\lambda : \Delta(iV) \Rightarrow id$ is a limit cone in $\cat{C}$. Furthermore if there exists a limit cone $\lambda' : \Delta(U) \Rightarrow id$ in $\cat{C}$, then there exists a limit cone $\lambda : \Delta(V) \Rightarrow d$ in $\cat{D}$ such that $iV = U$ and $i \lambda = \lambda'$,
    \item if $d : I \to \cat{D}$ is a (small) diagram, and $id : I \to \cat{C}$ admits a colimit $\colim \, id$ in $\cat{C}$, then $L(\colim \, id)$ is a colimit of $d$ in $\cat{D}$.
\end{itemize}
Therefore if $\cat{D} \hookrightarrow \cat{C}$ is a reflective subcategory, with $\cat{C}$ (co)complete, then so is $\cat{D}$. Furthermore limits in $\cat{D}$ agree with those in $\cat{C}$, while colimits in $\cat{D}$ are computed by applying $L$ to the colimit computed in $\cat{C}$.
\end{Prop}

\section{Locally Presentable Categories} \label{section locally presentable categories}
Directed posets and filtered categories are massively important objects in category theory, and show up often in these notes, especially in the theory of flat functors (Section \ref{section morphisms of sites}). In this section, we build up from these concepts to the theory of locally presentable categories. These constitute a huge class of categories that have very nice properties. We try to keep this discussion motivated and down to earth.

This section owes itself to the classic reference \cite{rosicky1994locally}, and the reader should consult there for more details. We also recommend the excellent introductory survey \cite{sarazola2017}.

To motivate locally presentable categories, let us begin with the following elementary observation about sets: suppose that $S$ is a set, and let $\cons{Sub}_{\text{fin}}(S)$ denote the poset of finite subsets of $S$. Then $S$ is a colimit of the included diagram $i_S : \cons{Sub}_{\fin}(S) \hookrightarrow \ncat{Set}$. Indeed, $S$ forms a cocone $\lambda : i_S \Rightarrow \Delta(S)$ over $i_S$ by including every finite subset into $S$. If $T$ is another set with a cocone $\lambda' : i_S \Rightarrow \Delta(T)$, then we can define a map $h : S \to T$ by sending $x \in S$ to wherever $\lambda'_{\{x\}} : \{ x \} \to T$ sends $x$ in $T$. This is easily seen to be well-defined and unique, showing that $S$ is a colimit of $i_S$. This shows that the large category $\ncat{Set}$ is generated in some sense by the essentially small full subcategory $\ncat{FinSet}$. Large categories in general are unwieldy and can be hard to work with directly. Locally presentable categories are large categories that are controlled in a very strong way by a small full subcategory. Let us now try and abstract this important property of $\ncat{Set}$.

\subsection{Directed and Filtered Categories}

\begin{Def} \label{def directed poset}
We say that a poset $I$ is \textbf{finitely directed} if:
\begin{enumerate}
    \item it is nonempty, and
    \item every pair of elements $x, y \in I$ have an upper bound in $I$, i.e. there exists a $z \in I$ such that $x \leq z$ and $y \leq z$.
\end{enumerate}
The second condition is equivalent to every finite subset $A \subseteq I$ having an upper bound, i.e. for all $a \in A$ there is a $z \in I$ such that $a \leq z$. We say a diagram $d: I \to \cat{C}$ is finitely directed if $I$ is a finitely directed poset, and say that $\colim \, d$ is a finitely directed colimit.
\end{Def}

Directed posets are those posets which ``flow'' in one direction, namely the direction of the upper bounds.

\begin{Ex}
If $S$ is a set, then $\Sub_{\fin}(S)$ is directed, since if $A, B \in \Sub_{\fin}(S)$, then $A \cup B$ is a finite subset of $S$ and $A, B \subseteq A \cup B$.
\end{Ex}

Now note that if $S$ is a countably infinite set, say $\mathbb{N} = \omega$, then there will be infinite subsets of $\Sub_{\fin}(\omega)$ with no upper bound. For example, the set $\{ \{0 \}, \{0, 1 \},  \{0, 1, 2 \}, \dots \}$ has no upper bound in $\Sub_{\fin}(\omega)$. This motivates the next definition.

\begin{Def}
Given a regular cardinal (Definition \ref{def regular cardinal}) $\kappa$, we say that a poset $I$ is \textbf{$\kappa$-directed}, if for every subset $A \subseteq I$ with $|A| < \kappa$, $A$ has an upper bound in $I$. We say that $I$ is directed if it is $\kappa$-directed for some $\kappa$.
\end{Def}

Thus being finitely directed is equivalent to being $\omega$-directed. Thus $\Sub_{\fin}(\omega)$ is $\omega$-directed, but not $\aleph_1$-directed. However the poset of all subsets $\Sub(\omega)$ is $\kappa$-directed for all cardinals $\kappa$ since $\omega$ is an upper bound for every subset of $\Sub(\omega)$. This same idea as above extends powerfully to all diagrams to give the following result.

\begin{Lemma} \label{lem every colimit is iso to a filtered colimit}
Let $\cat{C}$ be a category with all small colimits, and let $d: I \to \cat{C}$ be a small diagram. Let $\Sub_{\fin}(I)$ denote the poset of finite full subcategories of $I$ ordered by inclusion. This is a finitely directed poset, and furthermore we obtain a functor $s(d) : \Sub_{\fin}(I) \to \cat{C}$ given by sending $J \in \Sub_{\fin}(I)$ to $\colim \, d|_J$. We have the following
\begin{equation*}
    \colim \, d \cong \ncolim{J \in \Sub_{\fin}(I)} s(d)(J) \cong \ncolim{J \in \Sub_{\fin}(I)} \colim \, d|_J.
\end{equation*}
In other words, every colimit is isomorphic to a directed colimit of finite ``partial colimits''.
\end{Lemma}

\begin{proof}
Let us show that $\colim \, d$ forms a cocone over $s(d)$. Let $\lambda : d \Rightarrow \Delta(\colim \, d)$ denote the colimit cocone over $d$, and $\lambda^J : d|_J \Rightarrow \Delta(\colim \, d|_J)$ the colimit cocone of the restriction. Then for $J \in \Sub_{\fin}(I)$, we get a cocone $\lambda|_J : d|_J \Rightarrow \Delta(\colim \, d)$, and thus a unique map $\sigma_J : s(d)(J) \to \colim \, d$ such that $\Delta(\sigma_J) \circ \lambda^J = \lambda|_J$. Now if $J \hookrightarrow J'$, then its easy to see that these maps are compatible, i.e. the $\sigma_J$ form a cocone $\sigma : s(d) \Rightarrow \Delta(\colim \, d)$, again by the uniqueness part of the universal property of colimits. Now if $\tau : s(d) \Rightarrow \Delta(V)$ is a cocone over $s(d)$, let us define a map $h : \colim \, d \to V$, which is the same thing as defining a cocone $h : d \Rightarrow \Delta(V)$. Given $i \in I$, let $h(i) : d(i) \to V$ denote the map given by considering the finite full subcategory $\{ i \}$ of $I$ given by the single object $i$. This defines a map $\tau_{\{ i \}} : s(d)(\{ i \}) \to V$, but of course $s(d)(\{i \}) = \colim \, d|_{\{ i \}} \cong d(i)$, with colimit cocone component $\lambda^{\{i\}}_i : d(i) \to \colim \, d|_{\{i \}}$ given by the identity. So we have a map $\tau_{\{ i \}} : d(i) \to V$. So let $h(i) = \tau_{\{i \}}$. We still have to prove that this forms a cocone over $d$. So suppose that $f : i \to j$ is a morphism in $I$. We need to show that $h(j) d(f) = h(i)$. But this follows by considering the following commutative diagram
\begin{equation*}
\begin{tikzcd}
	& V \\
	\\
	{\colim \, d|_{\{i \}}} & {\colim \, d|_{\{i, j \}}} & {\colim \, d|_{\{j\}}} \\
	{d(i)} && {d(j)}
	\arrow["{h(i) = \tau_{\{i\}}}", from=3-1, to=1-2]
	\arrow[from=3-1, to=3-2]
	\arrow["{\tau_{\{i,j\}}}"', from=3-2, to=1-2]
	\arrow["{h(j) = \tau_{\{j\}}}"', from=3-3, to=1-2]
	\arrow[from=3-3, to=3-2]
	\arrow[Rightarrow, no head, from=4-1, to=3-1]
	\arrow["{\lambda^{\{i,j\}}_i}"{description}, from=4-1, to=3-2]
	\arrow["{d(f)}"{description}, from=4-1, to=4-3]
	\arrow["{\lambda_j^{\{i,j\}}}"{description}, from=4-3, to=3-2]
	\arrow[Rightarrow, no head, from=4-3, to=3-3]
\end{tikzcd}
\end{equation*}
In other words, the $\lambda^{\{i,j\}}$ being compatible with $f$ forces the $\tau_{\{i\}}$ to be compatible with it as well. Thus we have defined a cocone $h : d \Rightarrow \Delta(V)$, and therefore a unique map $h : \colim \, d \to V$ such that $\tau_{\{i \}} = h \lambda_i$ for every $i \in I$. Now all we have to prove is that $\Delta(h) \sigma = \tau$, or equivalently for each $J \in \Sub_{\fin}(I)$, that $h \sigma_J = \tau_J$. Now $h \sigma_J : \colim \, d|_J \to V$ and $\tau_J : \colim \, d|_J \to V$ are maps out of $\colim \, d|_J$ and therefore are equivalent to cocones over $d|_J$. Thus it is enough to check that that they are equal as cocones. But since $\Delta(\sigma_J) \circ \lambda^J = \lambda|_J$, this implies that for every $i \in I$ we have $\sigma_J \lambda^J_i = \lambda_i$, and $\tau_J$ precomposed with $\lambda^J_i : d(i) \to \colim \, d|_J$ is just $\tau_{\{i \}}$. Thus since $\tau_{\{i \}} = h \lambda_i$ for each $i \in I$, this proves that $\sigma$ is a colimit cocone.
\end{proof}

Given a category $\cat{C}$ with a full subcategory $\cat{C}_0$, let us say that $\cat{C}$ is \textbf{generated by finitely directed colimits} from $\cat{C}_0$ if every object $U \in \cat{C}$ is isomorphic to \textit{some} finitely directed colimit in $\cat{C}_0$. In other words, there is a finitely directed diagram $d : I \to \cat{C}_0$, such that $U$ is a colimit in $\cat{C}$ of the resulting diagram
\begin{equation*}
    I \xrightarrow{d} \cat{C}_0 \hookrightarrow \cat{C}.
\end{equation*}

So what we have discovered is that the category $\ncat{Set}$, which is a large category, is generated by finitely directed colimits from the essentially small full subcategory $\ncat{FinSet}$ of finite sets. 

It turns out that we can actually strengthen the above description of how $\ncat{Set}$ is generated from $\ncat{FinSet}$. Above we said that every set is a colimit of \textit{some} finitely directed diagram. However, we can actually assign to every set a canonical diagram that it is the colimit of.

\begin{Def} \label{def canonical diagram}
Let $\cat{C}$ be a category, and $\cat{C}_0$ a full subcategory of $\cat{C}$. For an object $U \in \cat{C}$, consider the diagram $\pi: (\cat{C}_0 \downarrow U) \to \cat{C}$, where $\pi(A \to U) = A$ is the forgetful functor. We call this the \textbf{canonical diagram of $U$} with respect to $\cat{C}_0$. We say that $\cat{C}_0$ is \textbf{dense} in $\cat{C}$ if every object $U \in \cat{C}$ is a colimit of its canonical diagram with respect to $\cat{C}_0$,
$$U \cong \colim \left( (\cat{C}_0 \downarrow U) \xrightarrow{\pi} \cat{C} \right).$$
\end{Def}

The concept of density can be generalized as follows.

\begin{Def}
We say that a functor $i : \cat{C}_0 \to \cat{C}$ is \textbf{dense} if $\text{Lan}_i i$ exists and is naturally isomorphic to the identity functor on $\cat{C}$. 
\end{Def}

\begin{Lemma} \label{lem equiv def of density}
Given a category $\cat{C}$ and a full subcategory $\cat{C}_0$, the inclusion functor $i : \cat{C}_0 \hookrightarrow \cat{C}$ is dense if and only if $\cat{C}_0$ is dense in $\cat{C}$.
\end{Lemma}

\begin{proof}
The functor $i$ is dense if and only for every $U_0 \in \cat{C}$ we have
\begin{equation*}
    \text{Lan}_i i(U) \cong \colim \left( (\cat{C}_0 \downarrow U) \xrightarrow{\pi'} \cat{C}_0 \xrightarrow{i} \cat{C} \right) \cong U.
\end{equation*}
But the functor $\pi : (\cat{C}_0 \downarrow U) \to \cat{C}$ is equal to $i \pi'$. Thus $\cat{C}_0$ is dense in $\cat{C}$ if and only if $i$ is a dense functor.
\end{proof}

Note that in the above definition if we take $\cat{C} = \ncat{Set}$ and $\cat{C}_0 = \ncat{FinSet}$, then $\ncat{FinSet}$ is dense in $\ncat{Set}$ if every set $S$ can be written as the colimit of the canonical diagram $(\ncat{FinSet} \downarrow S) $ of all maps from finite sets into $S$. This is different from $\Sub_{\fin}(S)$, which is a poset, but it is still the case that $S$ is the colimit of its canonical diagram with respect to $\ncat{FinSet}$, using practically the same argument.

So we have managed to strengthen \textit{some colimit} above to \textit{a particular colimit}, but we have lost something in the sense that these particular colimits are no longer directed. Thus we introduce the following notion, which is a generalization of directed posets.

\begin{Def} \label{def filtered category}
A small category $\cat{C}$ is \textbf{finitely filtered} if for every finite category $I$ and diagram $d: I \to \cat{C}$, there is a cocone $\lambda : d \Rightarrow \Delta(U)$ in $\cat{C}$. More generally a small category $\cat{C}$ is \textbf{$\kappa$-filtered} if every $\kappa$-small diagram $d$ admits a cocone in $\cat{C}$. We say that $\cat{C}$ is filtered if it is $\kappa$-filtered for some regular cardinal $\kappa$. We say a diagram $d: I \to \cat{C}$ is filtered if $I$ is a filtered category.

We say a category $\cat{C}$ is ($\kappa$-)\textbf{cofiltered} if $\cat{C}^\op$ is ($\kappa$-)filtered.
\end{Def}

The following alternate description of finite filteredness is also very useful.

\begin{Lemma} \label{lem alternate def of filtered}
A category $\cat{C}$ is finitely filtered if and only if
\begin{enumerate}
    \item it is nonempty,
    \item for every pair of objects $U, V \in \cat{C}$, there exists an object $W$ and morphisms $U \to W$ and $V \to W$,
    \item for every pair of parallel morphisms $f, g: U \to V$, there exists a morphism $h : V \to W$ such that $hf = hg$.
\end{enumerate}
\end{Lemma}

The category $(\ncat{FinSet} \downarrow S)$ is finitely filtered for any set $S$. Indeed, for every $A, B \in \ncat{FinSet}$ and every pair of maps $f : A \to S$, $g: B \to S$, then $f + g : A + B \to S$ is a map such that the following diagram commutes
\begin{equation*}
    \begin{tikzcd}
	A & {A + B} & B \\
	\\
	& S
	\arrow[from=1-1, to=1-2]
	\arrow["f"', curve={height=12pt}, from=1-1, to=3-2]
	\arrow["{f + g}"{description}, from=1-2, to=3-2]
	\arrow[from=1-3, to=1-2]
	\arrow["g", curve={height=-12pt}, from=1-3, to=3-2]
\end{tikzcd}
\end{equation*}
Furthermore, suppose that $a: A \to S$ and $b: A \to S$ are maps, and suppose that there are parallel maps $f, g : A \to B$ over $S$, with $A$ and $B$ finite. Then $b f = b g = a$, as can be seen in the following commutative diagram
\begin{equation*}
    \begin{tikzcd}
	A && B & S \\
	& S
	\arrow["g"', shift right=2, from=1-1, to=1-3]
	\arrow["f", shift left=2, from=1-1, to=1-3]
	\arrow["a"', from=1-1, to=2-2]
	\arrow["b", from=1-3, to=1-4]
	\arrow["b", from=1-3, to=2-2]
	\arrow[curve={height=-12pt}, Rightarrow, no head, from=1-4, to=2-2]
\end{tikzcd}
\end{equation*}
So by Lemma \ref{lem alternate def of filtered}, $(\ncat{FinSet} \downarrow S)$ is finitely filtered.

We lose very little by abstracting from directed posets to filtered categories. Let us first recall the notion of final functor.

\begin{Def} \label{def final functor}
We say that a functor $F: \cat{C} \to \cat{D}$ between cocomplete categories is \textbf{final} if for every functor $G: \cat{D} \to \cat{E}$, the canonical map
\begin{equation*}
  \ncolim{U \in \cat{C}} (G \circ F)(U) \to \ncolim{V \in \cat{D}} G(V)
\end{equation*}
is an isomorphism. Similarly, we say that $F$ is \textbf{initial} if the canonical map 
\begin{equation*}
    \lim_{V \in \cat{D}} G(V) \to \lim_{U \in \cat{C}} (G \circ F)(U)
\end{equation*}
is an isomorphism. Equivalently, $F$ is initial if $F^\op : \cat{C}^\op \to \cat{D}^\op$ is final.
\end{Def}

\begin{Prop}[{\cite[Lemma 8.3.4]{riehl2017category}}] \label{prop final functors}
A functor $F: \cat{C} \to \cat{D}$ is final if and only if for every $V \in \cat{D}$, the comma category $(V \downarrow F)$ is nonempty and connected. Similarly, $F$ is initial if and only if for all $V \in \cat{D}$, the comma category $(F \downarrow V)$ is nonempty and connected.
\end{Prop}

\begin{Lemma}[{\cite[Theorem 1.5 and Remark 1.21
]{rosicky1994locally}}]
For every small $\kappa$-filtered category $\cat{C}$ there exists a small $\kappa$-directed poset $\cat{D}$ and a final\footnote{The reference refers to these as cofinal functors, an unfortunate clash of terminology.} functor $F : \cat{D} \to \cat{C}$. In other words, an object can be written as a filtered colimit if and only if it can be written as a directed colimit.
\end{Lemma}

Now before we go further, let us better understand filtered colimits in $\ncat{Set}$. This detour will have a surprisingly large pay-off (Proposition \ref{prop filtered colimits commute with finite limits in Set}), but we must first muck about with set-level details.

Let us show that filtered colimits in $\ncat{Set}$ admit a convenient description. First, let us consider an arbitrary diagram $d : I \to \ncat{Set}$. By the usual way of decomposing a colimit into a coequalizer of coproducts we have
\begin{equation*}
    \underset{i \in I}{\colim} \, d(i) \cong \text{coeq} \left( \sum_{f : i \to j} d(i) \underset{\psi}{\overset{\varphi}{\rightrightarrows}} \sum_{j \in I} d(j) \right)
\end{equation*}
where $\varphi$ takes a pair $(f : i \to j, \, x \in d(i))$ to $d(f)(x) \in d(j)$ and $\psi$ takes $(f : i \to j, \, x \in d(i))$ to $x \in d(i)$.

Thus we can describe $\colim_{i \in I} d(i)$ as the set $\sum_{i \in I} d(i)$ quotiented by the smallest equivalence relation $\approx$ generated by the relation where if $x \in d(i)$, then $x \sim d(f)(x)$. We can describe $\approx$ completely as the reflexive, symmetric, transitive closure of the relation $\sim$. Another way of saying this is that $x \approx y$ if and only if there is a finite zig-zag in $I$ of the following form (or of the form where we reverse the directions of the arrows)
\begin{equation} \label{eq zig-zag filtered colimit}
\begin{tikzcd}
	i & {k_0} & {k_1} & \dots & {k_n} & j
	\arrow["f", from=1-1, to=1-2]
	\arrow["{h_0}"', from=1-3, to=1-2]
	\arrow["{h_1}", from=1-3, to=1-4]
	\arrow["{h_{n-1}}", from=1-4, to=1-5]
	\arrow["g"', from=1-6, to=1-5]
\end{tikzcd}    
\end{equation}
and there exist elements $x_0 \in d(k_0)$, $x_1 \in d(k_1)$, $\dots$, $x_n \in d(k_n)$ such that
\begin{equation*}
    d(f)(x) = x_0 = d(h_0)(x_1), \; d(h_1)(x_1) = x_2 = d(h_3)(x_3), \, \dots, \, d(h_{n-1})(x_{n-1}) = x_n = d(g)(y).
\end{equation*}
So far, this is just the general description of an equivalence class $[x]$ of an element in a general colimit in $\ncat{Set}$. Now if $I$ is filtered, and $x \approx y$, then the zig-zag (\ref{eq zig-zag filtered colimit}) admits a cocone
\begin{equation*}
    \begin{tikzcd}
	&& k \\
	\\
	i & {k_0} & {k_1} & \dots & {k_n} & j
	\arrow["p", from=3-1, to=1-3]
	\arrow["f"', from=3-1, to=3-2]
	\arrow["{\ell_0}"{description}, from=3-2, to=1-3]
	\arrow["{\ell_1}"{description}, from=3-3, to=1-3]
	\arrow["{h_0}", from=3-3, to=3-2]
	\arrow["{h_1}"', from=3-3, to=3-4]
	\arrow["{h_{n-1}}"', from=3-4, to=3-5]
	\arrow["{\ell_n}"{description}, from=3-5, to=1-3]
	\arrow["q"', from=3-6, to=1-3]
	\arrow["g", from=3-6, to=3-5]
\end{tikzcd}
\end{equation*}
and since
\begin{equation*}
d(p)(x) = d(\ell_0 f)(x) = d(\ell_0) d(f)(x) = d(\ell_0)(x_0) = d(\ell_0) d(h_0)(x_1) = d(\ell_1)(x_1) = \dots = d(\ell_n)(x_n) = d(q)(y)
\end{equation*}
this implies that there is a $z = d(p)(x) = d(q)(y)$ in $d(k)$ such that $x \sim z$ and $y \sim z$. We have thus reduced general zig-zags to spans $i \to k \leftarrow j$, and therefore have proved the following result.

\begin{Lemma} \label{lem filtered colimits description in sets}
Suppose that $d: I \to \ncat{Set}$ is a filtered diagram of sets, then $\colim_{i \in I} \, d(i)$ is isomorphic to the set $\sum_{i \in I} d(i)$ quotiented by the equivalence relation $\approx$ where $x \approx y$ if there exists a $k \in I$ and maps $f: i \to k$, $g : j \to k$ such that $d(f)(x) = d(g)(y)$.
\end{Lemma}

Let $\cat{C}$ be a category with $I$ and $J$-shaped colimits and limits. Suppose that $d : I \times J \to \cat{C}$ is a diagram. There is a canonical map
\begin{equation} \label{eq commuting canonical map}
    \sigma : \underset{i \in I}{\colim} \lim_{j \in J} d(i,j) \to \lim_{j \in J} \underset{i \in I}{\colim} \, d(i,j),
\end{equation}
that can be described as follows. For every $i_0 \in I$ and $j_0 \in J$, there are canonical maps
\begin{equation}
    \lim_{j \in J} d(i_0, j) \to d(i_0, j_0) \to \underset{i \in I}{\colim} \, d(i, j_0),
\end{equation}
given by the respective (co)limit (co)cones. Now the composite map forms a cone over the diagram $\colim_{i \in I} \, d(i, -) : J \to \cat{C}$ for if $f: j_0 \to j'_0$ is a morphism in $J$, then the following diagram commutes
\begin{equation*}
    \begin{tikzcd}
	& {\lim_{j \in J} d(i_0,j)} \\
	{d(i_0, j_0)} && {d(i_0, j'_0)} \\
	{\colim_{i \in I} \, d(i, j_0)} && {\colim_{i \in I} \, d(i, j'_0)}
	\arrow[from=1-2, to=2-1]
	\arrow[from=1-2, to=2-3]
	\arrow["{d(i_0, f)}", from=2-1, to=2-3]
	\arrow[from=2-1, to=3-1]
	\arrow[from=2-3, to=3-3]
	\arrow["{\colim_{i \in I} \, d(i,f)}", from=3-1, to=3-3]
\end{tikzcd}
\end{equation*}
So for every $i_0 \in I$ there is a map
\begin{equation*}
    \lim_{j \in J} d(i_0, j) \to \lim_{j \in J} \underset{i \in I}{\colim} \, d(i, j),
\end{equation*}
and a similar argument shows this provides a cocone over $\lim_{j \in J} d(-, j) : I \to \cat{C}$, and thus gives the map $\sigma$ in (\ref{eq commuting canonical map}).

\begin{Def} \label{def commuting (co)limits}
We say that \textbf{$I$-colimits commute with $J$-limits} in $\cat{C}$ if the map $\sigma$ in (\ref{eq commuting canonical map}) is an isomorphism for every diagram $d: I \times J \to \cat{C}$.
\end{Def} 

\begin{Prop}[{\cite[Theorem 3.8.9]{riehl2017category}}] \label{prop filtered colimits commute with finite limits in Set}
Finitely filtered colimits commute with finite limits in $\ncat{Set}$.
\end{Prop}

\begin{proof}
Suppose that $I$ is a (small) finitely filtered category, $J$ is a finite category and $d: I \times J \to \ncat{Set}$ is a diagram. We want to show that $\sigma$ is a bijection. First note that an element of $\lim_{j \in J} \colim_{i \in I} \, d(i,j)$ consists of a finite collection of elements $[x_{j_0}] \in \colim_{i  \in I} \, d(i,j_0)$ which are compatible in the sense that for every map $f: j \to j'$ in $J$, $\colim_{i \in I}\, d(i,f)([x_j]) = [d(i_0, f)(x_j)] = [x_{j'}]$. Since $J$ is finite and $I$ is filtered, by Lemma \ref{lem filtered colimits description in sets}, this means that there exists a fixed $k \in I$, and elements $y_{j_0} \in d(k, j_0)$ such that $[x_{j_0}] = [y_{j_0}]$ for every $j_0 \in J$. In other words, we can represent an element of $\lim_{j \in J} \colim_{i \in I} \, d(i,j)$ by the collection $\{ y_j \}_{j \in J}$, and they will be compatible on the nose, $d(k,f)(y_j) = y_{j'}$. But such a collection is the same thing as an element of the set $\lim_{j \in J} d(k, j)$. In other words, $\sigma$ is surjective.

Now an element of $\colim_{i \in I} \lim_{j \in J} d(i,j)$ consists of an equivalence class $[ \{ y_j \}_{j \in J}]$ of compatible families of elements $y_j \in d(i, j)$, where $\{ y_j \} \approx \{ y'_j \}$ if there are maps $f : i \to k$, $g : i' \to k$ such that $d(f,j)(y_j) = d(g,j)(y'_j)$ for every $j \in J$, by Lemma \ref{lem filtered colimits description in sets}. Now if $[\{y_j \}]$ and $[\{y'_j \}]$ are elements such that $\sigma([\{y_j \}]) = \sigma([\{y'_j\}])$, then we know that there is a $k \in I$ and a family $\{ z_j \}$ with each $z_j \in d(k, j)$ such that $[\{y_j \}] = [\{ z_j \}] = [\{ y'_j \}]$ such that the $z_j$ are compatible on the nose. But this implies that $\{y_j \} \approx \{z_j \} \approx \{y'_j \}$. Thus $\sigma$ is injective.
\end{proof}

In fact Proposition \ref{prop filtered colimits commute with finite limits in Set} has a converse.

\begin{Prop}[{\cite[Lemma 2.1]{bjerrum2015notes}}] \label{prop commute with finite limits implies filtered}
Suppose that $I$ is a small category and $I$-colimits commute with all finite limits in $\ncat{Set}$. Then $I$ is a finitely filtered category.
\end{Prop}

Before proving Proposition \ref{prop commute with finite limits implies filtered}, let us prove the following.

\begin{Lemma} \label{lem colimits of representables are singletons}
Let $\cat{C}$ be a small category with $U \in \cat{C}$ and consider the representable presheaf $y(U) \in \Pre(\cat{C})$. Then
\begin{equation*}
   \ncolim{V \in \cat{C}^{\op}} y(U)(V) \cong \lim_{V \in \cat{C}} y^{\text{co}}(U)(V) \cong * 
\end{equation*}
where $*$ denotes a singleton set.
\end{Lemma}

\begin{proof}
Given an arbitrary set $S$, we have
\begin{equation*}
    \begin{aligned}
        \ncat{Set}\left( \ncolim{V \in \cat{C}^{\op}} y(U)(V), S) \right) & \cong \Pre(\cat{C})(y(U), \Delta(S)) \\
        & \cong \Delta(S)(U) \\
        & \cong S \\
        & \cong \ncat{Set}(*, S),
    \end{aligned}
\end{equation*}
where the first isomorphism is the definition of the colimit functor as representing the constant functor. Since $S$ was arbitrary, by the Yoneda lemma, this implies that $\colim_{V \in \cat{C}} y(U)(V) \cong *$. A similar argument proves that $\lim_{V \in \cat{C}} y^{\text{co}}(U)(V) \cong *$, where $y^{\text{co}}(U) = \cat{C}(U,-)$.
\end{proof}

\begin{proof}[Proof of Proposition \ref{prop commute with finite limits implies filtered}]
Suppose that $d : J^{\op} \to I$ is a finite diagram. Consider the functor $F: I \times J \to \ncat{Set}$ defined by $F(i,j) = I(d(j), i)$. Now by Lemma \ref{lem colimits of representables are singletons}, for each $j \in J$, we have 
\begin{equation*}
    \lim_{i \in I} F(i,j) = \lim_{i \in I} y^{\text{co}}(d(j))(i) \cong *.
\end{equation*}
Thus $\colim_{j \in J} \lim_{i \in I} F(i,j) \cong *$. But 
\begin{equation*}
\ncolim{j \in J} F(i,j) \cong \ncolim{j \in J} I(d(j),i) \cong I(\ncolim{j \in J} d(j), i) \cong \text{coCone}(d)(i),
\end{equation*}
where $\text{coCone}(d)(i)$ is the set of cocones over $d$ with vertex $i$. Thus if $\text{coCone}(d)(i)$ is empty for all $i$, then $\lim_{i \in I} \Cone(d)(i) \cong \varnothing$. So if $I$-shaped colimits commute with finite limits, then $d$ must have a cocone. Since $d$ was an arbitrary finite diagram, this proves that $I$ is finitely filtered.
\end{proof}

\begin{Cor} \label{cor colimits commute with finite limits iff filtered}
Given a small category $I$, $I$-colimits commute with finite limits in $\ncat{Set}$ if and only if $I$ is finitely filtered.
\end{Cor}

This result can be strengthened to the following.

\begin{Cor} \label{cor kappa filtered commutes with kappa small in set}
Given a small category $I$ and a regular cardinal $\kappa$, $I$-colimits commute with $\kappa$-small limits in $\ncat{Set}$ if and only if $I$ is $\kappa$-filtered.
\end{Cor}

\begin{Cor} \label{cor filtered colimits commute in presheaf topoi}
Given a small category $I$ and a regular cardinal $\kappa$, $I$-colimits commute with $\kappa$-small limits in presheaf topoi if and only if $I$ is $\kappa$-filtered.
\end{Cor}

Now with Proposition \ref{prop filtered colimits commute with finite limits in Set} under our belt, let us return to our new goal. We wish to abstract the following property of $\ncat{Set}$: the category $\ncat{FinSet}$ is dense in $\ncat{Set}$. First we want a notion similar to that of finiteness for an arbitrary category.

\begin{Def} \label{def small objects}
Let $\cat{C}$ be a category and $U \in \cat{C}$. We say that $U$ is \textbf{finitely presentable} if for every finitely filtered diagram $d: I \to \cat{C}$, the canonical map
\begin{equation} \label{eq canonical map for small object def}
    \tau: \underset{i \in I}{\colim} \, \cat{C}(U, d(i)) \to \cat{C}(U, \underset{i \in I}{\colim} \, d(i))
\end{equation}
is an isomorphism. In other words $U$ is finitely presentable if and only if $\cat{C}(U,-)$ preserves finitely filtered colimits. More generally we say that $U$ is $\kappa$-presentable if $\cat{C}(U,-)$ preserves $\kappa$-filtered colimits.
\end{Def}

\begin{Rem} \label{rem kappa pres implies kappa' pres}
If $U \in \cat{C}$ is $\kappa$-presentable, then it is $\kappa'$-presentable for any pair of regular cardinals $\kappa < \kappa'$. This is just because $\kappa'$-filtered categories are $\kappa$-filtered. In particular if $U$ is finitely presentable, then it is $\kappa$-presentable for all infinite regular cardinals $\kappa$.
\end{Rem}

We can equivalently swap the word ``filtered'' with ``directed'' in the above definition thanks to the following result.

\begin{Prop}[{\cite[Corollary of Theorem 1.5 and Remark 1.21]{rosicky1994locally}}]
A category $\cat{C}$ admits $\kappa$-filtered colimits if and only if it admits $\kappa$-directed colimits. Furthermore a functor $F$ preserves $\kappa$-filtered colimits if and only if it preserves $\kappa$-directed colimits.
\end{Prop}

From Lemma \ref{lem filtered colimits description in sets}, we can understand the map (\ref{eq canonical map for small object def}) explicitly. The left hand side is the set of maps $s_i : U \to d(i)$ quotiented by the equivalence relation $\approx$, where $s_i \approx s_j$ if there are maps $f: i \to k$ and $g : j \to k$ such that $d(f) s_i = d(g) s_j$. The right hand side is the set of maps from $U$ to the colimit of $d$. If we let $\lambda_i : d(i) \to \colim \, d$ denote the colimit cocone components, then the map $\tau$ sends an element $[s_i]$ to the map $ \lambda_i s_i$.

So $\tau$ is surjective if and only if every map $U \to \colim \, d$ factors through some map $U \to d(i)$. Now $\tau$ is injective if and only if whenever we have maps $s_i : U \to d(i)$ and $s_j : U \to d(j)$ such that $\lambda_i s_i = \lambda_j s_j$, then $s_i \approx s_j$. In other words, $\tau$ is a bijection if and only if every map $U \to \colim \, d$ factors through some $U \to d(i)$ uniquely up to $\approx$.

\begin{Lemma} \label{lem finitely presentable in Set iff finite}
A set $S$ is finitely presentable in $\ncat{Set}$ if and only if it is finite.
\end{Lemma}

\begin{proof}
$(\Rightarrow)$ Suppose that $S$ is finitely small. Then $\ncat{Set}(S, -)$ preserves the colimit of the diagram $\Sub_{\fin}(S) \hookrightarrow \ncat{Set}$. In other words, the identity function $1_S : S \to S$ factors through some finite subset of $S$. Thus $S$ is finite.

$(\Leftarrow)$ Suppose that $S$ is a finite set and $d : I \to \ncat{Set}$ is a finitely filtered diagram. If we let $S = \{x_0, \dots, x_n \}$, then $S \cong \sum_{i = 0}^n \{ x_i \} \cong \sum_{i = 0}^n *$. Thus by Proposition \ref{prop filtered colimits commute with finite limits in Set},
\begin{equation*}
    \ncat{Set}(S, \underset{i \in I}{\colim} d(i)) \cong \prod_{i = 0}^n \ncat{Set}(*, \underset{i \in I}{\colim} d(i)) \cong \prod_{i = 0}^n \underset{i \in I}{\colim} d(i) \cong \underset{i \in I}{\colim} \prod_{i = 0}^n d(i) \cong \underset{i \in I}{\colim} \, \ncat{Set}(S, d(i)). 
\end{equation*}
So $S$ is finitely presentable.
\end{proof}

\begin{Rem}
We offer a more hands-on proof of Lemma \ref{lem finitely presentable in Set iff finite}.$(\Leftarrow)$ here, we think this proof really gets the idea of finite presentability across. If $h : S \to \colim \, d$ is a function, then by Lemma \ref{lem filtered colimits description in sets}, we know that $\colim \, d = \sum_{i \in I} d(i) /{\approx}$. So since $S$ is finite, if we let $S = \{ x_0, \dots, x_n \}$, then there are objects $i_\ell \in I$ and elements $y_\ell \in d(i_\ell)$ such that $[y_\ell] = h(x_\ell)$ for all $0 \leq \ell \leq n$. Since $I$ is finitely filtered, there exists a $k \in I$ with maps $f_\ell : i_\ell \to k$. So there is a function $h_k : S \to d(k)$ that sends $x_\ell$ to $d(i_\ell)(y_\ell)$, and furthermore $[d(i_\ell)(y_\ell)] = [y_\ell] = h(x_\ell)$. In other words, $h$ factors through $h_k$. Thus $\tau$ in (\ref{eq canonical map for small object def}) is surjective. Making different choices of the $i_\ell$ and $y_\ell$ could result in different choices of $k$, but the resulting maps $h_k : S \to d(k)$ and $h_{k'} : S \to d(k')$ would have the property that $h_k \approx h_{k'}$. Thus $\tau$ is injective.
\end{Rem}

A very similar argument proves the following.

\begin{Cor} \label{cor kappa presentable sets}
Given a regular cardinal $\kappa$, a set $S$ is $\kappa$-presentable if and only if $|S| < \kappa$.
\end{Cor}

Let us also note that $\kappa$-presentable objects are closed under $\kappa$-small colimits.

\begin{Lemma} \label{lem colimits of presentable are presentable}
Let $\cat{C}$ be a category with all $\kappa$-small colimits, and let $c : I \to \cat{C}$ be a $\kappa$-small diagram. If $d$ factors through $\cat{C}_\kappa$, i.e. every $c(i)$ is $\kappa$-presentable, then $\colim \, c$ is $\kappa$-presentable.
\end{Lemma}

\begin{proof}
Given a $\kappa$-small diagram $c : I \to \cat{C}$ and a $\kappa$-filtered diagram $d : J \to \cat{C}$, we have
\begin{equation*}
\begin{aligned}
	\cat{C}(\ncolim{i \in I} c(i), \ncolim{j \in J} d(j)) & \cong \lim_{i \in I} \cat{C}(c(i), \ncolim{j \in J} d(j)) \\
	& \cong \lim_{i \in I} \ncolim{j \in J} \cat{C}(c(i), d(j)) \\
	& \cong \ncolim{j \in J} \lim_{i \in I} \cat{C}(c(i), d(j)) \\
	& \cong \ncolim{j \in J} \cat{C}(\ncolim{i \in I} c(i), d(j)).
\end{aligned}
\end{equation*}
where the second isomorphism holds because each $c(i)$ is $\kappa$-presentable and the third holds by Corollary \ref{cor kappa filtered commutes with kappa small in set}. Thus $\ncolim{i \in I} c(i)$ is $\kappa$-presentable.
\end{proof}

\begin{Ex}
There are many interesting characterizations of finitely presentable objects in various categories, see \cite[Example 1.2]{rosicky1994locally}. As just one example, to see where the terminology comes from, a group $G$ is a finitely presentable object in $\ncat{Grp}$ if it is literally finitely presentable, i.e. if it can be presented using finitely many generators and finitely many relations.
\end{Ex}

\begin{Lemma} \label{lem representables are presentable}
Let $\cat{C}$ be a small category with $U \in \cat{C}$. Then the representable presheaf $y(U)$ is finitely presentable in $\Pre(\cat{C})$.
\end{Lemma}

\begin{proof}
This follows from the Yoneda lemma. Let $d : I \to \Pre(\cat{C})$ be a finitely filtered diagram, then
\begin{equation*}
    \begin{aligned}
        \Pre(\cat{C})(y(U), \ncolim{i \in I} d(i) ) & \cong [\ncolim{i \in I} d(i)](U) \\
        & \cong \ncolim{i \in I} d(i)(U) \\
        & \cong \ncolim{i \in I} \Pre(\cat{C})(y(U), d(i)).
    \end{aligned}
\end{equation*}
\end{proof}

\begin{Rem} \label{rem finitely presentable objects in presheaf topoi}
In fact, it turns out that the finitely presentable objects in a presheaf topos are precisely those presheaves that can be written as finite colimits of representables, see \cite{campion2020}.
\end{Rem}

Given a category $\cat{C}$, let $\cat{C}_{\fp}$ denote the full subcategory on the finitely presentable objects. As we have seen, every set $S$ can be written as the colimit over its canonical diagram with respect to $\ncat{Set}_{\fp} \simeq \ncat{FinSet}$. In other words $\ncat{Set}_{\fp}$ is dense in $\ncat{Set}$. We can similarly define $\cat{C}_\kappa$ to be the full subcategory on the $\kappa$-presentable objects. Note however, that technically speaking, $\ncat{FinSet}$ is not a small category. It is essentially small though. Indeed, it is equivalent to the full subcategory $\ncat{FinSet}'$ on the sets $\{ 0, 1, \dots, n \}$, which is small.

\begin{Def} \label{def locally presentable category}
We say that a (locally small) category $\cat{C}$ is \textbf{locally finitely presentable} if:
\begin{enumerate}
    \item it is cocomplete, and
    \item the category $\cat{C}_{\fp}$ of finitely presentable objects is essentially small and dense in $\cat{C}$.
\end{enumerate}
More generally we say that $\cat{C}$ is locally $\kappa$-presentable if $\cat{C}_k$ is essentially small and fully dense in $\cat{C}$. We say that $\cat{C}$ is \textbf{locally presentable} if it is locally $\kappa$-presentable for some infinite regular cardinal $\kappa$.
\end{Def}

We can actually give an equivalent definition for locally presentable categories that is easier to check in practice, thanks to the following pair of technical results.

\begin{Lemma}[{\cite[Proposition 1.22]{rosicky1994locally}}] \label{lem every object is colimit of its canonical diagram}
Let $\cat{C}$ be a cocomplete category. Then an object $U \in \cat{C}$ can be written as \textit{some} $\kappa$-filtered colimit of $\kappa$-presentable objects if and only if $U$ can be written as the colimit of its canonical diagram with respect to $\cat{C}_\kappa$. Furthermore the category $(\cat{C}_\kappa \downarrow U)$ is $\kappa$-filtered for every $U \in \cat{C}$. 
\end{Lemma}

\begin{Lemma}[{\cite[Remark 1.9, 1.19]{rosicky1994locally}}]
A cocomplete category $\cat{C}$ has a set $S$ of $\kappa$-presentable objects such that every object of $\cat{C}$ is a $\kappa$-filtered colimit of objects in $S$ if and only if every object in $\cat{C}$ is a $\kappa$-filtered colimit of $\kappa$-presentable objects and there exists, up to isomorphism, only a set of $\kappa$-presentable objects.
\end{Lemma}

\begin{Cor} \label{cor equiv def of locally presentable category}
A cocomplete category $\cat{C}$ is locally $\kappa$-presentable if and only if there is a set $S$ of $\kappa$-presentable objects such that every object $U \in \cat{C}$ can be written as a $\kappa$-filtered colimit of objects in $S$. 
\end{Cor} 

\begin{Lemma} \label{lem presheaf topoi are locally presentable}
Given a small category $\cat{C}$, its presheaf topos $\Pre(\cat{C})$ is locally finitely presentable.
\end{Lemma}

\begin{proof}
By the coYoneda Lemma \ref{lem coyoneda lemma}, we know that any presheaf $X$ can be written as a colimit $X \cong \colim_{y(U) \to X} \, y(U)$. Now by Lemma \ref{lem representables are presentable}, we know that representables are finitely presentable objects in $\Pre(\cat{C})$. Unfortunately, the category $(y \downarrow X) \cong \int X$ is not filtered in general. However, by Lemma \ref{lem every object is colimit of its canonical diagram}, if we let $d : \int X \to \Pre(\cat{C})$ denote the diagram $d(x,U) = y(U)$, then $s(d) : \Sub_{\fin}(\int X) \to \Pre(\cat{C})$ is the functor defined by $s(d)(J) = \colim \, d|_J$. Now $\colim \, d|_J$ is a finite colimit of representables, so by Lemma \ref{lem colimits of presentable are presentable}, each $\colim \, d|_J$ is a finitely presentable object in $\Pre(\cat{C})$. Since $\Sub_{\fin}(\int X)$ is a directed poset, it is filtered, so $X$ can be written as a filtered colimit of finitely presentable objects. Finally by Corollary \ref{cor equiv def of locally presentable category}, this implies that $\Pre(\cat{C})$ is a locally finitely presentable category.
\end{proof}

Suppose that $\cat{C}$ is a locally $\kappa$-presentable category for some regular cardinal $\kappa$. Thus by Corollary \ref{cor equiv def of locally presentable category}, there exists a set $S$ of $\kappa$-presentable objects that generate $\cat{C}$ by $\kappa$-filtered colimits. Let $\cat{C}_0$ denote the full subcategory of $\cat{C}$ on the set $S$. For every $U \in \cat{C}$, we can consider the functor $y_0(U) : \cat{C}_0^\op \to \ncat{Set}$ that sends a $V \in \cat{C}_0$ to $\cat{C}(j(V),U)$. In other words, $y_0 : \cat{C} \to \Pre(\cat{C}_0)$ is the Yoneda embedding restricted to $\cat{C}_0$. 

The functor $y_0$ has a left adjoint defined as follows. Let $j : \cat{C}_0 \hookrightarrow \cat{C}$ denote inclusion functor. Consider the Yoneda extension of $j$, i.e. the left Kan extension
\begin{equation*}
\begin{tikzcd}
	{\cat{C}_0} & \cat{C} \\
	{\Pre(\cat{C}_0)}
	\arrow["j", hook, from=1-1, to=1-2]
	\arrow["y"', hook, from=1-1, to=2-1]
	\arrow["{\text{Lan}_{y} j}"', from=2-1, to=1-2]
\end{tikzcd}	
\end{equation*}
which exists because $\cat{C}$ is cocomplete.

Let $L \coloneqq \text{Lan}_{y} j$. Then for any $X \in \Pre(\cat{C}_0)$ we have
$$L(X) = (\text{Lan}_{y_0} j)(X) = \int^{U_0 \in \cat{C_0}} \Pre(\cat{C}_0)(y(U_0), X) \otimes j(U_0) \cong \ncolim{y(U_0) \to X} j(U_0).$$

It is easy to see that this defines a left adjoint to $y_0$. Given $X \in \Pre(\cat{C}_0)$ and $V \in \cat{C}$, we have
\begin{equation*}
\begin{aligned}
	\cat{C}(L(X), V) & \cong \cat{C}\left( \int^{U_0 \in \cat{C_0}} \Pre(\cat{C}_0)(y(U_0),X) \otimes j(U_0), V \right) \\
	& \cong \int_{U_0 \in \cat{C}_0} \cat{C}( X(U_0) \otimes j(U_0), V) \\
	& \cong \int_{U_0 \in \cat{C}_0} \ncat{Set}(X(U_0), \cat{C}(j(U_0), V)) \\
	& \cong \Pre(\cat{C}_0)(X, y_0(V)).
\end{aligned}
\end{equation*}
Thus $y_0$ is fully faithful, and has a left adjoint. It is good to note here that since $y_0$ is fully faithful, it is conservative, i.e. it reflects isomorphisms. So if $f : U \to V$ is a morphism such that $y_0(f) : y_0(U) \to y_0(V)$ is an isomorphism of presheaves, then $f$ is an isomorphism. This explicitly points out the philosophical idea that objects in $\cat{C}$ are completely determined by the objects in $\cat{C}_0$.

\begin{Lemma} \label{lem locally presentable implies refl subcat of presheaves}
For $\cat{C}$ a locally presentable category, the functor $y_0 : \cat{C} \to \Pre(\cat{C}_0)$ is fully faithful. In other words, $\cat{C}$ is a reflective localization of a presheaf topos,
\begin{equation*}
\begin{tikzcd}
	{\cat{C}} && {\Pre(\cat{C}_0)}
	\arrow[""{name=0, anchor=center, inner sep=0}, "y_0"', shift right=2, hook, from=1-1, to=1-3]
	\arrow[""{name=1, anchor=center, inner sep=0}, "L"', shift right=2, from=1-3, to=1-1]
	\arrow["\dashv"{anchor=center, rotate=-90}, draw=none, from=1, to=0]
\end{tikzcd}	
\end{equation*}
and furthermore $y_0 : \cat{C} \to \Pre(\cat{C}_0)$ preserves filtered colimits.
\end{Lemma}

\begin{proof}
Suppose that $\cat{C}$ is locally $\kappa$-presentable for some regular cardinal $\kappa$ with dense small full subcategory $\cat{C}_0 \hookrightarrow \cat{C}_\kappa$. To prove that $y_0$ is fully faithful, it is equivalent to show that the counit of the adjunction $\varepsilon: L y_0 \Rightarrow 1_{\cat{C}}$ is a natural isomorphism. For $V \in \cat{C}$, we have
$$L y_0(V) = (\text{Lan}_y j)(y_0(V)) \cong \underset{y(U) \to y_0(V)}{\colim} \, j(U) \cong \underset{L y(U) \to V}{\colim} \, j(U) \cong \underset{j(U) \to V}{\colim} \, j(U) \cong \text{Lan}_j j(V) \cong V$$
where the third isomorphism follows from Lemma \ref{lem yoneda extension identity on representables} and the last isomorphism follows from Lemma \ref{lem equiv def of density}. This shows that $\varepsilon$ is a natural isomorphism and thus that $y_0$ is fully faithful.

Now let us show that $y_0$ preserves $\kappa$-filtered colimits. Suppose that $d : I \to \cat{C}$ is a $\kappa$-filtered diagram, and $V \in \cat{C}_0$, then
\begin{equation*}
    \begin{aligned}
        \Pre(\cat{C}_0)(y(V), \ncolim{i \in I} y_0(d(i)))) & \cong \ncolim{i \in I} \Pre(\cat{C}_0)(y(V), y_0(d(i))) \\
        & \cong \ncolim{i \in I} \cat{C}(j(V), d(i)) \\
        & \cong \cat{C}(j(V), \ncolim{i \in I} d(i)) \\
        & \cong \Pre(\cat{C}_0) \left( y_0j(V), y_0 \left[ \ncolim{i \in I} d(i) \right] \right) \\
        & \cong \Pre(\cat{C}_0) \left( y(V), y_0 \left[ \ncolim{i \in I} d(i) \right] \right),
    \end{aligned} 
\end{equation*}
where the first isomorphism follows because representables are finitely presentable in $\Pre(\cat{C}_0)$ by Lemma \ref{lem representables are presentable} and therefore $\kappa$-presentable, and the third isomorphism holds because $V$ is $\kappa$-presentable. In other words, $\ncolim{i \in I} y_0(d(i))$ and $y_0(\ncolim{i \in I} d(i))$ are objectwise isomorphic as presheaves, and therefore isomorphic. Thus $y_0$ preserves $\kappa$-filtered colimits.
\end{proof}

\begin{Th} \label{th local pres iff refl subcat of presheaf topos} 
A category $\cat{C}$ is locally presentable if and only if it is a reflective subcategory of a presheaf topos such that the right adjoint preserves filtered colimits.
\end{Th}

\begin{proof}
$(\Rightarrow)$ is the content of Lemma \ref{lem locally presentable implies refl subcat of presheaves}.

$(\Leftarrow)$ Suppose that $\cat{C}$ is a reflective subcategory of a presheaf topos
\begin{equation*}
    \begin{tikzcd}
	{\cat{C}} && {\Pre(\cat{C}_0)}
	\arrow[""{name=0, anchor=center, inner sep=0}, "j"', shift right=2, hook, from=1-1, to=1-3]
	\arrow[""{name=1, anchor=center, inner sep=0}, "L"', shift right=2, from=1-3, to=1-1]
	\arrow["\dashv"{anchor=center, rotate=-90}, draw=none, from=1, to=0]
\end{tikzcd}
\end{equation*}
where $\cat{C}_0$ is a small category, and $j$ preserves $\kappa$-filtered colimits for some regular cardinal $\kappa$. We want to show that $\cat{C}$ is locally $\kappa$-presentable. 

By Proposition \ref{prop (co)limits in reflective subcategories}, $\cat{C}$ is cocomplete, so we need to show that every object $U \in \cat{C}$ can be written as a $\kappa$-filtered colimit of $\kappa$-presentable objects in $\cat{C}$.

By the coYoneda Lemma (Lemma \ref{lem coyoneda lemma}) we have
\begin{equation*}
    j(U) \cong \ncolim{y(U_0) \to j(U)} y(U_0)
\end{equation*}
But since $j$ is fully faithful, this means that the counit $\varepsilon : Lj \Rightarrow 1_{\cat{C}}$ is a natural isomorphism, so for $U \in \cat{C}$, we have
\begin{equation} \label{eq loc pres density}
    U \cong Lj(U) \cong L \left[\ncolim{y(U_0) \to j(U)} y(U_0) \right] \cong \ncolim{ y(U_0) \to j(U)} Ly(U_0) \cong \ncolim{Ly(U_0) \to U} Ly(U_0) \cong (\text{Lan}_{Ly} Ly)(U).
\end{equation}
Thus the functor $Ly : \cat{C}_0 \to \cat{C}$ is dense, and if $\cat{C}'_0$ denotes the essential image of $Ly$, then $\cat{C}'_0 \hookrightarrow \cat{C}$ is dense in $\cat{C}$ and essentially small. Thus every object $U \in \cat{C}$ can be written as a $\kappa$-filtered colimit (by Lemma \ref{lem every object is colimit of its canonical diagram}) of objects of the form $Ly(U_0)$ for $U_0 \in \cat{C}_0$.

Let us show that each $Ly(U_0)$ for $U_0 \in \cat{C}_0$ is $\kappa$-presentable in $\cat{C}$. Let $d : I \to \cat{C}$ be a $\kappa$-filtered diagram, then
\begin{equation*}
    \begin{aligned}
        \cat{C}(Ly(U_0), \ncolim{i \in I} d(i)) & \cong \Pre(\cat{C}_0) \left( y(U_0), j \left[ \ncolim{i \in I} d(i) \right] \right) \\
        & \cong \Pre(\cat{C}_0)(y(U_0), \ncolim{i \in I} j(d(i))) \\
        & \cong \left[ \ncolim{i \in I} j(d(i)) \right](U_0) \\
        & \cong \ncolim{i \in I} \left[ j(d(i))(U_0) \right] \\
        & \cong \ncolim{i \in I} \Pre(\cat{C}_0)(y(U_0), j(d(i))) \\
        & \cong \ncolim{i \in I} \cat{C}(Ly(U_0), d(i)),
    \end{aligned}
\end{equation*}
where the second isomorphism holds because $j$ was assumed to preserve $\kappa$-filtered colimits.

Thus every object in $\cat{C}$ can be written as a $\kappa$-filtered colimit of a set $\{ Ly(U_0) \}_{U_0 \in \cat{C}_0}$ of $\kappa$-presentable objects, and thus by Corollary \ref{cor equiv def of locally presentable category}, this proves that $\cat{C}$ is locally $\kappa$-presentable.
\end{proof}

Now we turn to one of the most important applications of the theory of locally presentable categories.

\begin{Th} \label{th adjoint functor thm for pres cats}
Let $\cat{C}$ and $\cat{D}$ be locally presentable categories, and let $F: \cat{C} \to \cat{D}$ be a functor that preserves small colimits. Then $F$ has a right adjoint.
\end{Th}

Before we prove this result, let us prove it in the case where the domain is a presheaf topos. 

\begin{Lemma} \label{lem adjoint functor thm for presheaf topoi}
Let $F: \Pre(\cat{C}) \to \cat{D}$ be a functor where $\cat{D}$ is a locally small and cocomplete category. If $F$ preserves small colimits, then $F$ has a right adjoint.
\end{Lemma}

\begin{proof}
Let $G: \cat{D} \to \Pre(\cat{C})$ be the functor defined objectwise for $U \in \cat{C}$ and $Y \in \cat{D}$ by
\begin{equation*}
   G(Y)(U) = \cat{D}(F(y(U)), Y). 
\end{equation*}
Let us show that this defines a right adjoint to $F$. Let $X \in \Pre(\cat{C})$, then
\begin{equation*}
    \begin{aligned}
        \cat{D}(F(X), Y) & \cong \cat{D}\left( F \left( \ncolim{y(U) \to X} y(U) \right), Y \right) \\
        & \cong \cat{D} \left( \ncolim{y(U) \to X} F(y(U)), Y \right) \\
        & \cong \lim_{y(U) \to X} \cat{D}(F(y(U)), Y) \\
        & \cong \lim_{y(U) \to X} G(Y)(U) \\
        & \cong \lim_{y(U) \to X} \Pre(\cat{C})(y(U), G(Y)) \\
        & \cong \Pre(\cat{C})\left( \ncolim{y(U) \to X} y(U), G(Y) \right) \\
        & \cong \Pre(\cat{C})(X, G(Y)).
    \end{aligned} 
\end{equation*}
where the first and last isomorphisms follow from the coYoneda Lemma \ref{lem coyoneda lemma}, and the second isomorphism follows because $F$ was assumed to preserve colimits. 
\end{proof}

\begin{proof}[Proof of Theorem \ref{th adjoint functor thm for pres cats}]
Suppose that $\cat{C}$ and $\cat{D}$ are locally presentable categories and $F: \cat{C} \to \cat{D}$ is a cocontinuous functor. Since $\cat{C}$ is locally presentable, by Theorem \ref{th local pres iff refl subcat of presheaf topos} there is a small full subcategory $\cat{C}_0 \hookrightarrow \cat{C}$ such that $\cat{C}$ is a reflective localization of its presheaf topos $j: \cat{C} \hookrightarrow \Pre(\cat{C}_0)$, where $j = y_0$ is the restricted Yoneda embedding, and with reflector $L : \Pre(\cat{C}_0) \to \cat{C}$. The functor $FL : \Pre(\cat{C}_0) \to \cat{D}$ preserves small colimits, so by Lemma \ref{lem adjoint functor thm for presheaf topoi}, it has a right adjoint $G' : \cat{D} \to \Pre(\cat{C}_0)$.

Now by Proposition \ref{prop local objects equiv to localization}, $\cat{C}$ is also equivalent to the category of $W$-local objects in $\Pre(\cat{C}_0)$, where $W = L^{-1}(\text{iso})$. Let us show that for any $V \in \cat{D}$, that $G'(V)$ is an $W$-local object.

Let $f: X \to Y$ be a map of presheaves over $\cat{C}_0$ that is a $W$-local equivalence (Definition \ref{def W local object and local equivalence}), by Lemma \ref{lem W-local objects and equivalences of reflective localizations}, this is equivalent to being a map inverted by $L$. So suppose that $f$ is inverted by $L$. We want to show that the map
\begin{equation*}
    \Pre(\cat{C}_0)(Y, G'(V)) \xrightarrow{\Pre(\cat{C}_0)(f, G'(V))} \Pre(\cat{C}_0)(X, G'(V))
\end{equation*}
is an isomorphism. But since $G'$ is a right adjoint, we have the following commutative diagram
\begin{equation*}
    \begin{tikzcd}
	{\Pre(\cat{C}_0)(Y,G'(V))} && {\Pre(\cat{C}_0)(X,G'(V))} \\
	{\cat{D}(FL(Y),V)} && {\cat{D}(FL(X),V)}
	\arrow["{\Pre(\cat{C}_0)(f, G'(V))}", from=1-1, to=1-3]
	\arrow["\cong"', from=1-1, to=2-1]
	\arrow["\cong", from=1-3, to=2-3]
	\arrow["{\cat{D}(FL(f),V)}"', from=2-1, to=2-3]
\end{tikzcd}
\end{equation*}
But $f$ is inverted by $L$, thus the bottom horizontal map is an isomorphism, so the top horizontal map is also an isomorphism. Thus $G'(V)$ is $W$-local for every $V \in \cat{D}$. Thus there exists an essentially unique $G(V) \in \cat{C}$ such that $jG(V) \cong G'(V)$. This can be shown to define a functor $G : \cat{D} \to \cat{C}$ such that $j G \cong G'$, or equivalently $G \cong L j G \cong L G'$.

Then $G : \cat{D} \to \cat{C}$ is a right adjoint to $F$, as we have
\begin{equation*}
    \begin{aligned}
        \cat{D}(F(U), V) & \cong \cat{D}\left( F \left( \ncolim{Ly(U_0) \to U} Ly(U_0) \right), V \right) \\
        & \cong \lim_{Ly(U_0) \to U} \cat{D}(FL(y(U_0)), V) \\
        & \cong \lim_{Ly(U_0) \to U} \Pre(\cat{C}_0)(y(U_0), G'(V)) \\
        & \cong \lim_{Ly(U_0) \to U} \cat{C}(U_0, G(V)) \\
        & \cong \cat{C}\left(\ncolim{Ly(U_0) \to U} U_0, G(V)\right) \\
        & \cong \cat{C}\left(\ncolim{U_0 \to U} U_0, G(V)\right) \\
        & \cong \cat{C}(U, G(V)).
    \end{aligned}
\end{equation*}
The first isomorphism holds since if $U_0 \in \cat{C}_0$, then $y(U_0) = j(U_0)$ and thus $Ly(U_0) \cong U_0$, and $\cat{C}_0$ is assumed to be dense in $\cat{C}$. The fourth isomorphism holds because $j$ is fully faithful, with $j(U_0) = y(U_0)$ and $G'(V) \cong jG(V)$.
\end{proof}

In fact, the above result can be strengthened to the following result.

\begin{Th}[{\cite[Theorem 1.66]{rosicky1994locally}}]
Let $F : \cat{C} \to \cat{D}$ be a functor between locally presentable categories. Then
\begin{enumerate}
    \item if $F$ preserves colimits, then $F$ has a right adjoint, and
    \item if $F$ preserves limits and filtered colimits, then $F$ has a left adjoint.
\end{enumerate}
\end{Th}

\begin{Rem}
Requiring $F$ to preserve both limits and filtered colimits in order to have a left adjoint is necessary. A counterexample in the case where $F$ only preserves limits is given in \cite[Section 1]{adamek2001large}.
\end{Rem}

\printbibliography
\end{document}